\documentclass[letterpaper,11pt,final ]{amsart}
\usepackage[utf8]{inputenc}
\usepackage{amsmath, amsthm, amssymb, enumerate, scalerel, tikz,  color, accents,
  fancybox, rotating, comment, colortbl, stmaryrd, hyperref,  cleveref, setspace, tikz-3dplot}
\title{Symmetries of equivariant Khovanov--Rozansky homology}  
\author{You Qi}
\address{Department of Mathematics, University of Virginia,
  Charlottesville, VA 22904, USA}
\email{\href{mailto:yq2dw@virginia.edu}{yq2dw@virginia.edu}}
\author{Louis-Hadrien Robert}
 \address{Université Clermont Auvergne, LMBP, Campus des Cézeaux, 3 place Vasarely, TSA 60026, CS 60026, 63178 Aubière Cedex, France}
 \email{\href{mailto:louis\_hadrien.robert@uca.fr}{louis\_hadrien.robert@uca.fr}}
 \author{Joshua Sussan}
 \address{ Mathematics Program\\
 The Graduate Center, CUNY, New York, NY 10016, USA} 
 \address{Department of Mathematics, CUNY Medgar Evers, Brooklyn, NY,
   11225, USA} 
  \email{\href{mailto:jsussan@mec.cuny.ed}{jsussan@mec.cuny.edu}}
 \author{Emmanuel Wagner}
 \address{Univ Paris Cit\'e, IMJ-PRG, Univ Paris Sorbonne, UMR 7586 CNRS,
   F-75013, Paris, France} 
 \email{\href{mailto:emmanuel.wagner@imj-prg.fr}{emmanuel.wagner@imj-prg.fr}}
\makeatletter
\@namedef{subjclassname@2020}{\textup{2020} Mathematics Subject Classification}
\makeatother
 \hypersetup{
    pdftoolbar=true,        
    pdfmenubar=true,        
    pdffitwindow=false,     
    pdfstartview={FitH},    
    pdftitle={\shorttitle},    
    pdfauthor={\shortauthors},     
    pdfsubject={\shorttitle},   
    pdfcreator={\shortauthors},   
    pdfproducer={\shortauthors}, 
    pdfkeywords={}, 
    pdfnewwindow=true,      
    colorlinks=true,       
    linkcolor=darkgray,          
    citecolor=teal,        
    filecolor=magenta,      
    urlcolor=violet,          
    linkbordercolor=red,
    citebordercolor=teal,
    urlbordercolor=violet,  
    linktocpage=true
    }
\usetikzlibrary{arrows, decorations.markings, decorations, patterns,
positioning, decorations.pathreplacing, decorations.pathmorphing,
intersections, fit}
\tikzset{->-/.style={decoration={markings, mark=at position .5 with {\arrow{>}}},postaction={decorate}}}
\tikzset{-<-/.style={decoration={markings, mark=at position .5 with {\arrow{<}}},postaction={decorate}}}
\let\oldtocsubsection\tocsubsection
\renewcommand\tocsubsection[3]{\hspace{0.5cm}\oldtocsubsection{#1}{#2}{#3}}
\let\oldtocsubsubsection\tocsubsubsection
\renewcommand\tocsubsubsection[3]{\hspace{1cm}\oldtocsubsubsection{#1}{#2}{#3}}

\newcounter{res}[section]
\numberwithin{res}{section}
\newtheorem{thm}[res]{Theorem}
\newtheorem*{theo}{Theorem}

\newtheorem{lem}[res]{Lemma}
\newtheorem{lem-dfn}[res]{Lemma-Definition}

\newtheorem{prop}[res]{Proposition}
\newtheorem{cor}[res]{Corollary}

\theoremstyle{definition}
\newtheorem{notation}[res]{Notation}
\newtheorem{dfn}[res]{Definition}
\newtheorem{rmk}[res]{Remark}
\newtheorem{exa}[res]{Example}

\newtheorem{conv}[res]{Convention}
\def\co{\colon\thinspace}

\newcommand{\NB}[1]{\ensuremath{\vcenter{\hbox{#1}}}}

\newcommand{\NN}{\ensuremath{\mathbb{N}}}
\newcommand{\ZZ}{\ensuremath{\mathbb{Z}}}

\newcommand{\RR}{\ensuremath{\mathbb{R}}}

 \newcommand
{\Id}{\operatorname{Id}} \newcommand{\id}{\mathrm{Id}}

\newcommand{\Hom}{\ensuremath{\mathrm{Hom}}}

\newcommand{\HOM}{\ensuremath{\mathrm{HOM}}}

\newcommand{\gll}{\ensuremath{\mathfrak{gl}}}
\newcommand{\sll}{\ensuremath{\mathfrak{sl}}}

\renewcommand{\deg}[2][{}]{\ensuremath{\mathrm{deg}_{#1}(#2)}}

\newcommand{\Links}{\ensuremath{\mathsf{Links}}}
\newcommand{\Linkscon}{\ensuremath{\mathsf{Links}_{Con}}}

\newcommand{\qbinom}[2]{\ensuremath
\begin{bmatrix}
  #1 \\
  #2
\end{bmatrix}}

 \newcommand{\Sym}{\ensuremath{\mathrm{Sym}}}

\newcommand{\gFm}{\ensuremath{\mathsf{gFoam}}}
\newcommand{\ch}{\ensuremath{\mathrm{Ch}}}
\newcommand{\MGH}{\ensuremath{\mathrm{MGH}}}
\newcommand{\MGS}{\ensuremath{\mathrm{MGS}}}
\newcommand{\MGO}{\ensuremath{\mathrm{MG1}}}
\newcommand{\MGTW}{\ensuremath{\mathrm{MG2}}}
\newcommand{\MGTH}{\ensuremath{\mathrm{MG3}}}

\newcommand{\scalars}{\ensuremath{\Bbbk}}
\newcommand{\de}{\ensuremath{\mathbf{e}}}
\newcommand{\df}{\ensuremath{\mathbf{f}}}
\renewcommand{\dh}{\ensuremath{\mathbf{h}}}
\newcommand{\Le}{\ensuremath{\mathsf{e}}}
\newcommand{\Lf}{\ensuremath{\mathsf{f}}}
\newcommand{\Lh}{\ensuremath{\mathsf{h}}}

\newcommand{\dotnewtoni}[1][i]{\ensuremath{\textcolor[rgb]{0,0,0.6}{\spadesuit_{#1}}}}
\newcommand{\wdotnewtoni}[1][i]{\ensuremath{\textcolor[rgb]{0,0,0.6}{\widehat{\spadesuit}_{#1}}}}

\newcommand{\NNN}{\ensuremath{\mathbb{N}_{-1}}}
\newcommand{\bracketN}[1]{\left\langle #1 \right\rangle_{\myN}}
\newcommand{\bracket}[1]{\left\langle #1 \right\rangle}

\newcommand{\KN}{\ensuremath{\Bbbk_\myN}}

\newcommand{\mymovie}[3][]{
  \NB{
    \begin{tikzpicture}[#1]
\begin{scope}
  \draw[gray, thick] (0, -0.05) -- +(0, 3.1);
  \draw[gray, thick] (4, -0.05) -- +(0, 3.1);
  \draw[gray, thick] (8, -0.05) -- +(0, 3.1);
  \draw[gray, line width=1mm] (0,0) -- +(8,0);
  \draw[white, densely dotted, line width=0.6mm] (0,0) -- +(8,0);
  \draw[gray, line width=1mm] (0,3) -- +(8,0);
  \draw[white, densely dotted, line width=0.6mm] (0,3) -- +(8,0);
  \node (Frame1) at (2, 1.5) {#2};
  \node (Frame1) at (6, 1.5) {#3};
\end{scope}
    \end{tikzpicture}
    }
}

\newcommand{\tone}{{t_1}}
\newcommand{\ttwo}{{t_2}}

\newcommand{\dif}{\ensuremath{\partial}}
\newcommand{\Fp}{\ensuremath{\mathbb{F}_p}}
\newcommand{\tqftfunc}[1][]{\ensuremath{\mathcal{F}_\myN^{#1}}}
\newcommand{\statespaceN}[2][]{\ensuremath{\tqftfunc[#1]\left(#2\right)}}

\newcommand{\RN}{\ensuremath{\mathbb{Z}_\myN}}

\newcommand{\myN}{\ensuremath{N}}
\newcommand{\web}{\ensuremath{\Gamma}}
\newcommand{\foam}{\ensuremath{F}}
\newcommand{\degN}[1]{\ensuremath{\mathrm{deg}_\myN\left(#1\right)}}
\newcommand{\facet}{\ensuremath{f}}
\newcommand{\surface}{\ensuremath{\Sigma}}

\newcommand{\foamcat}[1][]{\ensuremath{\mathsf{Foam}_{#1}}}

\newcommand{\vectweb}[1]{\ensuremath{V\left(#1\right)}}

\def\mc{\mathcal}

\def\lra{{\longrightarrow}}
\def\dmod{{\mathrm{\mbox{\textrm{-}}mod}}}  
\newcommand{\KR}{\mathrm{KR}}
\newcommand{\mapX}{\ensuremath{\chi}}
\newcommand{\mapH}{\ensuremath{\eta}}
\newcommand{\mapB}{\ensuremath{\upsilon}}
\newcommand{\mapU}{\ensuremath{\zeta}}
\newcommand{\mapA}{\ensuremath{\alpha}}

\newcommand{\mapI}{\ensuremath{\iota}}
\newcommand{\mapS}{\ensuremath{\sigma}}
\newcommand{\mapC}{\ensuremath{\kappa}}
\newcommand{\Com}{\ensuremath{\mathrm{Com}}}

\def\gmod{{\mathrm{\mbox{-}gmod}}}
\def\mod{{\mathrm{\mbox{-}mod}}}

\newcommand{\gdot}{\ensuremath{\NB{\tikz[thin, green!50!black]{\draw (0,0) circle (0.5mm);}}}}

\newcommand{\gsoliddot}{\ensuremath{\NB{\tikz[thin, green!50!black]{\filldraw[draw= green!50!black, fill = green] (0,0) circle (0.5mm);}}}}

 \newcommand{\imagesfolder}{.}

\allowdisplaybreaks
\begin{document}
\begin{abstract}
We construct an $\mathfrak{sl}_2$-action on equivariant
  $\mathfrak{gl}_N$-link homologies. As a consequence we obtain an action of $\sll_2$ on these homologies as well as a $p$-DG structures for $p$ a prime number.
  We explore topological applications of these structures. 
\end{abstract}
\maketitle
\setcounter{tocdepth}{1}
\tableofcontents

\section{Introduction}
\label{sec:intro}
Symmetries on the homology of a link have been constructed in various
contexts.  A starting point for this work was an action of a
subalgebra of the graded Witt algebra on HOMFLYPT homology given by Khovanov
and Rozansky \cite{KRWitt}.  They were partially motivated by work of
Lipschitz and Sarkar \cite{LS1, LS2} who constructed an action of the
Steenrod algebra on Khovanov homology, as well observations by Gorsky,
Oblomkov, and Rasmussen \cite{GOR2} who conjectured that some colored
link homologies have graded dimensions equal to characters of
representations of affine Lie algebras.  Another symmetry of HOMFLYPT
homology was studied by Gorsky, Hogancamp, and Mellit \cite{GHM}.
They found an $\mathfrak{sl}_2$-action which seems distinct from the
symmetries studied in \cite{KRWitt}.  In particular, the
$\mathfrak{sl}_2$-action in \cite{GHM} acts in the homological
direction as opposed to the Khovanov--Rozansky action which preserves
homological degree.  This was applied in \cite{ChGo} to study some
structural properties of $\mathfrak{gl}_N$-homology.  In another
direction, Grigsby, Licata, and Wehrli \cite{GLW} constructed an
$\mathfrak{sl}_2$-action on annular Khovanov homology.  In this work,
we construct an action of a graded version of $\mathfrak{sl}_2$ on (equivariant)
Khovanov--Rozansky $\mathfrak{gl}_N$-homologies.  It would be
interesting to find an analogous $\mathfrak{sl}_2$-action in the
annular case developed by Akhmechet and Khovanov \cite{AkhKh} and
compare it to the construction of \cite{GLW}.

The construction in \cite{KRWitt} has been utilized and extended in
various directions.  Over a field of characteristic $p$, the action of
the degree-2 Witt generator $L_1$ has been used to define a $p$-DG structure on
HOMFLYPT homology \cite{KR2, Khovtriply} and
$\mathfrak{gl}_{-2}$-homology \cite{Cautisremarks}, \cite{QRS},
\cite{RW2}, leading to categorifications of the Jones and colored
Jones polynomial at a root of unity \cite{QiSussanLink, QRSW1}. Work
of Wang \cite{Wang} uses the degree-($-2$) nilpotent part of the $\mathfrak{sl}_2$-action (that we
extend here) to deduce various structural properties of
Khovanov--Rozansky $\mathfrak{gl}_N$-link homologies \cite{KR1}.
Actions of $\mathfrak{sl}_2$ have been discovered by Elias and one of
the authors on various categorical structures \cite{EliasQisl2}.  It
was earlier observed by Khovanov that there is a Witt-type action on
KLR algebras \cite{KhovNote}.

The first connection between foams and link homology was Khovanov's
categorification of the $\mathfrak{sl}_3$-link invariant \cite{Khsl3}.
This was extended to the $\mathfrak{gl}_N$ case by Mackaay,
Sto{\v{s}}i{\'c}, and Vaz \cite{MSV} who used the Kapustin--Li formula
\cite{KapLi, KRLG} to evaluate closed foams.  Categorified quantum
groups play a role in link homology via categorified skew Howe duality
and was first explored by Cautis, Kamnitzer, and Licata
\cite{CautisKamnitzerLicataSkewHowe} and further developed in several
papers such as \cite{CautisClasp}.  Lauda, Queffelec, and Rose showed
that these constructions factor through a foam category \cite{LQR,
  QR1} extending earlier work of Mackaay, Pan, and Tubbenhauer
\cite{Mac, MPT}.  The relations on the foam category imposed in
\cite{QR1} are shown to be a consequence of the foam evaluation
formula defined in \cite{RW1}.

We now summarize the main results of this work.
\begin{theo}
  \begin{itemize}  
  \item The equivariant Khovanov--Rozansky $\mathfrak{gl}_N$-homology
    of a link $L$ is a graded $\mathfrak{sl}_2$-module.
  \item This action is functorial in the sense that there is a functor
    from the category of links and cobordisms to a foam category
    enriched in graded $\mathfrak{sl}_2$-modules.
  \item The Rasmussen invariant is equal to the highest weight of a
    certain quotient representation.
\item If $C$ is a ribbon concordance from $L_0$ to $L_1$, then
  the equivariant homology of $L_0$ is a direct summand of
the equivariant homology of $L_1$ as $\sll_2$-modules.
  \item Over a field of characteristic $p$, certain homogeneous nilpotent generators of
    $\mathfrak{sl}_2$ give rise to $p$-DG structures on link homology.
\end{itemize}
\end{theo}

While we only consider the uncolored case (links are colored only by
the natural representation), we expect that this construction extends
to the setting where link components are colored by exterior powers of
the fundamental representation.  Furthermore, we anticipate that the
$\mathfrak{sl}_2$-action really comes from the Witt-type action on
foams constructed in \cite{QRSW2}. Note that in \cite{QRSW2}, we
introduced a smaller class of foams, called \emph{spherical foams} for
which the Witt- and $\mathfrak{sl}_2$-actions have a bit more
flexibility. We expect this flexibility to remain if working with
spherical foams only, which could be achieved if one only works with
braid closures instead of arbitrary link diagrams.

\subsection{Outline}
\label{sec:outline}
We begin in Section \ref{homological:sec} with some background
material on the homological algebra needed later in the paper.

Section \ref{sl2foams:sec} contains a review of the
$\mathfrak{sl}_2$-action on foams constructed in \cite{QRSW2}.
This action is shown to extend to homology of links in Section
\ref{homology:sec}.

The compatibility of the $\mathfrak{sl}_2$-action with functoriality
of link homology is explained in Section \ref{functoriality:sec}.  Our
argument is a direct application of functoriality proved by Ehrig,
Tubbenhauer, and Wedrich \cite{ETW}.

Section \ref{sec:special} contains various specializations of the
$\mathfrak{sl}_2$-action on link homology.  Applying the Zuckerman
functor (a certain Lie theoretic construction), we obtain a locally
finite submodule of equivariant link homology which has the potential
of being finite-dimensional.  In another direction, over a field of
characteristic $p$, we obtain a $p$-DG structure on link homology.  In
fact, we have two different such constructions: one coming from the
action of $\df$ and one coming from the action of $\de$.  Using the
$\df$ action, we must use symmetric polynomials as the ground ring in
order to have a non-trivial action.  We are not able to make any
statements about the Euler characteristic.  However, using the $\de$
action, we may take the ground ring just to be a field of
characteristic $p$ and we categorify the $\mathfrak{gl}_N$-link
invariant at a root of unity. We also relate our construction to the
Rasmussen invariant and to recent work by Wang \cite{Wang}.

\subsection{Conventions}
\label{sec:conventions}
Let $\NN$ stand for the set of non-negative
integers and set
$\NNN= \{ k\in \ZZ, k\geq -1\}$. 

For a ring with unity $\scalars$, and for $x \in \scalars$, we set $\bar{x}=1-x$.

Throughout most of the paper, we will fix a natural number $N$. The
algebras $\RN=\ZZ[X_1, \dots,X_\myN]^{S_\myN}$ and
$\KN = \scalars[X_1, \dots, X_\myN]^{S_\myN}$ of symmetric polynomials
will play central roles in this paper. They are non-negatively graded
by imposing that $\deg{X_i} =2$. The $i$th elementary, complete
homogeneous, and power sum symmetric polynomials in
$X_1,\dots, X_\myN$ are denoted by $E_i$, $H_i$ and $P_i$
respectively, so that
\[
  \RN=\ZZ[E_1, \dots, E_\myN] \qquad \text{and} \qquad \KN=
  \scalars[E_1, \dots, E_\myN].
\]
Throughout most of this paper we invert $2$ in the ground ring.

For $a \in \NN$, $\Sym_{a}$ denotes the ring of symmetric polynomials
in $a$ variables with $\ZZ$ coefficients, in particular
$\RN=\Sym_\myN$. When working in such a ring, we will
let $e_i, h_i$ and $p_i$ be the $i$th elementary, complete
homogeneous, and power sum symmetric polynomials respectively without
reference to the variables.
The ring $\Sym_{a}$ is graded by imposing that $e_i$ is homogeneous of
degree $2i$. With this setting, we have:
\[ \deg{e_i} = \deg{h_i} = \deg {p_i} = 2i.\]

For $n \in \ZZ$, let $[n]=\frac{q^n-q^{-n}}{q-q^{-1}}$, for $k \in
\NN$, we let $[k]!= \prod_{j=1}^k[j]$. Finally, for $m\in \ZZ$ and $a
\in \NN$,  define: 
\[
  \qbinom{m}{a}=\prod_{i=1}^a \frac{[m+1-i]}{[i]}.\]
Note that if $m$ is non-negative, one has $\qbinom{m}{a}= \frac{[m]!}{[a]![m-a]!}$.

For a $\mathbb{Z}$-graded vector space $V$, let $V_i$ denote the
subspace in degree $i$. Let $q^n V$ denote the $\mathbb{Z}$-graded
vector space where $(q^n V)_i=V_{i-n}$.

Complexes will be taken to be cohomologically graded.  For a complex
$C$, we let $t^i C$ denote the shifted complex whose piece in
cohomological degree $i+j$ is the piece of $C$ in cohomological degree
$j$, in other words, $(t^j C)_i = C_{i-j}$.

Foams are read from bottom to top.

\subsection{Acknowledgments}
\label{sec:acknoledgment}
We would like to thank Rostislav Akhmechet, Ben Elias, Lev Rozansky, Pedro Vaz, and Joshua
Wang for interesting and enlightening conversations. Let us thank Joshua Wang again for spotting a mistake in a previous version of this paper. We especially
would like to thank Mikhail Khovanov who encouraged us to search for
algebra actions on foams and link homologies.  We are grateful to
Felix Roz for helpful comments on an earlier draft of the paper.

Various ideas motivating this paper partially arose
during the hybrid workhshop ``Foam Evaluation'' held at ICERM. The Universit\'e Paris Cit\'e  and its Programme d'invitations internationales scientifiques further supported the collaboration of the authors. We thank the institutions for their hospitality.

Some figures are recycled from papers of various subsets of the authors
with or without other collaborators.

 Y.Q.{} is partially supported by the Simons Foundation Collaboration Grants for Mathematicians. J.S.{} is
partially supported by the NSF grant DMS-1807161 and PSC CUNY Award
64012-00 52.
LH.R.{} was supported by the Luxembourg National Research Fund PRIDE17/1224660/GPS.
E.W.{} is partially supported by the ANR projects AlMaRe
(ANR-19-CE40-0001-01), AHA (JCJC ANR-18-CE40-0001) and CHARMES (ANR-19-CE40-0017).
 
\section{Homological background}
\label{homological:sec}

\label{sec:homol-non-sense}

Throughout this work, we will study $H$-modules, where $H$ is some
cocommutative Hopf algebra. We will freely make use of Sweedler's
notation $\Delta(h)=\sum_h h_1\otimes h_2$ and $(\Delta\otimes
\id_H)\circ \Delta(h)=\sum_h h_1\otimes h_2\otimes h_3$, etc.  In the
main results of this paper, the Hopf algebra $H$ will be the universal enveloping algebra of $\mathfrak{sl}_2$.

Let $A$ be an $H$-module algebra.  This means that $A$ is an algebra
object in the module category of the Hopf algebra $H$. In particular,
for $h$ in $H$ and $a$ and $b$ in $A$, one has:
\begin{equation}
h\cdot (ab)=\sum_h (h_1\cdot a)(h_2\cdot b). 
\end{equation}
We may then form the \emph{smash product algebra}  $A\# H$.
As an abelian group, $A\# H$ is isomorphic to
$A\otimes H$.  The multiplicative structure is determined by
\begin{equation}
  (a\otimes h)(b\otimes k)=\sum_h a(h_1\cdot b)\otimes h_2 k,
\end{equation}
for any $a,b\in A$ and $h,k\in H$.
Thus, for an $H$-module algebra $A$,  $A\otimes 1$ and $1\otimes H$ sit in
$A\# H$ as subalgebras by construction. 
We will often refer to modules and morphisms
in $A\#H\dmod$ 
as \emph{$H$-equivariant} $A$-modules and morphisms.

There is an exact forgetful functor between the usual homotopy
categories of chain complexes of graded modules
\begin{equation}\label{eqn-forgetful-functor}
  \mathrm{For}: \mc{C}(A\# H)\lra \mc{C}(A).
\end{equation}
An object $K_\bullet $ in $\mc{C}(A\# H)$ is annihilated by the forgetful functor if and only if, when forgetting the $H$-module structure
on each term of $K_\bullet$, the complex of graded $A$-modules
$\mathrm{For}(K_\bullet)$ is null-homotopic. The null-homotopy map on
$\mathrm{For}(K_\bullet)$ though, is not required to intertwine
$H$-actions. We will refer to such complexes as \emph{relatively null-homotopic}. Also, for ease of notation, we will usually just abbreviate the complex notation $K_\bullet$ as $K$ when no confusion can be caused.

Since $H$ is a cocommutative Hopf algebra, if $B$ is an $H$-module algebra, then so is its opposite algebra $B^{\mathrm{op}}$. Thus given two $H$-module algebras $A$ and $B$, then their tensor product $A\otimes B^{\mathrm{op}}$ is an $H$-module algebra with its natural $H$-module structure. Following the
definition given above, the multiplication on $A\otimes B^{\mathrm{op}}$ is given by:
\[
  (a_1\otimes b_1 \otimes h )(a_2\otimes b_2 \otimes k) = \sum_h a_1 (h_1\cdot a_2)
  \otimes (h_2\cdot b_1)b_2\otimes h_3k,
\]
for any $a_1,a_2\in A$, $b_1,b_2 \in B$ and $h,k\in H$. A complex in $\mc{C}((A\otimes B^{\mathrm{op}})\# H)$ will also be referred to as a complex of $H$-equivariant $(A,B)$-bimodules, while morphisms between such complexes as $H$-equivariant bimodules homomorphisms.

\begin{dfn}\label{def-relative-homotopy-category}
Given an $H$-module algebra $A$,
  the \emph{relative homotopy
  category} is the Verdier quotient
  \[\mc{C}^H(A):=\dfrac{\mc{C}(A\#
    H)}{\mathrm{Ker}(\mathrm{For})}.\]
\end{dfn}
The superscript $H$ in the definition is there to remind the reader of the
$H$-module structures on the objects.

The usual homology functor factors through the kernel of $\mathrm{For}$, and induces well-defined homology functors on the relative homotopy category. The homology of objects in $\mc{C}^H(A)$ carries natural actions by the smash product algebra $A\# H$.

When $A=\Bbbk$ is a field, then the functor $\mathrm{For}:\mc{C}(H)\lra \mc{C}(\Bbbk)$ agrees with the usual total homology functor, since a complex of $H$-modules is acyclic if and only if it is null-homotopic over the ground field. It follows that, in this case $\mc{C}^H(\Bbbk)=\mc{D}(H)$, the usual derived category of $H$-modules. In this paper, we will consider ground rings that are more general than fields.

The category $\mc{C}^H(A)$ is triangulated. By construction,
there is a factorization of the forgetful functor
\begin{gather}
  \NB{
    \tikz[xscale=3, yscale =0.8]{
      \node (AH) at (-1, 1) {$\mc{C}(A\# H)$};
      \node (A) at ( 1, 1) {$\mc{C}(A)$};
      \node (Adif) at ( 0, -1) {$\mc{C}^H(A)$};
      \draw[-to] (AH) -- (A) node[pos =0.5, above] {$\mathrm{For}$};
      \draw[-to] (AH) -- (Adif);
      \draw[-to] (Adif) -- (A);      
    }
  }  \ .
\end{gather}

Let us briefly comment on the triangulated structure of the relative homotopy category $\mc{C}^H(A)$. By construction, the homological shift functors, denoted $t^i$ for $i\in \mathbb{Z}$, are inherited from the shift functors on $\mc{C}(A\# H)$, which shift complexes $i$ steps to the left if $i\geq 0$, and $-i$ steps to the right if $i<0$. 

For the usual homotopy category $\mc{C}(A)$ of an algebra, standard distinguished triangles arise from short exact sequences
  \[
 0 \lra M \stackrel{f}{\lra} N\stackrel{g}{\lra} L \lra 0
  \]
of $A$-modules that are termwise split exact. The class of distinguished triangles in $\mc{C}(A)$ are declared to be those that are isomorphic to standard ones.
For distinguished triangles in the relative homotopy category, similarly, we have the following construction.

\begin{lem} \cite[Lemma 2.3]{QRSW1} \label{lem-construction-of-triangle}
  A short exact sequence of chain complexes of $A\#H$-modules
  \[
 0 \lra M \stackrel{f}{\lra} N \stackrel{g}{\lra} L \lra 0
  \]
 that is termwise $A$-split exact gives rise to a distinguished triangle in $\mc{C}^H(A)$. Conversely, any distinguished triangle in $\mc{C}^H(A)$ is isomorphic to one that arises in this form.
\end{lem}

A morphism of $A\# H$-modules $f:M\lra N$ becomes an isomorphism in $\mc{C}^H(A)$ if and only if its cone $\mathrm{C}(f)$ in $\mc{C}(A\# H)$ is relatively null-homotopic. To see this, replace $N$ by the usual mapping cylinder $\mathrm{Cyl}(f)= N\oplus M\oplus t M$, which is homotpic to $N$ in $\mc{C}(A\# H)$. The inclusion of $M$ into $\mathrm{Cyl}(f)$ is a homotopy equivalence. Furthermore, we have a short exact sequence of $A\# H$-modules
\begin{equation}
    0 \lra M \stackrel{\iota}{\lra} \mathrm{Cyl}(f) \lra \mathrm{C}(f) \lra 0,
\end{equation}
where 
$\iota: M \lra \mathrm{Cyl}(f)$ is the map $m\mapsto (f(m), m, 0)$. The sequence is termwise $A$-split (even $A\# H$-split), and thus gives rise to a distinguished triangle in $\mc{C}^H(A)$. Thus $\iota: M \lra \mathrm{Cyl}(f)$ descends to an isomorphism in $\mc{C}^H(A)$ if and only if $\mathrm{C}(f)$ is annihilated under $\mathrm{For}$, which is equivalent to saying that it is relatively null-homotopic.

Our main goal in this section is to describe an enriched structure on the relative homotopy category of an $H$-module algebra. The enriched hom spaces will naturally carry structures of $H$-representations. As usual, the enriched hom will be naturally right adjoint to tensor product of chain complexes.

Let $A$ be an $H$-module algebra, and $M$, $N$ be two (bounded) complexes of left $A\# H$-modules. 
We set $\HOM_A(M,N)$ to be the usual space of all $A$-linear homomorphisms
\begin{equation}
    \HOM_A(M,N)=\oplus_{i\in \ZZ} \Hom_A(M,t^i N),
\end{equation}
equipped with the natural differential
\begin{equation}
(df)(m)=d(f(m))-(-1)^if(dm)
\end{equation}
for any $f\in \Hom_A(M, t^iN)$.
Then the Hopf algebra $H$ acts on this morphism space by
\begin{equation}\label{eqn-H-action}
    (h \cdot f)(m):=\sum_h h_1 \cdot (f(S(h_2)\cdot m))
\end{equation}
for any $h\in H$, $f\in \HOM_A(M, N)$ and $m\in M$. This action preserves $A$-linearity since, for any $a\in A$, we have that
\begin{align*}
    (h \cdot f)(am)& =\sum_h h_1 \cdot f(S(h_2)\cdot (am)) \\
     & = \sum_h h_1\cdot f((S(h_3)\cdot a)(S(h_2)\cdot m)) \\
        & = \sum_h h_1\cdot (S(h_4)\cdot a) (h_2 \cdot f(S(h_3)\cdot m)) \\
         & = \sum_h h_1\cdot (S(h_2)\cdot a) (h_3 \cdot f(S(h_4)\cdot m)) \\
          & = \sum_h a h_1 \cdot f(S(h_2)\cdot m)) = a (h\cdot f)(m).
\end{align*}
Here, in the fourth equality, we have used that $H$ is cocommutative, so that
\[
\sum_h h_1\otimes h_2\otimes h_3\otimes h_4=\sum_h h_1\otimes h_4\otimes h_2\otimes h_3.
\] 
One can also show, as in \cite[Lemma 5.2]{QYHopf}, that, under this $H$-action, the space of $H$-invariants is equal to
\begin{equation}\label{eqn-H-inv-in-HOM}
    \HOM_A(M,N)^H \cong \HOM_{A\# H} (M, N).
\end{equation}

Notice that we restrict the objects inserted into the bifunctor $\HOM_A(\mbox{-},\mbox{-}) $ to those in $\mc{C}(A\# H)$, i.e., those chain complexes of $A$-modules that are $H$-equivariant. When $M$ is an $H$-equivariant complex of $(A,B)$-bimodules, where $B$ is a second $H$-module algebra, we obtain a functor on the usual homotopy categories:
\begin{equation}
\HOM_A(M,\mbox{-}): \mc{C}(A\#H) \lra \mc{C}(B\# H).
\end{equation}

On the other hand, if $M$ is an $H$-equivariant complex of $(A,B)$-bimodules and $K$ is a complex of $H$-equivariant $B$-modules, the natural tensor product
\begin{equation}
    M\otimes_B K= \oplus_{i \in \ZZ}(M\otimes_B K)_i,\quad \quad (M\otimes_B K)_i:=\oplus_{k\in \ZZ} M_k\otimes_B K_{i-k}.
\end{equation}
gives one an $H$-equivariant complex of $A$-modules.
Here, the differential acts by, for any $m\in M_i$ and $k\in K_j$,
\begin{equation}
    d(m\otimes k)=d(m)\otimes k + (-1)^{i}m\otimes d(k);
\end{equation}
while the Hopf algebra $H$ acts by
\begin{equation}
    h\cdot (m\otimes k)=\sum_h (h_1\cdot m)\otimes (h_2\otimes k).
\end{equation}
In this way, tensor product with $M$ over $B$ defines a functor
\begin{equation}
    M\otimes_B(\mbox{-}): \mc{C}(B\#H) \lra \mc{C}(A\#H)
\end{equation}

It is then an easy exercise to check that the usual tensor-hom adjunction preserves the action by the Hopf algebra $H$ (see, for instance, \cite[Lemma 8.5]{QYHopf}):
\begin{equation}
    \HOM_A(M\otimes_B K, N) \cong \HOM_{B}(K,\HOM_A(M,N)).
\end{equation}
Taking $H$-invariants on both sides gives an isomorphism of the graded hom spaces in complexes of $B\# H$-modules and $A\# H$-modules (equation \eqref{eqn-H-inv-in-HOM}). Then taking the zeroth cohomology with respect to the differentials gives us the adjunction in the usual homotopy category:
\begin{equation}\label{eqn-homotopy-adj}
    \Hom_{\mc{C}(A\#H)}(M\otimes_B K, N) \cong \Hom_{\mc{C}(B\#H)}(K,\HOM_A(M,N)).
\end{equation}

Our first result is to show that tensor and hom descend to the relative homotopy category.

\begin{lem}\label{lem-tensor-hom-preserve-null-homotopy}
    The tensor and hom bifunctors in equation \eqref{eqn-homotopy-adj} preserve the class of relatively null-homotopic objects.
\end{lem}
\begin{proof}
    We show this for the tensor case, and leave the similar hom case to the reader. To do so, suppose $M$ is a complex of $A\# H$-modules. Let $K\in \mc{C}(A^{\mathrm{op}}\#H)$ be a relatively null-homotopic complex. 
    Since $K$ is relatively null-homotopic, there is an $A$-linear map $h: K\lra tK$ such that $\Id_K = dh+hd$. 
    Then $K\otimes_A M$ is also null-homotopic over $\Bbbk$, with a null-homotopy given by $h\otimes_A \Id_M$. 
\end{proof}

\begin{lem}\label{lem-relative-tensor-hom-exact}
The tensor and hom bifunctors in equation \eqref{eqn-homotopy-adj} descend to exact bifunctors on relative homotopy categories. 
\end{lem}
\begin{proof}
We only deal with the hom functor, and leave the tensor functor case as an exercise. To do so, we need to show that $\HOM_A(\mbox{-},\mbox{-})$ preserve the triangulated category structure on $\mc{C}^H(A)$. By definition, it commutes with homological shifts. It suffices to show that, the functor
$$\HOM_A(M, \mbox{-}): \mc{C}(A\#H)\lra \mc{C}(B\# H)$$
sends distinguished triangles to distinguished triangles.   By Lemma \ref{lem-construction-of-triangle}, triangles in $\mc{C}^H(A)$ arises from short exact sequence of $A\# H$-modules 
    \[
   0 \lra K \stackrel{f}{\lra} L \stackrel{g}{\lra} N \lra 0
    \]
that are split exact when forgetting the $H$-actions. Therefore, $\HOM_A(M,\mbox{-})$ applied to such an exact sequence results in a short exact sequence with induced maps from $f$ and $g$:
   \[
   0 \lra \HOM_A(M, K) \stackrel{f_*}{\lra} \HOM_A(M,L) \stackrel{g_*}{\lra} \HOM_A(M,N) \lra 0.
    \]
 Under the action of \eqref{eqn-H-action}, it is clear that $f_*$ and $g_*$ are $H$-intertwining. Furthermore, the sequence splits termwise as $B$-modules when forgetting the $H$-actions. By Lemma \ref{lem-construction-of-triangle} again, this short exact sequence descends to a triangle in the relative homotopy category $\mc{C}^H(B)$. 
 
 The argument for $\HOM_A(\mbox{-},N) $ is similar. The result follows. 
\end{proof}

\begin{thm}\label{thm-tensor-hom-adj}
    Let $M$ be a complex of $H$-equivariant $(A,B)$-bimodules.  Then there is an isomorphism
    \[
    \Hom_{\mc{C}^H(A)}(M\otimes_B K, N) \cong \Hom_{\mc{C}^H(B)}( K, \HOM_A(M, N))
    \]
    functorial in $K$ and $N$.
\end{thm}
\begin{proof}
Suppose that $\mc{C}$, $\mc{D}$ are triangulated categories, and $\mc{L}: \mc{C}\lra \mc{D}$ is an exact functor admitting a right adjoint $\mc{R}$. Let $\mc{A}(\mc{C})$ and $\mc{A}(\mc{D})$ be classes of objects that are idempotent complete and satisfy $\mc{L}(\mc{A}(\mc{C})) \subset \mc{A}(\mc{D})$ and $\mc{R}(\mc{A}(\mc{D}))\subset \mc{A}(\mc{C}) $. Then it is immediate from the
universal property of Verdier localization that $\mc{L}$ and $\mc{R}$ induce a pair of adjoint functors between
the Verdier quotient categories $\mc{C}/\mc{A}(\mc{C})$ and $\mc{D}/\mc{A}(\mc{D})$.

Now, in our situation, take $\mc{C}=\mc{C}(A\# H)$ and $\mc{D}=\mc{C}(B\# H)$, and set $\mc{L}=M\otimes_B(\mbox{-})$ and $R=\HOM_A(M,\mbox{-})$ 
respectively. By Lemma \ref{lem-tensor-hom-preserve-null-homotopy}, these functors preserve relatively null-homotopic objects, and the general construction follows.
\end{proof}

\begin{cor}\label{cor-hom-as-derived-invariants}
    Let $M$ be a complex of $H$-equivariant $B$-modules. Then
    \[
    \Hom_{\mc{C}^{H}(\Bbbk)}(\Bbbk, \HOM_A(M,N)) \cong \Hom_{\mc{C}^H(A)}(M,N).
    \]
\end{cor}
\begin{proof}
    Taking $B=\Bbbk$ to be the ground field, and $K$ to be the one-term complex $\Bbbk$  equipped with the trivial $H$-action sitting in homological degree zero, the result is a special case of Theorem \ref{thm-tensor-hom-adj}.
\end{proof}

Again, when $\Bbbk$ is a field, the left-hand side of Corollary \ref{cor-hom-as-derived-invariants} gives us a way of computing morphisms spaces in the relative homotopy category:
\[
  \Hom_{\mc{C}^H(A)}(M,N)  \cong  \Hom_{\mc{D}(H)}(\Bbbk, \HOM_A(M,N)).
\]
The latter space is the zeroth (hyper)cohomology of the complex of $H$-modules $\HOM_A(M,N)$.

\section{\texorpdfstring{$\sll_2$}{sl(2)}-action on foams}
\label{sl2foams:sec}
In this section we review a family of actions of $\mathfrak{sl}_2$ on foams constructed in \cite[Sections 4.2 and 4.4]{QRSW2}.

\subsection{Webs and foams}
\begin{dfn}
  Let $\surface$ be a surface. A \emph{closed web} or simply a
  \emph{web}
is a finite, oriented, trivalent graph
  $\web = (V(\Gamma), E(\Gamma))$ embedded in the interior of
  $\surface$ and endowed with a \emph{thickness function}
  $\ell\co E(\Gamma) \to \NN$ satisfying a flow condition: vertices
  and thicknesses of their adjacent edges must be one of these two types:
  \begin{equation} \label{mergesplit}
    \NB{\tikz[scale=0.6]{\begin{scope}[font=\tiny]
  \begin{scope}
    \draw [-<] (0,0) -- (-90:1) node[pos =1, below] {$a+b$};
    \draw [->] (0,0) -- (30:1) node[pos =1, above] {$b$};
    \draw [->] (0,0) -- (150:1) node[pos =1, above] {$a$};
  \end{scope}
\end{scope}
%%% Local Variables:
%%% mode: latex
%%% TeX-master: t
%%% End:
}} \qquad\text{or}\qquad \NB{\tikz[scale=0.6]{\begin{scope}[font=\tiny]
    \draw [->] (0,0) -- (90:1) node[pos =1, above] {$a+b$};
    \draw [-<] (0,0) -- (-30:1) node[pos =1, below] {$b$};
    \draw [-<] (0,0) -- (-150:1) node[pos =1, below] {$a$};
\end{scope}
%%% Local Variables:
%%% mode: latex
%%% TeX-master: t
%%% End:
}}.
  \end{equation}
  The first type is called a \emph{split} vertex, the second a \emph{merge}
  vertex. In each of these types, there is one \emph{thick} edge and
  two \emph{thin} edges. Oriented circles with non-negative thickness
  are regarded as edges without vertices and can be part of a web. The
  embedding of $\web$ in $\surface$ is smooth outside its vertices,
  and at the vertices should fit with the local models in \eqref{mergesplit}.
\end{dfn}

\begin{figure}
  \centering
  \NB{\tikz[font=\tiny]{\input{\imagesfolder/pf_exa_web}}}
    \caption{Example of a web in $\RR^2$.}
    \label{fig:exa-web}
\end{figure}

\begin{dfn}\label{def:foam}
  Let $M$ be an~oriented smooth 3-manifold with a~collared boundary.
  A~\emph{foam} $\foam \subset M$ is a~collection of \emph{facets},
  that are compact oriented surfaces labeled with non-negative integers
  and glued together along their boundary points, such that every
  point $p$ of $\foam$ has a~closed neighborhood homeomorphic to one
  of the~following:
  \begin{enumerate}
  \item a~disk, when $p$ belongs to a~unique facet,
  \item \label{it:Y}$Y \times [0,1]$, where $Y$ is the neighborhood of
    a~merge or split vertex of a web, when $p$ belongs to three facets, 
  \item the~cone over the~1-skeleton of a~tetrahedron with $p$ as
    the~vertex of the~cone (so that it belongs to six facets).
  \end{enumerate}
  See Figure~\ref{fig:foam-local-model} for a pictorial representation
  of these three cases. The set of points of the~second type is
  a~collection of curves called \emph{bindings} and the~points of
  the~third type are called \emph{singular vertices}.
  The~\emph{boundary} $\partial\foam$ of $\foam$ is the~closure of
  the~set of boundary points of facets that do not belong to
  a~binding. It is understood that $\foam$ coincides with
  $\partial\foam\times[0,1]$ on the~collar of $\partial M$. For each facet
  $\facet$ of $\foam$, we denote by $\ell(\facet)$ its~label, 
  called the \emph{thickness of $\facet$}. A~foam $\foam$ is
  \emph{decorated} if each facet $\facet$ of $\foam$ is assigned
  a~symmetric polynomial $P_f \in \Sym_{\ell(\facet)}$.  In the second local
  model (\ref{it:Y}) of Definition \ref{def:foam}, it is implicitly understood that thicknesses of
  the three facets are given by that of the edges in $Y$. In
  particular, it satisfies a flow condition and locally one has a thick
  facet and two thin ones. We also require that orientations of
  bindings are induced by that of the thin facets and by the opposite
  of the thick facet. Foams are regarded up to ambient isotopy relative to
  boundary. Foams without boundary are said to be \emph{closed}.
\end{dfn}

\begin{rmk}
    Diagrammatically, decorations on facets are depicted by dots
      placed on facets 
      adorned with symmetric polynomials in the correct number of
      variables (the thickness of the facet they sit on). The
      decoration of a given facet is the product of all adornments of
      dots sitting on that facet. 
\end{rmk}

\begin{figure}[ht]
  \centering
  \NB{\tikz[]{\input{\imagesfolder/pf_foam-3-local-models}}}
  \caption{The three local models of a foam. Taking into account the
    thicknesses, the model in the middle is denoted $Y^{(a,b)}$,
    and the model on the right is denoted $T^{(a,b,c)}$.}
  \label{fig:foam-local-model}
\end{figure}

\begin{notation}
  For a foam $\foam$, we write:
  \begin{itemize}
  \item $\foam^2$ for the~collection of facets of $\foam$,
  \item $\foam^1$ for the collection of bindings,
  \item $\foam^0$ for the set of singular vertices of $\foam$.
  \end{itemize}
  We partition $\foam^1$ as follows:
  $\foam^1= \foam^1_\circ\sqcup \foam^1_{-}$, where $\foam^1_\circ$ is
  the collection of circular bindings and $\foam^1_-$ is the
  collection of bindings diffeomorphic to intervals. If $s \in
  \foam^1_-$, any of its points has a neighborhood diffeomorphic to $Y^{(a,b)}$ for a given $a$ and $b$, and we set:
  \begin{equation}
    \degN{s} = ab + (a+b)(\myN-a-b).\label{eq:deg-binding}
  \end{equation}
  If $v \in
  \foam^0$, it has a neighborhood diffeomorphic to $T^{(a,b,c)}$ and we set:
  \begin{equation}
    \degN{v} = ab +bc + ac + (a+b+c)(\myN-a-b-c).\label{eq:deg-sing}
  \end{equation}
\end{notation}

\begin{dfn}\label{dfn:deg-foam}
  Let $\foam$ be a decorated foam and suppose that all decorations are
  homogeneous. For all $\myN$ in $\NN$, the \emph{$\myN$-degree of
    $\foam$} is the integer $\degN{\foam} \in \ZZ$ given by the
  following formula: 
  \begin{align}\label{eq:deg-foam}
    \degN{\foam}: = & \sum_{\facet \in \foam^2}
    \Big(\deg{P_\facet} -\ell(\facet)(\myN-\ell(\facet))\chi(\facet)
    \Big)  \\
     & + \sum_{s \in \foam^1_-} \degN{s} - \sum_{v \in \foam^0} \degN{v}. \nonumber
  \end{align}
\end{dfn}

The~boundary of a~foam $\foam\subset M$ is a~web in $\partial M$. In the
case $M = \surface\times[0,1]$ is a~thickened surface, a~generic
section $\foam_t := \foam \cap (\surface\times\{t\})$ is
a~web. The~bottom and top webs $\foam_0$ and $\foam_1$ are called the~\emph{input} and \emph{output} of $\foam$ respectively.

If $\surface$ is a surface, $\foamcat[\surface]$ is the category which
has webs in $\surface$ as objects and 
\[
  {\Hom}_{\foamcat[\surface]}\left( \web_0,\web_1\right)
  = \left\{ \begin{array}{c}\text{decorated foams $\foam$ in $\surface\times [0,1]$}\\ \text{ with $\foam_i =\web_i$ for $i\in \{0,1\}$}\end{array}\right\}.
\]
Composition is given by stacking foams on one another and
rescaling. Decorations behave multiplicatively. The identity of $\web$
is $\web\times [0,1]$ decorated by the constant polynomial $1$ on
every facet. The $\myN$-degree of foams is additive under composition.
If $\web$ is a web in a surface $\surface$ and
$h:\surface\times[0,1] \to \surface$ is a smooth isotopy\footnote{For
  the sake of satisfying the collared condition, one should assume
  that $h_t=\id_\surface$ for $t\in [0,\epsilon[\cup]1-\epsilon, 1]$
  for an $\epsilon\in ]0,1]$.} of $\surface$, one can define the foam
$\foam(h)$ to be the trace of $h(\web)$ in $\surface\times[0,1]$: for
all $t \in [0,1]$, $\foam(h)_t=h_t(\web)$. Such foams are called
\emph{traces of isotopies}. They have degree $0$.

\begin{dfn}
 A foam in a surface $\surface\times[0,1]$ is \emph{basic} if it is a trace of isotopy or if it
 is equal to $\web\times[0,1]$ outside a cylinder $B\times[0,1]$, and where
 it is given inside by one of the local models given in Figure~\ref{fig:basic-foam}.  
 \begin{figure}[ht]
   \centering
   \begin{tikzpicture}[xscale=2.4, yscale =-3.2]
     \node (pol) at (0,0) {\NB{\tikz[font=\tiny]{\input{\imagesfolder/pf_foam-polR}}}};
     \node[yshift= -1.4cm] at (pol) {polynomial};
     \node[yshift = -1.8cm] at (pol) {$\deg{R}$};
     \node (asso) at(1.85,0) {\NB{\tikz[font=\tiny]{\input{\imagesfolder/pf_foam-assoc}}}};
     \node[yshift = -1.4cm] at (asso) {associativity};
     \node[yshift= -1.8cm] at (asso) {$0$};
     \node (coasso) at (3.7,0) {\NB{\tikz[font=\tiny]{\input{\imagesfolder/pf_foam-coassoc}}}};
     \node[yshift = -1.4cm] at (coasso) {associativity};
     \node[yshift= -1.8cm] at (coasso) {$0$};
     \node (digcup) at (0,1) {\NB{\tikz[font=\tiny]{\input{\imagesfolder/pf_foam-digon-cup}}}};
     \node[yshift = -1.64cm] at (digcup) {digon-cup};
     \node[yshift= -2.04cm] at (digcup) {$-ab$};
     \node (digcap) at (2,1.1) {\NB{\tikz[font=\tiny]{\input{\imagesfolder/pf_foam-digon-cap}}}};
     \node[yshift = -1.4cm] at (digcap) {digon-cap};
     \node[yshift= -1.8cm] at (digcap) {$-ab$};
     \node (zip) at (4,1.1) {\NB{\tikz[font=\tiny]{\input{\imagesfolder/pf_foam-unzip}}}};
     \node[yshift = -1.4cm] at (zip) {unzip};
     \node[yshift= -1.8cm] at (zip) {$ab$};
     \node (unzip) at (0,2.1) {\NB{\tikz[font=\tiny]{\input{\imagesfolder/pf_foam-zip}}}};
     \node[yshift = -1.4cm] at (unzip) {zip};
     \node[yshift= -1.8cm] at (unzip) {$ab$};
     \node (cup) at (1.4,2.1) {\NB{\tikz[font=\tiny]{\input{\imagesfolder/pf_foam-cupa}}}};
     \node[yshift = -1.4cm] at (cup) {cup};
     \node[yshift= -1.8cm] at (cup) {$-a(\myN-a)$};
     \node (cap) at (2.6,2.1) {\NB{\tikz[font=\tiny]{\input{\imagesfolder/pf_foam-capa}}}};
     \node[yshift = -1.4cm] at (cap) {cap};
     \node[yshift= -1.8cm] at (cap) {$-a(\myN-a)$};
     \node (saddle) at (4,2.1) {\NB{\tikz[font=\tiny]{\input{\imagesfolder/pf_foam-saddle}}}};
     \node[yshift = -1.4cm] at (saddle) {saddle};
     \node[yshift= -1.8cm] at (saddle) {$a(\myN-a)$};
   \end{tikzpicture}
   \caption{The degree of a basic foam is given below the name of each
     of the local models.}\label{fig:basic-foam}
 \end{figure}

 A foam in $\surface \times [0,1]$ is \emph{in good position} if it is
 a composition of basic foams. 

If $\web$ is a web in $\RR^2$, we denote by $\vectweb{\web}$ the free
$\scalars$-module generated by foams in good position in $\RR^2\times
[0,1]$ with $\emptyset$ as input and $\web$ as output.
\end{dfn}

\begin{rmk}
  Every foam in $\surface \times [0,1]$ is isotopic to a foam in good
  position. 
\end{rmk}

In \cite{RW1}, a \emph{$\gll_\myN$-evaluation} of a foam $\foam$ was defined.  This evaluation
$\bracketN{\foam}$ is an element of the ring $\mathbb{Z}[E_1,\ldots,E_N]$.
Its exact formula is not important for this paper and we refer to \cite{RW1} for details of the construction.

Decorations of foams that we consider will often be power sums, and so 
we introduce the following notation: 
\begin{equation}
  \NB{\tikz[scale=1.5, font=\small]{\begin{scope}
  \draw (0,0) rectangle (1,1) coordinate [midway] (A);
  % \fill (A) circle (0.5mm)
  \node at (A) {$\dotnewtoni$};
\end{scope}}} \ =\
  \NB{\tikz[scale=1.5, font=\small]{\begin{scope}
  \draw (0,0) rectangle (1,1) coordinate [midway] (A);
  \fill (A) circle (0.5mm) node[below] {$p_i$};
\end{scope}}}.
\end{equation}
Note that in particular, on a facet of thickness $a$,
$\dotnewtoni[0]=a$. 
We will also extend the decorations allowed and let
$\wdotnewtoni$ denote the $i$th power sum in the variables which are
not in the facet. In other words,
\begin{equation}
  \label{eq:18}
  \NB{\tikz[scale=1.5, font=\small]{\begin{scope}
  \draw (0,0) rectangle (1,1) coordinate [midway] (A);
  % \fill (A) circle (0.5mm)
  \node at (A) {$\wdotnewtoni$};
\end{scope}}} \ =\
  P_i\cdot\ \NB{\tikz[scale=1.5, font=\small]{\begin{scope}
  \draw (0,0) rectangle (1,1) coordinate [midway] (A);
\end{scope}}}\   -\ 
  \NB{\tikz[scale=1.5, font=\small]{}}.
\end{equation}

\subsection{$\mathfrak{sl}_2$-action} \label{sec:sl2action}
\begin{dfn}
Let $\mathfrak{sl}_2$ be the Lie algebra over $\mathbb{Z}$ generated by $\Le, \Lf, \Lh$ with relations
\[
[\Lh, \Le]=2 \Le, \quad
[\Lh, \Lf]=-2 \Lf, \quad
[\Le, \Lf]= \Lh  \ .
\]
\end{dfn}

We fix parameters $t_1, t_2$ and define operators $\de$, $\df$ and $\dh$ below on basic foams and extend the action to satisfy
the Leibniz rule with respect to composition of foams. They map traces of isotopies to $0$. The operator $\de$ acts via 
$-\sum_i \frac{\partial}{\partial x_i} $ on
polynomials and by $0$ on any other basic foam. The operators $\dh$ and $\df$ are
defined as follows. 

\begin{gather}
\label{eq:h-act-pol} \dh\left(\NB{\tikz[scale=1.5, font=\small]{\begin{scope}
  \draw (0,0) rectangle (1,1) coordinate [midway] (A); %node[pos=0.2, font=\tiny] {$a$};
  \fill (A) circle (0.5mm) node[below] {$R$};
\end{scope}}}\right) 
  =-\deg{R}\cdot\  \NB{\tikz[scale=1.5, font=\small]{}} \\
\label{eq:h-act-assoc}  \dh\left(\NB{\tikz[scale=0.6, font=\tiny]{\begin{scope}[xscale = -1 ]
  \begin{scope}
    \coordinate (LL) at (0,0);
    \coordinate (L) at (0.5,0);
    \coordinate (R1) at (2.2,0.4);
    \coordinate (R2) at (2,0);
    \coordinate (R3) at (1.8, -0.4);
    \draw[-<-] (LL) -- (L);
    \draw (L) .. controls +(0,0) and +(-0.5, 0) .. (R1)
    coordinate[pos =0.5] (M);
    \draw (L) .. controls +(0,0) and +(-0.5, 0) .. (R3) ;
    \draw (M) .. controls +(0,0) and +(-0.5, 0) .. (R2) ;
  \end{scope}  
 \begin{scope}[yshift = -1.5cm]
    \coordinate (LLB) at (0,0);
    \coordinate (LB) at (0.5,0);
    \coordinate (R1B) at (2.2,0.4);
    \coordinate (R2B) at (2,0);
    \coordinate (R3B) at (1.8, -0.4);
    \draw[-<-] (LLB) -- (LB)  node[right, pos =0] {$a+b+c$};
    \draw (LB) .. controls +(0,0) and +(-0.5, 0) .. (R1B)    node[left] {$a$};
    \draw (LB) .. controls +(0,0) and +(-0.5, 0) .. (R3B)   node[left] {$b$}
    coordinate[pos =0.5] (MB);
    \draw (MB) .. controls +(0,0) and +(-0.5, 0) .. (R2B)    node[left] {$c$};
  \end{scope}  
  \draw (LL) -- (LLB);
  \draw (R1) -- (R1B);
  \draw (R2) -- (R2B);
  \draw (R3) -- (R3B);
  \draw[name path=path1, thick] (M) .. controls +(0,-0.5) and +(0,
  0.5) .. (LB);
  \draw[name path=path2, thick] (L) .. controls +(0,-0.5) and +(0,
  0.5) .. (MB);
  \path [name intersections={of=path1 and path2,by=O}];
  
\end{scope}}}\right)=
  \dh\left(\NB{\tikz[scale=0.6,font=\tiny]{\begin{scope}
  \begin{scope}
    \coordinate (LL) at (0,0);
    \coordinate (L) at (0.5,0);
    \coordinate (R1) at (2.2,0.4);
    \coordinate (R2) at (2,0);
    \coordinate (R3) at (1.8, -0.4);
    \draw[->-] (LL) -- (L);
    \draw (L) .. controls +(0,0) and +(-0.5, 0) .. (R1)
    coordinate[pos =0.5] (M);
    \draw (L) .. controls +(0,0) and +(-0.5, 0) .. (R3) ;
    \draw (M) .. controls +(0,0) and +(-0.5, 0) .. (R2) ;
  \end{scope}  
 \begin{scope}[yshift = -1.5cm]
    \coordinate (LLB) at (0,0);
    \coordinate (LB) at (0.5,0);
    \coordinate (R1B) at (2.2,0.4);
    \coordinate (R2B) at (2,0);
    \coordinate (R3B) at (1.8, -0.4);
    \draw[->-] (LLB) -- (LB)  node[left, pos =0] {$a+b+c$};
    \draw (LB) .. controls +(0,0) and +(-0.5, 0) .. (R1B)    node[right] {$a$};
    \draw (LB) .. controls +(0,0) and +(-0.5, 0) .. (R3B)   node[right] {$b$}
    coordinate[pos =0.5] (MB);
    \draw (MB) .. controls +(0,0) and +(-0.5, 0) .. (R2B)    node[right] {$c$};
  \end{scope}  
  \draw (LL) -- (LLB);
  \draw (R1) -- (R1B);
  \draw (R2) -- (R2B);
  \draw (R3) -- (R3B);
  \draw[name path=path1, thick] (M) .. controls +(0,-0.5) and +(0,
  0.5) .. (LB);
  \draw[name path=path2, thick] (L) .. controls +(0,-0.5) and +(0,
  0.5) .. (MB);
  \path [name intersections={of=path1 and path2,by=O}];  
\end{scope}
}} \right) =0 \\
\label{eq:h-act-dig-cup}  \dh\left( \NB{\tikz[font=\tiny]{\begin{scope}
  \begin{scope}
    \coordinate (L) at (0,0);
    \coordinate (R) at (2,0);
    \coordinate (ML) at (0.5, 0);
    \coordinate (MR) at (1.5, 0);
    \draw[->-] (L) -- (ML);
    \draw[->-] (MR) -- (R) node[right] {$a+b$};
    \draw[->-] (ML).. controls + (0.4, 0.4) and +(-0.2, 0.4) .. (MR)
    node[above, midway] {$a$};
    \draw[->-] (ML).. controls + (0.2, -0.4) and +(-0.4, -0.4) .. (MR)
    node[below, midway] {$b$};
  \end{scope}  
 \begin{scope}[yshift = -1.3cm]
    \coordinate (LB) at (0,0);
    \coordinate (RB) at (2,0);
    \draw[->-] (LB) -- (RB);
  \end{scope}  
  \draw (R) -- (RB);
  \draw (L) -- (LB);
  \draw[thick] (ML) .. controls +(0, -1) and +(0, -1) .. (MR);
\end{scope}}}\right)\ =\
  ab(\tone+\ttwo)\cdot\ \NB{\tikz[font=\tiny]{}} \\
\label{eq:h-act-dig-cap}  \dh\left( \NB{\tikz[font=\tiny]{\begin{scope}
  \begin{scope}
    \coordinate (L) at (0,0);
    \coordinate (R) at (2,0);
    \coordinate (ML) at (0.5, 0);
    \coordinate (MR) at (1.5, 0);
    \draw[->-] (L) -- (ML);
    \draw[->-] (MR) -- (R) node[right] {$a+b$};
    \draw[->-] (ML).. controls + (0.4, 0.4) and +(-0.2, 0.4) .. (MR)
    node[above, midway] {$a$};
    \draw[->-] (ML).. controls + (0.2, -0.4) and +(-0.4, -0.4) .. (MR)
    node[below, midway] {$b$};
  \end{scope}  
 \begin{scope}[yshift = 1.3cm]
    \coordinate (LB) at (0,0);
    \coordinate (RB) at (2,0);
    \draw (LB) -- (RB);
  \end{scope}  
  \draw (R) -- (RB);
  \draw (L) -- (LB);
  \draw[thick] (ML) .. controls +(0, 1) and +(0, 1) .. (MR);
\end{scope}}}\right)\ =\
 ab(\overline{\tone}+\overline{\ttwo})\cdot\  \NB{\tikz[font=\tiny]{}} \\
\label{eq:h-act-dig-zip}  \dh\left( \NB{\tikz[font=\tiny]{\begin{scope}
  \begin{scope}
    \coordinate (L1) at (0.2,0.4);
    \coordinate (L2) at (0,0);
    \coordinate (R1) at (2.2,0.4);
    \coordinate (R2) at (2,0);
    \coordinate (ML) at (0.6, 0.2);
    \coordinate (MR) at (1.6, 0.2);
    \draw[->-] (ML) -- (MR) node[above, midway] {$a+b$};
    \draw (MR) .. controls +(0, 0) and +(-0.3,0) .. (R1) ;
    \draw (MR) .. controls +(0, 0) and +(-0.3,0) .. (R2);
    \draw (L1) .. controls +( 0.3, 0) and +(0,0) .. (ML);
    \draw (L2) .. controls +( 0.3, 0) and +(0,0) .. (ML);
  \end{scope}  
 \begin{scope}[yshift = -1cm]
    \coordinate (L1B) at (0.2,0.4);
    \coordinate (L2B) at (0,0);
    \coordinate (R1B) at (2.2,0.4);
    \coordinate (R2B) at (2,0);
    \draw[->-] (L1B) .. controls +( 0, 0) and +(0,0) .. (R1B) node [right, pos
    = 1] {$a$};
    \draw[->-] (L2B) .. controls +( 0, 0) and +(0,0) .. (R2B) node [right, pos
    = 1] {$b$};
 \end{scope}  
  \draw (R1) -- (R1B);
  \draw (R2) -- (R2B);
  \draw (L1) -- (L1B);
  \draw (L2) -- (L2B);
  \draw[thick] (ML) .. controls +(0, -0.6) and +(0, -0.6) .. (MR);
\end{scope}

%%% Local Variables:
%%% mode: latex
%%% TeX-master: t
%%% End:
}}\right)\ =\
-ab(\overline{\tone} + \overline{\ttwo}) \cdot\ \NB{\tikz[font=\tiny]{}} 
 \\
\label{eq:h-act-dig-unzip} \dh\left( \NB{\tikz[font=\tiny]{\begin{scope}
  \begin{scope}
    \coordinate (L1) at (0.2,0.4);
    \coordinate (L2) at (0,0);
    \coordinate (R1) at (2.2,0.4);
    \coordinate (R2) at (2,0);
    \coordinate (ML) at (0.6, 0.2);
    \coordinate (MR) at (1.6, 0.2);
    \draw[->-] (ML) -- (MR) node[below, midway] {$a+b$};
    \draw (MR) .. controls +(0, 0) and +(-0.3,0) .. (R1) ;
    \draw (MR) .. controls +(0, 0) and +(-0.3,0) .. (R2);
    \draw (L1) .. controls +( 0.3, 0) and +(0,0) .. (ML);
    \draw (L2) .. controls +( 0.3, 0) and +(0,0) .. (ML);
  \end{scope}  
 \begin{scope}[yshift = 1cm]
    \coordinate (L1B) at (0.2,0.4);
    \coordinate (L2B) at (0,0);
    \coordinate (R1B) at (2.2,0.4);
    \coordinate (R2B) at (2,0);
    \draw[->-] (L1B) .. controls +( 0, 0) and +(0,0) .. (R1B) node [left, pos
    = 0] {$a$};
    \draw[->-] (L2B) .. controls +( 0, 0) and +(0,0) .. (R2B) node [left, pos
    = 0] {$b$};
 \end{scope}  
  \draw (R1) -- (R1B);
  \draw (R2) -- (R2B);
  \draw (L1) -- (L1B);
  \draw (L2) -- (L2B);
  \draw[thick] (ML) .. controls +(0, 0.6) and +(0, 0.6) .. (MR);
\end{scope}
}}\right)\ =\
 -ab({\tone} + {\ttwo}) \cdot\ \NB{\tikz[font=\tiny]{}} 
  \\
\label{eq:h-act-cup}  \dh\left( \NB{\tikz[font=\tiny, scale=1.2]{\begin{scope}
  \draw (0,0) arc (180 :0: 0.5cm and 0.2cm) node[above, pos =
  0.5] {$a$};
  \draw[very thin] (0,0) arc (180 :0: 0.5cm and -0.6cm) node[pos=0.5,
  above] {};
  \draw (0,0) arc (180 :0: 0.5cm and -0.2cm);
\end{scope}

%%% Local Variables:
%%% mode: latex
%%% TeX-master: t
%%% End:
}}\right)\ =\
     a(N-a) \cdot\ \NB{\tikz[font=\tiny, scale=1.2]{}}
  \\[3pt]
\label{eq:h-act-cap}  \dh\left( \NB{\tikz[font=\tiny, scale=1.2]{\begin{scope}[-]
  \draw (0,0) arc (180 :0: 0.5cm and 0.2cm);
  \draw[very thin] (0,0) arc (180 :0: 0.5cm and 0.6cm) node[pos=0.5,
  below] {};
  \draw (0,0) arc (180 :0: 0.5cm and -0.2cm)node[ below, pos =
  0.5] {$a$};
\end{scope}

%%% Local Variables:
%%% mode: latex
%%% TeX-master: t
%%% End:
}}\right)\ =\
  a(N-a) \cdot\ \NB{\tikz[font=\tiny,
    scale=1.2]{}} \\
\label{eq:h-act-saddle}  \dh\left( \NB{\tikz[font=\tiny, scale=1.2]{\begin{scope}[scale=0.6]
\tdplotsetmaincoords{70}{25}
\begin{scope}[scale = 1.5, tdplot_main_coords]
  \tikzset{yxplane/.style={canvas is xy plane at z=#1}}
  \begin{scope}[yxplane=1]
    \coordinate (AT) at ({cos(  45)}, {sin(  45)});
    \coordinate (BT) at ({cos(135)}, {sin(135)});
    \coordinate (CT) at ({cos(225)}, {sin(225)});
    \coordinate (DT) at ({cos(315)}, {sin(315)});
    \coordinate (aT) at (  0.3, 0);
    \coordinate (bT) at ( -0.3, 0);
    \draw (AT) .. controls (aT) and (bT) .. (BT)
    coordinate[pos=0.5] (eT) node [pos =0.5, above] {$a$};
    \draw (DT) .. controls (aT) and (bT) .. (CT) coordinate[pos=0.5] (fT);
  \end{scope}

  \begin{scope}[yxplane=0]
    \coordinate (AM) at ({cos(  45)}, {sin(  45)});
    \coordinate (BM) at ({cos(135)}, {sin(135)});
    \coordinate (CM) at ({cos(225)}, {sin(225)});
    \coordinate (DM) at ({cos(315)}, {sin(315)});
    \coordinate (aM) at (0, 0.3);
    \coordinate (bM) at (0, -0.3);
    \draw (AM) .. controls (aM) and (bM) .. (DM) coordinate[pos=0.5] (eM);
    \draw (BM) .. controls (aM) and (bM) .. (CM) coordinate[pos=0.5] (fM);
\end{scope}
  \coordinate (OT) at (0,0, 0.5);
\draw (AM) -- (AT);
  \draw (BM) -- (BT);
  \draw (CM) -- (CT);
  \draw (DM) -- (DT);
  \draw[densely dotted] (eM) ..controls +(0,0,0.2) and + (0.1,0,0).. (OT);
  \draw[densely dotted] (fM) ..controls +(0,0,0.2) and + (-0.1,0,0).. (OT);
  \draw[densely dotted] (eT) ..controls +(0,0,-0.2) and + (0,0.1,0).. (OT);
  \draw[densely dotted] (fT) ..controls +(0,0,-0.2) and + (0,-0.1,0).. (OT);
\end{scope}  
\end{scope}

%%% Local Variables:
%%% mode: latex
%%% TeX-master: t
%%% End:
}}\right)\ =\
    -a(N-a) \cdot\ \NB{\tikz[font=\tiny,
    scale=1.2]{}} 
\end{gather}

\begin{gather}
\label{eq:e-act-pol}  \df\left(\NB{\tikz[scale=1.5, font=\small]{}}\right)
  =\ \NB{\tikz[scale=1.5, font=\small]{\begin{scope}
  \draw (0,0) rectangle (1.5,1) coordinate [midway] (A);
  \fill (A) circle (0.5mm) node[below] {$\sum_i x_i^2 \frac{\partial}{\partial x_i}(R)$};
\end{scope}}} \\
\label{eq:e-act-assoc}   \df\left(\NB{\tikz[scale=0.6, font=\tiny]{}}\right)=
  \df\left(\NB{\tikz[scale=0.6,font=\tiny]{}} \right) =0 \\
\label{eq:e-act-dig-cup}  \df\left( \NB{\tikz[font=\tiny]{}}\right)\ =\
 - \tone\cdot\ \NB{\tikz[font=\tiny]{\begin{scope}[font=\tiny]
  \begin{scope}
    \coordinate (L) at (0,0);
    \coordinate (R) at (2,0);
    \coordinate (ML) at (0.5, 0);
    \coordinate (MR) at (1.5, 0);
    \draw[->-] (L) -- (ML);
    \draw[->-] (MR) -- (R) node[right] {$a+b$};
    \draw[->-] (ML).. controls + (0.4, 0.4) and +(-0.2, 0.4) .. (MR)
    node[above, pos=0.7 ] {$a$} node[below, pos =0.3] {$\dotnewtoni[1]$};
    \draw[->-] (ML).. controls +(0.2, -0.4) and +(-0.4, -0.4) .. (MR)
    node[below, pos =0.3] {$b$} node[below, pos =0.75] {$\dotnewtoni[0]$};
  \end{scope}  
 \begin{scope}[yshift = -1.3cm]
    \coordinate (LB) at (0,0);
    \coordinate (RB) at (2,0);
    \draw[->-] (LB) -- (RB);
  \end{scope}  
  \draw (R) -- (RB);
  \draw (L) -- (LB);
  \draw[thick] (ML) .. controls +(0, -1) and +(0, -1) .. (MR);
\end{scope}

%%% Local Variables:
%%% mode: latex
%%% TeX-master: t
%%% End:
}} 
  - \ttwo\cdot \ \NB{\tikz[font=\tiny]{\begin{scope}[font=\tiny]
  \begin{scope}
    \coordinate (L) at (0,0);
    \coordinate (R) at (2,0);
    \coordinate (ML) at (0.5, 0);
    \coordinate (MR) at (1.5, 0);
    \draw[->-] (L) -- (ML);
    \draw[->-] (MR) -- (R) node[right] {$a+b$};
    \draw[->-] (ML).. controls + (0.4, 0.4) and +(-0.2, 0.4) .. (MR)
    node[above, pos=0.7 ] {$a$} node[below, pos =0.3] {$\dotnewtoni[0]$};
    \draw[->-] (ML).. controls +(0.2, -0.4) and +(-0.4, -0.4) .. (MR)
    node[below, pos =0.3] {$b$} node[below, pos =0.75] {$\dotnewtoni[1]$};
  \end{scope}  
 \begin{scope}[yshift = -1.3cm]
    \coordinate (LB) at (0,0);
    \coordinate (RB) at (2,0);
    \draw[->-] (LB) -- (RB);
  \end{scope}  
  \draw (R) -- (RB);
  \draw (L) -- (LB);
  \draw[thick] (ML) .. controls +(0, -1) and +(0, -1) .. (MR);
\end{scope}

%%% Local Variables:
%%% mode: latex
%%% TeX-master: t
%%% End:
}}  \\
\label{eq:e-act-dig-cap}   \df\left( \NB{\tikz[font=\tiny]{}}\right)\ =\
   - \overline{\tone}
  \cdot\ \NB{\tikz[font=\tiny]{\begin{scope}[font =\tiny]
  \begin{scope}
    \coordinate (L) at (0,0);
    \coordinate (R) at (2,0);
    \coordinate (ML) at (0.5, 0);
    \coordinate (MR) at (1.5, 0);
    \draw[->-] (L) -- (ML);
    \draw[->-] (MR) -- (R) node[right] {$a+b$};
    \draw[->-] (ML).. controls + (0.4, 0.4) and +(-0.2, 0.4) .. (MR)
    node[above, pos =0.7] {$a$}     node[above, pos =0.3] {$\dotnewtoni[1]$};
    \draw[->-] (ML).. controls + (0.2, -0.4) and +(-0.4, -0.4) .. (MR)
    node[left, pos =0.3] {$b$} node[above, pos =0.7, yshift = -0.5mm] {$\dotnewtoni[0]$};
  \end{scope}  
 \begin{scope}[yshift = 1.3cm]
    \coordinate (LB) at (0,0);
    \coordinate (RB) at (2,0);
    \draw (LB) -- (RB);
  \end{scope}  
  \draw (R) -- (RB);
  \draw (L) -- (LB);
  \draw[thick] (ML) .. controls +(0, 1) and +(0, 1) .. (MR);
\end{scope}

%%% Local Variables:
%%% mode: latex
%%% TeX-master: t
%%% End:
}}  
  \!\!\!\! - \overline{\ttwo}
  \cdot \ \NB{\tikz[font=\tiny]{\begin{scope}[font =\tiny]
  \begin{scope}
    \coordinate (L) at (0,0);
    \coordinate (R) at (2,0);
    \coordinate (ML) at (0.5, 0);
    \coordinate (MR) at (1.5, 0);
    \draw[->-] (L) -- (ML);
    \draw[->-] (MR) -- (R) node[right] {$a+b$};
    \draw[->-] (ML).. controls + (0.4, 0.4) and +(-0.2, 0.4) .. (MR)
    node[above, pos =0.7] {$a$}     node[above, pos =0.3] {$\dotnewtoni[0]$};
    \draw[->-] (ML).. controls + (0.2, -0.4) and +(-0.4, -0.4) .. (MR)
    node[left, pos =0.3] {$b$} node[above, pos =0.7, yshift = -0.5mm] {$\dotnewtoni[1]$};
  \end{scope}  
 \begin{scope}[yshift = 1.3cm]
    \coordinate (LB) at (0,0);
    \coordinate (RB) at (2,0);
    \draw (LB) -- (RB);
  \end{scope}  
  \draw (R) -- (RB);
  \draw (L) -- (LB);
  \draw[thick] (ML) .. controls +(0, 1) and +(0, 1) .. (MR);
\end{scope}

%%% Local Variables:
%%% mode: latex
%%% TeX-master: t
%%% End:
}}  \\
\label{eq:e-act-zip}   \df\left( \NB{\tikz[font=\tiny]{}}\right)\ =\
    \overline{\tone} \cdot\ \NB{\tikz[font=\tiny]{\begin{scope}
  \begin{scope}
    \coordinate (L1) at (0.2,0.4);
    \coordinate (L2) at (0,0);
    \coordinate (R1) at (2.2,0.4);
    \coordinate (R2) at (2,0);
    \coordinate (ML) at (0.6, 0.2);
    \coordinate (MR) at (1.6, 0.2);
    \draw[->-] (ML) -- (MR) node[above, midway] {$a+b$};
    \draw (MR) .. controls +(0, 0) and +(-0.3,0) .. (R1) ;
    \draw (MR) .. controls +(0, 0) and +(-0.3,0) .. (R2);
    \draw (L1) .. controls +( 0.3, 0) and +(0,0) .. (ML);
    \draw (L2) .. controls +( 0.3, 0) and +(0,0) .. (ML);
  \end{scope}  
 \begin{scope}[yshift = -1cm]
    \coordinate (L1B) at (0.2,0.4);
    \coordinate (L2B) at (0,0);
    \coordinate (R1B) at (2.2,0.4);
    \coordinate (R2B) at (2,0);
    \draw[->-] (L1B) .. controls +( 0, 0) and +(0,0) .. (R1B) node [right, pos
    = 1] {$a$};
    \draw[->-] (L2B) .. controls +( 0, 0) and +(0,0) .. (R2B) node [right, pos
    = 1] {$b$}   node [pos = 0.2, above] {$\dotnewtoni[0]$};
 \end{scope}  
  \draw (R1) -- (R1B) node [pos = 0.2, left] {$\dotnewtoni[1]$};
  \draw (R2) -- (R2B);
  \draw (L1) -- (L1B);
  \draw (L2) -- (L2B);
  \draw[thick] (ML) .. controls +(0, -0.6) and +(0, -0.6) .. (MR);
\end{scope}

%%% Local Variables:
%%% mode: latex
%%% TeX-master: t
%%% End:
}} 
  +  \overline{\ttwo}\cdot \ \NB{\tikz[font=\tiny]{\begin{scope}
  \begin{scope}
    \coordinate (L1) at (0.2,0.4);
    \coordinate (L2) at (0,0);
    \coordinate (R1) at (2.2,0.4);
    \coordinate (R2) at (2,0);
    \coordinate (ML) at (0.6, 0.2);
    \coordinate (MR) at (1.6, 0.2);
    \draw[->-] (ML) -- (MR) node[above, midway] {$a+b$};
    \draw (MR) .. controls +(0, 0) and +(-0.3,0) .. (R1) ;
    \draw (MR) .. controls +(0, 0) and +(-0.3,0) .. (R2);
    \draw (L1) .. controls +( 0.3, 0) and +(0,0) .. (ML);
    \draw (L2) .. controls +( 0.3, 0) and +(0,0) .. (ML);
  \end{scope}  
 \begin{scope}[yshift = -1cm]
    \coordinate (L1B) at (0.2,0.4);
    \coordinate (L2B) at (0,0);
    \coordinate (R1B) at (2.2,0.4);
    \coordinate (R2B) at (2,0);
    \draw[->-] (L1B) .. controls +( 0, 0) and +(0,0) .. (R1B) node [right, pos
    = 1] {$a$};
    \draw[->-] (L2B) .. controls +( 0, 0) and +(0,0) .. (R2B) node [right, pos
    = 1] {$b$}   node [pos = 0.2, above] {$\dotnewtoni[1]$};
 \end{scope}  
  \draw (R1) -- (R1B) node [pos = 0.2, left] {$\dotnewtoni[0]$};
  \draw (R2) -- (R2B);
  \draw (L1) -- (L1B);
  \draw (L2) -- (L2B);
  \draw[thick] (ML) .. controls +(0, -0.6) and +(0, -0.6) .. (MR);
\end{scope}

%%% Local Variables:
%%% mode: latex
%%% TeX-master: t
%%% End:
}}  \\
\label{eq:e-act-unzip}   \df\left( \NB{\tikz[font=\tiny]{}}\right)\ =\
    {\tone}\cdot\ \NB{\tikz[font=\tiny]{\begin{scope}
  \begin{scope}
    \coordinate (L1) at (0.2,0.4);
    \coordinate (L2) at (0,0);
    \coordinate (R1) at (2.2,0.4);
    \coordinate (R2) at (2,0);
    \coordinate (ML) at (0.6, 0.2);
    \coordinate (MR) at (1.6, 0.2);
    \draw[->-] (ML) -- (MR) node[below, midway] {$a+b$};
    \draw (MR) .. controls +(0, 0) and +(-0.3,0) .. (R1) ;
    \draw (MR) .. controls +(0, 0) and +(-0.3,0) .. (R2);
    \draw (L1) .. controls +( 0.3, 0) and +(0,0) .. (ML);
    \draw (L2) .. controls +( 0.3, 0) and +(0,0) .. (ML) node [above,
    pos=0.2] {$\dotnewtoni[0]$};
  \end{scope}  
 \begin{scope}[yshift = 1cm]
    \coordinate (L1B) at (0.2,0.4);
    \coordinate (L2B) at (0,0);
    \coordinate (R1B) at (2.2,0.4);
    \coordinate (R2B) at (2,0);
    \draw[->-] (L1B) .. controls +( 0, 0) and +(0,0) .. (R1B) node [left, pos
    = 0] {$a$} node
  [pos=0.8, below] {$\dotnewtoni[1]$};
    \draw[->-] (L2B) .. controls +( 0, 0) and +(0,0) .. (R2B) node [left, pos
    = 0] {$b$};
 \end{scope}  
  \draw (R1) -- (R1B);
  \draw (R2) -- (R2B);
  \draw (L1) -- (L1B);
  \draw (L2) -- (L2B);
  \draw[thick] (ML) .. controls +(0, 0.6) and +(0, 0.6) .. (MR);
\end{scope}

%%% Local Variables:
%%% mode: latex
%%% TeX-master: t
%%% End:
}} 
   +  {\ttwo}\cdot \ \NB{\tikz[font=\tiny]{\begin{scope}
  \begin{scope}
    \coordinate (L1) at (0.2,0.4);
    \coordinate (L2) at (0,0);
    \coordinate (R1) at (2.2,0.4);
    \coordinate (R2) at (2,0);
    \coordinate (ML) at (0.6, 0.2);
    \coordinate (MR) at (1.6, 0.2);
    \draw[->-] (ML) -- (MR) node[below, midway] {$a+b$};
    \draw (MR) .. controls +(0, 0) and +(-0.3,0) .. (R1) ;
    \draw (MR) .. controls +(0, 0) and +(-0.3,0) .. (R2);
    \draw (L1) .. controls +( 0.3, 0) and +(0,0) .. (ML);
    \draw (L2) .. controls +( 0.3, 0) and +(0,0) .. (ML) node [above,
    pos=0.2] {$\dotnewtoni[1]$};
  \end{scope}  
 \begin{scope}[yshift = 1cm]
    \coordinate (L1B) at (0.2,0.4);
    \coordinate (L2B) at (0,0);
    \coordinate (R1B) at (2.2,0.4);
    \coordinate (R2B) at (2,0);
    \draw[->-] (L1B) .. controls +( 0, 0) and +(0,0) .. (R1B) node [left, pos
    = 0] {$a$} node
  [pos=0.8, below] {$\dotnewtoni[0]$};
    \draw[->-] (L2B) .. controls +( 0, 0) and +(0,0) .. (R2B) node [left, pos
    = 0] {$b$};
 \end{scope}  
  \draw (R1) -- (R1B);
  \draw (R2) -- (R2B);
  \draw (L1) -- (L1B);
  \draw (L2) -- (L2B);
  \draw[thick] (ML) .. controls +(0, 0.6) and +(0, 0.6) .. (MR);
\end{scope}

%%% Local Variables:
%%% mode: latex
%%% TeX-master: t
%%% End:
}}  
  \\
\label{eq:e-act-cup}   \df\left( \NB{\tikz[font=\tiny, scale=1.2]{}}\right)\ =\
     - \frac{1}{2}  \cdot\ \NB{\tikz[font=\tiny, scale=1.2]{\begin{scope}
  \draw (0,0) arc (180 :0: 0.5cm and 0.2cm) node[above, pos =
  0.5] {$a$};
  \draw[very thin] (0,0) arc (180 :0: 0.5cm and -0.6cm) node[pos=0.5,
  above] {$\dotnewtoni[0] \wdotnewtoni[1]$};
  \draw (0,0) arc (180 :0: 0.5cm and -0.2cm);
\end{scope}

%%% Local Variables:
%%% mode: latex
%%% TeX-master: t
%%% End:
}}
    \ 
   -\   \frac{1}{2}  \cdot\ \NB{\tikz[font=\tiny, scale=1.2]{\begin{scope}
  \draw (0,0) arc (180 :0: 0.5cm and 0.2cm) node[above, pos =
  0.5] {$a$};
  \draw[very thin] (0,0) arc (180 :0: 0.5cm and -0.6cm) node[pos=0.5,
  above] {$\dotnewtoni[1] \wdotnewtoni[0]$};
  \draw (0,0) arc (180 :0: 0.5cm and -0.2cm);
\end{scope}

%%% Local Variables:
%%% mode: latex
%%% TeX-master: t
%%% End:
}}  \\[3pt]
\label{eq:e-act-cap}   \df\left( \NB{\tikz[font=\tiny, scale=1.2]{}}\right)\ =\
    -\frac{1}{2} \cdot\ \NB{\tikz[font=\tiny,
    scale=1.2]{\begin{scope}
  \draw (0,0) arc (180 :0: 0.5cm and 0.2cm);
  \draw[very thin] (0,0) arc (180 :0: 0.5cm and 0.6cm) node[pos=0.5,
  below] {$\dotnewtoni[0] \wdotnewtoni[1]$};
  \draw (0,0) arc (180 :0: 0.5cm and -0.2cm)node[ below, pos =
  0.5] {$a$};
\end{scope}

%%% Local Variables:
%%% mode: latex
%%% TeX-master: t
%%% End:
}} \ 
  - \  \frac{1}{2}  \cdot\ \NB{\tikz[font=\tiny, scale=1.2]{\begin{scope}
  \draw (0,0) arc (180 :0: 0.5cm and 0.2cm);
  \draw[very thin] (0,0) arc (180 :0: 0.5cm and 0.6cm) node[pos=0.5,
  below] {$\dotnewtoni[1] \wdotnewtoni[0]$};
  \draw (0,0) arc (180 :0: 0.5cm and -0.2cm)node[ below, pos =
  0.5] {$a$};
\end{scope}

%%% Local Variables:
%%% mode: latex
%%% TeX-master: t
%%% End:
}} \\
  \label{eq:e-act-saddle}   \df\left( \NB{\tikz[font=\tiny, scale=1.2]{}}\right)\ =\
    \frac{1}{2} \cdot\ \NB{\tikz[font=\tiny,
    scale=1.2]{\begin{scope}[scale=0.6]
\tdplotsetmaincoords{70}{25}
\begin{scope}[scale = 1.5, tdplot_main_coords]
  \tikzset{yxplane/.style={canvas is xy plane at z=#1}}
  \begin{scope}[yxplane=1]
    \coordinate (AT) at ({cos(  45)}, {sin(  45)});
    \coordinate (BT) at ({cos(135)}, {sin(135)});
    \coordinate (CT) at ({cos(225)}, {sin(225)});
    \coordinate (DT) at ({cos(315)}, {sin(315)});
    \coordinate (aT) at (  0.3, 0);
    \coordinate (bT) at ( -0.3, 0);
    \draw (AT) .. controls (aT) and (bT) .. (BT)
    coordinate[pos=0.5] (eT) node [pos =0.5, above] {$a$};
    \draw (DT) .. controls (aT) and (bT) .. (CT) coordinate[pos=0.5] (fT);
  \end{scope}
  \node at (-0.75, -0.1, 0.3) {$\dotnewtoni[0]$};
    \node at (0.5, 0.5, 0.3) {$\wdotnewtoni[1]$};
  \begin{scope}[yxplane=0]
    \coordinate (AM) at ({cos(  45)}, {sin(  45)});
    \coordinate (BM) at ({cos(135)}, {sin(135)});
    \coordinate (CM) at ({cos(225)}, {sin(225)});
    \coordinate (DM) at ({cos(315)}, {sin(315)});
    \coordinate (aM) at (0, 0.3);
    \coordinate (bM) at (0, -0.3);
    \draw (AM) .. controls (aM) and (bM) .. (DM) coordinate[pos=0.5] (eM);
    \draw (BM) .. controls (aM) and (bM) .. (CM) coordinate[pos=0.5] (fM);
\end{scope}
  \coordinate (OT) at (0,0, 0.5);
\draw (AM) -- (AT);
  \draw (BM) -- (BT);
  \draw (CM) -- (CT);
  \draw (DM) -- (DT);
  \draw[densely dotted] (eM) ..controls +(0,0,0.2) and + (0.1,0,0).. (OT);
  \draw[densely dotted] (fM) ..controls +(0,0,0.2) and + (-0.1,0,0).. (OT);
  \draw[densely dotted] (eT) ..controls +(0,0,-0.2) and + (0,0.1,0).. (OT);
  \draw[densely dotted] (fT) ..controls +(0,0,-0.2) and + (0,-0.1,0).. (OT);
\end{scope}  
\end{scope}

%%% Local Variables:
%%% mode: latex
%%% TeX-master: t
%%% End:
}} \ 
  + \  \frac{1}{2} \cdot\ \NB{\tikz[font=\tiny, scale=1.2]{\begin{scope}[scale=0.6]
\tdplotsetmaincoords{70}{25}
\begin{scope}[scale = 1.5, tdplot_main_coords]
  \tikzset{yxplane/.style={canvas is xy plane at z=#1}}
  \begin{scope}[yxplane=1]
    \coordinate (AT) at ({cos(  45)}, {sin(  45)});
    \coordinate (BT) at ({cos(135)}, {sin(135)});
    \coordinate (CT) at ({cos(225)}, {sin(225)});
    \coordinate (DT) at ({cos(315)}, {sin(315)});
    \coordinate (aT) at (  0.3, 0);
    \coordinate (bT) at ( -0.3, 0);
    \draw (AT) .. controls (aT) and (bT) .. (BT)
    coordinate[pos=0.5] (eT) node [pos =0.5, above] {$a$};
    \draw (DT) .. controls (aT) and (bT) .. (CT) coordinate[pos=0.5] (fT);
  \end{scope}
  \node at (-0.75, -0.1, 0.3) {$\dotnewtoni[1]$};
    \node at (0.5, 0.5, 0.3) {$\wdotnewtoni[0]$};
  \begin{scope}[yxplane=0]
    \coordinate (AM) at ({cos(  45)}, {sin(  45)});
    \coordinate (BM) at ({cos(135)}, {sin(135)});
    \coordinate (CM) at ({cos(225)}, {sin(225)});
    \coordinate (DM) at ({cos(315)}, {sin(315)});
    \coordinate (aM) at (0, 0.3);
    \coordinate (bM) at (0, -0.3);
    \draw (AM) .. controls (aM) and (bM) .. (DM) coordinate[pos=0.5] (eM);
    \draw (BM) .. controls (aM) and (bM) .. (CM) coordinate[pos=0.5] (fM);
\end{scope}
  \coordinate (OT) at (0,0, 0.5);
\draw (AM) -- (AT);
  \draw (BM) -- (BT);
  \draw (CM) -- (CT);
  \draw (DM) -- (DT);
  \draw[densely dotted] (eM) ..controls +(0,0,0.2) and + (0.1,0,0).. (OT);
  \draw[densely dotted] (fM) ..controls +(0,0,0.2) and + (-0.1,0,0).. (OT);
  \draw[densely dotted] (eT) ..controls +(0,0,-0.2) and + (0,0.1,0).. (OT);
  \draw[densely dotted] (fT) ..controls +(0,0,-0.2) and + (0,-0.1,0).. (OT);
\end{scope}  
\end{scope}

%%% Local Variables:
%%% mode: latex
%%% TeX-master: t
%%% End:
}} 
\end{gather}

We now record the fact that the formulas above do indeed give rise to an $\mathfrak{sl}_2$-action.
\begin{lem} \cite[Lemma 4.6]{QRSW2} \label{lem:sl2-acts-good-position}
  Mapping $\Le$ to $\de$, $\Lh$ to $\dh$ and $\Lf$ to $\df$ defines an
  action of $\sll_2$ on the $\scalars$-module generated by
  foams in good position.
\end{lem}

The actions of $\de, \df, \dh$ defined above descend to an action on symmetric polynomials.
The next fact is a direct consequence of \cite[Proposition 4.3]{QRSW2} and says that the action of $\mathfrak{sl}_2$ on the space of foams is compatible with the $\mathfrak{sl}_2$-action on symmetric polynomials via foam evaluation.

\begin{prop}\label{prop:sl2-acts-good-pos}
  Let $\foam$ be a closed foam in good position, then for 
  $\mathbf{x} \in \mathfrak{sl}_2$: \begin{equation} \bracketN{\mathbf{x} \cdot \foam}= \mathbf{x} \cdot \bracketN{\foam} .\end{equation}
\end{prop}

\subsection{Green dotted webs and state spaces}
\label{subsec:greendottedwebs}
Using foam evaluation (as in \cite{Khsl3, RW1, QRSW2}), one can associate an
$\RN$-module, called a {\em $\gll_N$-state space $\statespaceN{\web}$} to each web
$\web$. To prevent overloaded diagrams, we will omit
$\statespaceN{\cdot}$ when dealing with actual webs or foams.

The $\mathfrak{sl}_2$-action on foams gives rise to an
$\mathfrak{sl}_2$-action on the state space $\statespaceN{\web}$.
Proposition \ref{prop:sl2-acts-good-pos} guarantees that this action
is well-defined.

For each edge $e$ with label $a$ in a web $\web$, the
  algebra 
  \begin{equation}
  D_e = \scalars[x_1, \dots, x_a,
  y_1,\dots,y_{N-a}]^{S_a\times S_{N-a}}
  \end{equation}
  acts on the $\gll_N$-state space $\statespaceN{\web}$ associated to $\web$ by adding a decoration on
  the facet bounding to the edge $e$. Consequently the algebra
  \begin{equation}
      D_\web:=\bigotimes_{e\in E(\web)} D_e 
  \end{equation} acts on
  $\statespaceN{\web}$. For each edge $e$, the algebra $D_e$ is
  endowed with a natural $\mathfrak{sl}_2$-module structure and is an
  $\mathfrak{sl}_2$-module algebra and thus so is $D_\web$. 

It is possible to twist the $\mathfrak{sl}_2$-action on
  $\statespaceN{\web}$.  This idea goes back to \cite{KRWitt} where
  the Witt-type action was twisted in order to make differentials
  occuring in link homology equivariant with respect to the Witt
  action.  Similar ideas were then used in \cite{QRSW1,QiSussanLink}
  in order to categorify the colored Jones polynomial in the context
  of a $p$-DG structure.  The twisting of the Witt-type action on
  foams was described in \cite{QRSW2}.  Here we review this twisting
  in the simplified setting of an $\mathfrak{sl}_2$-action.

  Let $\foam\co \web_1 \to \web_2$ be a foam.  Then
  $\tqftfunc\left(\foam\right)$ induces
a $\RN$-linear map from $\statespaceN{\web_1}$ to
  $\statespaceN{\web_2}$. Due to the Leibniz rule, this map
  intertwines the action of $\mathfrak{sl}_2$ if and only if
  $\tqftfunc\left(g \cdot\foam\right) =0$ for all
  $g\in \mathfrak{sl}_2$. This prevents many foams from being
  morphisms in categories taking into account the
  $\mathfrak{sl}_2$-action. Having in mind the construction of a link
  homology with an $\mathfrak{sl}_2$-action, we must circumvent this
  problem.  We introduce twists on state spaces. These twists are very
  similar to that in \cite{QRSW1}. However the more refined structures
  we are considering here forces us to introduce several types of
  twists. We follow the diagrammatic formalism of \cite{QRSW1} and
  extend the definition of green-dotted webs.

  \begin{dfn}
    A \emph{green-dotted} web is a web $\web$ endowed with
a finite collection $D$ of \emph{green dots}, that are marked
    points with multiplicities (in $\scalars$) located in the interior
    of edges of $\web$. These green dots are of two types $\gdot$ and
$ \gsoliddot$.  If a given edge carries several green dots of the
    same type, they may be replaced by one green dot of that type on
    this edge with the sum of all multiplicities.
      
\end{dfn}

\begin{exa}
The following green-dotted web will come up later when proving invariance of link homology under a Reidemeister II move.
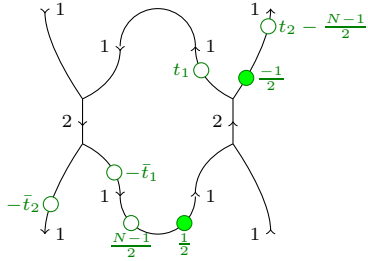
\[
\NB{\tikz[scale=1, font=\tiny]{\begin{scope}
  \coordinate (bl) at (-0.5, -1.5);
  \coordinate (br) at ( 0.5, -1);
  \coordinate (tr) at ( 0.5,  1);
  \coordinate (tl) at (-0.5,  1.5);
  % \coordinate (tr) at ( 0.5,  1);
  % \coordinate (bl) at (-0.5, -1);
  % \coordinate (br) at ( 0.5, -1);

  % \coordinate (tl) at (-0.5,  1);
  \coordinate (bm) at (  0,-0.3);
  \coordinate (tm) at (  0, 0.3);
  %%%
    \coordinate (bl2) at (1.5, -1);
    \coordinate (br2) at ( 2.5, -1.5);
%    \coordinate (br2) at ( 2.5, -1);
  \coordinate (bm2) at (  2,-0.3);
  \coordinate (tl2) at (1.5,  1);
  \coordinate (tr2) at ( 2.5,  1.5);
%  \coordinate (tr2) at ( 2.5,  1);
  \coordinate (tm2) at (  2, 0.3);
  %%%%%
  \draw[<-]  (bl) .. controls +( 0, 0.5) and +(0,0) .. (bm)
  node[right, pos = 0] {$1$} coordinate[pos = 0.25] (ha) ;
    \filldraw[draw= green!50!black, fill = white] (ha) circle (1mm)
  node[left, green!50!black] {$-\bar{t}_2$};
  \draw[<-]  (br) .. controls +( 0, 0.5) and +(0,0) .. (bm)
  node[left, pos = 0] {$1$} coordinate[pos = 0.25] (hb) ;
    \filldraw[draw= green!50!black, fill = white] (hb) circle (1mm)
  node[right, green!50!black] {$-\bar{t}_1$};
  \draw[>-]  (tl) .. controls +( 0, -0.5) and +(0,0) .. (tm)
  node[right, pos = 0] {$1$} coordinate[pos = 0.25] (ga) ;
  %  \filldraw[draw= green!50!black, fill = white] (ga) circle (1mm)
  %node[left, green!50!black] {$t_1$};
   \draw[>-]  (tr) .. controls +( 0, -0.5) and +(0,0) .. (tm)
  node[left, pos = 0] {$1$} coordinate[pos = 0.25] (gb) ;
 %   \filldraw[draw= green!50!black, fill = white] (gb) circle (1mm)
%  node[left, green!50!black] {$t_2$};
  \draw [-<-] (bm) -- (tm) node[left, pos = 0.5] {$2$};
%     \draw (.5, 1) arc (180:0:0.5) -- (1.5,-1) arc (180:0:-0.5);
\draw (.5, 1) arc (180:0:0.5);
\draw (1.5,-1) arc (180:0:-0.5) coordinate[pos = 0.25] (X) coordinate[pos = 0.75] (Y);
     %%%%%
 \draw[>-]  (bl2) .. controls +( 0, 0.5) and +(0,0) .. (bm2)
  node[right, pos = 0] {$1$};
  \draw[>-]  (br2) .. controls +( 0, 0.5) and +(0,0) .. (bm2)
  node[left, pos = 0] {$1$};
  \draw[<-]  (tl2) .. controls +( 0, -0.5) and +(0,0) .. (tm2)
  node[right, pos = 0] {$1$} coordinate[pos = 0.25] (ga2) ; 
      \filldraw[draw= green!50!black, fill = white] (ga2) circle (1mm)
  node[left, green!50!black] {${t}_1$};
    \draw [->-] (bm2) -- (tm2) node[left, pos = 0.5] {$2$};
  \draw[<-]  (tr2) .. controls +( 0, -0.5) and +(0,0) .. (tm2)
  node[left, pos = 0] {$1$} coordinate[pos = 0.15] (gb2) coordinate[pos = 0.6] (gb3);
   \filldraw[draw= green!50!black, fill = white] (gb2) circle (1mm)
  node[right, green!50!black] {${t}_2-\frac{N-1}{2}$};
    \filldraw[draw= green!50!black, fill = green] (gb3) circle (1mm)
  node[right, green!50!black] {$\frac{-1}{2}$};
    \filldraw[draw= green!50!black, fill = green] (X) circle (1mm)
  node[below, green!50!black] {$\frac{1}{2}$};
    \filldraw[draw= green!50!black, fill = white] (Y) circle (1mm)
  node[below, green!50!black] {$\frac{N-1}{2}$};
\end{scope}

}}
\]

\end{exa}

  Let $(\web, D)$ be a green-dotted web. For each
  green dot $d$ of multiplicity $\lambda \in \scalars$, define
  $\web_d $ to be the foam $\web\times [0,1]$ with a twisted action of
  $\mathfrak{sl}_2$.  Each green dot lives on an edge, hence each
  foam bounding $\web$ has a neighborhood of that green dots
  homeomorphic to $]0,1[\times [0,1[$. The twists induced by green
  dots is local and we depict the modified $\sll_2$-action on these
  neighborhoods. Recall that to act on a concrete foam, one uses the
  Leibniz rule, so that the twists induced by various green-dot on a
  given web add up.

\begin{gather}
\label{eq:e-act-pol-twisthollow} \de\left(\NB{\tikz[scale=1.5, font=\small]{\begin{scope}
  \draw[thin] (0,0.5) rectangle (1,1);
  \draw[thick] (0,1) -- (1,1)  coordinate [pos=0.4] (A);
  \filldraw[draw= green!50!black, fill = white] (A) circle (.5mm)
  node[below,  green!50!black] {$\lambda$};
\end{scope}}}\right) 
  =0 \\
\label{eq:e-act-pol-twistsolid}  \de\left(\NB{\tikz[scale=1.5, font=\small]{\begin{scope}
  \draw[thin] (0,0.5) rectangle (1,1);
  \draw[thick] (0,1) -- (1,1)  coordinate [pos=0.4] (A);
  \filldraw[draw= green!50!black, fill = green] (A) circle (.5mm)
  node[below,  green!50!black] {$\lambda$};
\end{scope}}}\right)=
0 \\
\label{eq:h-act-pol-twisthollow} \dh\left(\NB{\tikz[scale=1.5, font=\small]{}}\right) 
  =-\lambda\cdot\  \NB{\tikz[scale=1.5, font=\small]{\begin{scope}
  \draw[thin] (0,0.5) rectangle (1,1);
  \draw[thick] (0,1) -- (1,1)  coordinate [pos=0.4] (A);
  \filldraw[draw= green!50!black, fill = white] (A) circle (.5mm)
  node[below,  green!50!black] {$\lambda$};
  \node (B) at (.75,.75) {$\dotnewtoni[0]$};
\end{scope}

%\begin{scope}
%  \draw (0,0) rectangle (1,1) coordinate [midway] (A);
%  \fill (A) 
%  %circle (0.5mm) 
%  node[below] {$\dotnewtoni[0]$};
%\end{scope}
}} \\
\label{eq:h-act-pol-twistsolid}  \dh\left(\NB{\tikz[scale=1.5, font=\small]{}}\right)=
-\lambda \cdot\  \NB{\tikz[scale=1.5, font=\small]{\begin{scope}
  \draw[thin] (0,0.5) rectangle (1,1);
  \draw[thick] (0,1) -- (1,1)  coordinate [pos=0.4] (A);
  \filldraw[draw= green!50!black, fill = green] (A) circle (.5mm)
  node[below,  green!50!black] {$\lambda$};
  \node (B) at (.75,.75)  {$\wdotnewtoni[0]$};
\end{scope}}} \\
\label{eq:f-act-pol-twisthollow} \df\left(\NB{\tikz[scale=1.5, font=\small]{}}\right) 
  =\lambda\cdot\  \NB{\tikz[scale=1.5, font=\small]{\begin{scope}
  \draw[thin] (0,0.5) rectangle (1,1);
  \draw[thick] (0,1) -- (1,1)  coordinate [pos=0.4] (A);
  \filldraw[draw= green!50!black, fill = white] (A) circle (.5mm)
  node[below,  green!50!black] {$\lambda$};
  \node (B) at (.75,.75) {$\dotnewtoni[1]$};
\end{scope}

%\begin{scope}
%  \draw (0,0) rectangle (1,1) coordinate [midway] (A);
%  \fill (A) 
%  %circle (0.5mm) 
%  node[below] {$\dotnewtoni[0]$};
%\end{scope}
}} \\
\label{eq:f-act-pol-twistsolid}  \df\left(\NB{\tikz[scale=1.5, font=\small]{}}\right)=
\lambda \cdot\  \NB{\tikz[scale=1.5, font=\small]{\begin{scope}
  \draw[thin] (0,0.5) rectangle (1,1);
  \draw[thick] (0,1) -- (1,1)  coordinate [pos=0.4] (A);
  \filldraw[draw= green!50!black, fill = green] (A) circle (.5mm)
  node[below,  green!50!black] {$\lambda$};
  \node (B) at (.75,.75)  {$\wdotnewtoni[1]$};
\end{scope}}}  .
\end{gather}

It is convenient to introduce floating green dots on the plane: one
can view them as being on edges of thickness $0$. A floating hollow
green dot $\gdot$ does not alter the Lie algebra action since in this context
$\dotnewtoni[0]=\dotnewtoni[1]=0$. However a solid green dot
$\gsoliddot$ twists the action of $\df$ by $E_1$ and the action of
$\dh$ by $-N$ (because $\wdotnewtoni[0]=N$ and $\wdotnewtoni[1]=E_1$)

With this new convention, one has the following local relation:

\[
\NB{\tikz[]{\begin{scope}
  \coordinate (m) at (  0,  0);
  \coordinate (t) at (  2, 0);
  \draw[->] (m) -- (t) node[pos = 1, right] {$a$} coordinate[pos=0.3]
  (ga) coordinate[pos=0.7] (gb) ;
  \filldraw[draw= green!50!black, fill = green] (ga) circle (1mm)
  node[yshift= 2mm, green!50!black] {$r$};
  \filldraw[draw= green!50!black, fill = white] (gb) circle (1mm)
  node[yshift=2mm, green!50!black] {$r$};
\end{scope}
%%% Local Variables:
%%% mode: latex
%%% TeX-master: t
%%% End:
}} = \NB{\tikz[]{\begin{scope}
  \coordinate (m) at (  0,  0);
  \coordinate (t) at (  2, 0);
  \draw[->] (m) -- (t) node[pos = 1, right] {$a$} coordinate [pos=0.5,
  yshift = 2mm]
  (ga);
  \filldraw[draw= green!50!black, fill = green, above] (ga) circle (1mm)
  node[left, green!50!black] {$r$};
  % \filldraw[draw= green!50!black, fill = white] (gb) circle (1mm)
  % node[above, green!50!black] {$r$};
\end{scope}
%%% Local Variables:
%%% mode: latex
%%% TeX-master: t
%%% End:
}}.
\]

\subsection{About twists}
\renewcommand{\gg}{\ensuremath{\mathfrak{g}}}

The
$\gll_N$-state spaces are endowed with an $\sll_2$-action and a
$D_\web$-action. These two actions intertwine nicely. This allows us to
twist the first one by the second as we shall see. In order to show
what structures are at play, we will formulate this in a slightly more
general context.

Let $\gg$ be a Lie algebra, $A$ a commutative $\gg$-module algebra and $M$ an
$A\# \gg$-module\footnote{In the notation of Section \ref{homological:sec}, it should really be $A\# \mc{U}(\gg)$, where  $\mc{U}(\gg)$ denotes the universal enveloping algebra of $\gg$. However, we will write this smash product algebra as $A\# \gg$ for simplification.}, that a $\gg$-module with an $A$-module structure, which
satisfies the following identities for all $g$ in $\gg$, $a$ in $A$
and $m$ in $M$:
\[g\cdot_{\gg}(a\cdot_{A} m) = (g\cdot a)\cdot_A m +
  a\cdot_{A}(g\cdot_{\gg} m).\]
In our context, $\gg =\sll_2$, $M$ is a $\gll_N$-state space of a web $\web$
 and $A$ is $D_\web$, the polynomial algebra of decorations of $\Gamma$.

A linear map $\tau\co \gg \to A$ is \emph{flat} if for all $g_1, g_2$
in $\gg$, one has:
\begin{equation}
\tau([g_1,g_2]) = g_1\cdot \tau(g_2) - g_2\cdot \tau(g_1).
\end{equation}
If $\tau$ is flat, one can check that there is a new $A\# \gg$-module structure on the rank-one free module $A$ itself defined by
\begin{equation}
    g\cdot_{\gg^\tau} a:= g\cdot_\gg a+\tau(g)a.
\end{equation}
More generally, the following formula
defines a new $\gg$-action on any $A\# \gg$-module $M$ via:
\begin{equation}
g\cdot_{\gg^\tau} m := g\cdot_{\gg} m + \tau(g)\cdot_A m. 
\end{equation}
One also checks that the twisted action $\cdot_{\gg^\tau}$, similar as $\cdot_\gg$, is compatible with the $A$-action. In other words, this gives $M$ a new $A\#\gg$-module
structure. If $\tau$ is flat, we denote by $M^{\tau}$ the module endowed with the
action $\cdot_{\gg^{\tau}}$.
One can readily see that, as $A\# \gg$-modules,
\begin{equation}
    M^\tau = A^\tau\otimes_A M.
\end{equation}
This means in particular that if
$\tau, \sigma \co \gg \to A$ are flat, then 
$(M^\tau)^\sigma \cong M^{\tau+\sigma}$.

  \begin{prop}\label{prop:classification-sl2-twists}
    For any green-dotted web $(\web, D)$,
    twisting the actions of $\de,\df,$ and $\dh$ as above
    endows $\statespaceN{\web}$ with an $\mathfrak{sl}_2$-module structure.
  \end{prop}
  \begin{proof}

From the earlier discussion, we know that if there are no green
    dots, then the state space carries an $\mathfrak{sl}_2$-action.
    
    By additivity of twists, it is enough to prove the proposition for
    the case that a single edge $e$ of thickness $a$ contains a hollow dot
    labeled $\alpha$ and a solid dot labeled  $\beta$.  Let
    $\tau:\sll_2 \to D_e \subseteq D_\Gamma$ be the map encoded by these two green
    dots:
    \begin{equation}
      \label{eq:twist-gd-def}
      \tau(\de) = 0,\qquad \tau(\dh) = -a\alpha -(N-a)\beta,\qquad
      \tau(\df) = \alpha \dotnewtoni[1] + \beta \wdotnewtoni[1].
    \end{equation}
    Checking flatness of $\tau$ is a straightforward computation:
    \begin{align*}
      \tau([\dh,\de]) - \dh\cdot\tau(\de) + \de\cdot \tau(\dh)
      &= 2\tau(\de) - 0 + \de\cdot(-a\alpha -(N-a)\beta) \\ &=0.\\
      \tau([\dh,\df]) - \dh\cdot\tau(\df) + \df\cdot \tau(\dh)
      &= -2\tau(\df) - \dh\cdot (\alpha \dotnewtoni[1] + \beta \wdotnewtoni[1])
        + \df\cdot(-a\alpha -(N-a)\beta) \\
      &= -2(\alpha \dotnewtoni[1] + \beta \wdotnewtoni[1]) +2 (\alpha
        \dotnewtoni[1] + \beta \wdotnewtoni[1]) + 0 \\ &=0.\\
        \tau([\de, \df]) - \de\cdot\tau(\df) + \df\cdot \tau(\de) 
      &= \tau(\dh)- \de\cdot(\alpha \dotnewtoni[1] + \beta
        \wdotnewtoni[1]) \\ &=
        -a\alpha -(N-a)\beta + (\alpha \dotnewtoni[0] +
        \beta\wdotnewtoni[0]) \\ &=0.
    \end{align*}
    The result follows.
  \end{proof}

Alternatively, if $e$ is an edge of thickness $a$ in a web $\web$, the
  algebra 
  $$D_e = \scalars[x_1, \dots, x_a,
  y_1,\dots,y_{N-a}]^{S_a\times S_{N-a}}$$ 
  is an $\sll_2$-module algebra with
   $\de$, $\df$ and $\dh$ acting by differential operators
  \begin{equation}
      \de  =-\sum_{i=1}^a \frac{\partial}{\partial x_i}-\sum_{j=1}^{N-a} \frac{\partial}{\partial y_j}, \quad \quad 
      \df  =\sum_{i=1}^a x_i^2\frac{\partial}{\partial x_i}+\sum_{j=1}^{N-a} y_j^2\frac{\partial}{\partial y_j},
  \end{equation}
  \begin{equation}
 \dh =-2\sum_{i=1}^a x_i\frac{\partial}{\partial x_i}-2\sum_{j=1}^{N-a} y_j\frac{\partial}{\partial y_j}.
  \end{equation}
 Proposition \ref{prop:classification-sl2-twists} shows that there is a $2$-parameter $\sll_2$-equivariant $D_e$-bimodule structure on the rank-one module $D_e \cdot v_{\alpha,\beta}$, for some $\alpha,\beta\in \Bbbk$. Here $\de$, $\df$ and $\dh$ act on the module generator $v_{\alpha,\beta}$ by
 \begin{equation} \label{eqn-sl2-twist-on-mod}
     \de (v_{\alpha,\beta})=0, \quad \dh (v_{\alpha,\beta})=-(a\alpha+(N-a)\beta)v_{\alpha,\beta}, \quad \df(v_{\alpha,\beta})= (\alpha \dotnewtoni[1] + \beta \wdotnewtoni[1])v_{\alpha,\beta}.
 \end{equation}
This equivariant bimodule gives rise to the twist endo-functor on $D_e\# \sll_2$-modules as the tensor product
$   (D_e\cdot v_{\alpha,\beta})\otimes_{D_e}(\mbox{-}) $.

Basic properties of symmetric functions allow green dots to migrate
along a web in the following ways, which we record for later use.  See also \cite[Lemma
3.11]{QRSW1}. These manipulation of green dots, will be referred to as
\emph{green dot migration}.
  \begin{equation}
    \NB{\tikz[scale=0.45,font=\tiny]{\begin{scope}
  \coordinate (m) at (  0,  0);
  \coordinate (t) at (  0, 1);
  \coordinate (br) at (+.5,  -1);
  \coordinate (bl) at (-.5,  -1);
  \draw[>-] (bl) .. controls +(0,0.5) and + (0, 0) .. (m) node[pos =
  0, below] {$a$} coordinate[pos = 0.3] (ga);
  \draw[>-] (br) .. controls +(0,0.5) and + (0, 0) .. (m) node[pos =
  0, below] {$b$} coordinate[pos = 0.3] (gb);
  \draw[->] (m) -- (t) node[pos = 1, above] {$a+b$};
  \filldraw[draw= green!50!black, fill = white] (ga) circle (1mm)
  node[left, green!50!black] {$r$};
  \filldraw[draw= green!50!black, fill = white] (gb) circle (1mm)
  node[right, green!50!black] {$r$};
\end{scope}}} =
    \NB{\tikz[scale=0.45,font=\tiny]{\begin{scope}
  \coordinate (m) at (  0,  0);
  \coordinate (t) at (  0, 1);
  \coordinate (br) at (+.5,  -1);
  \coordinate (bl) at (-.5,  -1);
  \draw[>-] (bl) .. controls +(0,0.5) and + (0, 0) .. (m) node[pos =
  0, below] {$a$} coordinate[pos = 0.3] (ga);
  \draw[>-] (br) .. controls +(0,0.5) and + (0, 0) .. (m) node[pos =
  0, below] {$b$} coordinate[pos = 0.3] (gb);
  \draw[->] (m) -- (t) node[pos = 1, above] {$a+b$} coordinate[pos = 0.7] (gc);
  %\filldraw[draw= green!50!black, fill = white] (ga) circle (1mm)
  %node[left, green!50!black] {$r$};
  %\filldraw[draw= green!50!black, fill = white] (gb) circle (1mm)
 % node[right, green!50!black] {$r$};
    \filldraw[draw= green!50!black, fill = white] (gc) circle (1mm)
  node[right, green!50!black] {$r$};
\end{scope}}} , \quad 
   \NB{\tikz[scale=0.45,font=\tiny]{\begin{scope}
  \coordinate (m) at (  0,  0);
  \coordinate (b) at (  0, -1);
  \coordinate (tr) at (+.5,  1);
  \coordinate (tl) at (-.5,  1);
  \draw[->] (m) .. controls +(0,0) and + (0, -0.5) .. (tl) node[pos =
  1, above] {$a$} coordinate[pos = 0.7] (ga);
  \draw[->] (m) .. controls +(0,0) and + (0, -0.5) .. (tr) node[pos =
  1, above] {$b$} coordinate[pos = 0.7] (gb);
  \draw[>-] (b) -- (m) node[pos = 0, below] {$a+b$} coordinate[pos = 0.3] (gc);
  %\filldraw[draw= green!50!black, fill = white] (ga) circle (1mm)
  %node[left, green!50!black] {$r$};
  %\filldraw[draw= green!50!black, fill = white] (gb) circle (1mm)
  %node[right, green!50!black] {$r$};
   \filldraw[draw= green!50!black, fill = white] (gc) circle (1mm)
  node[left, green!50!black] {$r$};
\end{scope}}} =
   \NB{\tikz[scale=0.45,font=\tiny]{\begin{scope}
  \coordinate (m) at (  0,  0);
  \coordinate (b) at (  0, -1);
  \coordinate (tr) at (+.5,  1);
  \coordinate (tl) at (-.5,  1);
  \draw[->] (m) .. controls +(0,0) and + (0, -0.5) .. (tl) node[pos =
  1, above] {$a$} coordinate[pos = 0.7] (ga);
  \draw[->] (m) .. controls +(0,0) and + (0, -0.5) .. (tr) node[pos =
  1, above] {$b$} coordinate[pos = 0.7] (gb);
  \draw[>-] (b) -- (m) node[pos = 0, below] {$a+b$};
  \filldraw[draw= green!50!black, fill = white] (ga) circle (1mm)
  node[left, green!50!black] {$r$};
  \filldraw[draw= green!50!black, fill = white] (gb) circle (1mm)
  node[right, green!50!black] {$r$};
\end{scope}}} ,\quad
   \NB{\tikz[scale=0.45,font=\tiny]{\begin{scope}
  \coordinate (m) at (  0,  0);
  \coordinate (t) at (  0, 1);
  \coordinate (br) at (+.5,  -1);
  \coordinate (bl) at (-.5,  -1);
  \draw[>-] (bl) .. controls +(0,0.5) and + (0, 0) .. (m) node[pos =
  0, below] {$a$} coordinate[pos = 0.3] (ga);
  \draw[>-] (br) .. controls +(0,0.5) and + (0, 0) .. (m) node[pos =
  0, below] {$b$} coordinate[pos = 0.3] (gb);
  \draw[->] (m) -- (t) node[pos = 1, above] {$a+b$};
  \filldraw[draw= green!50!black, fill = green] (ga) circle (1mm)
  node[left, green!50!black] {$r$};
  % \filldraw[draw= green!50!black, fill = green] (gb) circle (1mm)
  % node[right, green!50!black] {$r$};
\end{scope}}} =
   \NB{\tikz[scale=0.45,font=\tiny]{\begin{scope}
  \coordinate (m) at (  0,  0);
  \coordinate (t) at (  0, 1);
  \coordinate (br) at (+.5,  -1);
  \coordinate (bl) at (-.5,  -1);
  \draw[>-] (bl) .. controls +(0,0.5) and + (0, 0) .. (m) node[pos =
  0, below] {$a$} coordinate[pos = 0.3] (ga);
  \draw[>-] (br) .. controls +(0,0.5) and + (0, 0) .. (m) node[pos =
  0, below] {$b$} coordinate[pos = 0.3] (gb);
  \draw[->] (m) -- (t) node[pos = 1, above] {$a+b$} coordinate[pos = 0.7] (gc);
  %\filldraw[draw= green!50!black, fill = white] (ga) circle (1mm)
  %node[left, green!50!black] {$r$};
  \filldraw[draw= green!50!black, fill=white] (gb) circle (1mm)
  node[right, green!50!black] {$r$};
  \filldraw[draw= green!50!black, fill = green] (gc) circle (1mm)
  node[right, green!50!black] {$r$};
\end{scope}}} , \quad
  \NB{\tikz[scale=0.45,font=\tiny]{\begin{scope}
  \coordinate (m) at (  0,  0);
  \coordinate (b) at (  0, -1);
  \coordinate (tr) at (+.5,  1);
  \coordinate (tl) at (-.5,  1);
  \draw[->] (m) .. controls +(0,0) and + (0, -0.5) .. (tl) node[pos =
  1, above] {$a$} coordinate[pos = 0.7] (ga);
  \draw[->] (m) .. controls +(0,0) and + (0, -0.5) .. (tr) node[pos =
  1, above] {$b$} coordinate[pos = 0.7] (gb);
  \draw[>-] (b) -- (m) node[pos = 0, below] {$a+b$};
  \filldraw[draw= green!50!black, fill = green] (ga) circle (1mm)
  node[left, green!50!black] {$r$};
  % \filldraw[draw= green!50!black, fill = green] (gb) circle (1mm)
  % node[right, green!50!black] {$r$};
\end{scope}}} =
  \NB{\tikz[scale=0.45,font=\tiny]{\begin{scope}
  \coordinate (m) at (  0,  0);
  \coordinate (b) at (  0, -1);
  \coordinate (tr) at (+.5,  1);
  \coordinate (tl) at (-.5,  1);
  \draw[->] (m) .. controls +(0,0) and + (0, -0.5) .. (tl) node[pos =
  1, above] {$a$} coordinate[pos = 0.7] (ga);
  \draw[->] (m) .. controls +(0,0) and + (0, -0.5) .. (tr) node[pos =
  1, above] {$b$} coordinate[pos = 0.7] (gb);
  \draw[>-] (b) -- (m) node[pos = 0, below] {$a+b$} coordinate[pos = 0.3] (gc);
  %\filldraw[draw= green!50!black, fill = white] (ga) circle (1mm)
  %node[left, green!50!black] {$r$};
  \filldraw[draw= green!50!black, fill = white] (gb) circle (1mm)
  node[right, green!50!black] {$r$};
   \filldraw[draw= green!50!black, fill = green] (gc) circle (1mm)
  node[left, green!50!black] {$r$};
\end{scope}}} .
\end{equation}

\subsection{Useful morphisms}
The next few lemmas are foam versions of $\mathfrak{sl}_2$-equivariant
maps of Soergel bimodules in \cite[Section 2.5]{QRSW2} or \cite[Lemma 3.21]{QRSW1} where
$t_1=t_2=0$ and we omit proofs of these basic facts.  These morphisms
play an important role in the link homology to be introduced in the
next section. 

\begin{lem}
  The two different orientations of the foam
  \[
    \NB
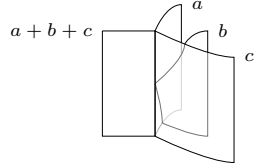
{\tikz[scale=0.7,font=\tiny]{\begin{scope}
  \coordinate (ABCt) at (0,2);
  \coordinate (Ot) at (1,2);
  \coordinate (At) at (1.5,2.5);
  \coordinate (Bt) at (2,2);
  \coordinate (Ct) at (2.5,1.5);
  \begin{scope}[yshift= -2cm]
    \coordinate (ABCb) at (0,2);
    \coordinate (Ob) at (1,2);
    \coordinate (Ab) at (1.5,2.5);
    \coordinate (Bb) at (2,2);
    \coordinate (Cb) at (2.5,1.5);
  \end{scope}
  \begin{scope}[fill = white, fill opacity = 0.45, draw = black,
    draw opacity =1, very thin]
    \filldraw (Ob) .. controls +(0,0) and +(-0.3, 0) .. (Ab)
    coordinate[pos=0.5] (ABb) -- (At)
    .. controls +(-0.3, 0) and +(0,0) .. (Ot) -- (Ob);
    \fill (Ob) .. controls +(0,0) and +(-0.5, 0) .. (Cb) -- (Ct)
    .. controls +(-0.5, 0) and +(0,0) .. (Ot) coordinate[pos=0.5]
    (BCt) -- (Ob) coordinate[pos=0.5] (Om);
    \filldraw (ABb) .. controls +(0,0) and +(-0.3, 0) .. (Bb)-- (Bt)
    .. controls +(-0.3, 0) and +(0,0) .. (BCt) .. controls +(0, -0.3) and
    +(0,0) .. (Om) .. controls +(0,0) and +(0, 0.3)  .. (ABb);
    \filldraw (Ob) .. controls +(0,0) and +(-0.5, 0) .. (Cb) -- (Ct)
    .. controls +(-0.5, 0) and +(0,0) .. (Ot) -- (Ob);
    \filldraw (ABCt) -- (Ot) -- (Ob) -- (ABCb) -- (ABCt);
  \end{scope}
  
  \begin{scope}
    \node[left] at (ABCt) {$a+b+c$};
    \node[right] at (At) {$a$};
    \node[right] at (Bt) {$b$};
    \node[right] at (Ct) {$c$};
  \end{scope}
\end{scope}

%%% Local Variables:
%%% mode: latex
%%% TeX-master: t
%%% End:
}}
  \]
  induce isomorphisms
\[
      \begin{array}{crcl}
        \mapA \colon\thinspace & \NB{\tikz[font=\tiny,
                                 scale =0.65]{\begin{scope}
  \coordinate (b) at ( 0,-0.50);
  \coordinate (m) at (0,0);
  \coordinate (tr) at (+1,  1);
  \coordinate (tm) at (0,  1);
  \coordinate (tl) at (-1,  1);
  \draw[>-, >= to] (b) -- (m) node[pos=0, below] {$a+b+c$};
  \draw[-to] (m) .. controls +(0,0) and + (0, -0.5) .. (tl) node[pos =
  1, above] {$a$} coordinate[pos= 0.5] (ml) coordinate[pos = 0.25] (lml);
  \draw[-to] (m) .. controls +(0,0) and + (0, -0.5) .. (tr) node[pos =
  1, above] {$c$};
  \draw[-to] (ml)  .. controls +(0,0) and + (0, -0.5) .. (tm) node[pos =
  1, above] {$b$};
  \node[left] (lml) {$a+b$};
\end{scope}}}  &
                                                                      \to &  \NB{\tikz[font=\tiny, scale =0.65]{\begin{scope}
  \coordinate (b) at ( 0,-0.50);
  \coordinate (m) at (0,0);
  \coordinate (tr) at (+1,  1);
  \coordinate (tm) at (0,  1);
  \coordinate (tl) at (-1,  1);
  \draw[>-, >= to] (b) -- (m) node[pos=0, below] {$a+b+c$};
  \draw[-to] (m) .. controls +(0,0) and + (0, -0.5) .. (tl) node[pos =
  1, above] {$a$};
  \draw[-to] (m) .. controls +(0,0) and + (0, -0.5) .. (tr) node[pos =
  1, above] {$c$}  coordinate[pos= 0.5] (mr) coordinate[pos = 0.25] (lmr);
  \draw[-to] (mr)  .. controls +(0,0) and + (0, -0.5) .. (tm) node[pos =
  1, above] {$b$};
  \node[right] (lmr) {$b+c$};
\end{scope}}}   \\
\end{array}
\quad
      and
      \quad
\begin{array}{crcl}
        \mapA \colon\thinspace & \NB{\tikz[font=\tiny, scale =0.65]{\begin{scope}
  \coordinate (t) at ( 0,0.50);
  \coordinate (m) at (0,0);
  \coordinate (br) at (+1,  -1);
  \coordinate (bm) at ( 0,  -1);
  \coordinate (bl) at (-1,  -1);
  \draw[-to] (m) -- (t) node[pos=1, above] {$a+b+c$};
  \draw[>-, >=to] (bl) .. controls +(0,0.5) and + (0, 0) .. (m) node[pos =
  0, below] {$a$} coordinate[pos= 0.5] (ml) coordinate[pos = 0.75] (lml);
  \draw[>-, >=to] (br) .. controls +(0,0.5) and + (0, -0) .. (m) node[pos =
  0, below] {$c$};
  \draw[>-, >=to] (bm)  .. controls +(0,0.5) and + (0, 0) .. (ml) node[pos =
  0, below] {$b$};
  \node[left] (lml) {$a+b$};
\end{scope}}}  & \to &
     \NB{\tikz[font=\tiny, scale =0.65]{\begin{scope}
  \coordinate (t) at ( 0,0.50);
  \coordinate (m) at (0,0);
  \coordinate (br) at (+1,  -1);
  \coordinate (bm) at ( 0,  -1);
  \coordinate (bl) at (-1,  -1);
  \draw[-to] (m) -- (t) node[pos=1, above] {$a+b+c$};
  \draw[>-, >=to] (bl) .. controls +(0,0.5) and + (0, 0) .. (m) node[pos =
  0, below] {$a$};
  \draw[>-, >=to] (br) .. controls +(0,0.5) and + (0, -0) .. (m) node[pos =
  0, below] {$c$}  coordinate[pos= 0.5] (mr) coordinate[pos = 0.75] (lmr);
  \draw[>-, >=to] (bm)  .. controls +(0,0.5) and + (0, 0) .. (mr) node[pos =
  0, below] {$b$};
  \node[right] (lmr) {$b+c$};
\end{scope}}}  \\
\end{array}
    \]
    of $\mathfrak{sl}_2$-equivariant $\gll_N$-state spaces associated with webs. Their inverses are also
    denoted by $\mapA$.
\end{lem}

\begin{lem}
  The foams
 \[
   \NB{\tikz[scale=1,font=\tiny]{\begin{scope}%[font=\tiny]
  \begin{scope}
    \coordinate (Lt) at (0,2);
    \coordinate (Rt) at (2,2);
    \coordinate (lt) at (0.5,2);
    \coordinate (rt) at (1.5,2);
    \coordinate (Lb) at (0,1);
    \coordinate (Rb) at (2,1);
  \end{scope}
  \draw (lt) .. controls +(0.3, 0.3) and +(-.3, 0.3) ..  (rt) node
  [pos=0.5, yshift= 0.1cm] {$a$};
  \draw (lt) .. controls +(0.3, -.3) and +(-.3, -.3) ..  (rt) node
  [pos=0.5, yshift= -0.1cm] {$b$};
  \draw (lt) .. controls +(0 , -.9) and +(0, -.9) ..  (rt);
  \draw (rt) -- (Rt) node[pos=1, right] {$a+b$} -- (Rb) -- (Lb) -- (Lt) -- (lt);
\end{scope}

%%% Local Variables:
%%% mode: latex
%%% TeX-master: t
%%% End:
}} \qquad \text{and} \qquad
   \NB{\tikz[scale=1,font=\tiny]{\begin{scope}[yscale=-1, font=\tiny]
  \begin{scope}
    \coordinate (Lt) at (0,2);
    \coordinate (Rt) at (2,2);
    \coordinate (lt) at (0.5,2);
    \coordinate (rt) at (1.5,2);
    \coordinate (Lb) at (0,1);
    \coordinate (Rb) at (2,1);
  \end{scope}
  \draw (lt) .. controls +(0.3, 0.3) and +(-.3, 0.3) ..  (rt) node
  [pos=0.5, yshift= -0.1cm] {$a$};
  \draw[white!50!black] (lt) .. controls +(0.3, -.3) and +(-.3, -.3) ..  (rt) node
  [pos=0.5, yshift= 0.1cm, black] {$b$};
  \draw (lt) .. controls +(0 , -.9) and +(0, -.9) ..  (rt);
  \draw (rt) -- (Rt) node[pos=1, right] {$a+b$} -- (Rb) -- (Lb) -- (Lt) -- (lt);
\end{scope}

%%% Local Variables:
%%% mode: latex
%%% TeX-master: t
%%% End:
}} 
  \]
  induce morphisms
    \[
      \begin{array}{crcl}
        \mapB \colon\thinspace &  \NB{\tikz[font=\tiny, scale =0.7]{\begin{scope}
  \coordinate (bm) at ( 0, -1);
  \coordinate (tm) at ( 0, 1);
  \draw[>->] (bm) -- (tm) node[above, pos =1] {$a+b$} node[below, pos =0] {$a+b$};
\end{scope}}}  & \to &  \NB{\tikz[font=\tiny, scale =0.7]{\begin{scope}
  \coordinate (bm) at ( 0, -1);
  \coordinate (cm) at ( 0, -0.6);
  \coordinate (sm) at ( 0,  0.6);
  \coordinate (tm) at ( 0, 1);
  \draw[->] (sm) -- (tm) node[above, pos =1] {$a+b$};
  \draw[>-] (bm) -- (cm) node[below, pos =0] {$a+b$};
  \draw[->-] (cm) .. controls + ( 0.6, 0.6) and + ( 0.6, -0.6) .. (sm)
  node [pos = 0.5, right] {$b$} coordinate[pos = 0.8] (gb);
  \draw[->-] (cm) .. controls + (-0.6, 0.6) and + (-0.6, -0.6) .. (sm)
  node [pos = 0.5,  left] {$a$} coordinate[pos = 0.2] (ga);
  \filldraw[draw= green!50!black, fill = white] (ga) circle (1mm)
  node[left, green!50!black] {$b {t}_1$};
  \filldraw[draw= green!50!black, fill = white] (gb) circle (1mm)
  node[right,  green!50!black] {$a {t}_2$};
\end{scope}}}   \end{array}
\quad
      \text{and}
      \quad
\begin{array}{crcl}
        \mapU \colon\thinspace
        & \NB{\tikz[font=\tiny, scale =0.7]{\begin{scope}
  \coordinate (bm) at ( 0, -1);
  \coordinate (cm) at ( 0, -0.6);
  \coordinate (sm) at ( 0,  0.6);
  \coordinate (tm) at ( 0, 1);
  \draw[->] (sm) -- (tm) node[above, pos =1] {$a+b$};
  \draw[>-] (bm) -- (cm) node[below, pos =0] {$a+b$};
  \draw[->-] (cm) .. controls + ( 0.6, 0.6) and + ( 0.6, -0.6) .. (sm)
  node [pos = 0.5, right] {$b$} coordinate[pos = 0.8] (gb);
  \draw[->-] (cm) .. controls + (-0.6, 0.6) and + (-0.6, -0.6) .. (sm)
  node [pos = 0.5,  left] {$a$} coordinate[pos = 0.2] (ga);
  \filldraw[draw= green!50!black, fill = white] (ga) circle (1mm)
  node[left, green!50!black] {$-b \bar{t}_1$};
  \filldraw[draw= green!50!black, fill = white] (gb) circle (1mm)
  node[right,  green!50!black] {$-a \bar{t}_2$};
\end{scope}}}
        & \to 
        &   \NB{\tikz[font=\tiny, scale =0.7]{}}   \end{array}
    \]
of $\mathfrak{sl}_2$-equivariant $\gll_N$-state spaces associated with webs.
\end{lem}

\begin{lem}
    The foams
 \[
   \NB{\tikz[xscale=1.4,font=\tiny]{\begin{scope}[xscale=0.5,yscale=0.5, font=\tiny]
  \begin{scope}
    \coordinate (At) at (0,0);
    \coordinate (Bt) at (0.5,1);
    \coordinate (Ct) at (3,1);
    \coordinate (Dt) at (2.5,0);
    \coordinate (Et) at (1,0.1);
    \coordinate (Ft) at (2,0.9);
  \end{scope}
  \coordinate (O) at (1.5, -0.5);
  \begin{scope}[yshift= -2cm]
    \coordinate (Ab) at (0,0);
    \coordinate (Bb) at (0.5,1);
    \coordinate (Cb) at (3,1);
    \coordinate (Db) at (2.5,0);
    \coordinate (Eb) at (0.8,0.5);
    \coordinate (Fb) at (2.2,0.5);
  \end{scope}

  \draw (Bt) -- (Bb) -- (Eb) -- (Fb) node[pos=0.5, below] {$a+b+c$}-- (Cb) -- (Ct) -- (Ft) -- (O);
  \filldraw[draw =  black, fill = white, fill opacity=0.5] (At) --
  (Et) -- (Dt) -- (Db) -- (Fb) -- (O) -- (Eb) -- (Ab) -- (At);
  \draw (Bt) -- (Ft)-- (Et) node[pos=0.5, xshift =-0.2cm] {$c$} -- (O);

  \begin{scope}
    \node[xshift = -0.3cm] at (Ab) {$b+c$};
    \node[xshift = -0.1cm] at (Bt) {$a$};
    \node[xshift = 0.3cm] at (Ct) {$a+c$};
    \node[xshift = 0.1cm] at (Db) {$b$};
    
  \end{scope}
  % \draw (At) -- (Et) -- (Dt);
  % \draw (Bt) -- (Ft) -- (Ct);
  % \draw (Et) -- (Ft);

  % \draw (At) -- (Et) -- (Bt);
  % \draw (Ct) -- (Ft) -- (Ct);
  % \draw (Et) -- (Ft);

%  \draw (Ab) -- (At)

\end{scope}

%%% Local Variables:
%%% mode: latex
%%% TeX-master: t
%%% End:
}} \qquad \text{and} \qquad
   \NB{\tikz[xscale=1.4,font=\tiny]{\begin{scope}[yscale=0.5, xscale=0.5, font=\tiny]
  \begin{scope}
    \coordinate (At) at (0,0);
    \coordinate (Bt) at (0.5,1);
    \coordinate (Ct) at (3,1);
    \coordinate (Dt) at (2.5,0);
    \coordinate (Et) at (2,0.1);
    \coordinate (Ft) at (1,0.9);
  \end{scope}
  \coordinate (O) at (1.5, -0.5);
  \begin{scope}[yshift= -2cm]
    \coordinate (Ab) at (0,0);
    \coordinate (Bb) at (0.5,1);
    \coordinate (Cb) at (3,1);
    \coordinate (Db) at (2.5,0);
    \coordinate (Eb) at (0.8,0.5);
    \coordinate (Fb) at (2.2,0.5);
  \end{scope}

  \draw (Bt) -- (Bb) -- (Eb) -- (Fb) node[pos=0.5, below] {$a+b+c$}-- (Cb) -- (Ct) -- (Ft) -- (O);
  \filldraw[draw =  black, fill = white, fill opacity=0.5] (At) --
  (Et) -- (Dt) -- (Db) -- (Fb) -- (O) -- (Eb) -- (Ab) -- (At);
  \draw (Bt) -- (Ft)-- (Et) node[pos=0.5, xshift =0.2cm] {$c$} -- (O);

  \begin{scope}
    \node[xshift = -0.1cm] at (Ab) {$b$};
    \node[xshift = -0.3cm] at (Bt) {$a+c$};
    \node[xshift = 0.1cm] at (Ct) {$a$};
    \node[xshift = 0.3cm] at (Db) {$b+c$};
    
  \end{scope}
  % \draw (At) -- (Et) -- (Dt);
  % \draw (Bt) -- (Ft) -- (Ct);
  % \draw (Et) -- (Ft);

  % \draw (At) -- (Et) -- (Bt);
  % \draw (Ct) -- (Ft) -- (Ct);
  % \draw (Et) -- (Ft);

%  \draw (Ab) -- (At)

\end{scope}

%%% Local Variables:
%%% mode: latex
%%% TeX-master: t
%%% End:
}} 
  \]
  induce morphisms
    \[
      \begin{array}{crcl}
        \mapH \colon\thinspace &  \NB{\tikz[font=\tiny, scale =0.7]{\begin{scope}
  \coordinate (bl) at (-0.5, -1);
  \coordinate (br) at ( 0.5, -1);
  \coordinate (bm) at (  0,-0.3);
  \coordinate (tl) at (-0.5,  1);
  \coordinate (tr) at ( 0.5,  1);
  \coordinate (tm) at (  0, 0.3);
  \draw[>-]  (bl) .. controls +( 0, 0.5) and +(0,0) .. (bm)
  node[below, pos = 0] {$a$} coordinate[pos =0.3] (ga);
  \draw[>-]  (br) .. controls +( 0, 0.5) and +(0,0) .. (bm)
  node[below, pos = 0] {$b+c$};
  \draw[<-]  (tl) .. controls +( 0, -0.5) and +(0,0) .. (tm)
  node[above, pos = 0] {$a+c$};
  \draw[<-]  (tr) .. controls +( 0, -0.5) and +(0,0) .. (tm)
  node[above, pos = 0] {$b$} coordinate[pos =0.3] (gb);
  \draw [->-] (bm) -- (tm) node[left, pos = 0.5] {$a+b+c$};
  \filldraw[draw= green!50!black, fill = white] (gb) circle (1mm)
  node[right, green!50!black] {$a {t}_2$};
  \filldraw[draw= green!50!black, fill = white] (ga) circle (1mm)
  node[left, green!50!black] {$b {t}_1$};
\end{scope}}}  & \to &  \NB{\tikz[font=\tiny, scale =0.7]{\begin{scope}
  \coordinate (bl) at (-0.5, -1);
  \coordinate (br) at ( 0.5, -1);
  \coordinate (tl) at (-0.5,  1);
  \coordinate (tr) at ( 0.5,  1);
  \draw[>->] (bl) -- (tl) node[pos = 0, below] {$a$} node[pos = 1,
  above] {$a+c$} coordinate[pos = 0.6] (ml);
  \draw[>->] (br) -- (tr) node[pos = 0, below] {$b+c$} node[pos = 1, above] {$b$} coordinate[pos = 0.4] (mr);
  \draw[->-] (mr) -- (ml) node [pos= 0.5, above] {$c$};
\end{scope}}}   \\
\end{array}
       \]
and
\[
      \begin{array}{crcl}
        \mapH \colon\thinspace &  \NB{\tikz[font=\tiny, scale =0.7]{\begin{scope}
  \coordinate (bl) at (-0.5, -1);
  \coordinate (br) at ( 0.5, -1);
  \coordinate (bm) at (  0,-0.3);
  \coordinate (tl) at (-0.5,  1);
  \coordinate (tr) at ( 0.5,  1);
  \coordinate (tm) at (  0, 0.3);
  \draw[>-]  (bl) .. controls +( 0, 0.5) and +(0,0) .. (bm)
  node[below, pos = 0] {$a+c$};
  \draw[>-]  (br) .. controls +( 0, 0.5) and +(0,0) .. (bm)
  node[below, pos = 0] {$b$} coordinate[pos =0.3] (gb);
  \draw[<-]  (tl) .. controls +( 0, -0.5) and +(0,0) .. (tm)
  node[above, pos = 0] {$a$} coordinate[pos =0.3] (ga);
  \draw[<-]  (tr) .. controls +( 0, -0.5) and +(0,0) .. (tm)
  node[above, pos = 0] {$b+c$};
  \draw [->-] (bm) -- (tm) node[left, pos = 0.5] {$a+b+c$};
  \filldraw[draw= green!50!black, fill = white] (gb) circle (1mm)
  node[right, green!50!black] {$a {t}_2$};
  \filldraw[draw= green!50!black, fill = white] (ga) circle (1mm)
  node[left, green!50!black] {$b {t}_1$};
\end{scope}}}  & \to &                                   \NB{\tikz[font=\tiny, scale =0.7]{\begin{scope}
  \coordinate (bl) at (-0.5, -1);
  \coordinate (br) at ( 0.5, -1);
  \coordinate (tl) at (-0.5,  1);
  \coordinate (tr) at ( 0.5,  1);
  \draw[>->] (bl) -- (tl) node[pos = 0, below] {$a+c$} node[pos = 1,
  above] {$a$} coordinate[pos = 0.4] (ml);
  \draw[>->] (br) -- (tr) node[pos = 0, below] {$b$} node[pos = 1, above] {$b+c$} coordinate[pos = 0.6] (mr);
  \draw[->-] (ml) -- (mr) node [pos= 0.5, above] {$c$};
\end{scope}}}   \\
\end{array}
    \]
     of $\mathfrak{sl}_2$-equivariant $\gll_N$-state spaces associated with webs.
\end{lem}

\begin{lem}
  The foams
   \[
   \NB{\tikz[xscale=1.4,font=\tiny]{\begin{scope}[yscale=-0.5, xscale=0.5, font=\tiny]
  \begin{scope}
    \coordinate (At) at (0,0);
    \coordinate (Bt) at (0.5,1);
    \coordinate (Ct) at (3,1);
    \coordinate (Dt) at (2.5,0);
    \coordinate (Et) at (2,0.1);
    \coordinate (Ft) at (1,0.9);
  \end{scope}
  \coordinate (O) at (1.5, -0.5);
  \begin{scope}[yshift= -2cm]
    \coordinate (Ab) at (0,0);
    \coordinate (Bb) at (0.5,1);
    \coordinate (Cb) at (3,1);
    \coordinate (Db) at (2.5,0);
    \coordinate (Eb) at (0.8,0.5);
    \coordinate (Fb) at (2.2,0.5);
  \end{scope}

  \draw (Bt) -- (Bb) -- (Eb) -- (Fb) node[pos=0.5, above] {$a+b+c$}-- (Cb) -- (Ct) -- (Ft) -- (O);
  \filldraw[draw =  black, fill = white, fill opacity=0.5] (At) --
  (Et) -- (Dt) -- (Db) -- (Fb) -- (O) -- (Eb) -- (Ab) -- (At);
  \draw (Bt) -- (Ft)-- (Et) node[pos=0.5, xshift =0.2cm] {$c$} -- (O);

  \begin{scope}
    \node[xshift = -0.1cm] at (Ab) {$b$};
    \node[xshift = -0.3cm] at (Bt) {$a+c$};
    \node[xshift = 0.1cm] at (Ct) {$a$};
    \node[xshift = 0.3cm] at (Db) {$b+c$};
    
  \end{scope}
  % \draw (At) -- (Et) -- (Dt);
  % \draw (Bt) -- (Ft) -- (Ct);
  % \draw (Et) -- (Ft);

  % \draw (At) -- (Et) -- (Bt);
  % \draw (Ct) -- (Ft) -- (Ct);
  % \draw (Et) -- (Ft);

%  \draw (Ab) -- (At)

\end{scope}

%%% Local Variables:
%%% mode: latex
%%% TeX-master: t
%%% End:
}} \qquad \text{and} \qquad
   \NB{\tikz[xscale=1.4,font=\tiny]{\begin{scope}[xscale=0.5,yscale=-0.5, font=\tiny]
  \begin{scope}
    \coordinate (At) at (0,0);
    \coordinate (Bt) at (0.5,1);
    \coordinate (Ct) at (3,1);
    \coordinate (Dt) at (2.5,0);
    \coordinate (Et) at (1,0.1);
    \coordinate (Ft) at (2,0.9);
  \end{scope}
  \coordinate (O) at (1.5, -0.5);
  \begin{scope}[yshift= -2cm]
    \coordinate (Ab) at (0,0);
    \coordinate (Bb) at (0.5,1);
    \coordinate (Cb) at (3,1);
    \coordinate (Db) at (2.5,0);
    \coordinate (Eb) at (0.8,0.5);
    \coordinate (Fb) at (2.2,0.5);
  \end{scope}

  \draw (Bt) -- (Bb) -- (Eb) -- (Fb) node[pos=0.5, above] {$a+b+c$}-- (Cb) -- (Ct) -- (Ft) -- (O);
  \filldraw[draw =  black, fill = white, fill opacity=0.5] (At) --
  (Et) -- (Dt) -- (Db) -- (Fb) -- (O) -- (Eb) -- (Ab) -- (At);
  \draw (Bt) -- (Ft)-- (Et) node[pos=0.5, xshift =-0.2cm] {$c$} -- (O);

  \begin{scope}
    \node[xshift = -0.3cm] at (Ab) {$b+c$};
    \node[xshift = -0.1cm] at (Bt) {$a$};
    \node[xshift = 0.3cm] at (Ct) {$a+c$};
    \node[xshift = 0.1cm] at (Db) {$b$};
    
  \end{scope}
  % \draw (At) -- (Et) -- (Dt);
  % \draw (Bt) -- (Ft) -- (Ct);
  % \draw (Et) -- (Ft);

  % \draw (At) -- (Et) -- (Bt);
  % \draw (Ct) -- (Ft) -- (Ct);
  % \draw (Et) -- (Ft);

%  \draw (Ab) -- (At)

\end{scope}

%%% Local Variables:
%%% mode: latex
%%% TeX-master: t
%%% End:
}} 
  \]
induce morphisms
  \[
      \begin{array}{crcl}
        \mapX \colon\thinspace &  \NB{\tikz[font=\tiny, scale =0.7]{}}  & \to &   \NB{\tikz[font=\tiny, scale =0.7]{\begin{scope}
  \coordinate (bl) at (-0.5, -1);
  \coordinate (br) at ( 0.5, -1);
  \coordinate (bm) at (  0,-0.3);
  \coordinate (tl) at (-0.5,  1);
  \coordinate (tr) at ( 0.5,  1);
  \coordinate (tm) at (  0, 0.3);
  \draw[>-]  (bl) .. controls +( 0, 0.5) and +(0,0) .. (bm)
  node[below, pos = 0] {$a$} coordinate[pos =0.3] (ga);
  \draw[>-]  (br) .. controls +( 0, 0.5) and +(0,0) .. (bm)
  node[below, pos = 0] {$b+c$};
  \draw[<-]  (tl) .. controls +( 0, -0.5) and +(0,0) .. (tm)
  node[above, pos = 0] {$a+c$};
  \draw[<-]  (tr) .. controls +( 0, -0.5) and +(0,0) .. (tm)
  node[above, pos = 0] {$b$} coordinate[pos =0.3] (gb);
  \draw [->-] (bm) -- (tm) node[left, pos = 0.5] {$a+b+c$};
  \filldraw[draw= green!50!black, fill = white] (gb) circle (1mm)
  node[right, green!50!black] {$-a \bar{t}_2$};
  \filldraw[draw= green!50!black, fill = white] (ga) circle (1mm)
  node[left, green!50!black] {$-b \bar{t}_1$};
\end{scope}}}   \\
\end{array}
       \]
and
\[
      \begin{array}{crcl}
        \mapX \colon\thinspace & \NB{\tikz[font=\tiny, scale =0.7]{}}  & \to &
 \NB{\tikz[font=\tiny, scale =0.7]{\begin{scope}
  \coordinate (bl) at (-0.5, -1);
  \coordinate (br) at ( 0.5, -1);
  \coordinate (bm) at (  0,-0.3);
  \coordinate (tl) at (-0.5,  1);
  \coordinate (tr) at ( 0.5,  1);
  \coordinate (tm) at (  0, 0.3);
  \draw[>-]  (bl) .. controls +( 0, 0.5) and +(0,0) .. (bm)
  node[below, pos = 0] {$a+c$};
  \draw[>-]  (br) .. controls +( 0, 0.5) and +(0,0) .. (bm)
  node[below, pos = 0] {$b$} coordinate[pos =0.3] (gb);
  \draw[<-]  (tl) .. controls +( 0, -0.5) and +(0,0) .. (tm)
  node[above, pos = 0] {$a$} coordinate[pos =0.3] (ga);
  \draw[<-]  (tr) .. controls +( 0, -0.5) and +(0,0) .. (tm)
  node[above, pos = 0] {$b+c$};
  \draw [->-] (bm) -- (tm) node[left, pos = 0.5] {$a+b+c$};
  \filldraw[draw= green!50!black, fill = white] (gb) circle (1mm)
  node[right, green!50!black] {$-a \bar{t}_2$};
  \filldraw[draw= green!50!black, fill = white] (ga) circle (1mm)
  node[left, green!50!black] {$-b \bar{t}_1$};
\end{scope}}}   \\
\end{array}
    \]
of $\mathfrak{sl}_2$-equivariant $\gll_N$-state spaces associated with webs.

\end{lem}

\begin{lem} \label{lem:cup-cap-maps}
  The foams
   \[
   \NB{\tikz[scale=1,font=\tiny]{}} \qquad \text{and} \qquad
   \NB{\tikz[scale=1,font=\tiny]{}} 
  \]
induce morphisms
  \[
      \begin{array}{crcl}
        \mapC \colon\thinspace &  \emptyset  & \to &  \NB{\tikz[font =\tiny, scale= 0.7]{\begin{scope}[scale= 1]
  \draw[->] (0,0) arc (0:360: 0.8) node[pos=1, right] {$a$} coordinate[pos=0.45]
  (ga) coordinate[pos=0.55] (gb); 
  \filldraw[draw= green!50!black, fill = green] (ga) circle (1mm)
  node[left, green!50!black] {$\frac{a}{2}$};
  \filldraw[draw= green!50!black, fill = white] (gb) circle (1mm)
  node[left, green!50!black] {$\frac{N-a}{2}$};
\end{scope}
%%% Local Variables:
%%% mode: latex
%%% TeX-master: t
%%% End:
}}      \\
\end{array}
       \]
and
\[
      \begin{array}{crcl}
        \mapC \colon\thinspace & \NB{\tikz[font =\tiny, scale= 0.7]{\begin{scope}[scale= 1]
  \draw[->] (0,0) arc (0:360: 0.8) node[pos=1, right] {$a$} coordinate[pos=0.45]
  (ga) coordinate[pos=0.55] (gb); 
  \filldraw[draw= green!50!black, fill = green] (ga) circle (1mm)
  node[left, green!50!black] {$\frac{-a}{2}$};
  \filldraw[draw= green!50!black, fill = white] (gb) circle (1mm)
  node[left, green!50!black] {$\frac{a-N}{2}$};
\end{scope}
%%% Local Variables:
%%% mode: latex
%%% TeX-master: t
%%% End:
}}   & \to & \emptyset   \\
\end{array}
    \]
of $\mathfrak{sl}_2$-equivariant $\gll_N$-state spaces associated with webs.
\end{lem}

\begin{lem}
  The foam
   \[
   \NB{\tikz[scale=1,font=\tiny]{}}
  \]
induces a morphism
  \[
      \begin{array}{crcl}
        \mapS \colon\thinspace & \NB{\tikz[font =\tiny, scale= 0.7]{\begin{scope}[scale= 1]
  \draw [->] (45:0.8) arc (-45:-135:0.8) node [pos=1, above] {$a$};
  \draw [<-] (-45:0.8) arc (45: 135:0.8) node [pos=0, below] {$a$};
\end{scope}
}}     & \to &    \NB{\tikz[font =\tiny, scale= 0.7]{\begin{scope}[scale= 1]
  \draw [->] (45:0.8) arc (135:225:0.8) coordinate[pos=0.3] (ga) node
  [pos=1, below] {$a$};
  \draw [<-] (135:0.8) arc (-135:-225:-0.8) coordinate[pos=0.7] (gb)
  node [pos=0, above] {$a$};
  \filldraw[draw= green!50!black, fill = green] (ga) circle (1mm)
  node[right, green!50!black] {$\frac{-a}{2}$};
  \filldraw[draw= green!50!black, fill = white] (gb) circle (1mm)
  node[left, green!50!black] {$\frac{a-N}{2}$};
\end{scope}
}}    \\
\end{array}
       \]
    of $\mathfrak{sl}_2$-equivariant webs.
\end{lem}

These morphisms can be composed in order to construct more sophisticated
morphisms between $\sll_2$-equivariant $\gll_N$-state spaces associated with webs.

\begin{exa}
  The composition of elementary foams
  \begin{equation}
  \NB{\tikz[xscale = 3, yscale = 3]{
       \node (i0) at (0, 0) { \NB{\tikz[font= \tiny,
  scale=0.6, yscale=1]{\begin{scope}
  \coordinate (b) at ( 0,-0.50);
  \coordinate (m) at (0,0);
  \coordinate (tr) at (+1,  1);
  \coordinate (tm) at (0,  1);
  \coordinate (tl) at (-1,  1);
  \draw[ >= to] (b) -- (m) node[pos=0, below] {};
  \draw[-to] (m) .. controls +(0,0) and + (0, -0.5) .. (tl) node[pos =
  1, above] {$2$} coordinate[pos= 0.5] (ml) coordinate[pos = 0.25] (lml);
  \draw[-to] (m) .. controls +(0,0) and + (0, -0.5) .. (tr) node[pos =
  1, above] {$1$};
  % \draw[-to] (ml)  .. controls +(0,0) and + (0, -0.5) .. (tm) node[pos =
  % 1, above] {$1$};
  \node[left] (lml) {};
  %%%%%
    \coordinate (tX) at ( 0,-0.50);
  \coordinate (mX) at (0,-1);
  \coordinate (brX) at (+1,  -2);
  \coordinate (bmX) at ( 0,  -2);
  \coordinate (blX) at (-1,  -2);
  \draw[-to] (mX) -- (tX) node[pos=0.5, left] {$3$};
  \draw[>-, >=to] (blX) .. controls +(0,0.5) and + (0, 0) .. (mX) node[pos =
  0, below] {$2$} coordinate[pos= 0.5] (mlX) coordinate[pos = 0.75] (lmlX);
  \draw[>-, >=to] (brX) .. controls +(0,0.5) and + (0, -0) .. (mX) node[pos =
  0, below] {$1$} coordinate[pos = 0.25] (gbU) ;
    \filldraw[draw= green!50!black, fill = white] (gbU) circle (1mm)
  node[left, green!50!black] {$t_2$};
  % \draw[>-, >=to] (bmX)  .. controls +(0,0.5) and + (0, 0) .. (mlX) node[pos =
  % 0, below] {$1$};
  \node[left] (lmlX) {};
\end{scope}

%%% Local Variables:
%%% mode: latex
%%% TeX-master: t
%%% End:
}} };
    \node (i1) at (1,0) { \NB{\tikz[font= \tiny,
  scale=0.6, yscale=1]{\begin{scope}
  \coordinate (b) at ( 0,-0.50);
  \coordinate (m) at (0,-0.3);
  \coordinate (tr) at (+1,  1);
  \coordinate (tm) at (-.5,  1);
  \coordinate (tl) at (-1,  1);
 \draw[ >= to] (b) -- (m) node[pos=0, below] {};
  \draw[-to] (m) .. controls +(0,0) and + (0, -0.8) .. (tl) node[pos =
  1, above] {$2$} coordinate[pos= 0.75] (ml) coordinate[pos = 0.2] (lml) coordinate[pos = 0.9] (lmlh) coordinate[pos = 0.5] (lmlm);
  \draw[-to] (m) .. controls +(0,0) and + (0, -0.5) .. (tr) node[pos =
  1, above] {$1$};
  % \draw[-to] (ml)  .. controls +(0,0) and + (0, -0.5) .. (tm) node[pos =
  % 1, above] {$1$};
    \draw (lml)  .. controls +(0,.4) and + (.4, 0) .. (lmlh) node[pos =
  1, above] {} coordinate[pos = 0.5] (X);
    \filldraw[draw= green!50!black, fill = white] (X) circle (1mm)
  node[right, green!50!black] {$t_2$};
      \filldraw[draw= green!50!black, fill = white] (lmlm) circle (1mm)
  node[left, green!50!black] {$t_1$};
 % \node[left] (lml) {};
  %%%%%
    \coordinate (tX) at ( 0,-0.50);
  \coordinate (mX) at (0,-1.2);
  \coordinate (brX) at (+1,  -2);
  \coordinate (bmX) at ( 0,  -2);
  \coordinate (blX) at (-1,  -2);
  \draw[-to] (mX) -- (tX) node[pos=1, above] {};
  \draw[>-, >=to] (blX) .. controls +(0,0.5) and + (0, 0) .. (mX) node[pos =
  0, below] {$2$} coordinate[pos= 0.5] (mlX) coordinate[pos = 0.75] (lmlX);
  \draw[>-, >=to] (brX) .. controls +(0,0.5) and + (0, -0) .. (mX) node[pos =
  0, below] {$1$} coordinate[pos = 0.25] (gbU) ;
    \filldraw[draw= green!50!black, fill = white] (gbU) circle (1mm)
  node[left, green!50!black] {$t_2$};
  % \draw[>-, >=to] (bmX)  .. controls +(0,0.5) and + (0, 0) .. (mlX) node[pos =
  % 0, below] {$1$};
  \node[left] (lmlX) {};
\end{scope}

%%% Local Variables:
%%% mode: latex
%%% TeX-master: t
%%% End:
}} };
      \node (i2) at (2, 0) { \NB{\tikz[font= \tiny,
  scale=0.6, yscale=1]{\begin{scope}
  \coordinate (b) at ( 0,-0.50);
  \coordinate (m) at (0,-0.3);
  \coordinate (tr) at (+1,  1);
  \coordinate (tm) at (-.5,  1);
  \coordinate (tl) at (-1,  1);
 \draw[ >= to] (b) -- (m) node[pos=0, below] {};
  \draw[-to] (m) .. controls +(0,0) and + (0, -0.8) .. (tl) node[pos =
  1, above] {$2$} coordinate[pos= 0.75] (ml)  coordinate[pos = 0.9] (lmlh) coordinate[pos = 0.6] (lmlm);
  \draw[-to] (m) .. controls +(0,0) and + (0, -0.5) .. (tr) node[pos =
  1, above] {$1$} coordinate[pos = 0.5] (lml);
  % \draw[-to] (ml)  .. controls +(0,0) and + (0, -0.5) .. (tm) node[pos =
  % 1, above] {$1$};
    \draw (lml)  .. controls +(0,0) and + (.4, 0) .. (lmlh) node[pos =
  1, above] {} coordinate[pos = 0.5] (X);
    \filldraw[draw= green!50!black, fill = white] (X) circle (1mm)
  node[above, green!50!black] {$t_2$};
      \filldraw[draw= green!50!black, fill = white] (lmlm) circle (1mm)
  node[left, green!50!black] {$t_1$};
 % \node[left] (lml) {};
  %%%%%
    \coordinate (tX) at ( 0,-0.50);
  \coordinate (mX) at (0,-1.2);
  \coordinate (brX) at (+1,  -2);
  \coordinate (bmX) at ( 0,  -2);
  \coordinate (blX) at (-1,  -2);
  \draw[-to] (mX) -- (tX) node[pos=1, above] {};
  \draw[>-, >=to] (blX) .. controls +(0,0.5) and + (0, 0) .. (mX) node[pos =
  0, below] {$2$} coordinate[pos= 0.5] (mlX) coordinate[pos = 0.75] (lmlX);
  \draw[>-, >=to] (brX) .. controls +(0,0.5) and + (0, -0) .. (mX) node[pos =
  0, below] {$1$} coordinate[pos = 0.25] (gbU) ;
    \filldraw[draw= green!50!black, fill = white] (gbU) circle (1mm)
  node[left, green!50!black] {$t_2$};
  % \draw[>-, >=to] (bmX)  .. controls +(0,0.5) and + (0, 0) .. (mlX) node[pos =
  % 0, below] {$1$};
  \node[left] (lmlX) {};
\end{scope}

%%% Local Variables:
%%% mode: latex
%%% TeX-master: t
%%% End:
}} };
    \node (i3) at (3, 0) { \NB{\tikz[font= \tiny,
  scale=0.6]{\begin{scope}

  \coordinate (b) at ( 0,-0.50);
  \coordinate (m) at (0,-0.3);
  \coordinate (tr) at (+1,  1);
  \coordinate (tm) at (-.5,  1);
  \coordinate (tl) at (-1,  1);
  \coordinate (tX) at ( 0,-0.50);
  \coordinate (mX) at (0,-1.2);
  \coordinate (brX) at (+1,  -2);
  \coordinate (bmX) at ( 0,  -2);
  \coordinate (blX) at (-1,  -2);

  \draw[>->] (blX) -- (tl) node[above] {$2$} node[pos=0, below] {$2$}
  coordinate[pos=0.2] (ml1) coordinate[pos=0.8] (ml2); 
  \draw[>->] (brX) -- (tr) node[above] {$1$} node[pos=0, below] {$1$}
  coordinate[pos=0.3] (mr1) coordinate[pos=0.7] (mr2); 
  \draw[->-] (ml1) -- (mr1);
  \draw[->-] (mr2) -- (ml2) coordinate[pos=0.3] (X);
  
  % \draw[ >= to] (b) -- (m) node[pos=0, below] {};
  % \draw[-to] (m) .. controls +(0,0) and + (0, -0.8) .. (tl) node[pos =
  % 1, above] {$2$} coordinate[pos= 0.75] (ml)  coordinate[pos = 0.9] (lmlh) coordinate[pos = 0.6] (lmlm);
  % \draw[-to] (m) .. controls +(0,0) and + (0, -0.5) .. (tr) node[pos =
  % 1, above] {$1$} coordinate[pos = 0.5] (lml);
  % % \draw[-to] (ml)  .. controls +(0,0) and + (0, -0.5) .. (tm) node[pos =
  % % 1, above] {$1$};
  %   \draw (lml)  .. controls +(0,0) and + (.4, 0) .. (lmlh) node[pos =
  % 1, above] {} coordinate[pos = 0.5] (X);
    \filldraw[draw= green!50!black, fill = white] (X) circle (1mm)
  node[above, green!50!black] {$t_2$};
 %      \filldraw[draw= green!50!black, fill = white] (lmlm) circle (1mm)
 %  node[left, green!50!black] {$t_1$};
 % % \node[left] (lml) {};
 %  %%%%%
 %    \coordinate (tX) at ( 0,-0.50);
 %  \coordinate (mX) at (0,-1.2);
 %  \coordinate (brX) at (+1,  -2);
 %  \coordinate (bmX) at ( 0,  -2);
 %  \coordinate (blX) at (-1,  -2);
 %  \draw[-to] (mX) -- (tX) node[pos=1, above] {};
 %  \draw[>-, >=to] (blX) .. controls +(0,0.5) and + (0, 0) .. (mX) node[pos =
 %  0, below] {$2$} coordinate[pos= 0.5] (mlX) coordinate[pos = 0.75] (lmlX);
 %  \draw[>-, >=to] (brX) .. controls +(0,0.5) and + (0, -0) .. (mX) node[pos =
 %  0, below] {$1$} coordinate[pos = 0.25] (gbU) ;
 %    \filldraw[draw= green!50!black, fill = white] (gbU) circle (1mm)
 %  node[left, green!50!black] {$t_2$};
 %  % \draw[>-, >=to] (bmX)  .. controls +(0,0.5) and + (0, 0) .. (mlX) node[pos =
 %  % 0, below] {$1$};
 %  \node[left] (lmlX) {};
\end{scope}

%%% Local Variables:
%%% mode: latex
%%% TeX-master: t
%%% End:
}} };
  \draw[->] (i0) -- (i1)  node[pos=0.5, above] {$\mapB$};
      \draw[->] (i1) -- (i2)  node[pos=0.5, above] {$\mapA$};
       \draw[->] (i2) -- (i3)  node[pos=0.5, above] {$\mapH$};
      }} 
\end{equation}
  induces a morphism of $\sll_2$-equivariant $\gll_N$-state spaces associated with webs. Adding a green dot
  with multiplicity $-t_2$ on the bottom right leg, we end up with a
  morphism
  \begin{equation}
    \NB{
      \tikz{
        \node (a) at (0,0) {\NB{\tikz[scale= 0.6, font=\tiny]{\begin{scope}
  \coordinate (b) at ( 0,-0.50);
  \coordinate (m) at (0,0);
  \coordinate (tr) at (+1,  1);
  \coordinate (tm) at (0,  1);
  \coordinate (tl) at (-1,  1);
  \draw[ >= to] (b) -- (m) node[pos=0, below] {};
  \draw[-to] (m) .. controls +(0,0) and + (0, -0.5) .. (tl) node[pos =
  1, above] {$2$} coordinate[pos= 0.5] (ml) coordinate[pos = 0.25] (lml);
  \draw[-to] (m) .. controls +(0,0) and + (0, -0.5) .. (tr) node[pos =
  1, above] {$1$};
  % \draw[-to] (ml)  .. controls +(0,0) and + (0, -0.5) .. (tm) node[pos =
  % 1, above] {$1$};
  \node[left] (lml) {};
  %%%%%
    \coordinate (tX) at ( 0,-0.50);
  \coordinate (mX) at (0,-1);
  \coordinate (brX) at (+1,  -2);
  \coordinate (bmX) at ( 0,  -2);
  \coordinate (blX) at (-1,  -2);
  \draw[-to] (mX) -- (tX) node[pos=0.5, left] {$3$};
  \draw[>-, >=to] (blX) .. controls +(0,0.5) and + (0, 0) .. (mX) node[pos =
  0, below] {$2$} coordinate[pos= 0.5] (mlX) coordinate[pos = 0.75] (lmlX);
  \draw[>-, >=to] (brX) .. controls +(0,0.5) and + (0, -0) .. (mX) node[pos =
  0, below] {$1$} coordinate[pos = 0.25] (gbU) ;
    %\filldraw[draw= green!50!black, fill = white] (gbU) circle (1mm)
    %node[left, green!50!black] {$t_2$};
    
  % \draw[>-, >=to] (bmX)  .. controls +(0,0.5) and + (0, 0) .. (mlX) node[pos =
  % 0, below] {$1$};
  %\node[left] (lmlX) {};
\end{scope}

%%% Local Variables:
%%% mode: latex
%%% TeX-master: t
%%% End:
}}};
        \node (b) at (4,0) {\NB{\tikz[scale=0.6,font=\tiny]{\begin{scope}

  \coordinate (b) at ( 0,-0.50);
  \coordinate (m) at (0,-0.3);
  \coordinate (tr) at (+1,  1);
  \coordinate (tm) at (-.5,  1);
  \coordinate (tl) at (-1,  1);
  \coordinate (tX) at ( 0,-0.50);
  \coordinate (mX) at (0,-1.2);
  \coordinate (brX) at (+1,  -2);
  \coordinate (bmX) at ( 0,  -2);
  \coordinate (blX) at (-1,  -2);

  \draw[>->] (blX) -- (tl) node[above] {$2$} node[pos=0, below] {$2$}
  coordinate[pos=0.2] (ml1) coordinate[pos=0.8] (ml2); 
  \draw[>->] (brX) -- (tr) node[above] {$1$} node[pos=0, below] {$1$}
  coordinate[pos=0.3] (mr1) coordinate[pos=0.7] (mr2)
  coordinate[pos=0.15] (Y); 
  \draw[->-] (ml1) -- (mr1);
  \draw[->-] (mr2) -- (ml2) coordinate[pos=0.3] (X);
  
  % \draw[ >= to] (b) -- (m) node[pos=0, below] {};
  % \draw[-to] (m) .. controls +(0,0) and + (0, -0.8) .. (tl) node[pos =
  % 1, above] {$2$} coordinate[pos= 0.75] (ml)  coordinate[pos = 0.9] (lmlh) coordinate[pos = 0.6] (lmlm);
  % \draw[-to] (m) .. controls +(0,0) and + (0, -0.5) .. (tr) node[pos =
  % 1, above] {$1$} coordinate[pos = 0.5] (lml);
  % % \draw[-to] (ml)  .. controls +(0,0) and + (0, -0.5) .. (tm) node[pos =
  % % 1, above] {$1$};
  %   \draw (lml)  .. controls +(0,0) and + (.4, 0) .. (lmlh) node[pos =
  % 1, above] {} coordinate[pos = 0.5] (X);
    \filldraw[draw= green!50!black, fill = white] (X) circle (1mm)
  node[above, green!50!black] {$t_2$};
    \filldraw[draw= green!50!black, fill = white] (Y) circle (1mm)
  node[right, green!50!black] {$-t_2$};

  %\filldraw[draw= green!50!black, fill = white] (lmlm) circle (1mm)
 %  node[left, green!50!black] {$t_1$};
 % % \node[left] (lml) {};
 %  %%%%%
 %    \coordinate (tX) at ( 0,-0.50);
 %  \coordinate (mX) at (0,-1.2);
 %  \coordinate (brX) at (+1,  -2);
 %  \coordinate (bmX) at ( 0,  -2);
 %  \coordinate (blX) at (-1,  -2);
 %  \draw[-to] (mX) -- (tX) node[pos=1, above] {};
 %  \draw[>-, >=to] (blX) .. controls +(0,0.5) and + (0, 0) .. (mX) node[pos =
 %  0, below] {$2$} coordinate[pos= 0.5] (mlX) coordinate[pos = 0.75] (lmlX);
 %  \draw[>-, >=to] (brX) .. controls +(0,0.5) and + (0, -0) .. (mX) node[pos =
 %  0, below] {$1$} coordinate[pos = 0.25] (gbU) ;
 %    \filldraw[draw= green!50!black, fill = white] (gbU) circle (1mm)
 %  node[left, green!50!black] {$t_2$};
 %  % \draw[>-, >=to] (bmX)  .. controls +(0,0.5) and + (0, 0) .. (mlX) node[pos =
 %  % 0, below] {$1$};
 %  \node[left] (lmlX) {};
\end{scope}

%%% Local Variables:
%%% mode: latex
%%% TeX-master: t
%%% End:
}}};
        \draw [-to] (a) -- (b);
        }
    }
  \end{equation}
  of $\sll_2$-equivariant $\gll_N$-state spaces associated with webs.
\end{exa}

\section{Link homology, definition and invariance}
\label{homology:sec}
In the previous section, we associated a $\RN$-module to any closed
web.  This construction categorifies the MOY calculus which describes
the category of $\mathcal{U}_q(\mathfrak{gl}_N)$ representations
generated by exterior powers of the fundamental representation in a
diagrammatic fashion.  It leads to a definition of
Khovanov--Rozansky $\mathfrak{gl}_N$-link homology.  For more details,
see \cite{RW1}.

We also saw in the previous section how $\mathfrak{sl}_2$ acts on the
state space associated to a web.  In this section we show how the
$\mathfrak{sl}_2$-action extends to Khovanov--Rozansky homology.

\begin{conv}
  In web diagrams, for space purposes, we will often not write the
  thicknesses of the edges when they have thickness $1$ or $2$. A
  ``normal'' edge (\NB{\tikz{\draw (0,0) -- (0.5,0); }}) means
  thickness $1$, a double edge (\NB{\tikz{\draw[double] (0,0) --
    (0.5,0);}}) means thickness $2$. Edges of higher thickness (only
    thickness $3$ will appear in the proof of Reidemeister III move)
    will be thicker (\NB{\tikz{\draw[thick] (0,0) -- (0.5,0);)}}) and their thicknesses will be displayed.

Also, fix from now on parameters $t_1,t_2\in\Bbbk$ such that $t_1+t_2=1$. Then $\bar{t}_1=1-t_1$ and $\bar{t}_2=1-t_2$ also add up to 1. Thanks to equations \eqref{eq:h-act-pol}--\eqref{eq:h-act-saddle}, such choices of parameters make the $\dh$-operator on foams easier to read off: it is just the negative of the usual degree operator.
  \end{conv}

\subsection{Link homology}
\label{sec:lh}

As usual for Khovanov-like link homology theories, we associate a
hypercube shaped complex to any link diagram by specifying locally a
length-2 complex to any crossing. 
We define the following
cohomologically graded braiding complexes:

\begin{equation} \label{eqn:def-T}
  T= \NB{\tikz[xscale = 0.6]{\begin{scope}[font=\tiny]
  \draw[->] (0.5, -0.5) ..controls +(0,0.3) and +(0,-0.3) .. (-0.5,
  0.5);% node[pos=1, above] {} coordinate[pos =0.2] (t1);
  \fill[white] (0,0) circle (2mm);
  \draw[->] (-0.5, -0.5) ..controls +(0,0.3) and +(0,-0.3) .. (0.5,
  0.5);% node[pos=1, above] {} coordinate[pos =0.2] (t2);
  % \filldraw[draw= green!50!black, fill = white] (t2) circle (1mm)
  % node[left, green!50!black] {$c$};
  % \filldraw[draw= green!50!black, fill = white] (t1) circle (1mm)
  % node[right, green!50!black] {$d$};
\end{scope}}} :=
\NB{\tikz[xscale = 3.5, yscale = 3]{
    \node (i0) at (0, 0) {$q^{-1}$ \NB{\tikz[font= \tiny,
  scale=0.6]{\begin{scope}
  \coordinate (bl) at (-0.5, -1);
  \coordinate (br) at ( 0.5, -1);
  \coordinate (tl) at (-0.5,  1);
  \coordinate (tr) at ( 0.5,  1);
    \coordinate (ml) at (-0.5,  .6);
 \coordinate (mr) at (0.5,  .6);
 \coordinate (G) at (0.8,-0.1);

   \draw[->] (bl) -- (tl);
        \filldraw[draw= green!50!black, fill = white] (ml) circle (1mm) 
  node[left, green!50!black] 
  {$-t_1$};
  
    \draw[->] (br) -- (tr);
        \filldraw[draw= green!50!black, fill = white] (mr) circle (1mm) 
  node[right, green!50!black] 
  {$-t_2$};

\end{scope}}} };
     \node (i1) at (-1, 0) {\NB{\tikz[font= \tiny,
  scale=0.6]{\begin{scope}
  \coordinate (bl) at (-0.5, -1);
  \coordinate (br) at ( 0.5, -1);
  \coordinate (bm) at (  0,-0.3);
  \coordinate (tl) at (-0.5,  1);
  \coordinate (tr) at ( 0.5,  1);
  \coordinate (tm) at (  0, 0.3);
  \draw[>-]  (bl) .. controls +( 0, 0.5) and +(0,0) .. (bm);
%  node[below, pos = 0] {$1$};
  \draw[>-]  (br) .. controls +( 0, 0.5) and +(0,0) .. (bm);
%  node[below, pos = 0] {$1$};
  \draw[<-]  (tl) .. controls +( 0, -0.5) and +(0,0) .. (tm);
%  node[above, pos = 0] {$1$} coordinate[pos = 0.25] (ga) ;
  %  \filldraw[draw= green!50!black, fill = white] (ga) circle (1mm)
  %node[left, green!50!black] {$t_1$};

  \draw[<-]  (tr) .. controls +( 0, -0.5) and +(0,0) .. (tm);
%  node[above, pos = 0] {$1$} coordinate[pos = 0.25] (gb) ;
 %   \filldraw[draw= green!50!black, fill = white] (gb) circle (1mm)
%  node[left, green!50!black] {$t_2$};
  \draw [double] (bm) -- (tm);% node[left, pos = 0.5] {$2$};
 
\end{scope}}} };
\draw[->] (i1) -- (i0) coordinate[pos=0.5] (a);
\node[above] at (a) {\NB{\tikz[font=\tiny, scale=.5]{\begin{scope}
  \begin{scope}
    \coordinate (L1) at (0.2,0.4);
    \coordinate (L2) at (0,0);
    \coordinate (R1) at (2.2,0.4);
    \coordinate (R2) at (2,0);
    \coordinate (ML) at (0.6, 0.2);
    \coordinate (MR) at (1.6, 0.2);
    \draw[double] (ML) -- (MR);% node[below, midway] {$2$};
    \draw (MR) .. controls +(0, 0) and +(-0.3,0) .. (R1) ;
    \draw (MR) .. controls +(0, 0) and +(-0.3,0) .. (R2);
    \draw (L1) .. controls +( 0.3, 0) and +(0,0) .. (ML);
    \draw (L2) .. controls +( 0.3, 0) and +(0,0) .. (ML);
  \end{scope}  
 \begin{scope}[yshift = 1cm]
    \coordinate (L1B) at (0.2,0.4);
    \coordinate (L2B) at (0,0);
    \coordinate (R1B) at (2.2,0.4);
    \coordinate (R2B) at (2,0);
    \draw (L1B) .. controls +( 0, 0) and +(0,0) .. (R1B); %node  [left, pos  = 0] {$1$};
    \draw (L2B) .. controls +( 0, 0) and +(0,0) .. (R2B);% node     [left, pos   = 0] {$1$};
 \end{scope}  
  \draw (R1) -- (R1B);
  \draw (R2) -- (R2B);
  \draw (L1) -- (L1B);
  \draw (L2) -- (L2B);
  \draw[thick] (ML) .. controls +(0, 0.6) and +(0, 0.6) .. (MR);
\end{scope}
}}};
  }}
\end{equation}
    \begin{equation}\label{eqn:def-T-prime}
  T'= \NB{\tikz[xscale = 0.6]{\begin{scope}[font=\tiny]
  \draw[->] (-0.5, -0.5) ..controls +(0,0.3) and +(0,-0.3) .. (0.5,
  0.5);% node[pos=1, above] {} coordinate[pos =0.2] (t2);
  \fill[white] (0,0) circle (2mm);
  \draw[->] (0.5, -0.5) ..controls +(0,0.3) and +(0,-0.3) .. (-0.5,
  0.5);% node[pos=1, above] {} coordinate[pos =0.2] (t1);
  % \filldraw[draw= green!50!black, fill = white] (t2) circle (1mm)
  % node[left, green!50!black] {$c$};
  % \filldraw[draw= green!50!black, fill = white] (t1) circle (1mm)
  % node[right, green!50!black] {$d$};
\end{scope}}}:=
    \NB{\tikz[xscale = 3.5, yscale = 3]{
    \node (i0) at (-1, 0) { $q$\ \NB{\tikz[font= \tiny,
  scale=0.6]{\begin{scope}
  \coordinate (bl) at (-0.5, -1);
  \coordinate (br) at ( 0.5, -1);
  \coordinate (tl) at (-0.5,  1);
  \coordinate (tr) at ( 0.5,  1);
    \coordinate (ml) at (-0.5,  .6);
 \coordinate (mr) at (0.5,  .6);
 \coordinate (G) at (0.8,-0.1);

   \draw[->] (bl) -- (tl);
        \filldraw[draw= green!50!black, fill = white] (ml) circle (1mm) 
  node[left, green!50!black] 
  {$\bar{t}_1$};
  
    \draw[->] (br) -- (tr);
        \filldraw[draw= green!50!black, fill = white] (mr) circle (1mm) 
  node[right, green!50!black] 
  {$\bar{t}_2$};

\end{scope}}} };
     \node (i1) at (0, 0) { \NB{\tikz[font= \tiny,
  scale=0.6]{}} };
\draw[->] (i0) -- (i1) coordinate[pos=0.5] (b);
\node[above] at (b) {\NB{\tikz[font=\tiny, scale=.5]{\begin{scope}
  \begin{scope}
    \coordinate (L1) at (0.2,0.4);
    \coordinate (L2) at (0,0);
    \coordinate (R1) at (2.2,0.4);
    \coordinate (R2) at (2,0);
    \coordinate (ML) at (0.6, 0.2);
    \coordinate (MR) at (1.6, 0.2);
    \draw[double] (ML) -- (MR);% node[above, midway] {$2$};
    \draw (MR) .. controls +(0, 0) and +(-0.3,0) .. (R1) ;
    \draw (MR) .. controls +(0, 0) and +(-0.3,0) .. (R2);
    \draw (L1) .. controls +( 0.3, 0) and +(0,0) .. (ML);
    \draw (L2) .. controls +( 0.3, 0) and +(0,0) .. (ML);
  \end{scope}  
 \begin{scope}[yshift = -1cm]
    \coordinate (L1B) at (0.2,0.4);
    \coordinate (L2B) at (0,0);
    \coordinate (R1B) at (2.2,0.4);
    \coordinate (R2B) at (2,0);
    \draw (L1B) .. controls +( 0, 0) and +(0,0) .. (R1B);% node
                                % [right, pos     = 1] {$1$};
    \draw (L2B) .. controls +( 0, 0) and +(0,0) .. (R2B); % node [right, pos    = 1] {$1$};
 \end{scope}  
  \draw (R1) -- (R1B);
  \draw (R2) -- (R2B);
  \draw (L1) -- (L1B);
  \draw (L2) -- (L2B);
  \draw[thick] (ML) .. controls +(0, -0.6) and +(0, -0.6) .. (MR);
\end{scope}

%%% Local Variables:
%%% mode: latex
%%% TeX-master: t
%%% End:
}}}; 
      }} \,
    \end{equation}
where in both complexes we assume (as in \cite{QRSW1}) that the terms
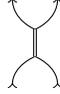
\[
  \NB{\tikz[font= \tiny,
  scale=0.6]{}} 
\]
sit in cohomological degree $0$. In these diagrams $\statespaceN{\cdot}$ has
been omitted to maintain readability.

For a link $L$, define
$\KR_N^{\mathfrak{sl}_2}(L;R) :=
\KR_{N;t_1,t_2}^{\mathfrak{sl}_2}(L;R) $ to be the Khovanov--Rozansky
$\gll_N$-homology of $L$ with coefficients in a ring $R$,
equipped with the action of the Hopf algebra
$\mathcal{U}(\mathfrak{sl}_2)$. When the coefficient ring $R$ is clear from context we will also 
write $\KR_N^{\mathfrak{sl}_2}(L)$ for simplicity.

In what follows, we will simply draw (pieces of) link diagrams (with green dots) 
to represent the $\sll_2$-equivariant Khovanov--Rozansky
$\gll_N$-homology of these diagrams. In other words, we will mostly drop
$\KR_N^{\mathfrak{sl}_2}(\cdot\,;R)$ around diagrams to prevent
overloaded figures.

\begin{rmk}
  As indicated above, the definition of the link invariant depends
  upon the parameters $t_1$ and $t_2$, but for ease
  of reading, we often remove $t_1$ and $t_2$ from the notation.
\end{rmk}

The rest of the section will be devoted to proving that this link homology is invariant under Reidemeister moves.

\begin{thm} \label{thm:sl2inv} The homology
$\KR_{N;t_1,t_2}^{\mathfrak{sl}_2}(L;\RN)$ is an invariant of framed
  oriented links.
\end{thm}

For the proof of this theorem (more precisely, the Reidemeister move-invariance), it will also be convenient, in what follows, to use a green-dot shifted version of the complexes \eqref{eqn:def-T} and \eqref{eqn:def-T-prime}: 

\begin{equation}\label{eqn:def-T-tw}
 \NB{\tikz[scale = 0.9]{\begin{scope}[font=\tiny]
  \draw[->] (0.5, -0.5) ..controls +(0,0.3) and +(0,-0.3) .. (-0.5,
  0.5) node[pos=1, above] {} coordinate[pos =0.8] (t1);
  \fill[white] (0,0) circle (2mm);
  \draw[->] (-0.5, -0.5) ..controls +(0,0.3) and +(0,-0.3) .. (0.5,
  0.5) node[pos=1, above] {} coordinate[pos =0.8] (t2);
  \filldraw[draw= green!50!black, fill = white] (t2) circle (1mm)
  node[right, green!50!black] {$t_2$};
  \filldraw[draw= green!50!black, fill = white] (t1) circle (1mm)
  node[left, green!50!black] {$t_1$};
\end{scope}}} :=
\NB{\tikz[xscale = 3, yscale = 3]{
    \node (i0) at (0, 0) {$q^{-1}$ \NB{\tikz[font= \tiny,
  scale=0.6]{\begin{scope}
  \coordinate (bl) at (-0.5, -1);
  \coordinate (br) at ( 0.5, -1);
  \coordinate (tl) at (-0.5,  1);
  \coordinate (tr) at ( 0.5,  1);
    \coordinate (ml) at (-0.5,  -.8);
        \coordinate (Ml) at (-0.5,  .8);
 \coordinate (mr) at (0.5,  -.6);
\coordinate (Mr) at (0.5,  .6);

 %\draw[>->] (bl) -- (tl) node[pos = 0, below] {$2$} node[pos = 1,
  %above] {$1$} coordinate[pos = 0.4] (ml);
   \draw[->] (bl) -- (tl);% node[pos = 0, below] {$1$} node[pos = 1,above] {$1$};
  %\draw[>->] (br) -- (tr) node[pos = 0, below] {$1$} node[pos = 1, above] {$2$} coordinate[pos = 0.6] (mr);
  
    \draw[->] (br) -- (tr);% node[pos = 0, below] {$1$} node[pos = 1, above] {$1$};
  
  %\draw[->-] (ml) -- (mr) node [pos= 0.5, above] {$1$};
  %  \draw[->-] (Ml) -- (Mr) node [pos= 0.5, above] {$1$};

\end{scope}}} };
     \node (i1) at (-1, 0) {\NB{\tikz[font= \tiny,
  scale=0.6]{\begin{scope}
  \coordinate (bl) at (-0.5, -1);
  \coordinate (br) at ( 0.5, -1);
  \coordinate (bm) at (  0,-0.3);
  \coordinate (tl) at (-0.5,  1);
  \coordinate (tr) at ( 0.5,  1);
  \coordinate (tm) at (  0, 0.3);
  \draw[>-]  (bl) .. controls +( 0, 0.5) and +(0,0) .. (bm);
  %node[below, pos = 0] {$1$};
     % \filldraw[draw= green!50!black, fill = white] (bm) circle (1mm)
 % node[left, green!50!black] {$X$};
  \draw[>-]  (br) .. controls +( 0, 0.5) and +(0,0) .. (bm);
  %node[below, pos = 0] {$1$};
  \draw[<-]  (tl) .. controls +( 0, -0.5) and +(0,0) .. (tm) coordinate[pos = 0.25] (ga) ;
  % node[above, pos = 0] {$1$} 
    \filldraw[draw= green!50!black, fill = white] (ga) circle (1mm)
  node[left, green!50!black] {$t_1$};

  \draw[<-]  (tr) .. controls +( 0, -0.5) and +(0,0) .. (tm)
  coordinate[pos = 0.25] (gb) ; %  node[above, pos = 0] {$1$} 
    \filldraw[draw= green!50!black, fill = white] (gb) circle (1mm)
  node[right, green!50!black] {$t_2$};
  \draw [double] (bm) -- (tm);% node[left, pos = 0.5] {$2$};
 
\end{scope}}} };
\draw[->] (i1) -- (i0) coordinate[pos=0.5] (a);
\node[above] at (a) {\NB{\tikz[font=\tiny, scale=.5]{}}};
  }} \ ,
\end{equation}
    \begin{equation}\label{eqn:T-prime-tw}
   \NB{\tikz[scale = 0.9]{\begin{scope}[font=\tiny]
  \draw[->] (-0.5, -0.5) ..controls +(0,0.3) and +(0,-0.3) .. (0.5,
  0.5) node[pos=1, above] {} coordinate[pos =0.8] (t2);
  \fill[white] (0,0) circle (2mm);
  \draw[->] (0.5, -0.5) ..controls +(0,0.3) and +(0,-0.3) .. (-0.5,
  0.5) node[pos=1, above] {} coordinate[pos =0.8] (t1);
  \filldraw[draw= green!50!black, fill = white] (t1) circle (1mm)
  node[left, green!50!black] {$-\bar{t}_1$};
  \filldraw[draw= green!50!black, fill = white] (t2) circle (1mm)
  node[right, green!50!black] {$-\bar{t}_2$};
\end{scope}}}:=
    \NB{\tikz[xscale = 3, yscale = 3]{
    \node (i0) at (-1, 0) { $q$\ \NB{\tikz[font= \tiny,
  scale=0.6]{}} };
     \node (i1) at (0, 0) { \NB{\tikz[font= \tiny,
  scale=0.6]{\begin{scope}
  \coordinate (bl) at (-0.5, -1);
  \coordinate (br) at ( 0.5, -1);
  \coordinate (bm) at (  0,-0.3);
  \coordinate (tl) at (-0.5,  1);
  \coordinate (tr) at ( 0.5,  1);
  \coordinate (tm) at (  0, 0.3);
  \draw[>-]  (bl) .. controls +( 0, 0.5) and +(0,0) .. (bm);
  %node[below, pos = 0] {$1$};
  \draw[>-]  (br) .. controls +( 0, 0.5) and +(0,0) .. (bm);
  %node[below, pos = 0] {$1$};
  \draw[<-]  (tl) .. controls +( 0, -0.5) and +(0,0) .. (tm) coordinate[pos = 0.25] (ga) ;
  %node[above, pos = 0] {$1$} 
    \filldraw[draw= green!50!black, fill = white] (ga) circle (1mm)
  node[left, green!50!black] {$-\bar{t}_1$};

  \draw[<-]  (tr) .. controls +( 0, -0.5) and +(0,0) .. (tm) coordinate[pos = 0.25] (gb) ;%  node[above, pos = 0] {$1$} 
    \filldraw[draw= green!50!black, fill = white] (gb) circle (1mm)
  node[right, green!50!black] {$-\bar{t}_2$};
  \draw [double] (bm) -- (tm);% node[left, pos = 0.5] {$2$};
 
\end{scope}}} };
\draw[->] (i0) -- (i1) coordinate[pos=0.5] (b);
\node[above] at (b) {\NB{\tikz[font=\tiny, scale=.5]{}}}; 
      }} \ .
    \end{equation}

\subsection{Green dots can slide through crossings}
\label{sec:green-dots-can}

We consider link diagrams with green dots. These green dots
should appear on any resolutions used to compute their
homology. As we shall see, the homology of a link diagram with green dots
does not depend on the precise location of the green dots on a
given component. Hence the green dot data can be thought as a framing
(and will be shown to have a non-trivial interaction with the framing,
see section~\ref{sec:fram-reid-1}).

We next show that we could slide green dots through crossings in the relative homotopy category.  These ideas originate from \cite{KRWitt}.  See also \cite{QRSW1} for an adaptation of these ideas to the $p$-DG setting.

\begin{lem}\label{lem:twotwists}
 Let $c,d \in \scalars$. Then there are isomorphisms of complexes in the relative homotopy category
  \begin{align} \label{homslide1}
    \NB{\tikz[font=\tiny]{\begin{scope}[font=\tiny]
  \draw[->] (0.5, -0.5) ..controls +(0,0.3) and +(0,-0.3) .. (-0.5,
  0.5) node[pos=1, above] {} coordinate[pos =0.2] (t1);
  \fill[white] (0,0) circle (2mm);
  \draw[->] (-0.5, -0.5) ..controls +(0,0.3) and +(0,-0.3) .. (0.5,
  0.5) node[pos=1, above] {} coordinate[pos =0.2] (t2);
  \filldraw[draw= green!50!black, fill = white] (t2) circle (1mm)
  node[left, green!50!black] {$c$};
  \filldraw[draw= green!50!black, fill = white] (t1) circle (1mm)
  node[right, green!50!black] {$d$};
\end{scope}}}
    \simeq
   \NB{\tikz[font=\tiny]{\begin{scope}[font=\tiny]
  \draw[->] (0.5, -0.5) ..controls +(0,0.3) and +(0,-0.3) .. (-0.5,
  0.5) node[pos=1, above] {} coordinate[pos =0.8] (t1);
  \fill[white] (0,0) circle (2mm);
  \draw[->] (-0.5, -0.5) ..controls +(0,0.3) and +(0,-0.3) .. (0.5,
  0.5) node[pos=1, above] {} coordinate[pos =0.8] (t2);
  \filldraw[draw= green!50!black, fill = white] (t2) circle (1mm)
  node[right, green!50!black] {$c$};
  \filldraw[draw= green!50!black, fill = white] (t1) circle (1mm)
  node[left, green!50!black] {$d$};
\end{scope}}}
    \qquad \text{and} \qquad 
    \NB{\tikz[font=\tiny]{\begin{scope}[font=\tiny]
  \draw[->] (-0.5, -0.5) ..controls +(0,0.3) and +(0,-0.3) .. (0.5,
  0.5) node[pos=1, above] {} coordinate[pos =0.2] (t2);
  \fill[white] (0,0) circle (2mm);
  \draw[->] (0.5, -0.5) ..controls +(0,0.3) and +(0,-0.3) .. (-0.5,
  0.5) node[pos=1, above] {} coordinate[pos =0.2] (t1);
  \filldraw[draw= green!50!black, fill = white] (t2) circle (1mm)
  node[left, green!50!black] {$c$};
  \filldraw[draw= green!50!black, fill = white] (t1) circle (1mm)
  node[right, green!50!black] {$d$};
\end{scope}}}
    \simeq
    \NB{\tikz[font=\tiny]{\begin{scope}[font=\tiny]
  \draw[->] (-0.5, -0.5) ..controls +(0,0.3) and +(0,-0.3) .. (0.5,
  0.5) node[pos=1, above] {} coordinate[pos =0.8] (t2);
  \fill[white] (0,0) circle (2mm);
  \draw[->] (0.5, -0.5) ..controls +(0,0.3) and +(0,-0.3) .. (-0.5,
  0.5) node[pos=1, above] {} coordinate[pos =0.8] (t1);
  \filldraw[draw= green!50!black, fill = white] (t1) circle (1mm)
  node[left, green!50!black] {$d$};
  \filldraw[draw= green!50!black, fill = white] (t2) circle (1mm)
  node[right, green!50!black] {$c$};
\end{scope}}},
  \end{align}
    \begin{align} \label{homslide2}
    \NB{\tikz[font=\tiny]{\begin{scope}[font=\tiny]
  \draw[->] (0.5, -0.5) ..controls +(0,0.3) and +(0,-0.3) .. (-0.5,
  0.5) node[pos=1, above] {} coordinate[pos =0.2] (t1);
  \fill[white] (0,0) circle (2mm);
  \draw[->] (-0.5, -0.5) ..controls +(0,0.3) and +(0,-0.3) .. (0.5,
  0.5) node[pos=1, above] {} coordinate[pos =0.2] (t2);
  \filldraw[draw= green!50!black, fill = green] (t2) circle (1mm)
  node[left, green!50!black] {$c$};
  \filldraw[draw= green!50!black, fill = green] (t1) circle (1mm)
  node[right, green!50!black] {$d$};
\end{scope}}}
    \simeq
   \NB{\tikz[font=\tiny]{\begin{scope}[font=\tiny]
  \draw[->] (0.5, -0.5) ..controls +(0,0.3) and +(0,-0.3) .. (-0.5,
  0.5) node[pos=1, above] {} coordinate[pos =0.8] (t1);
  \fill[white] (0,0) circle (2mm);
  \draw[->] (-0.5, -0.5) ..controls +(0,0.3) and +(0,-0.3) .. (0.5,
  0.5) node[pos=1, above] {} coordinate[pos =0.8] (t2);
  \filldraw[draw= green!50!black, fill = green] (t2) circle (1mm)
  node[right, green!50!black] {$c$};
  \filldraw[draw= green!50!black, fill = green] (t1) circle (1mm)
  node[left, green!50!black] {$d$};
\end{scope}}}
    \qquad \text{and} \qquad 
    \NB{\tikz[font=\tiny]{\begin{scope}[font=\tiny]
  \draw[->] (-0.5, -0.5) ..controls +(0,0.3) and +(0,-0.3) .. (0.5,
  0.5) node[pos=1, above] {} coordinate[pos =0.2] (t2);
  \fill[white] (0,0) circle (2mm);
  \draw[->] (0.5, -0.5) ..controls +(0,0.3) and +(0,-0.3) .. (-0.5,
  0.5) node[pos=1, above] {} coordinate[pos =0.2] (t1);
  \filldraw[draw= green!50!black, fill = green] (t2) circle (1mm)
  node[left, green!50!black] {$c$};
  \filldraw[draw= green!50!black, fill = green] (t1) circle (1mm)
  node[right, green!50!black] {$d$};
\end{scope}}}
    \simeq
    \NB{\tikz[font=\tiny]{\begin{scope}[font=\tiny]
  \draw[->] (-0.5, -0.5) ..controls +(0,0.3) and +(0,-0.3) .. (0.5,
  0.5) node[pos=1, above] {} coordinate[pos =0.8] (t2);
  \fill[white] (0,0) circle (2mm);
  \draw[->] (0.5, -0.5) ..controls +(0,0.3) and +(0,-0.3) .. (-0.5,
  0.5) node[pos=1, above] {} coordinate[pos =0.8] (t1);
  \filldraw[draw= green!50!black, fill = green] (t1) circle (1mm)
  node[left, green!50!black] {$d$};
  \filldraw[draw= green!50!black, fill = green] (t2) circle (1mm)
  node[right, green!50!black] {$c$};
\end{scope}}}.
  \end{align}
\end{lem}

\begin{proof}
We only prove
 \begin{equation} \label{homslide3}
  {}_xC:=  \NB{\tikz[font=\tiny]{\begin{scope}[font=\tiny]
  \draw[->] (0.5, -0.5) ..controls +(0,0.3) and +(0,-0.3) .. (-0.5,
  0.5) node[pos=1, above] {} coordinate[pos =0.2] (t1);
  \fill[white] (0,0) circle (2mm);
  \draw[->] (-0.5, -0.5) ..controls +(0,0.3) and +(0,-0.3) .. (0.5,
  0.5) node[pos=1, above] {} coordinate[pos =0.2] (t2);
  % \filldraw[draw= green!50!black, fill = white] (t1) circle (1mm)
  % node[left, green!50!black] {$c$};
  \filldraw[draw= green!50!black, fill = white] (t2) circle (1mm)
  node[left, green!50!black] {$1$};
\end{scope}}}
    \simeq
    \NB{\tikz[font=\tiny]{\begin{scope}[font=\tiny]
  \draw[->] (0.5, -0.5) ..controls +(0,0.3) and +(0,-0.3) .. (-0.5,
  0.5) node[pos=1, above] {} coordinate[pos =0.8] (t1);
  \fill[white] (0,0) circle (2mm);
  \draw[->] (-0.5, -0.5) ..controls +(0,0.3) and +(0,-0.3) .. (0.5,
  0.5) node[pos=1, above] {} coordinate[pos =0.8] (t2);
  % \filldraw[draw= green!50!black, fill = white] (t1) circle (1mm)
  % node[left, green!50!black] {$c$};
  \filldraw[draw= green!50!black, fill = white] (t2) circle (1mm)
  node[right, green!50!black] {$1$};
\end{scope}}} := C^y.
  \end{equation}
The general case of the first isomorphism in \eqref{homslide1} follows similarly.  The other isomorphism in \eqref{homslide1} could be proved in a similar fashion. Alternatively, it follows from the first isomorphism and Reidemeister II (to be proved later).
The isomorphisms in \eqref{homslide2} follow from the isomorphisms in \eqref{homslide1} using the fact that a solid green dot and a hollow green dot together form a symmetric function in the ground ring.

In order to prove \eqref{homslide3}, we begin by enlarging the ground ring from 
$\KN= \scalars[E_1, \dots, E_\myN]$ to $\KN[z]= \scalars[E_1, \dots, E_\myN][z]$.
The $\mathfrak{sl}_2$-action on $\KN$ extends to $\KN[z]$ in the obvious way:
\begin{equation} \label{homslide4}
\Le(z)=-1, \quad \Lh(z)=-2z, \quad \Lf(z)=z^2.
\end{equation}

Let $\lambda, \mu, \gamma \in \scalars$.
We will use the following abbreviations of state spaces of webs and their twistings.

\[
{}_{\lambda x}B^{\mu y}=
  \NB{\tikz[font= \tiny,
  scale=0.6]{\begin{scope}
  \coordinate (bl) at (-0.5, -1);
  \coordinate (br) at ( 0.5, -1);
  \coordinate (bm) at (  0,-0.3);
  \coordinate (tl) at (-0.5,  1);
  \coordinate (tr) at ( 0.5,  1);
  \coordinate (tm) at (  0, 0.3);
  \draw[>-]  (bl) .. controls +( 0, 0.5) and +(0,0) .. (bm) 
  coordinate[pos = 0.25] (gc);
  %node[below, pos = 0] {$1$};
      \filldraw[draw= green!50!black, fill = white] (gc) circle (1mm)
  node[left, green!50!black] {$\lambda$};
  \draw[>-]  (br) .. controls +( 0, 0.5) and +(0,0) .. (bm);
  %node[below, pos = 0] {$1$};
  \draw[<-]  (tl) .. controls +( 0, -0.5) and +(0,0) .. (tm)
  coordinate[pos = 0.25] (ga) ; %  node[above, pos = 0] {$1$} 
    \filldraw[draw= green!50!black, fill = white] (ga) circle (1mm)
  node[left, green!50!black] {$t_1$};

  \draw[<-]  (tr) .. controls +( 0, -0.5) and +(0,0) .. (tm)
  coordinate[pos = 0.25] (gb) ; %  node[above, pos = 0] {$1$} 
    \filldraw[draw= green!50!black, fill = white] (gb) circle (1mm)
  node[right, green!50!black] {$t_2+\mu$};

  %    \filldraw[draw= green!50!black, fill = white] (tr) circle (1mm)
  % node[right, green!50!black] {$\mu$};
  
  \draw [double] (bm) -- (tm);% node[left, pos = 0.5] {$2$};
 
\end{scope}}} , \quad \quad 
{}_{\lambda x}A^{\mu y}=
  \NB{\tikz[font= \tiny,
  scale=0.6]{\begin{scope}
  \coordinate (bl) at (-0.5, -1);
  \coordinate (br) at ( 0.5, -1);
  \coordinate (tl) at (-0.5,  1);
  \coordinate (tr) at ( 0.5,  1);
    \coordinate (ml) at (-0.5,  -.8);
        \coordinate (Ml) at (-0.5,  .8);
            \coordinate (mtr) at ( 0.5,  .5);

 \coordinate (mr) at (0.5,  -.6);
\coordinate (Mr) at (0.5,  .6);

 %\draw[>->] (bl) -- (tl) node[pos = 0, below] {$2$} node[pos = 1,
  %above] {$1$} coordinate[pos = 0.4] (ml);
   \draw[->] (bl) -- (tl)
   coordinate[pos = 0.2] (ga);% node[pos = 0, below] {$1$} node[pos = 1,  above] {$1$};
  %\draw[>->] (br) -- (tr) node[pos = 0, below] {$1$} node[pos = 1, above] {$2$} coordinate[pos = 0.6] (mr);
        \filldraw[draw= green!50!black, fill = white] (ga) circle (1mm)
  node[left, green!50!black] {$\lambda$};
    \draw[->] (br) -- (tr);% node[pos = 0, below] {$1$} node[pos = 1, above] {$1$};

            \filldraw[draw= green!50!black, fill = white] (mtr) circle (1mm)
  node[left, green!50!black] {$\mu$};
  
  %\draw[->-] (ml) -- (mr) node [pos= 0.5, above] {$1$};
  %  \draw[->-] (Ml) -- (Mr) node [pos= 0.5, above] {$1$};

\end{scope}
%%% Local Variables:
%%% mode: latex
%%% TeX-master: t
%%% End:
}} .
\]
Let $B[z]:=B \otimes \scalars[z]$ and
$A[z]:=A \otimes \scalars[z]$ with the $\mathfrak{sl}_2$-actions coming from \eqref{homslide4}.

The action of $\mathfrak{sl}_2$ on state spaces
${}_{\lambda x}B^{\mu y}[z]^{\gamma z}$ and ${}_{\lambda x}A^{\mu y}[z]^{\gamma z}$ are determined by the actions on ${}_{\lambda x}B^{\mu y}$ and ${}_{\lambda x}A^{\mu y}$ respectively and 
\begin{equation} \label{homslide5}
\Le(z^n)=-nz^{n-1}, \quad \Lh(z^n)=-(2n+\gamma)z^n, \quad
\Lf(z^n)=(n+\gamma)z^{n+1}.
\end{equation}

Recall that we have unzip and zip foams respectively:
\[
\mapH=
  \NB{\tikz[font= \tiny,
  scale=.8]{\begin{scope}
  \begin{scope}
    \coordinate (L1) at (0.2,0.4);
    \coordinate (L2) at (0,0);
    \coordinate (R1) at (2.2,0.4);
    \coordinate (R2) at (2,0);
    \coordinate (ML) at (0.6, 0.2);
    \coordinate (MR) at (1.6, 0.2);
    \draw[double] (ML) -- (MR) node[below, midway] {}; %[->-]
    \draw (MR) .. controls +(0, 0) and +(-0.3,0) .. (R1) ;
    \draw (MR) .. controls +(0, 0) and +(-0.3,0) .. (R2);
    \draw (L1) .. controls +( 0.3, 0) and +(0,0) .. (ML);
    \draw (L2) .. controls +( 0.3, 0) and +(0,0) .. (ML);
  \end{scope}  
 \begin{scope}[yshift = 1cm]
    \coordinate (L1B) at (0.2,0.4);
    \coordinate (L2B) at (0,0);
    \coordinate (R1B) at (2.2,0.4);
    \coordinate (R2B) at (2,0);
    \draw[->-] (L1B) .. controls +( 0, 0) and +(0,0) .. (R1B) node [left, pos
    = 0] {};
    \draw[->-] (L2B) .. controls +( 0, 0) and +(0,0) .. (R2B) node [left, pos
    = 0] {};
 \end{scope}  
  \draw (R1) -- (R1B);
  \draw (R2) -- (R2B);
  \draw (L1) -- (L1B);
  \draw (L2) -- (L2B);
  \draw[thick] (ML) .. controls +(0, 0.6) and +(0, 0.6) .. (MR);
\end{scope}

%%% Local Variables:
%%% mode: latex
%%% TeX-master: t
%%% End:
}} , \quad \quad 
\mapX=
  \NB{\tikz[font= \tiny,
  scale=.8]{\begin{scope}
  \begin{scope}
    \coordinate (L1) at (0.2,0.4);
    \coordinate (L2) at (0,0);
    \coordinate (R1) at (2.2,0.4);
    \coordinate (R2) at (2,0);
    \coordinate (ML) at (0.6, 0.2);
    \coordinate (MR) at (1.6, 0.2);
    \draw[double] (ML) -- (MR) node[above, midway] {};
    \draw (MR) .. controls +(0, 0) and +(-0.3,0) .. (R1) ;
    \draw (MR) .. controls +(0, 0) and +(-0.3,0) .. (R2);
    \draw (L1) .. controls +( 0.3, 0) and +(0,0) .. (ML);
    \draw (L2) .. controls +( 0.3, 0) and +(0,0) .. (ML);
  \end{scope}  
 \begin{scope}[yshift = -1cm]
    \coordinate (L1B) at (0.2,0.4);
    \coordinate (L2B) at (0,0);
    \coordinate (R1B) at (2.2,0.4);
    \coordinate (R2B) at (2,0);
    \draw[->-] (L1B) .. controls +( 0, 0) and +(0,0) .. (R1B) node [right, pos
    = 1] {};
    \draw[->-] (L2B) .. controls +( 0, 0) and +(0,0) .. (R2B) node [right, pos
    = 1] {};
 \end{scope}  
  \draw (R1) -- (R1B);
  \draw (R2) -- (R2B);
  \draw (L1) -- (L1B);
  \draw (L2) -- (L2B);
  \draw[thick] (ML) .. controls +(0, -0.6) and +(0, -0.6) .. (MR);
\end{scope}

%%% Local Variables:
%%% mode: latex
%%% TeX-master: t
%%% End:
}} .
\]
In what follows below we will need foams of the form $\lambda \mapX$ and $\lambda z \mapX$ which just mean the foam $\mapX$ with coefficients $\lambda$ or $\lambda z$ in the ground ring.  We will also need foams
\[
x_1 \mapX=
  \NB{\tikz[font= \tiny,
  scale=0.8]{\begin{scope}
  \begin{scope}
        \coordinate (X1) at (0.5,-0.1);
    \coordinate (L1) at (0.2,0.4);
    \coordinate (L2) at (0,0);
    \coordinate (R1) at (2.2,0.4);
    \coordinate (R2) at (2,0);
    \coordinate (ML) at (0.6, 0.2);
    \coordinate (MR) at (1.6, 0.2);
    \draw[double] (ML) -- (MR) node[above, midway] {}; %[->-]
    \draw (MR) .. controls +(0, 0) and +(-0.3,0) .. (R1) ;
    \draw (MR) .. controls +(0, 0) and +(-0.3,0) .. (R2);
    \draw (L1) .. controls +( 0.3, 0) and +(0,0) .. (ML);
    \draw (L2) .. controls +( 0.3, 0) and +(0,0) .. (ML);
  \end{scope}  
 \begin{scope}[yshift = -1cm]
    \coordinate (L1B) at (0.2,0.4);
    \coordinate (L2B) at (0,0);
    \coordinate (R1B) at (2.2,0.4);
    \coordinate (R2B) at (2,0);
    \draw[->-] (L1B) .. controls +( 0, 0) and +(0,0) .. (R1B) node [right, pos
    = 1] {};
    \draw[->-] (L2B) .. controls +( 0, 0) and +(0,0) .. (R2B) node [right, pos
    = 1] {}   node [pos = 0.2, above] {};
 \end{scope}  
  \draw (R1) -- (R1B) node [pos = 0.2, left] {};
  \draw (R2) -- (R2B);
  \draw (L1) -- (L1B);
  \draw (L2) -- (L2B);
  \draw[thick] (ML) .. controls +(0, -0.6) and +(0, -0.6) .. (MR);

      \fill (X1) 
  %circle (0.5mm) 
  node[below] {$\dotnewtoni[1]$};
\end{scope}

%%% Local Variables:
%%% mode: latex
%%% TeX-master: t
%%% End:
}} , \quad \quad 
y_2 \mapX=
  \NB{\tikz[font= \tiny,
  scale=0.8]{\begin{scope}
\begin{scope}
        \coordinate (Y2) at (1.7,-0.5);
    \coordinate (L1) at (0.2,0.4);
    \coordinate (L2) at (0,0);
    \coordinate (R1) at (2.2,0.4);
    \coordinate (R2) at (2,0);
    \coordinate (ML) at (0.6, 0.2);
    \coordinate (MR) at (1.6, 0.2);
    \draw[double] (ML) -- (MR) node[above, midway] {}; %[->-]
    \draw (MR) .. controls +(0, 0) and +(-0.3,0) .. (R1) ;
    \draw (MR) .. controls +(0, 0) and +(-0.3,0) .. (R2);
    \draw (L1) .. controls +( 0.3, 0) and +(0,0) .. (ML);
    \draw (L2) .. controls +( 0.3, 0) and +(0,0) .. (ML);
  \end{scope}  
 \begin{scope}[yshift = -1cm]
    \coordinate (L1B) at (0.2,0.4);
    \coordinate (L2B) at (0,0);
    \coordinate (R1B) at (2.2,0.4);
    \coordinate (R2B) at (2,0);
    \draw[->-] (L1B) .. controls +( 0, 0) and +(0,0) .. (R1B) node [right, pos
    = 1] {};
    \draw[->-] (L2B) .. controls +( 0, 0) and +(0,0) .. (R2B) node [right, pos
    = 1] {}   node [pos = 0.2, above] {};
 \end{scope}  
  \draw (R1) -- (R1B) node [pos = 0.2, left] {};
  \draw (R2) -- (R2B);
  \draw (L1) -- (L1B);
  \draw (L2) -- (L2B);
  \draw[thick] (ML) .. controls +(0, -0.6) and +(0, -0.6) .. (MR);

      \fill (Y2) 
  %circle (0.5mm) 
  node[below] {$\dotnewtoni[1]$};
\end{scope}

}} .
\]
Identity maps from $B$ to $B$ or $A$ to $A$ are given by foams which are the webs times an interval.  For these foams, we will let $x_1$ or $y_2$ be shorthand for these foams with decorations $\dotnewtoni[1]$ on the lower left and upper right facets respectively.

We can fit ${}_x C$ in a short exact sequence of
$\mathfrak{sl}_2$-equivariant state spaces:
  \begin{equation}
\NB{
  \begin{tikzpicture}[xscale =4, yscale=2]
    \node (CT) at (1.65,3) {${}_xC_{}$};
    \node (Ct) at (1.65,2) {${}_xC'(\lambda)$};
    \node (Cb) at (1.65,1) {${}_xC''(\lambda)$};
    \node (zeroT) at (1.65,3.5) {$0$};
    \node (zerob) at (1.65,0.5) {$0$};
    \draw[densely dotted] (CT) -- (1.2, 3);
    \draw[densely dotted] (Ct) -- (1.2, 2);
    \draw[densely dotted] (Cb) -- (1.2, 1);
    \draw[densely dotted, rounded corners] (-0.2, 3.2) rectangle (1.2, 2.8);
    \draw[densely dotted, rounded corners] (-1.25, 2.38) rectangle (1.2, 1.8);
    \draw[densely dotted, rounded corners] (-1.25, 1.38) rectangle
    (1.2,0.7);
    \draw [-to] (CT) -- (zeroT);
    \draw [-to] (Ct) -- (CT);
    \draw [-to] (Cb) -- (Ct);
    \draw [-to] (zerob) -- (Cb);
    \node (BT) at (0,3) {${}_xB^{}$};
    \node (AT) at (1,3) {${}_xA^{}$};
    \node (Zt) at (-1,2) {${}_{1-\lambda}B^{\lambda}[z]^{2z}$};
    \node (Bt) at (0,2) {$B[z]^{z}\oplus {}_{1-\lambda}A^{\lambda}[z]^{2z}$};
    \node (At) at (1,2) {$A[z]^{z}$};
    \node (Zb) at (-1,1) {${}_{1-\lambda}B^{\lambda}[z]^{2z}$};
    \node (Bb) at (0,1) {${}_{x}B[z]^{2z}\oplus {}_{1-\lambda}A^{\lambda}[z]^{2z}$};
    \node (Ab) at (1,1) {${}_{x}A[z]^{2z}$};
    \draw[-to] (BT) -- (AT) node[pos=0.5, above, scale = 0.7] {$\mapH$};
    \draw[-to] (Bt) -- (At) node[pos=0.5, above, scale = 0.7] {$\left(\mapH \,\,
        x_1 -z\right)$};
    \draw[-to] (Bb) -- (Ab) node[pos=0.5, above, scale = 0.7] {$\left(\mapH \,\,
        1 \right)$};
    \draw[-to] (Zt) -- (Bt) node[pos=0.5, above, scale = 0.7]
    {$\left( \begin{array}{c}z-x_1 \\ \mapH\end{array}\right)$};    
    \draw[-to] (Zb) -- (Bb) node[pos=0.5, above, scale = 0.7]
    {$\left( \begin{array}{c} 1 \\ -\mapH\end{array}\right)$};
    \draw[-to] (Zb) -- (Zt) node[pos=0.5, left, scale=0.7] {$-\Id$};
    \draw[-to] (Ab) -- (At) node[pos=0.5, left, scale=0.7] {$x_1-z$};
    \draw[-to] (Bb) -- (Bt) node[pos=0.65, left, scale=0.7]
    {$\begin{pmatrix}  x_1-z & 0 \\ 0 & 1 \end{pmatrix}
      $};
    \draw[-to] (Bt) -- (BT) node[pos=0.5, left, scale=0.7]
    {$\left(z\mapsto x_1 \,\, 0 \right)$};
    \draw[-to] (At) -- (AT) node[pos=0.5, left, scale=0.7]
    {$z\mapsto x_1$};
     \draw[-to, green!50!black] ($(Bt)+(0.15, 0.1)$) .. controls
    +(-0.1,0.1) and +(0.1, 0.1) .. +(-0.3, 0)
    node[scale=0.7,green!50!black, pos=0.1, above, xshift=0.4cm]
    {$\lambda(z-x_1)  \mapX$};\draw[-to, green!50!black] ($(Bb)+(0.15, -0.1)$) .. controls
    +(-0.1,-0.1) and +(0.1, -0.1) .. +(-0.3, 0)
    node[scale=0.7,green!50!black, pos=0.1, below] {$-\lambda \mapX$};\end{tikzpicture}}
\ ,
\end{equation}
where the green arrows indicate that the actions of $\Lf$ get further twisted by the foams on the labels of the arrow.  The actions of $\Le$ and $\Lh$ are not twisted any further.
We let the reader check that all maps are indeed
$\mathfrak{sl}_2$-equivariant. The complex ${}_x C''(\lambda)$ is null-homotopic and 
this sequence splits when one forgets the $\mathfrak{sl}_2$-module
structures. Hence ${}_xC$ and ${}_x C'(\lambda)$ are isomorphic in the
relative homotopy category.

Similarly, for any $\lambda \in \scalars$, one has an isomorphism
$C^{y}\cong (C^y)'(\lambda)$ in the 
relative homotopy category with:
\[
(C^y)'(\lambda) :=  \left(\NB{\tikz[xscale =4, yscale=2]{
    \node (Zt) at (-1,2) {${}_{-\lambda x} B^{(\lambda+1)y}[z]^ {2z}$};
    \node (Bt) at (0,2) {$B[z]^{z}\oplus {}_{-\lambda x} A^{(\lambda+1)y}[z]^ {2z}$};
    \node (At) at (1,2) {$A[z]^{z}$};
    \draw[-to] (Bt) -- (At) node[pos=0.5, above, scale = 0.7] {$\left(\mapH \,\,
        y_2-z\right)$};
    \draw[-to] (Zt) -- (Bt) node[pos=0.5, above, scale = 0.7]
    {$\left( \begin{array}{c}z-y_2 \\ \mapH\end{array}\right)$};    
     \draw[-to, green!50!black] ($(Bt)+(0.15, 0.1)$) .. controls
    +(-0.1,0.1) and +(0.1, 0.1) .. +(-0.5, 0)
    node[scale=0.7,green!50!black, pos=0.1, above] {$\lambda(z-y_2)  \mapX$}}}
\right).
\]

The complexes ${}_x C'(0)$ and $(C^y)'(-1)$ are isomorphic as
$\mathfrak{sl}_2$-equivariant complexes of state spaces. An isomorphism is
given below:
\begin{align}
\NB{
  \begin{tikzpicture}[xscale =4, yscale=2]
    \node (Ct) at (1.75,2) {${}_xC'(0)$};
    \node (Cb) at (1.75,1) {$(C^y)'(-1)$};
    \draw[densely dotted] (Cb) -- (1.2, 1);
    \draw[densely dotted] (Ct) -- (1.2, 2);
    \draw[densely dotted, rounded corners] (-1.2, 2.38) rectangle (1.2, 1.8);
    \draw[densely dotted, rounded corners] (-1.2, 1.38) rectangle  (1.2, 0.68);
    \draw[-to] (Cb) -- (Ct);
    \node (Zt) at (-1,2) {${}_xB[z]^{2z}$};
    \node (Bt) at (0,2) {$B[z]^{z}\oplus {}_xA[z]^{2z}$};
    \node (At) at (1,2) {$A[z]^{z}$};
    \node (Zb) at (-1,1) {${}_xB[z]^{2z}$};
    \node (Bb) at (0,1) {$B[z]^{z}\oplus {}_xA[z]^{2z}$};
    \node (Ab) at (1,1) {$A[z]^{z}$};
    \draw[-to] (Bt) -- (At) node[pos=0.5, above, scale = 0.7] {$\left(\mapH \,\,
        x_1 -z\right)$};
    \draw[-to] (Bb) -- (Ab) node[pos=0.5, above, scale = 0.7] {$\left(\mapH \,\,
        y_2 -z \right)$};
    \draw[-to] (Zt) -- (Bt) node[pos=0.5, above, scale = 0.7]
    {$\left( \begin{array}{c}z-x_1 \\ \mapH\end{array}\right)$};    
    \draw[-to] (Zb) -- (Bb) node[pos=0.5, above, scale = 0.7]
    {$\left( \begin{array}{c} z-y_2 \\ \mapH\end{array}\right)$};
    \draw[-to] (Zb) -- (Zt) node[pos=0.5, left, scale=0.7] {$\Id$};
    \draw[-to] (Ab) -- (At) node[pos=0.5, left, scale=0.7] {$\Id$};
    \draw[-to] (Bb) -- (Bt) node[pos=0.65, left, scale=0.7]
    {$\begin{pmatrix}  1 & \mapX \\ 0 & 1 \end{pmatrix}
      $};
\draw[-to, green!50!black] ($(Bb)+(0.15, -0.1)$) .. controls
    +(-0.1,-0.1) and +(0.1, -0.1) .. +(-0.3, 0)
    node[scale=0.7,green!50!black, pos=0.1, below] {$(y_2-z)  \mapX$};
\end{tikzpicture}} \ .
\end{align}
This proves that ${}_xC \cong C^y$. 
  \end{proof}

\subsection{Framed Reidemeister I}
\label{sec:fram-reid-1}

\begin{prop}\label{prop:framed-RI}
There are isomorphisms in the relative homotopy category
\[
  \NB{\tikz[scale = 0.6]{%\begin{scope}
%  \draw (0, -1) -- +(0,2);
%\end{scope}
%\node at (1.25, 0) {$\rightsquigarrow$};
\begin{scope}[xshift = 2.5cm]
  \draw [<-] (0, 1) -- (0, 0.5) .. controls +(0,-0.5) and +(0, 0.5)
  .. (1, -0.5) arc (180:360:0.5) -- (2,0);  
    \fill[white] (0.5, 0) circle (1mm);
  \draw (0, -1) -- (0, -0.5) .. controls +(0,0.5) and +(0, -0.5)
  .. (1, 0.5) arc (180:0:0.5) -- (2,0);
 %%%%%%
\end{scope}}} \cong tq^{-N}  \NB{\tikz[scale= 0.6, font =\tiny]{\begin{scope}
  \draw [->] (0, -1) -- +(0,2);
 %   \coordinate (A) at (0, 0);
  %      \filldraw[draw= green!50!black, fill = white] (A) circle (1mm)
  %node[left, green!50!black] 
  %%
     \coordinate (A) at (0, -.6);
        \filldraw[draw= green!50!black, fill = white] (A) circle (1mm) 
  node[left, green!50!black] 
  {$-\frac{1}{2}(N+1)$};
       \coordinate (B) at (0, .6);
        \filldraw[draw= green!50!black, fill = green] (B) circle (1mm)
  node[left, green!50!black] 
  {$-\frac{1}{2}$};
  
\end{scope}}}, \quad \quad \quad \quad
  \NB{\tikz[scale= 0.6, yscale =-1]{%\begin{scope}
%  \draw (0, -1) -- +(0,2);
%\end{scope}
%\node at (1.25, 0) {$\rightsquigarrow$};
\begin{scope}[xshift = 2.5cm]
  \draw  (0, 1) -- (0, 0.5) .. controls +(0,-0.5) and +(0, 0.5)
  .. (1, -0.5) arc (180:360:0.5) -- (2,0);  
    \fill[white] (0.5, 0) circle (1mm);
  \draw[<-] (0, -1) -- (0, -0.5) .. controls +(0,0.5) and +(0, -0.5)
  .. (1, 0.5) arc (180:0:0.5) -- (2,0);
 %%%%%%
\end{scope}
%%% Local Variables:
%%% mode: latex
%%% TeX-master: t
%%% End:
}} \cong t^{-1}q^{N}\NB{\tikz[scale= 0.6, font =\tiny]{\begin{scope}
  \draw [->] (0, -1) -- +(0,2);
   \coordinate (A) at (0, -.6);
        \filldraw[draw= green!50!black, fill = white] (A) circle (1mm)
  node[left, green!50!black] 
  {$\frac{1}{2}(N+1)$};
       \coordinate (B) at (0, .6);
        \filldraw[draw= green!50!black, fill = green] (B) circle (1mm)
  node[left, green!50!black] 
  {$\frac{1}{2}$};
\end{scope}}} .
\]
\end{prop}

\begin{proof}
Both isomorphisms are similar and we will only prove the first one. Using Lemma \ref{lem:twotwists}, the desired isomorphism is equivalent to 
\begin{equation}\label{eqn:twisted-curl}
  \NB{\tikz[font=\tiny, scale=0.6]{
\begin{scope}[xshift = 2.5cm]
  \draw [<-] (0, 1) -- (0, 0.5) .. controls +(0,-0.5) and +(0, 0.5)
  .. (1, -0.5) arc (180:360:0.5) -- (2,0);  
    \fill[white] (0.5, 0) circle (1mm);
  \draw (0, -1) -- (0, -0.5) .. controls +(0,0.5) and +(0, -0.5)
  .. (1, 0.5) arc (180:0:0.5) -- (2,0);
       \coordinate (A) at (2, -.6);
        \filldraw[draw= green!50!black, fill = white] (A) circle (1mm) 
  node[right, green!50!black] 
  {$\frac{1}{2}(N+1)$};
       \coordinate (B) at (2, .6);
        \filldraw[draw= green!50!black, fill = green] (B) circle (1mm)
  node[right, green!50!black] 
  {$\frac{1}{2}$};
  
\end{scope}}} \cong   \ tq^{-N}\NB{\tikz[scale=0.6]{\begin{scope}
  \draw [->] (0, -1) -- +(0,2);
\end{scope}}} \ ,
\end{equation}
which we will show. By definition and Lemma \ref{lem:twotwists}, the green-dotted curl $\NB{\tikz[font= \tiny,
  scale=0.4]{}}$ is given, up to green-dotted shifting, by the complex
\[
   \NB{\tikz[xscale = 3, yscale = 3]{
       \node (i0) at (0, 0) { \NB{\tikz[font= \tiny,
  scale=0.6]{\begin{scope}
  \coordinate (bl) at (-0.5, -1);
  \coordinate (br) at ( 0.5, -1);
  \coordinate (bm) at (  0,-0.3);
  \coordinate (tl) at (-0.5,  1);
  \coordinate (tr) at ( 0.5,  1);
  \coordinate (tm) at (  0, 0.3);
  \draw[>-]  (bl) .. controls +( 0, 0.5) and +(0,0) .. (bm);
  \draw[>-]  (br) .. controls +( 0, 0.5) and +(0,0) .. (bm);
  \draw[<-]  (tl) .. controls +( 0, -0.5) and +(0,0) .. (tm) coordinate[pos = 0.25] (ga) ; 
    \filldraw[draw= green!50!black, fill = white] (ga) circle (1mm)
  node[left, green!50!black] {$t_1$};

  \draw[<-]  (tr) .. controls +( 0, -0.5) and +(0,0) .. (tm)
  coordinate[pos = 0.25] (gb) ;
    \filldraw[draw= green!50!black, fill = white] (gb) circle (1mm)
  node[left, green!50!black] {$t_2$};
  \draw [double] (bm) -- (tm);
  \draw (.5, 1) arc (180:0:0.5) -- (1.5,-1) arc (180:0:-0.5);
       \coordinate (A) at (1.5, -.8);
        \filldraw[draw= green!50!black, fill = white] (A) circle (1mm) 
  node[right, green!50!black] 
  {$\frac{1}{2}(N-1)$};
       \coordinate (B) at (1.5, .8);
        \filldraw[draw= green!50!black, fill = green] (B) circle (1mm)
  node[right, green!50!black] 
  {$\frac{1}{2}$};
\end{scope}}} };
   \node (i1) at (1.5, 0) {$q^{-1}$ \NB{\tikz[font= \tiny,
  scale=0.6]{\begin{scope}
  \coordinate (bl) at (-0.5, -1);
  \coordinate (br) at ( 0.5, -1);
  \coordinate (tl) at (-0.5,  1);
  \coordinate (tr) at ( 0.5,  1);
    \coordinate (ml) at (-0.5,  -.8);
        \coordinate (Ml) at (-0.5,  .8);
 \coordinate (mr) at (0.5,  -.6);
\coordinate (Mr) at (0.5,  .6);
   \draw[->] (bl) -- (tl);  
    \draw[->] (br) -- (tr);
    \draw (.5, 1) arc (180:0:0.5) -- (1.5,-1) arc (180:0:-0.5);
     \coordinate (A) at (1.5, -.8);
        \filldraw[draw= green!50!black, fill = white] (A) circle (1mm) 
  node[right, green!50!black] 
  {$\frac{1}{2}(N-1)$};
       \coordinate (B) at (1.5, .8);
        \filldraw[draw= green!50!black, fill = green] (B) circle (1mm)
  node[right, green!50!black] 
  {$\frac{1}{2}$};
\end{scope}}} };
        \draw[->] (i0) -- (i1)  node[pos=0.5, above] {$\mapH$};
  }} .
\]
We can embed the above complex into a short exact sequence of $\sll_2$-equivariant complexes that splits when forgetting the $\sll_2$-actions as follows: 
\[
   \NB{\tikz[xscale = 3, yscale = 3]{
       \node (i0) at (0, 0) {\NB{\tikz[font= \tiny,
  scale=.6]{}} };
   \node (i1) at (1.5, 0) { $q^{-1}$\NB{\tikz[font= \tiny,
  scale=0.6]{}} };
    \node (i2) at (1.5, 1) { $q^{-N}$\NB{\tikz[font= \tiny,
  scale=0.6]{}} };
 \node (i3) at (1.5, -2) {$\oplus [N-1]$ \NB{\tikz[font= \tiny, scale=0.6]{\begin{scope}
  \draw [->] (0, -1) -- +(0,2);
    \coordinate (A) at (0, 0);
       % \filldraw[draw= green!50!black, fill = white] (A) circle (1mm)
 % node[left, green!50!black] {$-\frac{1}{2}$};
\end{scope}}} };
       \node (i4) at (0, -2) {$\oplus [N-1]$ \NB{\tikz[font= \tiny, scale=0.6]{}}};
   \draw[->] (i0) -- (i1)  node[pos=0.5, above] {$\mapH$};
      \draw[->] (i2) -- (i1)  node[pos=0.5, right] {\NB{\tikz[font= \tiny,
  scale=0.9,-]{\begin{scope}
  \draw (0,0) arc (180 :0: 0.5cm and 0.2cm) node[above, pos =
  0.5] {};
  \draw[very thin] (0,0) arc (180 :0: 0.5cm and -0.6cm) node[pos=0.5,
  above] {};
  \draw (0,0) arc (180 :0: 0.5cm and -0.2cm);
\end{scope}

%%% Local Variables:
%%% mode: latex
%%% TeX-master: t
%%% End:
}}};
         \draw[->] (i1) -- (i3)  node[pos=0.5, right] {$\begin{pmatrix} 
         \NB{\tikz[font= \tiny, scale=0.9]{\begin{scope}[-]
  \draw (0,0) arc (180 :0: 0.5cm and 0.2cm);
  \draw[very thin] (0,0) arc (180 :0: 0.5cm and 0.6cm) node[pos=0.5,
  below] {};
  \draw (0,0) arc (180 :0: 0.5cm and -0.2cm)node[ below, pos =
  0.5] {};
\end{scope}

%%% Local Variables:
%%% mode: latex
%%% TeX-master: t
%%% End:
}} \\
         \cdots
         \\[.5cm]
         \NB{\tikz[font= \tiny, scale=0.9]{\begin{scope}[-]
  \draw (0,0) arc (180 :0: 0.5cm and 0.2cm);
  \draw[very thin] (0,0) arc (180 :0: 0.5cm and 0.6cm) node[pos=0.5,
  below] {$\dotnewtoni[N-2]$};
  \draw (0,0) arc (180 :0: 0.5cm and -0.2cm)node[ below, pos =
  0.5] {};
\end{scope}

%%% Local Variables:
%%% mode: latex
%%% TeX-master: t
%%% End:
}}
         \end{pmatrix}$};
       \draw[->] (i0) -- (i4)  node[pos=0.5, left] {$\begin{pmatrix} 
         \NB{\tikz[font= \tiny, scale=0.9]{}} \circ \mapH \\
         \cdots
         \\[.5cm]
         \NB{\tikz[font= \tiny, scale=0.9]{}} \circ \mapH
         \end{pmatrix}$};
         \draw[->] (i4) -- (i3)  node[pos=0.5, above] {$\Id$};
  }} .
\]

The objects in the bottom row are copies of the identity web with some unspecified twisted $\sll_2$-structure, in particular the $\sll_2$-structure is not compatible with the direct sum. The $\sll_2$-structures arise by identifying these object with the quotient of the second complex with the first one.

The crucial point being that the bottom map is the identity map. Thus the bottom complex is contractible.  

Consequently, taking the inverse twisting of equation \eqref{eqn:twisted-curl} gives us the desired
\[
  \NB{\tikz[scale=0.6]{}} \cong tq^{-N} \NB{\tikz[font =\tiny,scale=0.6]{}} .
\]
The proposition follows.
\end{proof}

\subsection{Reidemeister II}
\label{sec:reidemeister-2-}

The next lemma is analogous to \cite[Lemma 4.3]{KRWitt} and
\cite[Lemma 3.5]{QiSussanLink}.  \begin{lem} \label{dumbdumbsplit:lem}
There is a short exact sequence of webs which splits when forgetting the action of $\mathfrak{sl}_2$:
\begin{equation}
  \NB{\tikz[xscale = 3, yscale = 3]{
       \node (i0) at (0, 0) { \NB{\tikz[font= \tiny,
  scale=0.6]{\begin{scope}
  \coordinate (bl) at (-0.5, -1);
  \coordinate (br) at ( 0.5, -1);
  \coordinate (bm) at (  0,-0.3);
  \coordinate (tl) at (-0.5,  1);
  \coordinate (tr) at ( 0.5,  1);
  \coordinate (tm) at (  0, 0.3);
  \draw[>-]  (bl) .. controls +( 0, 0.5) and +(0,0) .. (bm);
%  node[below, pos = 0] {$1$};
  \draw[>-]  (br) .. controls +( 0, 0.5) and +(0,0) .. (bm);
%  node[below, pos = 0] {$1$};
  \draw[<-]  (tl) .. controls +( 0, -0.5) and +(0,0) .. (tm) coordinate[pos = 0.25] (ga) ;
 % node[above, pos = 0] {$1$} 
    \filldraw[draw= green!50!black, fill = white] (ga) circle (1mm)
  node[left, green!50!black] {$1$};

  \draw[<-]  (tr) .. controls +( 0, -0.5) and +(0,0) .. (tm)
  coordinate[pos = 0.25] (gb) ; %  node[above, pos = 0] {$1$} 
    \filldraw[draw= green!50!black, fill = white] (gb) circle (1mm)
  node[left, green!50!black] {$1$};
  \draw [double] (bm) -- (tm);% node[left, pos = 0.5] {$2$};
 
\end{scope}}} };
    \node (i1) at (-1,0) { \NB{\tikz[font= \tiny,
  scale=0.6]{\begin{scope}
  \coordinate (bl) at (-0.5, -1);
  \coordinate (br) at ( 0.5, -1);
  \coordinate (bm) at (  0,-0.3);
  \coordinate (tl) at (-0.5,  1);
  \coordinate (tr) at ( 0.5,  1);
  \coordinate (tm) at (  0, 0.3);
    \coordinate (brr) at ( 1.5, -1);
        \coordinate (trr) at ( 1.5, 1);
\draw[>-]  (bl) .. controls +( 0, 0.5) and +(0,0) .. (bm);
  %node[below, pos = 0] {$1$};
  \draw[>-]  (br) .. controls +( 0, 0.5) and +(0,0) .. (bm);
 % node[below, pos = 0] {$1$};
  \draw[<-]  (tl) .. controls +( 0, -0.5) and +(0,0) .. (tm)
  coordinate[pos = 0.25] (ga) ; %  node[above, pos = 0] {} 
    \filldraw[draw= green!50!black, fill = white] (ga) circle (1mm)
  node[left, green!50!black] {$t_1$};

  \draw[<-]  (tr) .. controls +( 0, -0.5) and +(0,0) .. (tm)
  node[above, pos = 0] {} coordinate[pos = 0.25] (gb) ;
    \filldraw[draw= green!50!black, fill = white] (gb) circle (1mm)
  node[left, green!50!black] {$t_2$};
  \draw [double] (bm) -- (tm); % node[left, pos = 0.5] {$2$};
% \draw [->-] (brr) -- (trr);
 %%%%%
   \coordinate (blU) at (-0.5, 1);
  \coordinate (brU) at ( 0.5, 1);
  \coordinate (bmU) at (  0,1.7);
  \coordinate (tlU) at (-0.5,  3);
  \coordinate (trU) at ( 0.5,  3);
  \coordinate (tmU) at (  0, 2.3);
    \coordinate (brrU) at ( 1.5, 1);
        \coordinate (trrU) at ( 1.5, 3);
\draw[]  (blU) .. controls +( 0, 0.5) and +(0,0) .. (bmU);
%  node[left, pos = 0] {$1$};
  \draw[]  (brU) .. controls +( 0, 0.5) and +(0,0) .. (bmU);
%  node[right, pos = 0] {$1$};
  \draw[<-]  (tlU) .. controls +( 0, -0.5) and +(0,0) .. (tmU)
  coordinate[pos = 0.25] (gaU) ; %  node[above, pos = 0] {$1$} 
  %  \filldraw[draw= green!50!black, fill = white] (gaU) circle (1mm)
  %node[left, green!50!black] {};

  \draw[<-]  (trU) .. controls +( 0, -0.5) and +(0,0) .. (tmU)
  coordinate[pos = 0.25] (gbU) ; %  node[above, pos = 0] {$1$} 
    %\filldraw[draw= green!50!black, fill = white] (gbU) circle (1mm)
  %node[right, green!50!black] {};
  \draw [double] (bmU) -- (tmU);% node[left, pos = 0.5] {$2$};
% \draw [->-] (brr) -- (trrU)  node[below, pos = 0] {$1$};
\end{scope}
%%% Local Variables:
%%% mode: latex
%%% TeX-master: t
%%% End:
}} };
      \node (i3) at (-2, 0) { \NB{\tikz[font= \tiny,
  scale=0.6]{}} };
  \draw[->] (i3) -- (i1)  node[pos=0.5, above] {$\mapB$};
      \draw[->] (i1) -- (i0)  node[pos=0.5, above] {$\mapU$};
        \draw[-to] (i0) .. controls +(-.25, -0.75) and +(+.25, -0.75) .. (i1)
  node[pos=0.5, below] {$\mapB^X - {}^X \mapB$};
   \draw[-to] (i1) .. controls +(-.25, -0.75) and +(+.25, -0.75) .. (i3)
  node[pos=0.5, below] {$\mapU^X - {}^X \mapU$};
      }}
\end{equation}
where the splitting maps are given by
\[
\mapU^X-{}^X\mapU =
\NB{\tikz[]{\begin{scope}[font =\tiny]
  \begin{scope}
    \coordinate (L) at (0,0);
    \coordinate (R) at (2,0);
    \coordinate (ML) at (0.5, 0);
    \coordinate (MR) at (1.5, 0);
    \draw[->-] (ML).. controls + (0.4, 0.4) and +(-0.2, 0.4) .. (MR)
    node[above, pos =0.3] {$\dotnewtoni[0]$}; %    node[above, pos =0.7] {$1$}     
    \draw[->-] (ML).. controls + (0.2, -0.4) and +(-0.4, -0.4) .. (MR)
    node[above, pos =0.7, yshift = -0.5mm] {$\dotnewtoni[1]$}; %    node[left, pos =0.3] {$1$} 
  \end{scope}  
 \begin{scope}[yshift = 1.3cm]
    \coordinate (LB) at (0,0);
    \coordinate (RB) at (2,0);
    %\draw[double] (LB) -- (RB);
  \end{scope}  
  \draw[double] (MR)--(R) -- (RB) -- (LB) -- (L)-- (ML);
  %\draw[double] (L) -- (LB);
  \draw[thick] (ML) .. controls +(0, 1) and +(0, 1) .. (MR);
\end{scope}

%%% Local Variables:
%%% mode: latex
%%% TeX-master: t
%%% End:
}}
 - \NB{\tikz[]{\begin{scope}[font =\tiny]
  \begin{scope}
    \coordinate (L) at (0,0);
    \coordinate (R) at (2,0);
    \coordinate (ML) at (0.5, 0);
    \coordinate (MR) at (1.5, 0);
    \draw[->-] (ML).. controls + (0.4, 0.4) and +(-0.2, 0.4) .. (MR)
    node[above, pos =0.3] {$\dotnewtoni[1]$}; %    node[above, pos =0.7] {$1$}     
    \draw[->-] (ML).. controls + (0.2, -0.4) and +(-0.4, -0.4) .. (MR)
    node[above, pos =0.7, yshift = -0.5mm] {$\dotnewtoni[0]$}; %    node[left, pos =0.3] {$1$} 
  \end{scope}  
 \begin{scope}[yshift = 1.3cm]
    \coordinate (LB) at (0,0);
    \coordinate (RB) at (2,0);
    %\draw[double] (LB) -- (RB);
  \end{scope}  
  \draw[double] (MR)--(R) -- (RB) -- (LB) -- (L)-- (ML);
  %\draw[double] (L) -- (LB);
  \draw[thick] (ML) .. controls +(0, 1) and +(0, 1) .. (MR);
\end{scope}

%%% Local Variables:
%%% mode: latex
%%% TeX-master: t
%%% End:
}} 
\]

\[
\mapB^X-{}^X\mapB =
\NB{\tikz[]{\begin{scope}[font=\tiny]
  \begin{scope}
    \coordinate (L) at (0,0);
    \coordinate (R) at (2,0);
    \coordinate (ML) at (0.5, 0);
    \coordinate (MR) at (1.5, 0);
    \draw[->-] (ML).. controls + (0.4, 0.4) and +(-0.2, 0.4) .. (MR) node[below, pos =0.3] {$\dotnewtoni[0]$};%    node[above, pos=0.7 ] {$1$} 
    \draw[->-] (ML).. controls +(0.2, -0.4) and +(-0.4, -0.4) .. (MR) node[below, pos =0.75] {$\dotnewtoni[1]$};%    node[below, pos =0.3] {$1$} 
  \end{scope}  
 \begin{scope}[yshift = -1.3cm]
    \coordinate (LB) at (0,0);
    \coordinate (RB) at (2,0);

  \end{scope}  
  \draw[double] (MR) -- (R) -- (RB) -- (LB) -- (L) -- (ML);% node[right] {$2$};
  \draw[thick] (ML) .. controls +(0, -1) and +(0, -1) .. (MR);
\end{scope}

%%% Local Variables:
%%% mode: latex
%%% TeX-master: t
%%% End:
}}
-
 \NB{\tikz[]{\begin{scope}[font=\tiny]
  \begin{scope}
    \coordinate (L) at (0,0);
    \coordinate (R) at (2,0);
    \coordinate (ML) at (0.5, 0);
    \coordinate (MR) at (1.5, 0);
    \draw[->-] (ML).. controls + (0.4, 0.4) and +(-0.2, 0.4) .. (MR) node[below, pos =0.3] {$\dotnewtoni[1]$};%    node[above, pos=0.7 ] {$1$} 
    \draw[->-] (ML).. controls +(0.2, -0.4) and +(-0.4, -0.4) .. (MR) node[below, pos =0.75] {$\dotnewtoni[0]$};%    node[below, pos =0.3] {$1$} 
  \end{scope}  
 \begin{scope}[yshift = -1.3cm]
    \coordinate (LB) at (0,0);
    \coordinate (RB) at (2,0);

  \end{scope}  
  \draw[double] (MR) -- (R) -- (RB) -- (LB) -- (L) -- (ML);% node[right] {$2$};
  \draw[thick] (ML) .. controls +(0, -1) and +(0, -1) .. (MR);
\end{scope}

%%% Local Variables:
%%% mode: latex
%%% TeX-master: t
%%% End:
}} 
 \ .
\]
\end{lem}

\begin{proof}
This is a straightforward calculation and analogous to \cite[Lemma 3.5]{QiSussanLink}. We leave it as an exercise for the reader.
\end{proof}

The next result is a foam version of
\cite[Theorem 4.2]{KRWitt} and
\cite[Proposition 3.7]{QiSussanLink}, whose proof is included for the sake of completeness.
\begin{prop}
There are isomorphisms in the relative homotopy category
\[
  \NB{\tikz[scale= 0.6]{\begin{scope}[font=\tiny]
  \draw (0.5, -0.5) ..controls +(0,0.3) and +(0,-0.3) .. (-0.5,
  0.5);% node[pos=1, above] {} coordinate[pos =0.2] (t1);
  \fill[white] (0,0) circle (2mm);
  \draw (-0.5, -0.5) ..controls +(0,0.3) and +(0,-0.3) .. (0.5,
  0.5);
  \draw[->] (-0.5, 0.5) ..controls +(0,0.3) and +(0,-0.3) .. (0.5,
  1.5);
  \fill[white] (0,1) circle (2mm);
  \draw[->] (0.5, 0.5) ..controls +(0,0.3) and +(0,-0.3) .. (-0.5,
  1.5);% node[pos=1, above] {} coordinate[pos =0.2] (t1);

  % node[pos=1, above] {} coordinate[pos =0.2] (t2);
  % \filldraw[draw= green!50!black, fill = white] (t2) circle (1mm)
  % node[left, green!50!black] {$c$};
  % \filldraw[draw= green!50!black, fill = white] (t1) circle (1mm)
  % node[right, green!50!black] {$d$};
\end{scope}}}\cong \NB{\tikz[font= \tiny,
  scale=0.6]{}}  \cong \NB{\tikz[xscale= -1, scale=0.6]{}}.
\]
\end{prop}

\begin{proof}
By definition, up to green-dot shifting, the complex for the leftmost term $\NB{\tikz[scale= 0.4]{}}$ is given as the total complex
\begin{equation}    
  \NB{\tikz[xscale = 3, yscale = 3]{
       \node (i0) at (0, 0) { $q^{-1}$\NB{\tikz[font= \tiny,
  scale=0.6]{}} };
    \node (i1) at (-1, .5) { \NB{\tikz[font= \tiny,
  scale=0.6]{}} };
    \node (i2) at (-1, -.5) { \NB{\tikz[font= \tiny,
  scale=0.6]{\begin{scope}
  \coordinate (bl) at (-0.5, -1);
  \coordinate (br) at ( 0.5, -1);
  \coordinate (bm) at (  0,-0.3);
  \coordinate (tl) at (-0.5,  1);
  \coordinate (tr) at ( 0.5,  1);
  \coordinate (tm) at (  0, 0.3);
    \coordinate (brr) at ( 1.5, -1);
        \coordinate (trr) at ( 1.5, 1);
\draw[>-]  (bl) .. controls +( 0, 0.5) and +(0,0) .. (bm);
 % node[below, pos = 0] {$1$};
  \draw[>-]  (br) .. controls +( 0, 0.5) and +(0,0) .. (bm);
 % node[below, pos = 0] {$1$};
  \draw[<-]  (tl) .. controls +( 0, -0.5) and +(0,0) .. (tm) coordinate[pos = 0.25] (ga) ;
  %node[above, pos = 0] {} 
    \filldraw[draw= green!50!black, fill = white] (ga) circle (1mm)
  node[left, green!50!black] {$t_1$};

  \draw[<-]  (tr) .. controls +( 0, -0.5) and +(0,0) .. (tm)
  node[above, pos = 0] {} coordinate[pos = 0.25] (gb) ;
    \filldraw[draw= green!50!black, fill = white] (gb) circle (1mm)
  node[left, green!50!black] {$t_2$};
  \draw [double] (bm) -- (tm);% node[left, pos = 0.5] {$2$};
% \draw [->-] (brr) -- (trr);
 %%%%%
   \coordinate (blU) at (-0.5, 1);
  \coordinate (brU) at ( 0.5, 1);
  \coordinate (bmU) at (  0,1.7);
  \coordinate (tlU) at (-0.5,  3);
  \coordinate (trU) at ( 0.5,  3);
  \coordinate (tmU) at (  0, 2.3);
    \coordinate (brrU) at ( 1.5, 1);
        \coordinate (trrU) at ( 1.5, 3);
\draw[]  (blU) .. controls +( 0, 0.5) and +(0,0) .. (bmU);
 % node[left, pos = 0] {$1$};
  \draw[]  (brU) .. controls +( 0, 0.5) and +(0,0) .. (bmU);
%  node[right, pos = 0] {$1$};
  \draw[<-]  (tlU) .. controls +( 0, -0.5) and +(0,0) .. (tmU) coordinate[pos = 0.25] (gaU) ;
%  node[above, pos = 0] {$1$} 
    \filldraw[draw= green!50!black, fill = white] (gaU) circle (1mm)
  node[left, green!50!black] {$-\bar{t}_1$};

  \draw[<-]  (trU) .. controls +( 0, -0.5) and +(0,0) .. (tmU) coordinate[pos = 0.25] (gbU) ;%  node[above, pos = 0] {$1$} 
    \filldraw[draw= green!50!black, fill = white] (gbU) circle (1mm)
  node[right, green!50!black] {$-\bar{t}_2$};
  \draw [double] (bmU) -- (tmU);% node[left, pos = 0.5] {$2$};
% \draw [->-] (brr) -- (trrU)  node[below, pos = 0] {$1$};
\end{scope}
%%% Local Variables:
%%% mode: latex
%%% TeX-master: t
%%% End:
}} };
      \node (i3) at (-2, 0) {$q$ \NB{\tikz[font= \tiny,
  scale=0.6]{}} };
  \draw[->] (i3) -- (i1) node[pos=0.5, above] {$\mapH$};
    \draw[->] (i3) -- (i2) node[pos=0.5, above] {$-\mapX$};
      \draw[->] (i1) -- (i0) node[pos=0.5, above] {$\mapX$};
  \draw[->] (i2) -- (i0) node[pos=0.5, above] {$\mapH$};
  }} .
\end{equation}
Here we have used  \eqref{eqn:def-T-tw} and \eqref{eqn:T-prime-tw} for constructing the total complex.

Using Lemma \ref{dumbdumbsplit:lem} and the fact that  $t_1+\bar{t}_1=t_2+\bar{t}_2=1$, one can fit this complex into a short exact sequence of $\sll_2$-equivariant complexes, which splits when forgetting the $\sll_2$-structures:

\begin{equation}
  \label{eq:1}
  \NB{\tikz{ \node (A) at (0,0) {
    \NB{\tikz[xscale = 4, yscale = 3]{
       \node (i0) at (0, 0) { $q^{-1}$\NB{\tikz[font= \tiny,
  scale=0.6]{}} };
    \node (i1) at (-1, .5) { \NB{\tikz[font= \tiny,
  scale=0.6]{}} };
    \node (i2) at (-1, -.5) { \NB{\tikz[font= \tiny,
  scale=0.6]{}} };
      \node (i3) at (-2, 0) {$q$ \NB{\tikz[font= \tiny,
  scale=0.6]{}} };
  \draw[->] (i3) -- (i1) node[pos=0.5, above] {$\mapH$};
    \draw[->] (i3) -- (i2) node[pos=0.5, above] {$-\mapX$};
      \draw[->] (i1) -- (i0) node[pos=0.5, above] {$\mapX$};
  \draw[->] (i2) -- (i0) node[pos=0.5, above] {$\mapH$};
}}};
\node (B) at (0, 6) {
  \NB{\tikz[xscale = 4, yscale = 3]{
       \node (i0) at (0, 0) { $q^{-1}$\NB{\tikz[font= \tiny,
  scale=0.6]{}} };
    \node (i1) at (-1, .5) { \NB{\tikz[font= \tiny,
  scale=0.6]{}} };
    \node (i2) at (-1, -.5) { $q^{-1}$\NB{\tikz[font= \tiny,
  scale=0.6]{}} };
      \node (i3) at (-2, 0) {0}; \draw[->] (i3) -- (i1);
    \draw[->] (i3) -- (i2);
      \draw[->] (i1) -- (i0)  node[pos=0.5, above] {$\mapX$};
  \draw[->] (i2) -- (i0) node[pos=0.5, above] {$\Id$};
  }} 
};
\node (C) at (0, -6) {
\NB{\tikz[xscale = 4, yscale = 3]{
        \node (i0) at (0, 0) {0}; \node (i1) at (-1, .5) {0}; \node (i2) at (-1, -.5) {$q$ \NB{\tikz[font= \tiny,
  scale=0.6]{}} };
      \node (i3) at (-2, 0) {$q$ \NB{\tikz[font= \tiny,
  scale=0.6]{}} };
  \draw[->] (i3) -- (i1);\draw[->] (i3) -- (i2) node[pos=0.5, above] {$-\mathrm{Id}$};
      \draw[->] (i1) -- (i0);\draw[->] (i2) -- (i0); }}
};
\begin{scope}[ dotted, rounded corners]
  \draw (-5.5, -2.5) rectangle +(11,5);
  \draw (-5.5, -8) rectangle +(11,4);
  \draw (-5.5, 3.7) rectangle +(11,4.5);
\end{scope}
\draw[->] (B) -- (A);
\draw[->] (A) -- (C);
}}
\end{equation}
Since the last complex, is obviously  contractible, the two first one are isomorphic in the relative homotopy category.

Finally, contracting out the terms
\begin{equation}
     \NB{\tikz[xscale = 3, yscale = 3]{
       \node (i0) at (.5, 0) { $q^{-1}$\NB{\tikz[font= \tiny,
  scale=0.6]{}} };
  }} ,
\end{equation}
we are left with $\NB{\tikz[scale= 0.3]{}} \cong \NB{\tikz[font= \tiny,
  scale=0.3]{}}$ .
The proof of the other isomorphism is similar.
\end{proof}

We now prove the other Reidemeister II relation.

\begin{prop}
There are isomorphisms in the relative homotopy category
\[
\NB{\tikz[]{\input{\imagesfolder/pdg_horres}}} \cong  \NB{\tikz[]{\begin{scope}[xshift = 2.5cm]
  \draw [<-] (0, -0.5) .. controls +(.5,0) and +(-.5, 0)
  .. (1, 0.5); 
    \fill[white] (0.5, 0) circle (1mm);

    \draw(0, 0.5) .. controls +(.5,0) and +(-.5, 0)
  .. (1, -0.5);
 %%%%%%
  \draw[] (1, 0.5) .. controls +(.5,0) and +(-.5, 0)
  .. (2, -0.5);
    \fill[white] (1.5, 0) circle (1mm);  
      \draw[<-] (2, 0.5)  .. controls +(-.5,0) and +(.5, 0)
  .. (1, -0.5); 
\end{scope}

%\begin{scope}[xshift = 2.5cm]
%  \draw [<-] (0, -0.5) .. controls +(0,0.5) and +(0, -0.5)
 % .. (1, 0.5); 
 %   \fill[white] (0.5, 0) circle (1mm);

 %   \draw [>-](0, 0.5) .. controls +(0,-0.5) and +(0, 0.5)
 % .. (1, -0.5);
 %%%%%%
 % \draw (1, 0.5) .. controls +(0,-0.5) and +(0, 0.5)
 % .. (2, -0.5);

 %   \fill[white] (1.5, 0) circle (1mm);

 %     \draw (2, 0.5)  .. controls +(0,-0.5) and +(0, 0.5)
%  .. (1, -0.5); 
%\end{scope}
}} .
\]
\end{prop}

\begin{proof}
First note that in the relative homotopy category that green dots slide through crossings (Lemma \ref{lem:twotwists}), so we have
\begin{equation} \label{greendoublecrosseq}
     \NB{\tikz[]{\begin{scope}[xshift = 2.5cm]
  \draw [<-] (0, -0.5) .. controls +(.5,0) and +(-.5, 0)
  .. (1, 0.5); 
    \fill[white] (0.5, 0) circle (1mm);

    \draw(0, 0.5) .. controls +(.5,0) and +(-.5, 0)
  .. (1, -0.5);
 %%%%%%
  \draw[] (1, 0.5) .. controls +(.5,0) and +(-.5, 0)
  .. (2, -0.5);
    \fill[white] (1.5, 0) circle (1mm);  
      \draw[<-] (2, 0.5)  .. controls +(-.5,0) and +(.5, 0)
  .. (1, -0.5); 
\end{scope}

%\begin{scope}[xshift = 2.5cm]
%  \draw [<-] (0, -0.5) .. controls +(0,0.5) and +(0, -0.5)
 % .. (1, 0.5); 
 %   \fill[white] (0.5, 0) circle (1mm);

 %   \draw [>-](0, 0.5) .. controls +(0,-0.5) and +(0, 0.5)
 % .. (1, -0.5);
 %%%%%%
 % \draw (1, 0.5) .. controls +(0,-0.5) and +(0, 0.5)
 % .. (2, -0.5);

 %   \fill[white] (1.5, 0) circle (1mm);

 %     \draw (2, 0.5)  .. controls +(0,-0.5) and +(0, 0.5)
%  .. (1, -0.5); 
%\end{scope}
}} \cong \NB{\tikz[font=\tiny]{\begin{scope}[xshift = 2.5cm]
  \draw [<-] (0, -0.5) .. controls +(.5,0) and +(-.5, 0)
  .. (1, 0.5) ; 
    \fill[white] (0.5, 0) circle (1mm);

    \draw (0, 0.5) .. controls +(.5,0) and +(-.5, 0)
  .. (1, -0.5) coordinate[pos = 0.85] (A) ;
 %%%%%%
  \draw[<-] (1, 0.5) .. controls +(.5,0) and +(-.5, 0)
  .. (2, -0.5);

    \fill[white] (1.5, 0) circle (1mm);

      \draw (2, 0.5)  .. controls +(-.5,0) and +(.5, 0)
  .. (1, -0.5) coordinate[pos = 0] (C) coordinate[pos = .25] (D) coordinate[pos = .8] (B); 
\filldraw[draw= green!50!black, fill = white] (A) circle (1mm) node[below, green!50!black] {$\frac{N-1}{2}$};
\filldraw[draw= green!50!black, fill = green] (B) circle (1mm) node[right, green!50!black] {$\frac{1}{2}$};
\filldraw[draw= green!50!black, fill = white] (C) circle (1mm) node[right, green!50!black] {$-\frac{N-1}{2}$};
\filldraw[draw= green!50!black, fill = green] (D) circle (1mm) node[below, green!50!black] {$-\frac{1}{2}$};
\end{scope}

}} .
\end{equation}
Using this isomorphism and the definitions of the crossings \eqref{eqn:def-T-tw} and \eqref{eqn:T-prime-tw}, we can expand $\NB{\tikz[scale =0.4]{\begin{scope}[xshift = 2.5cm]
  \draw [<-] (0, -0.5) .. controls +(.5,0) and +(-.5, 0)
  .. (1, 0.5); 
    \fill[white] (0.5, 0) circle (1mm);

    \draw(0, 0.5) .. controls +(.5,0) and +(-.5, 0)
  .. (1, -0.5);
 %%%%%%
  \draw[] (1, 0.5) .. controls +(.5,0) and +(-.5, 0)
  .. (2, -0.5);
    \fill[white] (1.5, 0) circle (1mm);  
      \draw[<-] (2, 0.5)  .. controls +(-.5,0) and +(.5, 0)
  .. (1, -0.5); 
\end{scope}

%\begin{scope}[xshift = 2.5cm]
%  \draw [<-] (0, -0.5) .. controls +(0,0.5) and +(0, -0.5)
 % .. (1, 0.5); 
 %   \fill[white] (0.5, 0) circle (1mm);

 %   \draw [>-](0, 0.5) .. controls +(0,-0.5) and +(0, 0.5)
 % .. (1, -0.5);
 %%%%%%
 % \draw (1, 0.5) .. controls +(0,-0.5) and +(0, 0.5)
 % .. (2, -0.5);

 %   \fill[white] (1.5, 0) circle (1mm);

 %     \draw (2, 0.5)  .. controls +(0,-0.5) and +(0, 0.5)
%  .. (1, -0.5); 
%\end{scope}
}}$, up to green-dotted shifting, as the total complex
\begin{align}
  \NB{\tikz[xscale = 3, yscale = 3]{
       \node (i0) at (.5, 0) { $q^{-1}$\NB{\tikz[font= \tiny,
  scale=0.6]{\begin{scope}
  \coordinate (bl) at (-0.5, -1);
  \coordinate (br) at ( 0.5, -1);
    \coordinate (brr) at ( 2.5, -1);

  \coordinate (bm) at (  0,-0.3);
  \coordinate (tl) at (-0.5,  1);
  \coordinate (tr) at ( 0.5,  1);
    \coordinate (trr) at ( 2.5,  1);

  \coordinate (tm) at (  0, 0.3);
  \draw[<-]  (bl) .. controls +( 0, 0.5) and +(0,0) .. (bm) coordinate[pos = 0.25] (gb);%  node[right, pos = 0] {$1$}
  \draw[<-]  (br) .. controls +( 0, 0.5) and +(0,0) .. (bm)
  coordinate[pos = 0.25] (ga); %  node[left, pos = 0] {$1$}
  \draw[>-]  (tl) .. controls +( 0, -0.5) and +(0,0) .. (tm);
%  node[right, pos = 0] {$1$};
    \filldraw[draw= green!50!black, fill = white] (ga) circle (1mm)
  node[right, green!50!black] {$-\bar{t}_1$};

  \draw[>-]  (tr) .. controls +( 0, -0.5) and +(0,0) .. (tm);%  node[left, pos = 0] {$1$};
    \filldraw[draw= green!50!black, fill = white] (gb) circle (1mm)
  node[left, green!50!black] {$-\bar{t}_2$};
  \draw [double] (bm) -- (tm);% node[left, pos = 0.5] {$2$};
  \draw (.5, 1) arc (180:0:0.5) -- (1.5,-1) arc (180:0:-0.5) coordinate[pos = 0.25] (X) coordinate[pos = 0.75] (Y);
     \draw[->] (brr) -- (trr) coordinate[pos = 0.25] (A) coordinate[pos = 0.75] (B);% node[pos = 0, left] {$1$} node[pos = 1, left] {$1$} 
      \filldraw[draw= green!50!black, fill = green] (X) circle (1mm)
  node[below, green!50!black] {$\frac{1}{2}$};
    \filldraw[draw= green!50!black, fill = white] (Y) circle (1mm)
  node[below, green!50!black] {$\frac{N-1}{2}$};
\filldraw[draw= green!50!black, fill = white] (A) circle (1mm)
  node[right, green!50!black] {$-\frac{N-1}{2}$}; 
  \filldraw[draw= green!50!black, fill = green] (B) circle (1mm)
  node[right, green!50!black] {$\frac{-1}{2}$};
\end{scope}}} };
    \node (i1) at (-1, .5) { \NB{\tikz[font= \tiny,
  scale=0.6]{\begin{scope}
  \coordinate (bl) at (0, -1);
  \coordinate (br) at (2, -1);
  \coordinate (tl) at (0,  1);
  \coordinate (tr) at (2,  1);
    \coordinate (ml) at (-0.5,  -.8);
        \coordinate (Ml) at (-0.5,  .8);
 \coordinate (mr) at (0.5,  -.6);
\coordinate (Mr) at (0.5,  .6);

 %\draw[>->] (bl) -- (tl) node[pos = 0, below] {$2$} node[pos = 1,
  %above] {$1$} coordinate[pos = 0.4] (ml);
   \draw[<-<] (bl) -- (tl);% node[pos = 0, below] {$1$} node[pos = 1,  above] {$1$};
     \draw[>->] (br) -- (tr) coordinate[pos = 0.25] (A) coordinate[pos = 0.75] (B);% node[pos = 0, below] {$1$} node[pos = 1,  above] {$1$} 
    \draw [-<] (.5, 0) arc (180:0:0.5) ;
    \draw (.5, 0) arc (-180:0:0.5) coordinate[pos = 0.2] (X) coordinate[pos = 0.8] (Y);
    \filldraw[draw= green!50!black, fill = green] (X) circle (1mm)
  node[below, green!50!black] {$\frac{1}{2}$};
    \filldraw[draw= green!50!black, fill = white] (Y) circle (1mm)
  node[below, green!50!black] {$\frac{N-1}{2}$};
\filldraw[draw= green!50!black, fill = white] (A) circle (1mm)
  node[right, green!50!black] {$-\frac{N-1}{2}$}; 
  \filldraw[draw= green!50!black, fill = green] (B) circle (1mm)
  node[right, green!50!black] {$\frac{-1}{2}$};
\end{scope}}} };
    \node (i2) at (-1, -.5) {  \NB{\tikz[font= \tiny,
  scale=0.6]{\begin{scope}
  \coordinate (bl) at (-0.5, -1.5);
  \coordinate (br) at ( 0.5, -1);
  \coordinate (tr) at ( 0.5,  1);
  \coordinate (tl) at (-0.5,  1.5);
  % \coordinate (tr) at ( 0.5,  1);
  % \coordinate (bl) at (-0.5, -1);
  % \coordinate (br) at ( 0.5, -1);

  % \coordinate (tl) at (-0.5,  1);
  \coordinate (bm) at (  0,-0.3);
  \coordinate (tm) at (  0, 0.3);
  %%%
    \coordinate (bl2) at (1.5, -1);
    \coordinate (br2) at ( 2.5, -1.5);
%    \coordinate (br2) at ( 2.5, -1);
  \coordinate (bm2) at (  2,-0.3);
  \coordinate (tl2) at (1.5,  1);
  \coordinate (tr2) at ( 2.5,  1.5);
%  \coordinate (tr2) at ( 2.5,  1);
  \coordinate (tm2) at (  2, 0.3);
  %%%%%
  \draw[<-]  (bl) .. controls +( 0, 0.5) and +(0,0) .. (bm) coordinate[pos = 0.25] (ha) ;%   node[right, pos = 0] {$1$} 
    \filldraw[draw= green!50!black, fill = white] (ha) circle (1mm)
  node[left, green!50!black] {$-\bar{t}_2$};
  \draw[<-]  (br) .. controls +( 0, 0.5) and +(0,0) .. (bm)
  coordinate[pos = 0.25] (hb) ; %  node[left, pos = 0] {$1$} 
    \filldraw[draw= green!50!black, fill = white] (hb) circle (1mm)
  node[right, green!50!black] {$-\bar{t}_1$};
  \draw[>-]  (tl) .. controls +( 0, -0.5) and +(0,0) .. (tm)
  coordinate[pos = 0.25] (ga) ; %  node[right, pos = 0] {$1$} 
  %  \filldraw[draw= green!50!black, fill = white] (ga) circle (1mm)
  %node[left, green!50!black] {$t_1$};
   \draw[>-]  (tr) .. controls +( 0, -0.5) and +(0,0) .. (tm)
   coordinate[pos = 0.25] (gb) ; %node[left, pos = 0] {$1$} 
 %   \filldraw[draw= green!50!black, fill = white] (gb) circle (1mm)
%  node[left, green!50!black] {$t_2$};
  \draw [double] (bm) -- (tm);% node[left, pos = 0.5] {$2$};
%     \draw (.5, 1) arc (180:0:0.5) -- (1.5,-1) arc (180:0:-0.5);
\draw (.5, 1) arc (180:0:0.5);
\draw (1.5,-1) arc (180:0:-0.5) coordinate[pos = 0.25] (X) coordinate[pos = 0.75] (Y);
     %%%%%
 \draw[>-]  (bl2) .. controls +( 0, 0.5) and +(0,0) .. (bm2);
  %node[right, pos = 0] {$1$};
  \draw[>-]  (br2) .. controls +( 0, 0.5) and +(0,0) .. (bm2);
  %node[left, pos = 0] {$1$};
  \draw[<-]  (tl2) .. controls +( 0, -0.5) and +(0,0) .. (tm2) coordinate[pos = 0.25] (ga2) ; %  node[right, pos = 0] {$1$} 
      \filldraw[draw= green!50!black, fill = white] (ga2) circle (1mm)
  node[left, green!50!black] {${t}_1$};
    \draw [double] (bm2) -- (tm2);% node[left, pos = 0.5] {$2$};
  \draw[<-]  (tr2) .. controls +( 0, -0.5) and +(0,0) .. (tm2)
  coordinate[pos = 0.15] (gb2) coordinate[pos = 0.6] (gb3); %  node[left, pos = 0] {$1$} 
   \filldraw[draw= green!50!black, fill = white] (gb2) circle (1mm)
  node[right, green!50!black] {${t}_2-\frac{N-1}{2}$};
    \filldraw[draw= green!50!black, fill = green] (gb3) circle (1mm)
  node[right, green!50!black] {$\frac{-1}{2}$};
    \filldraw[draw= green!50!black, fill = green] (X) circle (1mm)
  node[below, green!50!black] {$\frac{1}{2}$};
    \filldraw[draw= green!50!black, fill = white] (Y) circle (1mm)
  node[below, green!50!black] {$\frac{N-1}{2}$};
\end{scope}

}}};
      \node (i3) at (-2.5, 0) {$q$ \ \NB{\tikz[font= \tiny,
  scale=0.6]{\begin{scope}
  \coordinate (bl) at (-0.5, -1);
    \coordinate (bll) at (-2.5, -1);
 \coordinate (br) at ( 0.5, -1);
  \coordinate (bm) at (  0,-0.3);
  \coordinate (tl) at (-.5,  1);
    \coordinate (tll) at (-2.5,  1);
  \coordinate (tr) at ( 0.5,  1);
  \coordinate (tm) at (  0, 0.3);
  \draw[>-]  (bl) .. controls +( 0, 0.5) and +(0,0) .. (bm);
  %node[right, pos = 0] {$1$};
  \draw[>-]  (br) .. controls +( 0, 0.5) and +(0,0) .. (bm);
  %node[left, pos = 0] {$1$};
  \draw[<-]  (tl) .. controls +( 0, -0.5) and +(0,0) .. (tm)
  coordinate[pos = 0.25] (ga) ; %  node[right, pos = 0] {$1$} 
    \filldraw[draw= green!50!black, fill = white] (ga) circle (1mm)
  node[left, green!50!black] {$t_1$};

  \draw[<-]  (tr) .. controls +( 0, -0.5) and +(0,0) .. (tm)
  coordinate[pos = 0.15] (gb) coordinate[pos = 0.65] (gb3); %  node[left, pos = 0] {$1$} 
  %   \filldraw[draw= green!50!black, fill = white] (gb) circle (1mm)
  % node[left, green!50!black] {$t_2$};
  \draw [double] (bm) -- (tm);% node[left, pos = 0.5] {$2$};
   \draw (-.5, 1) arc (-180:0:-0.5) -- (-1.5,-1) arc (-180:0:0.5) coordinate[pos = 0.25] (X) coordinate[pos = 0.75] (Y);
      \draw[<-<] (bll) -- (tll);% node[pos = 0, left] {$1$} node[pos = 1, left] {$1$};
       \filldraw[draw= green!50!black, fill = white] (gb) circle (1mm)
  node[right, green!50!black] {${t}_2-\frac{N-1}{2}$};
   \filldraw[draw= green!50!black, fill = green] (gb3) circle (1mm)
  node[right, green!50!black] {$\frac{-1}{2}$};
    \filldraw[draw= green!50!black, fill = green] (X) circle (1mm)
  node[below, green!50!black] {$\frac{1}{2}$};
    \filldraw[draw= green!50!black, fill = white] (Y) circle (1mm)
  node[below, green!50!black] {$\frac{N-1}{2}$};
\end{scope}}} };
  \draw[->] (i3) -- (i1) node[pos=0.5, above] {$\mapH$};
    \draw[->] (i3) -- (i2) node[pos=0.5, above] {$-\mapX$};
      \draw[->] (i1) -- (i0) node[pos=0.5, above] {$\mapX$};
  \draw[->] (i2) -- (i0) node[pos=0.5, above] {$\mapH$};
  }} \nonumber    .
\end{align}

There is a morphism of complexes
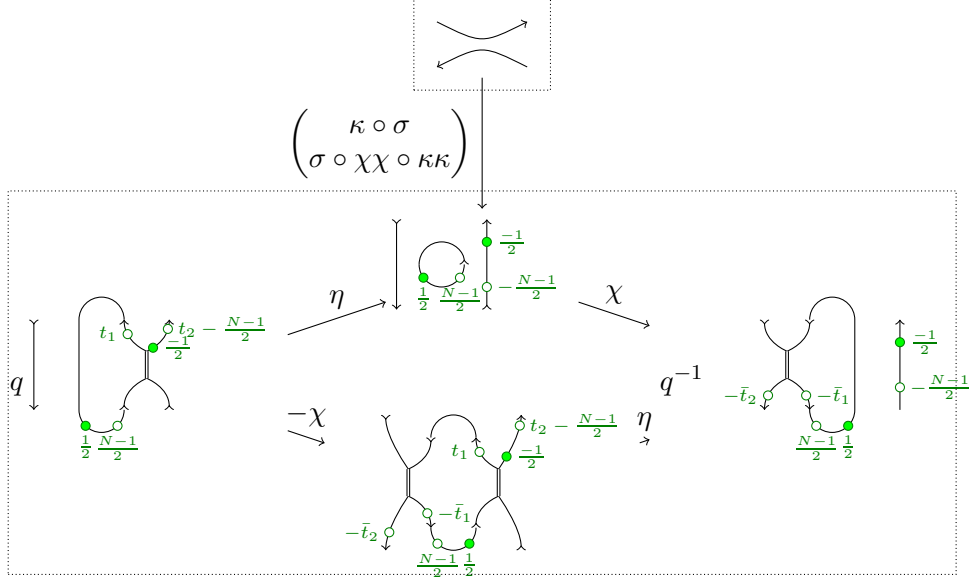
\begin{equation}
  \NB{\tikz[xscale = 3, yscale = 3]{
       \node (i0) at (.5, 0) {$q^{-1}$ \NB{\tikz[font= \tiny,
  scale=0.6]{}} };
    \node (i1) at (-1, .5) { \NB{\tikz[font= \tiny,
  scale=0.6]{}} };
    \node (i2) at (-1, -.5) {  \NB{\tikz[font= \tiny,
  scale=0.6]{}}};
      \node (i3) at (-2.5, 0) {$q$ \NB{\tikz[font= \tiny,
  scale=0.6]{}} };
    \node (i4) at (-1, 1.5) { \NB{\tikz[font= \tiny,
  scale=0.6]{\input{\imagesfolder/pdg_horres}}} };
      \node (i5) at (-3, -.75) { };
    \node (i6) at (1, .75) { };
 \draw[->] (i3) -- (i1) node[pos=0.5, above] {$\mapH$};
    \draw[->] (i3) -- (i2) node[pos=0.5, above] {$-\mapX$};
      \draw[->] (i1) -- (i0) node[pos=0.5, above] {$\mapX$};
  \draw[->] (i2) -- (i0) node[pos=0.5, above] {$\mapH$};
   \draw[->] (i4) -- (i1) node[pos=0.5, left] {$\begin{pmatrix} \mapC \circ \mapS \\ 
   \mapS \circ \mapX \mapX \circ \mapC \mapC
\end{pmatrix}$};
       \node[draw, densely dotted ,fit=(i4) (i4) ,inner sep=1ex,rectangle] {};
           \node[draw, densely dotted ,fit=(i5) (i6) ,inner sep=1ex,rectangle] {};
  }}    .
\end{equation}
where the first component of the vertical map $\mapC \circ \mapS$ is a composition of a saddle and a cup creation, and 
the second component $ \mapS \circ \mapX \mapX \circ \mapC \mapC $ is a composition of two cup creations, followed by two merges, followed by a saddle on a thickness two facet.
Note that the twists coming from $\mapC$ and $\mapS$ in the first component cancel each other, while the twists for the second component require a more detailed calculation.
Since this map is a homotopy equivalence when forgetting about the $\mathfrak{sl}_2$-action (see \cite[Proposition 9]{Khsl3} or \cite[Theorem 7.1]{MSV}),
in the relative homotopy category we have that the double crossing is isomorphic to 
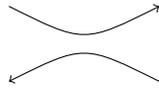
\[
\NB{\tikz[]{\input{\imagesfolder/pdg_horres}}} .
\]
This finishes the proof of the proposition.
\end{proof}

\subsection{Reidemeister III}
We will follow the proof of Reidemeister III given by Khovanov and Rozansky \cite{KRWitt} and its reformulation in \cite[Proposition 3.9]{QiSussanLink} .

\begin{prop}
  For $i=1,\ldots,n-1$, there are isomorphisms
  \begin{equation}\label{eq:R3-prop}
    \NB
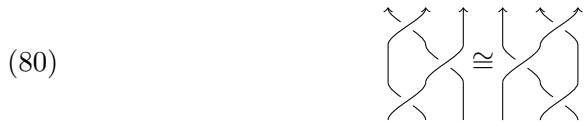
{\tikz[scale=0.5]{\begin{scope}[font=\tiny]
  \draw (0.5, -0.5) ..controls +(0,0.3) and +(0,-0.3) .. (-0.5,0.5);
  \fill[white] (0,0) circle (2mm);
  \draw (-0.5, -0.5) ..controls +(0,0.3) and +(0,-0.3) .. (0.5,0.5);
  \begin{scope}[xshift= 1cm, yshift= 1cm]
  \draw (0.5, -0.5) ..controls +(0,0.3) and +(0,-0.3) .. (-0.5,0.5);
  \fill[white] (0,0) circle (2mm);
  \draw (-0.5, -0.5) ..controls +(0,0.3) and +(0,-0.3) .. (0.5,0.5);
  \end{scope}
  \begin{scope}[xshift=0cm, yshift= 2cm]
  \draw[->] (0.5, -0.5) ..controls +(0,0.3) and +(0,-0.3) .. (-0.5,0.5);
  \fill[white] (0,0) circle (2mm);
  \draw[->] (-0.5, -0.5) ..controls +(0,0.3) and +(0,-0.3) .. (0.5,0.5);
\end{scope}
\draw (-0.5,0.5) -- +(0,1);
\draw (1.5,-0.5) -- +(0,1);
\draw[->] (1.5,1.5) -- +(0,1);
\end{scope}}}\cong
\NB{\tikz[xscale = -.5, yscale=.5]{\begin{scope}[yscale= -1, font=\tiny]
  \draw[<-] (0.5, -0.5) ..controls +(0,0.3) and +(0,-0.3) .. (-0.5,0.5);
  \fill[white] (0,0) circle (2mm);
  \draw[<-] (-0.5, -0.5) ..controls +(0,0.3) and +(0,-0.3) .. (0.5,0.5);
  \begin{scope}[xshift= 1cm, yshift= 1cm]
  \draw (0.5, -0.5) ..controls +(0,0.3) and +(0,-0.3) .. (-0.5,0.5);
  \fill[white] (0,0) circle (2mm);
  \draw (-0.5, -0.5) ..controls +(0,0.3) and +(0,-0.3) .. (0.5,0.5);
  \end{scope}
  \begin{scope}[xshift=0cm, yshift= 2cm]
  \draw (0.5, -0.5) ..controls +(0,0.3) and +(0,-0.3) .. (-0.5,0.5);
  \fill[white] (0,0) circle (2mm);
  \draw (-0.5, -0.5) ..controls +(0,0.3) and +(0,-0.3) .. (0.5,0.5);
\end{scope}
\draw (-0.5,0.5) -- +(0,1);
\draw[<-] (1.5,-0.5) -- +(0,1);
\draw (1.5,1.5) -- +(0,1);
\end{scope}}} 
  \end{equation}
  in the relative homotopy category.
\end{prop}

\begin{proof}
Throughout the proof we omit writing down explicitly the differentials
for space purposes.  We also drop $q$-shifts, since they could easily be reincorporated. One could follow the analogous proof in \cite[Proposition 3.9]{QiSussanLink} for such details.
By definition, the complex for $\NB{\tikz[scale=0.3]{\begin{scope}[font=\tiny]
  \draw (0.5, -0.5) ..controls +(0,0.3) and +(0,-0.3) .. (-0.5,0.5);
  \fill[white] (0,0) circle (2mm);
  \draw (-0.5, -0.5) ..controls +(0,0.3) and +(0,-0.3) .. (0.5,0.5);
  \begin{scope}[xshift= 1cm, yshift= 1cm]
  \draw (0.5, -0.5) ..controls +(0,0.3) and +(0,-0.3) .. (-0.5,0.5);
  \fill[white] (0,0) circle (2mm);
  \draw (-0.5, -0.5) ..controls +(0,0.3) and +(0,-0.3) .. (0.5,0.5);
  \end{scope}
  \begin{scope}[xshift=0cm, yshift= 2cm]
  \draw[->] (0.5, -0.5) ..controls +(0,0.3) and +(0,-0.3) .. (-0.5,0.5);
  \fill[white] (0,0) circle (2mm);
  \draw[->] (-0.5, -0.5) ..controls +(0,0.3) and +(0,-0.3) .. (0.5,0.5);
\end{scope}
\draw (-0.5,0.5) -- +(0,1);
\draw (1.5,-0.5) -- +(0,1);
\draw[->] (1.5,1.5) -- +(0,1);
\end{scope}}}$ on the left-hand side of \eqref{eq:R3-prop} is given as the total complex, up to green-dotted shifting:
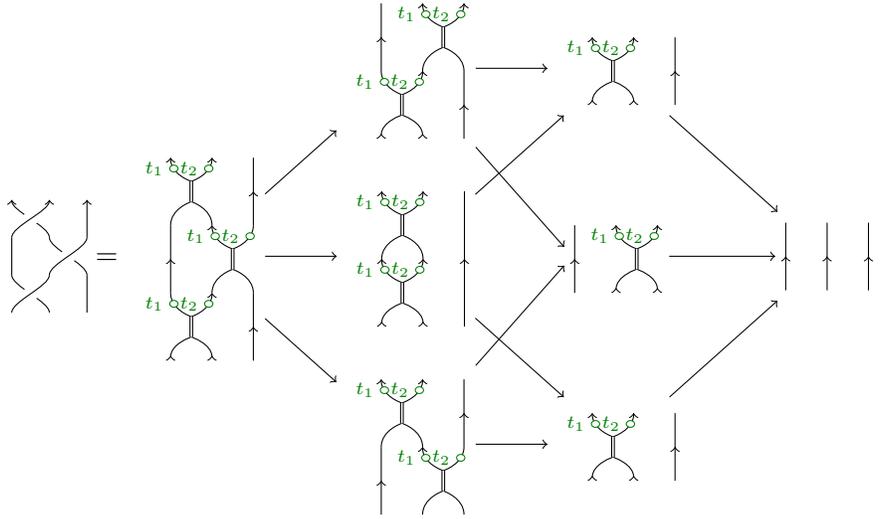
\begin{equation}
  \providecommand{\myxscale}{0.55}
  \providecommand{\myyscale}{0.45}
\NB{\tikz[xscale = 2.8, yscale = 2.5]{
       \node (i0) at (0, 0) { \NB{\tikz[font= \tiny,
  xscale=\myxscale, yscale=\myyscale]{\begin{scope}
  \coordinate (bl) at (-0.5, -1);
  \coordinate (br) at ( 0.5, -1);
  \coordinate (tl) at (-0.5,  1);
  \coordinate (tr) at ( 0.5,  1);
    \coordinate (ml) at (-0.5,  -.8);
        \coordinate (Ml) at (-0.5,  .8);
 \coordinate (mr) at (0.5,  -.6);
\coordinate (Mr) at (0.5,  .6);
  \coordinate (brr) at ( 1.5, -1);
    \coordinate (trr) at ( 1.5, 1);

 %\draw[>->] (bl) -- (tl) node[pos = 0, below] {$2$} node[pos = 1,
  %above] {$1$} coordinate[pos = 0.4] (ml);
   \draw[->-] (bl) -- (tl);% node[pos = 0, below] {$1$} node[pos = 1,  above] {$1$};
  %\draw[>->] (br) -- (tr) node[pos = 0, below] {$1$} node[pos = 1, above] {$2$} coordinate[pos = 0.6] (mr);
  
    \draw[->-] (br) -- (tr); %node[pos = 0, below] {$1$} node[pos = 1, above] {$1$};
       \draw[->-] (brr) -- (trr);% node[pos = 0, below] {$1$} node[pos = 1, above] {$1$};
  
  %\draw[->-] (ml) -- (mr) node [pos= 0.5, above] {$1$};
  %  \draw[->-] (Ml) -- (Mr) node [pos= 0.5, above] {$1$};

\end{scope}}} };
      \node (i1) at (-1, 0) { \NB{\tikz[font= \tiny,
  xscale=\myxscale, yscale=\myyscale]{\begin{scope}
  \coordinate (bl) at (-0.5, -1);
  \coordinate (br) at ( 0.5, -1);
  \coordinate (bm) at (  0,-0.3);
  \coordinate (tl) at (-0.5,  1);
  \coordinate (tr) at ( 0.5,  1);
  \coordinate (tm) at (  0, 0.3);
    \coordinate (bll) at ( -1.5, -1);
        \coordinate (tll) at ( -1.5, 1);
\draw[>-]  (bl) .. controls +( 0, 0.5) and +(0,0) .. (bm)
 ;% node[below, pos = 0] {$1$}
  \draw[>-]  (br) .. controls +( 0, 0.5) and +(0,0) .. (bm)
 ;% node[below, pos = 0] {$1$}
  \draw[<-]  (tl) .. controls +( 0, -0.5) and +(0,0) .. (tm)
  coordinate[pos = 0.25] (ga) ;% node[above, pos = 0] {$1$}
    \filldraw[draw= green!50!black, fill = white] (ga) circle (1mm)
  node[left, green!50!black] {$t_1$};

  \draw[<-]  (tr) .. controls +( 0, -0.5) and +(0,0) .. (tm)
  coordinate[pos = 0.25] (gb) ;% node[above, pos = 0] {$1$}
    \filldraw[draw= green!50!black, fill = white] (gb) circle (1mm)
  node[left, green!50!black] {$t_2$};
  \draw [double] (bm) -- (tm) ;%node[left, pos = 0.5] {$2$};
 \draw [->-] (bll) -- (tll);% node[above, pos = 1] {$1$}
\end{scope}}} };
   \node (i2) at (-1, 1) { \NB{\tikz[font= \tiny,
  xscale=\myxscale, yscale=\myyscale]{\begin{scope}
  \coordinate (bl) at (-0.5, -1);
  \coordinate (br) at ( 0.5, -1);
  \coordinate (bm) at (  0,-0.3);
  \coordinate (tl) at (-0.5,  1);
  \coordinate (tr) at ( 0.5,  1);
  \coordinate (tm) at (  0, 0.3);
    \coordinate (brr) at ( 1.5, -1);
        \coordinate (trr) at ( 1.5, 1);
\draw[>-]  (bl) .. controls +( 0, 0.5) and +(0,0) .. (bm)
 ;% node[below, pos = 0] {$1$}
  \draw[>-]  (br) .. controls +( 0, 0.5) and +(0,0) .. (bm)
 ;% node[below, pos = 0] {$1$}
  \draw[<-]  (tl) .. controls +( 0, -0.5) and +(0,0) .. (tm)
  coordinate[pos = 0.25] (ga) ;% node[above, pos = 0] {$1$}
    \filldraw[draw= green!50!black, fill = white] (ga) circle (1mm)
  node[left, green!50!black] {$t_1$};

  \draw[<-]  (tr) .. controls +( 0, -0.5) and +(0,0) .. (tm)
  coordinate[pos = 0.25] (gb) ;% node[above, pos = 0] {$1$}
    \filldraw[draw= green!50!black, fill = white] (gb) circle (1mm)
  node[left, green!50!black] {$t_2$};
  \draw [double] (bm) -- (tm) ;%node[left, pos = 0.5] {$2$};
 \draw [->-] (brr) -- (trr);% node[above, pos = 1] {$1$}

\end{scope}}} };
     \node (i3) at (-1, -1) { \NB{\tikz[font= \tiny,
  xscale=\myxscale, yscale=\myyscale]{}} };
      \node (i4) at (-2, 1) { \NB{\tikz[font= \tiny,
  xscale=\myxscale, yscale=\myyscale]{\begin{scope}
  \coordinate (bl) at (-0.5, -1);
  \coordinate (br) at ( 0.5, -1);
  \coordinate (bm) at (  0,-0.3);
  \coordinate (tl) at (-0.5,  1);
  \coordinate (tr) at ( 0.5,  1);
  \coordinate (tm) at (  0, 0.3);
    \coordinate (brr) at ( 1.5, -1);
        \coordinate (trr) at ( 1.5, 1);
\draw[>-]  (bl) .. controls +( 0, 0.5) and +(0,0) .. (bm)
 ;% node[below, pos = 0] {$1$}
  \draw[>-]  (br) .. controls +( 0, 0.5) and +(0,0) .. (bm)
 ;% node[below, pos = 0] {$1$}
  \draw[]  (tl) .. controls +( 0, -0.5) and +(0,0) .. (tm)
  node[left, pos = 0] {} coordinate[pos = 0.25] (ga) ;
    \filldraw[draw= green!50!black, fill = white] (ga) circle (1mm)
  node[left, green!50!black] {$t_1$};

  \draw[<-]  (tr) .. controls +( 0, -0.5) and +(0,0) .. (tm)
  node[above, pos = 0] {} coordinate[pos = 0.25] (gb) ;
    \filldraw[draw= green!50!black, fill = white] (gb) circle (1mm)
  node[left, green!50!black] {$t_2$};
  \draw [double] (bm) -- (tm) ;%node[left, pos = 0.5] {$2$};
 \draw [->-] (brr) -- (trr);% node[below, pos = 0] {$1$}
  %%%%%
  \coordinate (blU) at (0.5, 1);
  \coordinate (brU) at ( 1.5, 1);
  \coordinate (bmU) at (  1, 1.7);
  \coordinate (tlU) at (0.5,  3);
  \coordinate (trU) at ( 1.5,  3);
  \coordinate (tmU) at (  1, 2.3);
    \coordinate (bllU) at ( -.5, 1);
        \coordinate (tllU) at ( -.5, 3);
\draw[]  (blU) .. controls +( 0, 0.5) and +(0,0) .. (bmU)
 ;% node[right, pos = 0] {$1$}
  \draw[]  (brU) .. controls +( 0, 0.5) and +(0,0) .. (bmU)
  node[left, pos = 0] {};
  \draw[<-]  (tlU) .. controls +( 0, -0.5) and +(0,0) .. (tmU)
  coordinate[pos = 0.25] (gaU) ;% node[above, pos = 0] {$1$}
    \filldraw[draw= green!50!black, fill = white] (gaU) circle (1mm)
  node[left, green!50!black] {$t_1$};

  \draw[<-]  (trU) .. controls +( 0, -0.5) and +(0,0) .. (tmU)
  coordinate[pos = 0.25] (gbU) ;% node[above, pos = 0] {$1$}
    \filldraw[draw= green!50!black, fill = white] (gbU) circle (1mm)
  node[left, green!50!black] {$t_2$};
  \draw [double] (bmU) -- (tmU) ;%node[left, pos = 0.5] {$2$};
 \draw [->-] (bllU) -- (tllU) ;% node[above, pos = 1] {$1$}

\end{scope}}} };
      \node (i5) at (-2, 0) { \NB{\tikz[font= \tiny,
  xscale=\myxscale, yscale=\myyscale]{\begin{scope}
  \coordinate (bl) at (-0.5, -1);
  \coordinate (br) at ( 0.5, -1);
  \coordinate (bm) at (  0,-0.3);
  \coordinate (tl) at (-0.5,  1);
  \coordinate (tr) at ( 0.5,  1);
  \coordinate (tm) at (  0, 0.3);
    \coordinate (brr) at ( 1.5, -1);
        \coordinate (trr) at ( 1.5, 1);
\draw[>-]  (bl) .. controls +( 0, 0.5) and +(0,0) .. (bm)
 ;% node[below, pos = 0] {$1$}
  \draw[>-]  (br) .. controls +( 0, 0.5) and +(0,0) .. (bm)
 ;% node[below, pos = 0] {$1$}
  \draw[<-]  (tl) .. controls +( 0, -0.5) and +(0,0) .. (tm)
  node[above, pos = 0] {} coordinate[pos = 0.25] (ga) ;
    \filldraw[draw= green!50!black, fill = white] (ga) circle (1mm)
  node[left, green!50!black] {$t_1$};

  \draw[<-]  (tr) .. controls +( 0, -0.5) and +(0,0) .. (tm)
  node[above, pos = 0] {} coordinate[pos = 0.25] (gb) ;
    \filldraw[draw= green!50!black, fill = white] (gb) circle (1mm)
  node[left, green!50!black] {$t_2$};
  \draw [double] (bm) -- (tm) ;%node[left, pos = 0.5] {$2$};
% \draw [->-] (brr) -- (trr);
 %%%%%
   \coordinate (blU) at (-0.5, 1);
  \coordinate (brU) at ( 0.5, 1);
  \coordinate (bmU) at (  0,1.7);
  \coordinate (tlU) at (-0.5,  3);
  \coordinate (trU) at ( 0.5,  3);
  \coordinate (tmU) at (  0, 2.3);
    \coordinate (brrU) at ( 1.5, 1);
        \coordinate (trrU) at ( 1.5, 3);
\draw[]  (blU) .. controls +( 0, 0.5) and +(0,0) .. (bmU)
 ;% node[left, pos = 0] {$1$}
  \draw[]  (brU) .. controls +( 0, 0.5) and +(0,0) .. (bmU)
 ;% node[right, pos = 0] {$1$}
  \draw[<-]  (tlU) .. controls +( 0, -0.5) and +(0,0) .. (tmU)
  coordinate[pos = 0.25] (gaU) ;% node[above, pos = 0] {$1$}
    \filldraw[draw= green!50!black, fill = white] (gaU) circle (1mm)
  node[left, green!50!black] {$t_1$};

  \draw[<-]  (trU) .. controls +( 0, -0.5) and +(0,0) .. (tmU)
  coordinate[pos = 0.25] (gbU) ;% node[above, pos = 0] {$1$}
    \filldraw[draw= green!50!black, fill = white] (gbU) circle (1mm)
  node[left, green!50!black] {$t_2$};
  \draw [double] (bmU) -- (tmU) ;%node[left, pos = 0.5] {$2$};
 \draw [->-] (brr) -- (trrU) ;% node[below, pos = 0] {$1$}
\end{scope}}} };
         \node (i6) at (-2, -1) { \NB{\tikz[font= \tiny,
  xscale=\myxscale, yscale=\myyscale]{\begin{scope}
 
  %%%%%
  \coordinate (blU) at (0.5, 1);
  \coordinate (brU) at ( 1.5, 1);
  \coordinate (bmU) at (  1, 1.7);
  \coordinate (tlU) at (0.5,  3);
  \coordinate (trU) at ( 1.5,  3);
  \coordinate (tmU) at (  1, 2.3);
    \coordinate (bllU) at ( -.5, 1);
        \coordinate (tllU) at ( -.5, 3);
\draw[]  (blU) .. controls +( 0, 0.5) and +(0,0) .. (bmU)
 ;% node[below, pos = 0] {$1$}
  \draw[]  (brU) .. controls +( 0, 0.5) and +(0,0) .. (bmU)
 ;% node[below, pos = 0] {$1$}
  \draw[<-]  (tlU) .. controls +( 0, -0.5) and +(0,0) .. (tmU)
  coordinate[pos = 0.25] (gaU) ;% node[right, pos = 0] {$1$}
    \filldraw[draw= green!50!black, fill = white] (gaU) circle (1mm)
  node[left, green!50!black] {$t_1$};

  \draw[]  (trU) .. controls +( 0, -0.5) and +(0,0) .. (tmU)
  node[above, pos = 1] {} coordinate[pos = 0.25] (gbU) ;
    \filldraw[draw= green!50!black, fill = white] (gbU) circle (1mm)
  node[left, green!50!black] {$t_2$};
  \draw [double] (bmU) -- (tmU) ;%node[left, pos = 0.5] {$2$};
 \draw [->-] (bllU) -- (tllU) ;% node[below, pos = 0] {$1$}
 
 %%%%%
 
  \coordinate (blUU) at (-0.5, 3);
  \coordinate (brUU) at ( 0.5, 3);
  \coordinate (bmUU) at (  0,3.7);
  \coordinate (tlUU) at (-0.5,  5);
  \coordinate (trUU) at ( 0.5,  5);
  \coordinate (tmUU) at (  0, 4.3);
    \coordinate (brrUU) at ( 1.5, 3);
        \coordinate (trrUU) at ( 1.5, 5);
\draw[]  (blUU) .. controls +( 0, 0.5) and +(0,0) .. (bmUU)
  node[below, pos = 0] {};
  \draw[]  (brUU) .. controls +( 0, 0.5) and +(0,0) .. (bmUU)
  node[below, pos = 0] {};
  \draw[<-]  (tlUU) .. controls +( 0, -0.5) and +(0,0) .. (tmUU)
  coordinate[pos = 0.25] (gaUU) ;% node[above, pos = 0] {$1$}
    \filldraw[draw= green!50!black, fill = white] (gaUU) circle (1mm)
  node[left, green!50!black] {$t_1$};

  \draw[<-]  (trUU) .. controls +( 0, -0.5) and +(0,0) .. (tmUU) 
  coordinate[pos = 0.25] (gbUU) ;% node[above, pos = 0] {$1$}
    \filldraw[draw= green!50!black, fill = white] (gbUU) circle (1mm)
  node[left, green!50!black] {$t_2$};
  \draw [double] (bmUU) -- (tmUU) ;%node[left, pos = 0.5] {$2$};
 \draw [->-] (brrUU) -- (trrUU);% node[above, pos = 1] {$1$}

\end{scope}}} };
  \node (i7) at (-3, 0) { \NB{\tikz[font= \tiny,
  xscale=\myxscale, yscale=\myyscale]{\begin{scope}
  \coordinate (bl) at (-0.5, -1);
  \coordinate (br) at ( 0.5, -1);
  \coordinate (bm) at (  0,-0.3);
  \coordinate (tl) at (-0.5,  1);
  \coordinate (tr) at ( 0.5,  1);
  \coordinate (tm) at (  0, 0.3);
    \coordinate (brr) at ( 1.5, -1);
        \coordinate (trr) at ( 1.5, 1);
\draw[>-]  (bl) .. controls +( 0, 0.5) and +(0,0) .. (bm); %  node[below, pos = 0] {$1$};
  \draw[>-]  (br) .. controls +( 0, 0.5) and +(0,0) .. (bm);%  node[below, pos = 0] {$1$};
  \draw[]  (tl) .. controls +( 0, -0.5) and +(0,0) .. (tm)
  coordinate[pos = 0.25] (ga) ; %  node[left, pos = 0] {} 
    \filldraw[draw= green!50!black, fill = white] (ga) circle (1mm)
  node[left, green!50!black] {$t_1$};

  \draw[<-]  (tr) .. controls +( 0, -0.5) and +(0,0) .. (tm)
  node[above, pos = 0] {} coordinate[pos = 0.25] (gb) ;
    \filldraw[draw= green!50!black, fill = white] (gb) circle (1mm)
  node[left, green!50!black] {$t_2$};
  \draw [double] (bm) -- (tm);% node[left, pos = 0.5] {$2$};
 \draw [->-] (brr) -- (trr);% node[below, pos = 0] {$1$};
  %%%%%
  \coordinate (blU) at (0.5, 1);
  \coordinate (brU) at ( 1.5, 1);
  \coordinate (bmU) at (  1, 1.7);
  \coordinate (tlU) at (0.5,  3);
  \coordinate (trU) at ( 1.5,  3);
  \coordinate (tmU) at (  1, 2.3);
    \coordinate (bllU) at ( -.5, 1);
        \coordinate (tllU) at ( -.5, 3);
\draw[]  (blU) .. controls +( 0, 0.5) and +(0,0) .. (bmU);
%  node[right, pos = 0] {$1$};
  \draw[]  (brU) .. controls +( 0, 0.5) and +(0,0) .. (bmU);
%  node[left, pos = 0] {};
  \draw[<-]  (tlU) .. controls +( 0, -0.5) and +(0,0) .. (tmU)coordinate[pos = 0.25] (gaU) ;
%  node[above, pos = 0] {$1$} 
    \filldraw[draw= green!50!black, fill = white] (gaU) circle (1mm)
  node[left, green!50!black] {$t_1$};

  \draw[]  (trU) .. controls +( 0, -0.5) and +(0,0) .. (tmU)
  node[above, pos = 1] {} coordinate[pos = 0.25] (gbU) ;
    \filldraw[draw= green!50!black, fill = white] (gbU) circle (1mm)
  node[left, green!50!black] {$t_2$};
  \draw [double] (bmU) -- (tmU);% node[left, pos = 0.5] {$2$};
 \draw [->-] (bllU) -- (tllU);%  node[left, pos = 1] {$1$};
 
 %%%%%
 
  \coordinate (blUU) at (-0.5, 3);
  \coordinate (brUU) at ( 0.5, 3);
  \coordinate (bmUU) at (  0,3.7);
  \coordinate (tlUU) at (-0.5,  5);
  \coordinate (trUU) at ( 0.5,  5);
  \coordinate (tmUU) at (  0, 4.3);
    \coordinate (brrUU) at ( 1.5, 3);
        \coordinate (trrUU) at ( 1.5, 5);
\draw[]  (blUU) .. controls +( 0, 0.5) and +(0,0) .. (bmUU)
  node[below, pos = 0] {};
  \draw[]  (brUU) .. controls +( 0, 0.5) and +(0,0) .. (bmUU)
  node[below, pos = 0] {};
  \draw[<-]  (tlUU) .. controls +( 0, -0.5) and +(0,0) .. (tmUU) coordinate[pos = 0.25] (gaUU) ;%  node[above, pos = 0] {$1$} 
    \filldraw[draw= green!50!black, fill = white] (gaUU) circle (1mm)
  node[left, green!50!black] {$t_1$};

  \draw[<-]  (trUU) .. controls +( 0, -0.5) and +(0,0) .. (tmUU)
  coordinate[pos = 0.25] (gbUU) ; %  node[above, pos = 0] {$1$} 
    \filldraw[draw= green!50!black, fill = white] (gbUU) circle (1mm)
  node[left, green!50!black] {$t_2$};
  \draw [double] (bmUU) -- (tmUU); %node[left, pos = 0.5] {$2$};
 \draw [->-] (brrUU) -- (trrUU); %node[above, pos = 1] {$1$};

\end{scope}}} };
\draw[->] (i1) -- (i0);
  \draw[->] (i2) -- (i0);
    \draw[->] (i3) -- (i0);
\draw[->] (i4) -- (i2);
    \draw[->] (i5) -- (i2);
    \draw[->] (i5) -- (i3);
    \draw[->] (i6) -- (i3);
    \draw[->] (i4) -- ([yshift=0.05cm]i1.west);
    \draw[->] (i6) -- ([yshift=-0.05cm]i1.west);
\draw[->] (i7) -- (i4);
    \draw[->] (i7) -- (i5);
    \draw[->] (i7) -- (i6);
    }} .
\end{equation}

There are short exact sequences of complexes \eqref{RIII1eq}, where the morphisms are given by foams.  When forgetting about the $\mathfrak{sl}_2$-action, these maps split.
\begin{equation} \label{RIII1eq}
  \providecommand{\myxscale}{0.55}
  \providecommand{\myyscale}{0.45}
\NB{\tikz[xscale = 2.5, yscale = 2.5]{
       \node (i0) at (0, 0) { \NB{\tikz[font= \tiny,
  xscale= \myxscale, yscale=\myyscale]{}} };
      \node (i1) at (-1, 0) { \NB{\tikz[font= \tiny,
  xscale= \myxscale, yscale=\myyscale]{}} };
   \node (i2) at (-1, 1) { \NB{\tikz[font= \tiny,
  xscale= \myxscale, yscale=\myyscale]{}} };
     \node (i3) at (-1, -1) { \NB{\tikz[font= \tiny,
  xscale= \myxscale, yscale=\myyscale]{}} };
      \node (i4) at (-2, 1) { \NB{\tikz[font= \tiny,
  xscale= \myxscale, yscale=\myyscale]{}} };
      \node (i5) at (-2, 0) { \NB{\tikz[font= \tiny,
  xscale= \myxscale, yscale=\myyscale]{}} };
         \node (i6) at (-2, -1) { \NB{\tikz[font= \tiny,
  xscale= \myxscale, yscale=\myyscale]{}} };
  \node (i7) at (-3, 0) { \NB{\tikz[font= \tiny,
  xscale= \myxscale, yscale=\myyscale]{}} };
    \node (i8) at (-3, -1) {};
    \node (i9) at (0, -1) {};
    \node (i10) at (-3, 1) {};
    \node (i11) at (0, 1) {};
    \node (k12) at (-3,-2.5) { \NB{\tikz[font= \tiny,
  xscale= \myxscale, yscale=\myyscale]{\begin{scope}
  \coordinate (bl) at (-0.5, -1);
  \coordinate (br) at ( 0.5, -1);
  \coordinate (bm) at (  0,-0.3);
  \coordinate (tl) at (-0.5,  1);
  \coordinate (tr) at ( 0.5,  1);
  \coordinate (tm) at (  0, 0.3);
    \coordinate (brr) at ( 1.5, -1);
        \coordinate (trr) at ( 1.5, 1);
\draw[>-]  (bl) .. controls +( 0, 0.5) and +(0,0) .. (bm);
%  node[below, pos = 0] {$1$};
  \draw[>-]  (br) .. controls +( 0, 0.5) and +(0,0) .. (bm);
%  node[below, pos = 0] {$1$};
  \draw[<-]  (tl) .. controls +( 0, -0.5) and +(0,0) .. (tm) coordinate[pos = 0.25] (ga) ;%  node[above, pos = 0] {$1$} 
    \filldraw[draw= green!50!black, fill = white] (ga) circle (1mm)
  node[left, green!50!black] {$t_1+1$};

  \draw[<-]  (tr) .. controls +( 0, -0.5) and +(0,0) .. (tm) coordinate[pos = 0.25] (gb) ;%  node[above, pos = 0] {$1$} 
    \filldraw[draw= green!50!black, fill = white] (gb) circle (1mm)
  node[right, green!50!black] {$t_2+1$};
  \draw [double] (bm) -- (tm);% node[left, pos = 0.5] {$2$};
 \draw [->-] (brr) -- (trr);

\end{scope}}} };
      \node (k13) at (-2,-2.5) { \NB{\tikz[font= \tiny,
  xscale= \myxscale, yscale=\myyscale]{}} };
    \node (k14) at (0, -2) {};
               \node (j0) at (0, 3.5) { \NB{\tikz[font= \tiny,
  xscale= \myxscale, yscale=\myyscale]{}} };
      \node (j1) at (-1, 3.5) { \NB{\tikz[font= \tiny,
  xscale= \myxscale, yscale=\myyscale]{}} };
   \node (j2) at (-1, 4.5) { \NB{\tikz[font= \tiny,
  xscale= \myxscale, yscale=\myyscale]{}} };
     \node (j3) at (-1, 2.5) { \NB{\tikz[font= \tiny,
  xscale= \myxscale, yscale=\myyscale]{}} };
      \node (j4) at (-2, 4.5) { \NB{\tikz[font= \tiny,
  xscale= \myxscale, yscale=\myyscale]{}} };
      \node (j5) at (-2, 3.5) { \NB{\tikz[font= \tiny,
  xscale= \myxscale, yscale=\myyscale]{}} };
         \node (j6) at (-2, 2.5) { \NB{\tikz[font= \tiny,
  xscale= \myxscale, yscale=\myyscale]{}} };
  \node (j7) at (-3, 3.5) { \NB{\tikz[font= \tiny,
  xscale= \myxscale, yscale=\myyscale]{\begin{scope}
  \coordinate (bl) at (-0.5, -1);
  \coordinate (br) at ( 0.5, -1);
  \coordinate (bm) at (  0,-0.3);
  \coordinate (tl) at (-0.5,  1);
  \coordinate (tr) at ( 0.5,  1);
  \coordinate (tm) at (  0, 0.3);
    \coordinate (ttm) at (  0, 1);
    \coordinate (bbm) at (  0, -1);
\draw[>-]  (bl) .. controls +( 0, 0.5) and +(0,0) .. (bm);
%  node[below, pos = 0] {$1$};
  \draw[>-]  (br) .. controls +( 0, 0.5) and +(0,0) .. (bm);
%  node[below, pos = 0] {$1$};
  \draw[<-]  (tl) .. controls +( 0, -0.5) and +(0,0) .. (tm)
  coordinate[pos = 0.25] (ga) ; %  node[left, pos = 0] {$1$} 
    \filldraw[draw= green!50!black, fill = white] (ga) circle (1mm)
  node[left, green!50!black] {$2t_1$};

  \draw[<-]  (tr) .. controls +( 0, -0.5) and +(0,0) .. (tm) coordinate[pos = 0.25] (gb) ;%  node[right, pos = 0] {$1$} 
    \filldraw[draw= green!50!black, fill = white] (gb) circle (1mm)
  node[right, green!50!black] {$2t_2$};
  \draw [thick] (bm) -- (tm) node[left, pos = 0.5] {$3$};
    \draw [->-] (bbm) -- (bm);% node[below, pos = 0] {$1$};
        \draw [->-] (tm) -- (ttm)        coordinate[pos = 1] (gc) ;% node[left, pos = .85] {$1$}  

   \filldraw[draw= green!50!black, fill = white] (gc) circle (1mm)
        node[above, green!50!black] {$t_1+t_2$};
 
\end{scope}}} };
      \node (j8) at (-3, 2.5) {};
    \node (j9) at (0, 2.5) {};
    \node (j10) at (-3, 4.5) {};
    \node (j11) at (0, 4.5) {};
    \draw[->] (i1) -- (i0);
    \draw[->] (i2) -- (i0);
    \draw[->] (i3) -- (i0);
\draw[->] (i4) -- ([yshift=0.05cm]i1.west);
    \draw[->] (i6) -- ([yshift=-0.05cm]i1.west);
\draw[->] (i4) -- (i2);
    \draw[->] (i5) -- (i2);
    \draw[->] (i5) -- (i3);
    \draw[->] (i6) -- (i3);
\draw[->] (i7) -- (i4);
    \draw[->] (i7) -- (i5);
    \draw[->] (i7) -- (i6);
    \draw[->] ([xshift=-0.1cm]k12.east) -- ([xshift=0.1cm]k13.west);
            \draw[->] (i7) -- (k12) node[pos=0.5, right] {$ \mapU \circ  \mapH $};
            \draw[->] (i6) -- (k13)  node [right, pos = 0.5] {$\begin{pmatrix} 0  & \mapU & 0 \end{pmatrix} $};
 \node[draw, densely dotted ,fit=(i8) (i9) (i10) (i11) ,inner sep=10ex,rectangle] {};
  \node[draw, densely dotted ,fit=(j8) (j9) (j10) (j11) ,inner sep=10ex,rectangle] {};
      \node[draw, densely dotted ,fit=(k12) (k13) ,inner sep=1ex,rectangle] {};
  \draw[->] (j1) -- (j0);
  \draw[->] (j2) -- (j0);
    \draw[->] (j3) -- (j0);
     \draw[->] (j4) -- (j1);
    \draw[->] (j4) -- (j2);
    \draw[->] (j5) -- (j2);
    \draw[->] (j5) -- (j3);
    \draw[->] (j6) -- (j3);
    \draw[->] (j6) -- (j1);
    \draw[->] (j7) -- (j4);
    \draw[->] (j7) -- (j5);
    \draw[->] (j7) -- (j6);
        \draw[->] (j7) -- (i7) node[pos=0.6, right] {$\mapI$};
        \draw[->] (j6) -- (i4)  node [right, pos = 0.5] {$\begin{pmatrix} \Id & &  \\ & \mapB & \\ & & \Id \end{pmatrix} $};
    \draw[->] (j3) -- (i2)  node [right, pos = 0.5] {$\begin{pmatrix} \Id & & \\ & \Id & \\ & & \Id \end{pmatrix} $};
    \draw[->] (j0) -- (i0)  node [right, pos = 0.5] {$\Id$};
  }}
\end{equation}

Since the bottom complex in \eqref{RIII1eq} is acyclic, in the
relative homotopy category, $\NB{\tikz[scale=0.3]{\begin{scope}[font=\tiny]
  \draw (0.5, -0.5) ..controls +(0,0.3) and +(0,-0.3) .. (-0.5,0.5);
  \fill[white] (0,0) circle (2mm);
  \draw (-0.5, -0.5) ..controls +(0,0.3) and +(0,-0.3) .. (0.5,0.5);
  \begin{scope}[xshift= 1cm, yshift= 1cm]
  \draw (0.5, -0.5) ..controls +(0,0.3) and +(0,-0.3) .. (-0.5,0.5);
  \fill[white] (0,0) circle (2mm);
  \draw (-0.5, -0.5) ..controls +(0,0.3) and +(0,-0.3) .. (0.5,0.5);
  \end{scope}
  \begin{scope}[xshift=0cm, yshift= 2cm]
  \draw[->] (0.5, -0.5) ..controls +(0,0.3) and +(0,-0.3) .. (-0.5,0.5);
  \fill[white] (0,0) circle (2mm);
  \draw[->] (-0.5, -0.5) ..controls +(0,0.3) and +(0,-0.3) .. (0.5,0.5);
\end{scope}
\draw (-0.5,0.5) -- +(0,1);
\draw (1.5,-0.5) -- +(0,1);
\draw[->] (1.5,1.5) -- +(0,1);
\end{scope}}}$ is
isomorphic to the top complex in \eqref{RIII1eq}.

There is an isomorphism of complexes
\begin{equation} \label{RIII3eq}
  \providecommand{\myxscale}{0.55}
  \providecommand{\myyscale}{0.45}
  \NB{\tikz[xscale = 2.5, yscale = 2.5]{
       \node (i0) at (0, 0) { \NB{\tikz[font= \tiny,
  xscale=\myxscale, yscale=\myyscale]{}} };
      \node (i1) at (-1, 0) { \NB{\tikz[font= \tiny,
  xscale=\myxscale, yscale=\myyscale]{}} };
   \node (i2) at (-1, 1) { \NB{\tikz[font= \tiny,
  xscale=\myxscale, yscale=\myyscale]{}} };
     \node (i3) at (-1, -1) { \NB{\tikz[font= \tiny,
  xscale=\myxscale, yscale=\myyscale]{}} };
      \node (i4) at (-2, 1) { \NB{\tikz[font= \tiny,
  xscale=\myxscale, yscale=\myyscale]{}} };
      \node (i5) at (-2, 0) { \NB{\tikz[font= \tiny,
  xscale=\myxscale, yscale=\myyscale]{}} };
         \node (i6) at (-2, -1) { \NB{\tikz[font= \tiny,
  xscale=\myxscale, yscale=\myyscale]{}} };
  \node (i7) at (-3, 0) { \NB{\tikz[font= \tiny,
  xscale=\myxscale, yscale=\myyscale]{}} };
           \node (j0) at (0, -4) { \NB{\tikz[font= \tiny,
  xscale=\myxscale, yscale=\myyscale]{}} };
      \node (j1) at (-1, -4) { \NB{\tikz[font= \tiny,
  xscale=\myxscale, yscale=\myyscale]{}} };
   \node (j2) at (-1, -3) { \NB{\tikz[font= \tiny,
  xscale=\myxscale, yscale=\myyscale]{}} };
     \node (j3) at (-1, -5) { \NB{\tikz[font= \tiny,
  xscale=\myxscale, yscale=\myyscale]{}} };
      \node (j4) at (-2, -3) { \NB{\tikz[font= \tiny,
  xscale=\myxscale, yscale=\myyscale]{}} };
      \node (j5) at (-2, -4) { \NB{\tikz[font= \tiny,
  xscale=\myxscale, yscale=\myyscale]{}} };
         \node (j6) at (-2, -5) { \NB{\tikz[font= \tiny,
  xscale=\myxscale, yscale=\myyscale]{}} };
  \node (j7) at (-3, -4) { \NB{\tikz[font= \tiny,
  xscale=\myxscale, yscale=\myyscale]{}} };
  \draw[->] (i1) -- (i0);
  \draw[->] (i2) -- (i0);
    \draw[->] (i3) -- (i0);
    \draw[->] (i4) -- ([yshift=0.05cm]i1.west);
    \draw[->] (i6) -- ([yshift=-0.05cm]i1.west);
\draw[->] (i4) -- (i2);
    \draw[->] (i5) -- (i2);
    \draw[->] (i5) -- (i3);
    \draw[->] (i6) -- (i3);
\draw[->] (i7) -- (i4);
    \draw[->] (i7) -- (i5);
    \draw[->] (i7) -- (i6);
\draw[->] (j1) -- (j0);
  \draw[->] (j2) -- (j0);
    \draw[->] (j3) -- (j0);
\draw[->] (j4) -- (j2);
\draw[->] (j4) -- ([yshift=0.05cm]j1.west);
    \draw[->] (j6) -- ([yshift=-0.05cm]j1.west);
    \draw[->] (j5) -- (j3);
    \draw[->] (j6) -- (j3);
    \draw[->] (j6) -- (j2);
\draw[->] (j7) -- (j4);
    \draw[->] (j7) -- (j5);
    \draw[->] (j7) -- (j6);
    \draw[->] (i7) -- (j7)  node [right, pos = 0.5] {$\Id$};
    \draw[->] (i6) -- (j4)  node [right, pos = 0.5] {$\begin{pmatrix} \Id & 0 & 0 \\ 0 & \Id & 0 \\ 0 & 0 & \Id \end{pmatrix} $};
    \draw[->] (i3) -- (j2)  node [right, pos = 0.5] {$\begin{pmatrix} \Id & 0 & \Id \\ 0 & \Id & 0 \\ 0 & 0 & \Id \end{pmatrix} $};
        \draw[->] (i0) -- (j0) node [right, pos = 0.5] {$\Id$};
    }} .
\end{equation}

Therefore we now get that in the relative homotopy category that  $\NB{\tikz[scale=0.3]{\begin{scope}[font=\tiny]
  \draw (0.5, -0.5) ..controls +(0,0.3) and +(0,-0.3) .. (-0.5,0.5);
  \fill[white] (0,0) circle (2mm);
  \draw (-0.5, -0.5) ..controls +(0,0.3) and +(0,-0.3) .. (0.5,0.5);
  \begin{scope}[xshift= 1cm, yshift= 1cm]
  \draw (0.5, -0.5) ..controls +(0,0.3) and +(0,-0.3) .. (-0.5,0.5);
  \fill[white] (0,0) circle (2mm);
  \draw (-0.5, -0.5) ..controls +(0,0.3) and +(0,-0.3) .. (0.5,0.5);
  \end{scope}
  \begin{scope}[xshift=0cm, yshift= 2cm]
  \draw[->] (0.5, -0.5) ..controls +(0,0.3) and +(0,-0.3) .. (-0.5,0.5);
  \fill[white] (0,0) circle (2mm);
  \draw[->] (-0.5, -0.5) ..controls +(0,0.3) and +(0,-0.3) .. (0.5,0.5);
\end{scope}
\draw (-0.5,0.5) -- +(0,1);
\draw (1.5,-0.5) -- +(0,1);
\draw[->] (1.5,1.5) -- +(0,1);
\end{scope}}}$ is isomorphic to the bottom complex in \eqref{RIII3eq}.
There is a sequence of complexes \eqref{RIII5eq} which splits when forgetting the action of $\mathfrak{sl}_2$.
\begin{equation} \label{RIII5eq}
  \providecommand{\myxscale}{0.55}
  \providecommand{\myyscale}{0.45}
  \NB{\tikz[xscale = 2.5, yscale = 2.5]{
       \node (i0) at (0, 0) { \NB{\tikz[font= \tiny,
  xscale=\myxscale, yscale=\myyscale]{}} };
      \node (i1) at (-1, 0) { \NB{\tikz[font= \tiny,
  xscale=\myxscale, yscale=\myyscale]{}} };
   \node (i2) at (-1, 1) { \NB{\tikz[font= \tiny,
  xscale=\myxscale, yscale=\myyscale]{}} };
     \node (i3) at (-1, -1) { \NB{\tikz[font= \tiny,
  xscale=\myxscale, yscale=\myyscale]{}} };
      \node (i4) at (-2, 1) { \NB{\tikz[font= \tiny,
  xscale=\myxscale, yscale=\myyscale]{}} };
      \node (i5) at (-2, 0) { \NB{\tikz[font= \tiny,
  xscale=\myxscale, yscale=\myyscale]{}} };
         \node (i6) at (-2, -1) { \NB{\tikz[font= \tiny,
  xscale=\myxscale, yscale=\myyscale]{}} };
  \node (i7) at (-3, 0) { \NB{\tikz[font= \tiny,
  xscale=\myxscale, yscale=\myyscale]{}} };
     \node (i8) at (-3.25, -1.5) {};
    \node (i9) at (.25, 1.5) {};
  \node (j2) at (-1, 2.5) { \NB{\tikz[font= \tiny,
  xscale=\myxscale, yscale=\myyscale]{}} };
       \node (j5) at (-2, 2.5) { \NB{\tikz[font= \tiny,
  xscale=\myxscale, yscale=\myyscale]{}} };
    \node (k0) at (0, -3) { \NB{\tikz[font= \tiny,
  xscale=\myxscale, yscale=\myyscale]{}} };
     \node (k1) at (-1, -3.5) { \NB{\tikz[font= \tiny,
  xscale=\myxscale, yscale=\myyscale]{}} };
   \node (k2) at (-1, -2.5) { \NB{\tikz[font= \tiny,
  xscale=\myxscale, yscale=\myyscale]{}} };
\node (k4) at (-2, -2.5) { \NB{\tikz[font= \tiny,
  xscale=\myxscale, yscale=\myyscale]{}} };
\node (k6) at (-2, -3.5) { \NB{\tikz[font= \tiny,
  xscale=\myxscale, yscale=\myyscale]{}} };
  \node (k7) at (-3, -3) { \NB{\tikz[font= \tiny,
  xscale=\myxscale, yscale=\myyscale]{}} };
    \node (k8) at (-3.25, -2) {};
        \node (k9) at (.25, -4) {};
  \draw[->] (i1) -- (i0);
  \draw[->] (i2) -- (i0);
    \draw[->] (i3) -- (i0);
     \draw[->] (i4) -- (i1);
    \draw[->] (i4) -- (i2);
\draw[->] (i5) -- (i3);
    \draw[->] (i6) -- (i3);
    \draw[->] (i6) -- (i2);
        \draw[->] (i6) -- (i1);
    \draw[->] (i7) -- (i4);
    \draw[->] (i7) -- (i5);
    \draw[->] (i7) -- (i6);
        \draw[->] (j5) -- (j2) node [above, pos = 0.5] {$\Id $};
    \draw[->] (j5) -- (i4) node [right, pos = 0.5] {$\begin{pmatrix} 0  \\ \Id  \\  0 \end{pmatrix} $} ;
      \draw[->] (j2) -- (i2) node [right, pos = 0.5] {$\begin{pmatrix} 0  \\ 0  \\  \Id \end{pmatrix} $};
\draw[->] (k1) -- (k0);
  \draw[->] (k2) -- (k0);
\draw[->] (k4) -- (k1);
    \draw[->] (k4) -- (k2);
\draw[->] (k6) -- (k2);
       \draw[->] (k6) -- (k1);
    \draw[->] (k7) -- (k4);
\draw[->] (k7) -- (k6);    
        \draw[->] (i7) -- (k7) node [right, pos = 0.5] {$\Id$};  
        \draw[->] (i0) -- (k0) node [right, pos = 0.5] {$\Id$};  
        \draw[->] (i6) -- (k4) node [right, pos = 0.5] {$\begin{pmatrix} \Id & 0 & 0  \\ 0 & 0 &\Id  \end{pmatrix} $} ;   
      \draw[->] (i3) -- (k2) node [right, pos = 0.5] {$\begin{pmatrix} \Id & 0 & 0  \\ 0 & \Id & 0 \end{pmatrix} $} ;   
     \node[draw, densely dotted ,fit=(j2) (j5) ,inner sep=1ex,rectangle] {};
  \node[draw, densely dotted ,fit=(k8) (k9) ,inner sep=1ex,rectangle] {};
 \node[draw, densely dotted ,fit=(i8) (i9) ,inner sep=1ex,rectangle] {}; 
 }}
\end{equation}

Thus $\NB{\tikz[scale=0.3]{\begin{scope}[font=\tiny]
  \draw (0.5, -0.5) ..controls +(0,0.3) and +(0,-0.3) .. (-0.5,0.5);
  \fill[white] (0,0) circle (2mm);
  \draw (-0.5, -0.5) ..controls +(0,0.3) and +(0,-0.3) .. (0.5,0.5);
  \begin{scope}[xshift= 1cm, yshift= 1cm]
  \draw (0.5, -0.5) ..controls +(0,0.3) and +(0,-0.3) .. (-0.5,0.5);
  \fill[white] (0,0) circle (2mm);
  \draw (-0.5, -0.5) ..controls +(0,0.3) and +(0,-0.3) .. (0.5,0.5);
  \end{scope}
  \begin{scope}[xshift=0cm, yshift= 2cm]
  \draw[->] (0.5, -0.5) ..controls +(0,0.3) and +(0,-0.3) .. (-0.5,0.5);
  \fill[white] (0,0) circle (2mm);
  \draw[->] (-0.5, -0.5) ..controls +(0,0.3) and +(0,-0.3) .. (0.5,0.5);
\end{scope}
\draw (-0.5,0.5) -- +(0,1);
\draw (1.5,-0.5) -- +(0,1);
\draw[->] (1.5,1.5) -- +(0,1);
\end{scope}}}$ is isomorphic to 
\begin{equation} \label{RIII6eq}
    \providecommand{\myxscale}{0.55}
  \providecommand{\myyscale}{0.45}
    \NB{\tikz[xscale = 2.5, yscale = 2.5]{
       \node (i0) at (0, 0) { \NB{\tikz[font= \tiny,
  xscale=\myxscale, yscale=\myyscale]{}} };
     \node (i1) at (-1, -.5) { \NB{\tikz[font= \tiny,
  xscale=\myxscale, yscale=\myyscale]{}} };
   \node (i2) at (-1, .5) { \NB{\tikz[font= \tiny,
  xscale=\myxscale, yscale=\myyscale]{}} };
\node (i4) at (-2, .5) { \NB{\tikz[font= \tiny,
  xscale=\myxscale, yscale=\myyscale]{}} };
\node (i6) at (-2, -.5) { \NB{\tikz[font= \tiny,
  xscale=\myxscale, yscale=\myyscale]{}} };
  \node (i7) at (-3, 0) { \NB{\tikz[font= \tiny,
  xscale=\myxscale, yscale=\myyscale]{}} };
  \draw[->] (i1) -- (i0);
  \draw[->] (i2) -- (i0);
\draw[->] (i4) -- (i1);
    \draw[->] (i4) -- (i2);
\draw[->] (i6) -- (i2);
       \draw[->] (i6) -- (i1);
    \draw[->] (i7) -- (i4);
\draw[->] (i7) -- (i6);
    }} .
\end{equation}
This is left-right symmetric and hence isomorphic to  $\NB{\tikz[xscale = -.3, yscale=.3]{\begin{scope}[yscale= -1, font=\tiny]
  \draw[<-] (0.5, -0.5) ..controls +(0,0.3) and +(0,-0.3) .. (-0.5,0.5);
  \fill[white] (0,0) circle (2mm);
  \draw[<-] (-0.5, -0.5) ..controls +(0,0.3) and +(0,-0.3) .. (0.5,0.5);
  \begin{scope}[xshift= 1cm, yshift= 1cm]
  \draw (0.5, -0.5) ..controls +(0,0.3) and +(0,-0.3) .. (-0.5,0.5);
  \fill[white] (0,0) circle (2mm);
  \draw (-0.5, -0.5) ..controls +(0,0.3) and +(0,-0.3) .. (0.5,0.5);
  \end{scope}
  \begin{scope}[xshift=0cm, yshift= 2cm]
  \draw (0.5, -0.5) ..controls +(0,0.3) and +(0,-0.3) .. (-0.5,0.5);
  \fill[white] (0,0) circle (2mm);
  \draw (-0.5, -0.5) ..controls +(0,0.3) and +(0,-0.3) .. (0.5,0.5);
\end{scope}
\draw (-0.5,0.5) -- +(0,1);
\draw[<-] (1.5,-0.5) -- +(0,1);
\draw (1.5,1.5) -- +(0,1);
\end{scope}}} $.
\end{proof}

 \subsection{Unframed invariance}
\label{sec:unframed}

The result of the previous subsection shows that the $\sll_2$-enriched link homology is naturally a framed invariant of oriented links. However, one can introduce framing corrections to balance out the twists in Proposition \ref{prop:framed-RI}.

Define, 

\begin{equation}
  \widehat{T}:=
  \NB{\tikz[xscale = 0.5]{}} :=
  q^N\left(
    \NB{\tikz[xscale = 3.5, yscale = 3]{
    \node (i0) at (0, 0) {$q^{-1}$ \NB{\tikz[font= \tiny,
  scale=0.55]{\begin{scope}
  \coordinate (bl) at (-0.5, -1);
  \coordinate (br) at ( 0.5, -1);
  \coordinate (tl) at (-0.5,  1);
  \coordinate (tr) at ( 0.5,  1);
    \coordinate (ml) at (-0.5,  .6);
 \coordinate (mr) at (0.5,  .6);
 \coordinate (G) at (0.8,-0.1);

   \draw[->] (bl) -- (tl);
        \filldraw[draw= green!50!black, fill = white] (ml) circle (1mm) 
  node[left, green!50!black] 
  {${\frac{(N-2)t_1}{2}}$};
  
    \draw[->] (br) -- (tr);
        \filldraw[draw= green!50!black, fill = white] (mr) circle (1mm) 
  node[right, green!50!black] 
  {${\frac{(N-2)t_2}{2}}$};
          \filldraw[draw= green!50!black, fill = green] (G) circle (1mm) 
  node[right, green!50!black] 
  {$\frac{1}{2}$};

\end{scope}}} };
     \node (i1) at (-1.2, 0) {\NB{\tikz[font= \tiny,
  scale=0.55]{\begin{scope}
  \coordinate (bl) at (-0.5, -1);
  \coordinate (br) at ( 0.5, -1);
  \coordinate (bm) at (  0,-0.3);
  \coordinate (tl) at (-0.5,  1);
  \coordinate (tr) at ( 0.5,  1);
  \coordinate (tm) at (  0, 0.3);
  \draw[>-]  (bl) .. controls +( 0, 0.5) and +(0,0) .. (bm);
  %node[below, pos = 0] {$1$};
     % \filldraw[draw= green!50!black, fill = white] (bm) circle (1mm)
 % node[left, green!50!black] {$X$};
  \draw[>-]  (br) .. controls +( 0, 0.5) and +(0,0) .. (bm);
  %node[below, pos = 0] {$1$};
  \draw[<-]  (tl) .. controls +( 0, -0.5) and +(0,0) .. (tm) coordinate[pos = 0.25] (ga) ;
  % node[above, pos = 0] {$1$} 
    \filldraw[draw= green!50!black, fill = white] (ga) circle (1mm)
  node[left, green!50!black] {$\frac{Nt_1}{2}$};

  \draw[<-]  (tr) .. controls +( 0, -0.5) and +(0,0) .. (tm)
  coordinate[pos = 0.25] (gb) ; %  node[above, pos = 0] {$1$} 
    \filldraw[draw= green!50!black, fill = white] (gb) circle (1mm)
  node[right, green!50!black] {$\frac{Nt_2}{2}$};
  
  \draw [double] (bm) -- (tm) 
   coordinate[pos = 0.5, xshift = 2mm] (gc);  
  \filldraw[draw= green!50!black, fill = green, right] (gc) circle (1mm)
  node[right, green!50!black] {$\frac{1}{2}$};
\end{scope}}} };
\draw[->] (i1) -- (i0) coordinate[pos=0.5] (a);
\node[above] at (a) {\NB{\tikz[font=\tiny, scale=.4]{}}};
  }}
\right)
    \end{equation}
    \begin{equation}
  \widehat{T}'= \NB{\tikz[xscale = 0.5]{}}:= q^{-N}\left(
    \NB{\tikz[xscale = 3.5, yscale = 3]{
    \node (i0) at (-1.2, 0) { $q$\ \NB{\tikz[font= \tiny,
  scale=0.55]{\begin{scope}
  \coordinate (bl) at (-0.5, -1);
  \coordinate (br) at ( 0.5, -1);
  \coordinate (tl) at (-0.5,  1);
  \coordinate (tr) at ( 0.5,  1);
    \coordinate (ml) at (-0.5,  .6);
 \coordinate (mr) at (0.5,  .6);
 \coordinate (G) at (0.8,-0.1);

   \draw[->] (bl) -- (tl);
        \filldraw[draw= green!50!black, fill = white] (ml) circle (1mm) 
  node[left, green!50!black] 
  {${\frac{(2-N)\bar{t}_1}{2}}$};
  
    \draw[->] (br) -- (tr);
        \filldraw[draw= green!50!black, fill = white] (mr) circle (1mm) 
  node[right, green!50!black] 
  {${\frac{(2-N)\bar{t}_2}{2}}$};
          \filldraw[draw= green!50!black, fill = green] (G) circle (1mm) 
  node[right, green!50!black] 
  {$-\frac{1}{2}$};

\end{scope}}} };
     \node (i1) at (0, 0) { \NB{\tikz[font= \tiny,
  scale=0.55]{\begin{scope}
  \coordinate (bl) at (-0.5, -1);
  \coordinate (br) at ( 0.5, -1);
  \coordinate (bm) at (  0,-0.3);
  \coordinate (tl) at (-0.5,  1);
  \coordinate (tr) at ( 0.5,  1);
  \coordinate (tm) at (  0, 0.3);
  \draw[>-]  (bl) .. controls +( 0, 0.5) and +(0,0) .. (bm);
  %node[below, pos = 0] {$1$};
     % \filldraw[draw= green!50!black, fill = white] (bm) circle (1mm)
 % node[left, green!50!black] {$X$};
  \draw[>-]  (br) .. controls +( 0, 0.5) and +(0,0) .. (bm);
  %node[below, pos = 0] {$1$};
  \draw[<-]  (tl) .. controls +( 0, -0.5) and +(0,0) .. (tm) coordinate[pos = 0.25] (ga) ;
  % node[above, pos = 0] {$1$} 
    \filldraw[draw= green!50!black, fill = white] (ga) circle (1mm)
  node[left, green!50!black] {$-\frac{N\bar{t}_1}{2}$};

  \draw[<-]  (tr) .. controls +( 0, -0.5) and +(0,0) .. (tm)
  coordinate[pos = 0.25] (gb) ; %  node[above, pos = 0] {$1$} 
    \filldraw[draw= green!50!black, fill = white] (gb) circle (1mm)
  node[right, green!50!black] {$-\frac{N\bar{t}_2}{2}$};
  
  \draw [double] (bm) -- (tm) 
   coordinate[pos = 0.5, xshift = 2mm] (gc);  
  \filldraw[draw= green!50!black, fill = green, right] (gc) circle (1mm)
  node[right, green!50!black] {$-\frac{1}{2}$};
\end{scope}}} };
\draw[->] (i0) -- (i1) coordinate[pos=0.5] (b);
\node[above] at (b) {\NB{\tikz[font=\tiny, scale=.4]{}}}; 
      }} \right)
  \,
    \end{equation}
  Here in both complexes we assume  that the terms
\[
  \NB{\tikz[font= \tiny,
  scale=0.55]{}} ,\quad \quad
   \NB{\tikz[font= \tiny,
  scale=0.55]{}}
\]
sit in cohomological degree 0.

 In a similar vein as before, we define the \emph{unframed  Khovanov--Rozansky $\gll_N$-homology} of a link $L$ with an $\sll_2$-action, denoted $\widehat{\KR}_N^{\mathfrak{sl}_2}(L;\Bbbk) $, via the twisted complexes $\widehat{T}$ and $\widehat{T}^\prime$.

\begin{cor}
The homology
$\widehat{\KR}_N^{\mathfrak{sl}_2}(L;\Bbbk)$ is an invariant of unframed oriented links in $\mathbb{R}^3$.
\end{cor}
\begin{proof}
    Twists in the definition of $\widehat{T}$ and $\widehat{T}^\prime$ are introduced to cancel the twisting in the structures of the $\sll_2$-modules appearing in the RI invariance (Proposition \ref{prop:framed-RI}). The rest of the Reidemeister moves are proved similarly to the framed version in the previous section. 
\end{proof}

\begin{rmk}
We make a final remark on other possible framing corrections in some special cases. Firstly, thanks to Lemma \ref{lem:twotwists}, one can have more freedom in placing the dots in the definition of $\widehat{T}$ and $\widehat{T}^\prime$. We choose here the most symmetric one, and the other choices are relatively homotopic to the ones here.

Secondly, suppose we are in either of the following situations:
\begin{itemize}
    \item[(i)] $N=2$ and $\Bbbk$ is a general base ring in which $2$ is invertible.
    \item[(ii)] $N=p\ell+2$ for some $\ell\in \mathbb{N}$ and $\Bbbk$ is a ring of prime characteristic $p>2$. 
\end{itemize}
Under these assumptions, the RI moves in Proposition \ref{prop:framed-RI} 
becomes
 \begin{equation} \label{eqn:sl2-twisting-bimod-for-links-special}
  \NB{\tikz[scale = 0.6]{}} \cong  tq^{-N}\NB{\tikz[scale= 0.6, font =\tiny]{\begin{scope}
  \draw [->] (0, -1) -- +(0,2);
  \filldraw[draw= green!50!black, fill = green, left] (ga) circle (1mm)
  node[left, green!50!black] {$-\frac{1}{2}$}; 
\end{scope}
  
}} \ , \quad \quad \quad \quad
  \NB{\tikz[scale= 0.6, yscale =-1]{}} \cong t^{-1}q^{N}\NB{\tikz[scale= 0.6, font =\tiny]{\begin{scope}
  \draw [->] (0, -1) -- +(0,2);
  \filldraw[draw= green!50!black, fill = green, left] (ga) circle (1mm)
  node[left, green!50!black] {$\frac{1}{2}$}; 
\end{scope}
  
}} \ .
 \end{equation}
In this case, the unframed $\gll_N$-homology is  obtained from the framed one by:
\begin{itemize}
\item Shifting the cohomological degree by $-\mathtt{w}(L)$, where $\mathtt{w}(L)$ stands for the usual writhe number of $L$ (the number of positive crossings minus the number of negative crossings). 
\item Shifting the $q$-degree by $N\mathtt{w}(L)$.
\item Twisting by the overall factor $\frac{\mathtt{w}(L)}{2} \gsoliddot$.
\end{itemize}
\end{rmk}

\section{Functoriality}
\label{functoriality:sec}
Full functoriality of Khovanov--Rozansky $\mathfrak{gl}_N$-link homologies was proved by Ehrig, Tubbenhauer, and Wedrich \cite{ETW}. This extends earlier work of functoriality for the $\mathfrak{sl}_2$ case in \cite{BHPW}, \cite{blan1}, \cite{Cap}, \cite{CMW},  \cite{EST1},  and \cite{Vog}.  Functoriality for $\mathfrak{sl}_2$ case first established by Jacobsson \cite{Jac} and Khovanov \cite{Khcob} who proved it up to a sign.  

In this section we will slightly extend the work of \cite{ETW} and show that the target of their functor $\mathcal{F}_N$ lands in the relative homotopy category of foams, with inner hom spaces carrying $\mathcal{U}(\mathfrak{sl}_2)$-module structures.  See Section \ref{homological:sec} for the general framework. We remind the reader that this means that inner morphism sets are objects in the monoidal category and composition of morphisms corresponds to tensor product of objects.

\subsection{Foam category}
We begin by establishing some preparatory definitions and lemmas.

Let 
$\gFm$ be the category whose objects are green-dotted webs and whose
morphisms are (linear combination of) foams up to isotopy and equivalence relations induced
by foam evaluation \cite{RW1}.

The \emph{graded hom space} in $\gFm$, between two green-dotted webs $\Gamma_1$ and $\Gamma_2$ is given by foams whose incoming boundary is $\Gamma_1$ and outgoing boundary is $\Gamma_2$. We denote this hom space by $\Hom_{\gFm}(\Gamma_1,\Gamma_2)$. For an element $F\in \Hom_{\gFm}(\Gamma_1,\Gamma_2)$ in this category, and $g \in \mathfrak{sl}_2$, 
\begin{equation} \label{enriched0}
g * F = g(F) - g(\Gamma_1) F + g(\Gamma_2) F \ ,    
\end{equation} 
where $g(\Gamma_j) F$, for $j=1,2$, is the foam $F$ with extra
decorations coming from $g(\Gamma_j)$ (given by formulas
\eqref{eq:e-act-pol-twisthollow}--\eqref{eq:f-act-pol-twistsolid}). 

The flatness of the twists associated with green dots implies that
formula \eqref{enriched0} defines an $\sll_2$-action on the hom spaces
of $\gFm$. From the definition of $*$, one in turn obtains that $\gFm$
is enriched in the monoidal category of $\mathcal{U}(\sll_2)$.
Note also that a morphism $F$ is $\sll_2$-equivariant if and only if
$g * F =0$ for all $g\in \sll_2$. Using this action, we may identify
the state space $\mc{F}_N(\Gamma)$ associated with a web $\Gamma$ as the graded hom space $\Hom_{\gFm}(\emptyset, \Gamma)$, with the induced
$\sll_2$-action $*$. Thus $\mc{F}_N(\Gamma)$ is a graded equivariant module over the commutative $\sll_2$-module algebra $D_\Gamma$ (see Section \ref{subsec:greendottedwebs}).

Define $\gFm\#\sll_2$ to be the category that has the same objects of $\gFm$, but with morphisms that commute with the $\sll_2$-action.  It is equipped with an enriched structure (c.f. equation \eqref{eqn-H-action}) on its \emph{inner hom space}, which is equal to $\Hom_{\gFm}(\Gamma_1,\Gamma_2)$ for any green dotted webs. By equation \eqref{eqn-H-inv-in-HOM},
\begin{equation}
    \Hom_{\gFm\# \sll_2}(\Gamma_1,\Gamma_2)=\Hom_{\gFm}(\Gamma_1,\Gamma_2)^{\sll_2}.
\end{equation}
Also denote by $\Com(\gFm \# \sll_2 )$ the abelian category of complexes in $\gFm\# \sll_2$. The inner hom space extends naturally: given two complexes of green-dotted webs $C$ and $D$
\begin{equation}
    \HOM_{\gFm }(C,D)=\oplus_{i\in \ZZ}\Hom_{\gFm }(C,t^iD).
\end{equation}
Let $\mc{C}(\gFm \# \sll_2)$ be the corresponding homotopy category, with the enriched structure. 

Let $\Com(\gFm)$ denote the usual abelian category of graded foams, $\mc{C}(\gFm)$ be its homotopy category. 
As in Section \ref{homology:sec}, we define $\mc{C}^{\mathfrak{sl}_2}(\gFm)$ to be the corresponding
relative homotopy category, obtained as the Verdier localization of $\mc{C}(\gFm\#\sll_2)$ along the exact forgetful functor
\begin{equation}
    \mathrm{For}: \mc{C}(\gFm\#\sll_2) \lra \mc{C}(\gFm).
\end{equation}
Translating the general construction of Section \ref{homological:sec} into our context, we have the following.

\begin{prop} \label{prop:comenriched}
The category $\mc{C}^{\sll_2}(\gFm)$ is enriched in the category of $\mathcal{U}(\mathfrak{sl}_2)$-modules, with the inner hom space between two green-dotted webs acted on by $\sll_2$ according to equation \eqref{enriched0}. \hfill $\square$
\end{prop}

\subsection{Link cobordisms}

\begin{dfn}
Let $ \Links$ be the category whose objects are oriented framed
links in $\mathbb{R}^3$ and whose morphisms are two-dimensional oriented framed cobordisms between oriented links in $\mathbb{R}^3 \times \mathbb{R}$.
\end{dfn}

A presentation of this category was found by Beliakova and Wehrli
\cite{BW} building upon work of Carter and Saito \cite{CarSai} in the
unframed case.

Following the notation of \cite[Figure 9]{ETW}, we denote some of the
generating movie moves by $\MGH$, $\MGS$, $\MGTW$, and $\MGTH$ when reading the movies from left to right.
Reading the movies in the reverse direction, we denote the generating
moves by $\MGH'$, $\MGS'$, $\MGTW'$, and $\MGTH'$ respectively.  We omit writing down the other oriented variations of these generators.

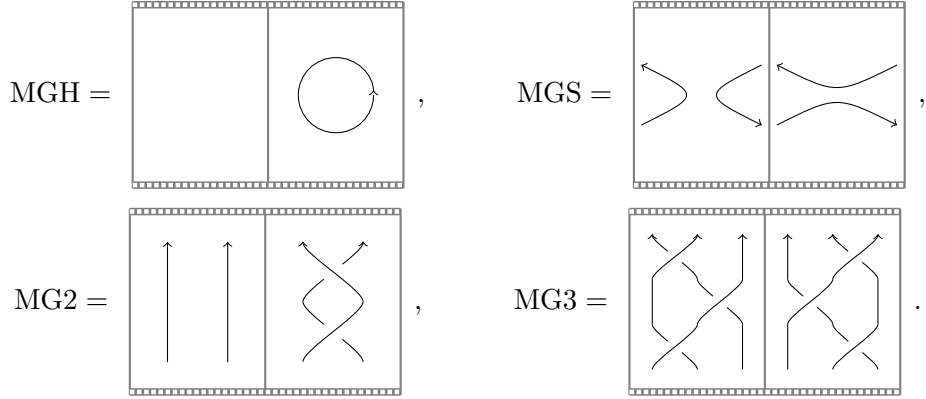
\[ \MGH = 
  \mymovie[yscale = 0.8, xscale =0.45]{}{\NB{\tikz[]{\begin{scope}
  \coordinate (bl) at (0, -1);
  \coordinate (br) at (2, -1);
  \coordinate (tl) at (0,  1);
  \coordinate (tr) at (2,  1);
    \coordinate (ml) at (-0.5,  -.8);
        \coordinate (Ml) at (-0.5,  .8);
 \coordinate (mr) at (0.5,  -.6);
\coordinate (Mr) at (0.5,  .6);

%   \draw[<-<] (bl) -- (tl) node[pos = 0, below] {$1$} node[pos = 1,
%  above] {$1$};
%     \draw[>->] (br) -- (tr) node[pos = 0, below] {$1$} node[pos = 1,
%  above] {$1$} coordinate[pos = 0.25] (A) coordinate[pos = 0.75] (B);
    \draw [-<] (.5, 0) arc (180:0:0.5) ;
    \draw (.5, 0) arc (-180:0:0.5) coordinate[pos = 0.2] (X) coordinate[pos = 0.8] (Y);
%    \filldraw[draw= green!50!black, fill = green] (X) circle (1mm)
 % node[below, green!50!black] {$\frac{1}{2}$};
 %   \filldraw[draw= green!50!black, fill = white] (Y) circle (1mm)
  %node[below, green!50!black] {$\frac{N-1}{2}$};
%\filldraw[draw= green!50!black, fill = white] (A) circle (1mm)
 % node[right, green!50!black] {$-\frac{N-1}{2}$}; 
%  \filldraw[draw= green!50!black, fill = green] (B) circle (1mm)
 % node[right, green!50!black] {$\frac{-1}{2}$};
\end{scope}}}}, \qquad \quad
\MGS = 
  \mymovie[yscale = 0.8, xscale =0.45]{\NB{\tikz[scale=0.8]{\begin{scope}[xshift = 2.5cm]
  \draw [->](0, 0) .. controls +(1,.5) 
  .. (0, 1);
    \draw [<-](2, 0) .. controls +(-1,.5) 
  .. (2, 1);
  %%%%%%
 % \draw [->](0, 0) .. controls +(1,.5) 
%  .. (2, 0);
%    \draw [<-](0, 1) .. controls +(1,-.5) 
%  .. (2, 1);  

\end{scope}}}}{\NB{\tikz[scale=0.8]{\input{\imagesfolder/pdg_horres2}}}},
  \]  
\[ \MGTW = 
  \mymovie[yscale = 0.8, xscale
  =0.45]{\NB{\tikz[scale=0.8]{\begin{scope}
  \coordinate (bl) at (-0.5, -1);
  \coordinate (br) at ( 0.5, -1);
  \coordinate (tl) at (-0.5,  1);
  \coordinate (tr) at ( 0.5,  1);
    \coordinate (ml) at (-0.5,  -.8);
        \coordinate (Ml) at (-0.5,  .8);
 \coordinate (mr) at (0.5,  -.6);
\coordinate (Mr) at (0.5,  .6);

 %\draw[>->] (bl) -- (tl) node[pos = 0, below] {$2$} node[pos = 1,
  %above] {$1$} coordinate[pos = 0.4] (ml);
   \draw[->] (bl) -- (tl) node[pos = 0, below] {} node[pos = 1,
  above] {};
  %\draw[>->] (br) -- (tr) node[pos = 0, below] {$1$} node[pos = 1, above] {$2$} coordinate[pos = 0.6] (mr);
  
    \draw[->] (br) -- (tr) node[pos = 0, below] {} node[pos = 1, above] {};
  
  %\draw[->-] (ml) -- (mr) node [pos= 0.5, above] {$1$};
  %  \draw[->-] (Ml) -- (Mr) node [pos= 0.5, above] {$1$};

\end{scope}}}}{\NB{\tikz[scale=0.8]{\begin{scope}
  \draw[] (3, 0) .. controls +(0, 0.2) and +(0, -0.2) ..  +(-1,1);
  \fill[white] (2.5, 0.5) circle (2mm);
  \draw (2, 0) .. controls +(0, 0.2) and +(0, -0.2) ..  +(1,1);
  %%%
    \draw[->] (2, 1) .. controls +(0, .2) and +(0, -0.2) ..  +(1,1); 
       \fill[white] (2.5, 1.5) circle (2mm);

   \draw[->] (3, 1) .. controls +(0, .2) and +(0, -0.2) ..  +(-1,1);
  % \draw (2, 0) arc (180:0:-0.5) -- (1,2) arc (180:0:0.5) coordinate[pos = 0.25] (X) coordinate[pos = 0.75] (Y);
%    \draw (3, 0) arc (-180:0:0.5) -- (4,2) arc (-180:0:-0.5) coordinate[pos = 0.25] (X) coordinate[pos = 0.75] (Y);
\end{scope}}}},\quad \qquad
\MGTH = 
  \mymovie[yscale = 0.8, xscale =0.45]{\NB{\tikz[font= \tiny,
  scale=0.6]{\begin{scope}
  \draw[] (3, 0) .. controls +(0, 0.2) and +(0, -0.2) ..  +(-1,1);
  \fill[white] (2.5, 0.5) circle (2mm);
  \draw (2, 0) .. controls +(0, 0.2) and +(0, -0.2) ..  +(1,1);
  \draw[] (4,0) -- (4,1);
  %%%
   \draw[] (2,1) -- (2,2);
   \draw[] (4, 1) .. controls +(0, 0.2) and +(0, -0.2) ..  +(-1,1);
  \fill[white] (3.5, 1.5) circle (2mm);
  \draw (3, 1) .. controls +(0, 0.2) and +(0, -0.2) ..  +(1,1);
  
 %%%% 
   \draw[->] (3, 2) .. controls +(0, .2) and +(0, -0.2) ..  +(-1,1);

       \fill[white] (2.5, 2.5) circle (2mm);

     \draw[->] (2, 2) .. controls +(0, .2) and +(0, -0.2) ..  +(1,1); 
  \draw[->] (4,2) -- (4,3);

\end{scope}}}}{\NB{\tikz[font= \tiny,
  scale=0.6]{\begin{scope}
  \draw[] (4, 0) .. controls +(0, 0.2) and +(0, -0.2) ..  +(-1,1);
  \fill[white] (3.5, 0.5) circle (2mm);
  \draw (3, 0) .. controls +(0, 0.2) and +(0, -0.2) ..  +(1,1);
  \draw[] (2,0) -- (2,1);
  %%%
   \draw[] (4,1) -- (4,2);
   \draw[] (3, 1) .. controls +(0, 0.2) and +(0, -0.2) ..  +(-1,1);
  \fill[white] (2.5, 1.5) circle (2mm);
  \draw (2, 1) .. controls +(0, 0.2) and +(0, -0.2) ..  +(1,1);
  
 %%%% 
    \draw[->] (4, 2) .. controls +(0, .2) and +(0, -0.2) ..  +(-1,1);
 
       \fill[white] (3.5, 2.5) circle (2mm);

   \draw[->] (3, 2) .. controls +(0, .2) and +(0, -0.2) ..  +(1,1); 
  \draw[->] (2,2) -- (2,3);

\end{scope}}}}.
  \]    
These generators satisfy certain relations described in \cite{BW, CarSai} and other references (such as \cite[Section 4]{ETW}).  
  
For each movie generator $G \colon X \rightarrow Y$, there is an associated foam or morphism of complexes of green-dotted webs $\mc{F}_N(G) \colon \mc{F}_N(X) \rightarrow \mc{F}_N(Y)$.
If $G=\MGO, \MGTH, \MGTH$, the morphism $\mc{F}_N(G)$ is a relative homotopy equivalence.  That is, if $G \colon X \rightarrow Y$, then one could write $\mc{F}_N(G)= (\alpha_{n-1}')^{-1} \circ \alpha_{n-1} \circ \cdots \circ (\alpha_{0}')^{-1} \circ \alpha_{0}$ as a sequence of ``roofs''
\begin{gather}
  \NB{
    \tikz[xscale=1, yscale =1]{
    \node (A1) at (0,0) {$\mc{F}_N(X)=X_0$};
    \node (A2) at (2,0) {$X_1$};
    \node (B1) at (1,-1) {$X_0'$};
    \node (C1) at (4,-.5) {$\cdots$};
 \node (A3) at (6,0) {$X_{n-1}$};
    \node (A4) at (8,0) {$X_{n}=\mc{F}_N(Y)$};
        \node (B2) at (7,-1) {$X_{n-1}'$};
              \node (B3) at (3,-1) {};
                \node (B4) at (5,-1) {};
 \draw[-to] (A1) -- (B1) node[pos =0.7, left] {$\alpha_0$};
      \draw[-to] (A2) -- (B1) node[pos =0.7, right] {$\alpha_0'$}; 
     \draw[-to] (A3) -- (B2) node[pos =0.7, left] {$\alpha_{n-1}$};
      \draw[-to] (A4) -- (B2) node[pos =0.7, right] {$\alpha_{n-1}'$};      
          \draw[-to] (A2) -- (B3) node[pos =0.7, left] {};
     \draw[-to] (A3) -- (B4) node[pos =0.7, left] {};
}
  }  \ ,
\end{gather}
where the $\alpha_i$ and $ \alpha_i'$ are relative homotopy equivalences such that $g * \alpha_i=0$ for $g \in \mathfrak{sl}_2$ and for all $i$.  For degree reasons (see \cite[Lemma 4.6]{ETW}), if $G=\MGO, \MGTW, \MGTH$, the map
$\mc{F}_N(G)$ must agree with the map in \cite{ETW} up to a scalar.  We choose the $\alpha_i$ and $\alpha_i'$ so that $\mc{F}_N(G)$ precisely agrees with the map in \cite{ETW}.  We now extend the functoriality result of $\mathfrak{gl}_N$-link homology \cite[Theorem 4.5]{ETW} to our setting.

\begin{thm} \label{thm:functorFN}
There is a functor $\KR_N^{\sll_2} \colon \Links \rightarrow
\mc{C}^{\mathfrak{sl}_2}(\gFm)$ in which the target category of the functor is enriched in the category of $\mathfrak{sl}_2$-modules. 

Furthermore, the $\mathfrak{sl}_2$-structure on the morphism space is an invariant of the framed links.
\end{thm}

\begin{proof}
The functor $\mathcal{F}_N \colon \Links \rightarrow \mc{C}(\gFm)$ is constructed in \cite[Theorem 4.5]{ETW}. 
That is, to every movie generator $M$, there is an associated map of complexes $\mathcal{F}_N(M)$ of (green-dotted) webs.  If there is a movie relation $M_1 = M_2$, then the maps $\mathcal{F}_N(M_1)$ and
$\mathcal{F}_N(M_2)$ are homotopic.

Let $\alpha, \beta \colon \mathcal{F}_N(L_1) \rightarrow \mathcal{F}_N(L_2)$ be two relative homotopic maps. Thus in the homotopy category $\mc{C}^{\sll_2}(\gFm)$, their classes $[\alpha]$ and $[\beta]$ are equal.  We will show that for any $g \in \mathcal{U}(\mathfrak{sl}_2)$, that we have an equality in the relative homotopy category $g*[\alpha]=g*[\beta]$. Since $[\alpha]=[\beta]$, there is a homotopy $H$, when forgetting the $\sll_2$-actions, such that
\begin{equation} \label{eq:gactonhom}
\alpha-\beta= H d + d H .    
\end{equation}
Acting by $g$ on both sides of \eqref{eq:gactonhom}, using the Leibniz rule and noting that $\sll_2$-actions commute with differentials, we obtain
\[
g * \alpha - g * \beta = (g*H) d + d (g*H) .
\]
Thus we obtain $g*[\alpha]=g*[\beta]$.

The last part follows from Theorem \ref{thm:sl2inv} and Proposition \ref{prop:comenriched}.
\end{proof}

Let $\Linkscon$ be the category whose objects are framed links and whose morphisms are connected link concordances. Recall that a link concordance is a link cobordism which is an embedding of cylinders inside $\mathbb{R}^3\times [0,1]$ with boundary components the given links. The Euler characteristic of a link concordance is always zero.  

\begin{cor} \label{cor:ribboninj}
The functor $\mathcal{F}_N$ from Theorem \ref{thm:functorFN} restricts to a functor
\[\mathcal{F}_N \colon \Linkscon \longrightarrow \mc{C}^{\mathfrak{sl}_2}(\Bbbk_N).\]
\end{cor}

\begin{proof}
Since the Euler characteristic of the concordance is zero, the number of saddles is equal to the sum of the number of cups and caps.  Using the formulas from Section \ref{sec:sl2action} (more specifically, the action of $\sll_2$ generators in which $\Le$ acts trivially on naked foams,  the $\Lh$-action is determined locally by equations \eqref{eq:h-act-cup}--\eqref{eq:h-act-saddle}, and the $\Lf$-action is determined locally by equations \eqref{eq:e-act-cup}--\eqref{eq:e-act-saddle}), and the fact that the concordance is connected, the $\mathfrak{sl}_2$-action on the concordance is actually zero.  Thus the induced map on $\mathfrak{sl}_2$-representations intertwines the $\mathfrak{sl}_2$-actions. 
\end{proof}

\begin{rmk}
    We chose to present functoriality for the category of framed links. However, using the  renormalization tricks from Section~\ref{sec:unframed}, this can also be formulated in the category of unframed links.
\end{rmk}

A concordance $C$ is \emph{ribbon} if the projection to the $[0, 1]$ factor
restricts to a Morse function on $C$ with only index 0 and 1 critical points.
Levine and Zemke established the following statement (where
$\mathrm{Kh}$ denotes Khovanov):

\begin{thm}[{\cite[Theorem 1]{LevineZemke}}]
If $C$ is a ribbon concordance from $L_0$ to $L_1$, the induced map
\[\mathrm{Kh}(C)\co \mathrm{Kh}(L_0) \to \mathrm{Kh}(L_1)\]
is injective, with left inverse given by $\mathrm{Kh}(\overline{C})$,
where $\overline{C}$ denotes the mirror image of $C$. 
\end{thm}

The proof relies on a neat topological argument by Zemke \cite{Zemke},
and therefore extends to Khovanov--Rozansky's setup (see \cite{CGLLSZ}). 
Incorporating the the $\sll_2$-action, one immediately obtains:
\begin{cor}
If $C$ is a ribbon concordance from $L_0$ to $L_1$, then
$\KR_N^{\mathfrak{sl}_2}(L_0;R)$ is a direct summand of $\KR_N^{\mathfrak{sl}_2}(L_1;R)$ as $\sll_2$-modules.
\end{cor}

\section{Specializations}
\label{sec:special}

\subsection{$\mathfrak{gl}_{p\ell}$-homology in characteristic $p$}
\label{sec:Fpslp}

In this subsection, we consider the setup where the base ring is
$R=\scalars$, a field of characteristic $p>2$ and look
at $\gll_{N}$-link homology when $N=p\ell$ is a multiple of the
characteristic of $\scalars$.

Shumakovitch \cite{Shum} and
Wang \cite{Wang} used an operator denoted $\nu$ or $\nabla$ to prove
that reduced $\mathfrak{gl}_p$-homology in characteristic $p$ is
independent of the choice of the basepoint of the link and  that non-reduced homology is isomorphic to
reduced homology tensored with $\scalars[x]/(x^p)$. It turns out that
$\nabla$ acts like $\Le$ in a non-equivariant setting.

Shumakovitch and Wang's operator fits into our general construction
(and was a partial motivation of our work), because the ideal $I$ of
$\scalars[E_1,\ldots,E_p]$ generated by $E_1,\ldots,E_p$ is preserved
by the action of $ \mathfrak{sl}_2$. As we shall see, working in
characteristic $p$ is crucial.

Shumakovitch's results were for $p=2$ and Wang extended this for any
prime $p$.  Since we consider the full $\mathfrak{sl}_2$-action which
forces us to invert $2$ and therefore we suppose that $p>2$.

\begin{lem} \label{thm:sl2invFp} The homology
  $\KR_{p \ell ;t_1,t_2}^{\mathfrak{sl}_2}(L;R)$ is an invariant of framed
  oriented links.
\end{lem}

\begin{proof}
  One only needs to show that the $\sll_2$-action is well-defined on
  state spaces. Recall that these state spaces are defined using the
  universal construction. In the equivariant setting, the
  $\sll_2$-structure on state spaces follows from the identity stated
  in Proposition~\ref{prop:sl2-acts-good-pos}.
  We will prove the same identity in the non-equivariant setting.
  In this setting, the evaluation of a foam $\foam$ is given by:
  \[
    \bracket{\foam}_{p \ell}^{\scalars} = \mathrm{ev}(\bracket{\foam}_{p\ell}).
  \]
  where the map
  \[
    \mathrm{ev}:\thinspace\scalars[E_1,\ldots,E_{p \ell}] \longrightarrow \scalars
  \]
  maps $E_1$, \dots, $E_{p\ell}$ to $0$.
  This is an $\sll_2$-map (the field $\scalars$ being endowed with the trivial
  $\sll_2$-action). Indeed, the only non-trivial fact is that
  $\mathrm{ev}(\Le\cdot E_1) =\Le \cdot \mathrm{ev}(E_1) =0$. This is
  the case since $\Le\cdot E_1 = p\ell  = 0 \in \scalars$.
  
  Hence if $\foam$ is a closed foam in good position, then for 
  $x \in \{\de, \df, \dh \}$
  \begin{equation} \bracket{x \cdot \foam}_{p\ell}^{\scalars}= x \cdot
    \mathrm{ev}(\bracket{x \cdot \foam}_{p\ell}) = \mathrm{ev}(x \cdot
    \bracket{\foam}_{p\ell})  = x \cdot  \mathrm{ev}(\bracket{\foam}_{p\ell}) = x \cdot \bracket{\foam}_{p\ell}^{\scalars}.\end{equation}
  
  Invariance of the $\sll_2$-action
  under framed isotopies follows from that in the general setting.
\end{proof}

\subsection{The Zuckerman functor and some representation theory}
\label{sec:zuckerman}
For this section let $\scalars=\mathbb{C}$ and $R=\mathbb{C}[E_1,\ldots, E_N] $.
We set up some representation theoretic machinery used in computing
examples in Section \ref{sec:examples}.  We also introduce the
Zuckerman functor as a speculative tool to extract something
finite-dimensional from equivariant link homology.
 An $\mathfrak{sl}_2$-module
is \emph{locally finite} if any of its elements generates a finite-dimensional $\mathfrak{sl}_2$-module.

The Zuckerman functor $\Gamma \colon \mathfrak{sl}_2 \mod \rightarrow \mathfrak{sl}_2 \mod$ is the functor which takes an $\mathfrak{sl}_2$-module $M$ and returns its maximally locally finite submodule $\Gamma M$ under the action of $\mathfrak{sl}_2$.

The Bernstein functor (sometimes known as the dual Zuckerman functor) 
$$\mathcal{Z} \colon \mathfrak{sl}_2 \mod \rightarrow \mathfrak{sl}_2 \mod$$
is the functor which takes an $\mathfrak{sl}_2$-module $M$ and returns its maximally locally finite quotient $\mathcal{Z} M$ under the action of $\mathfrak{sl}_2$.  

For an integer $\lambda \in \mathbb{Z}$, there is an $\mathfrak{sl}_2$-module $M(\lambda)$ which is isomorphic to $\Bbbk[\Lf]$ as a $\Bbbk[\Lf]$-module with $\Le \cdot 1=0$, and $\Lh \cdot 1 = \lambda$.  This object is known as a Verma module with highest weight $\lambda$.

For $\lambda \in \mathbb{Z}$, we let $L(\lambda)$ denote the irreducible $\mathfrak{sl}_2$-module whose highest weight is $\lambda$.  If $\lambda \geq 0$, $L(\lambda)$ is finite-dimensional.  Let $P(\lambda)$ and $P(-\lambda-2)$ be the indecomposable projective covers of $L(\lambda)$ and $L(-\lambda-2)$ respectively.  The Verma module $M(\lambda)$ is isomorphic to $P(\lambda)$ and $P(-\lambda-2)$ can be described as an extension of Verma modules
\begin{equation} \label{sesproj}
0 \rightarrow M(\lambda) \rightarrow P(-\lambda-2) \rightarrow M(-\lambda-2) \rightarrow 0 \ .
\end{equation}
We will also use the short exact sequence
\begin{equation} \label{sesverma}
0 \rightarrow L(-\lambda-2) \rightarrow M(\lambda) \rightarrow L(\lambda) \rightarrow 0 \ .
\end{equation}
Taking the dual of \eqref{sesverma} yields the short exact sequence
\begin{equation} \label{sesdualverma}
0 \rightarrow L(\lambda) \rightarrow M^*(\lambda) \rightarrow L(-\lambda-2) \rightarrow 0 
\end{equation}
since the simple objects are preserved by the contravariant duality functor $*$.
Sequences \eqref{sesverma} and \eqref{sesdualverma} describe the unique extensions between simples.

We record the following well-known result about the Zuckerman functor and Verma modules.

\begin{prop}
The Zuckerman functor $\Gamma$ annihilates Verma modules.
For the Bernstein functor we have
\[
\mathcal{Z}(M(\lambda)) \cong
\begin{cases}
L(\lambda) & \text{ if } \lambda \geq 0, \\
0 & \text{ if } \lambda < 0 .
\end{cases}
\]
\end{prop}

\begin{proof}
If $\lambda < 0$, then $M(\lambda)$ is irreducible and infinite-dimensional so its maximal locally finite quotient is zero.

If $\lambda \geq 0$, then $M(-\lambda -2)$ is an irreducible submodule of $M(\lambda)$ whose quotient is the finite-dimensional irreducible representation $L(\lambda)$.
\end{proof}

Let $\mathfrak{h}$ be the Cartan subalgebra of diagonal matrices of $\sll_2$.  It is spanned by $\Lh$.  
For $\sll_2$, category $\mathcal{O}$ is the full
subcategory of $\sll_2$-modules consisting of objects $M$ such that:
\begin{itemize}
\item $\Lh$ acts diagonally on $M$, in other words, $M$ is a weight module;
\item $M$ is finitely generated as a $\mathcal{U}(\sll_2)$-module;
\item for all $v \in M$, the $\mathcal{U}(\mathfrak{b})$-submodule of $M$
  generated by $v$ is finite-dimensional, where $\mathfrak{b}$ is the subalgebra of upper-triangular matrices.
\end{itemize}
Category $\mathcal{O}$ decomposes into a direct sum of categories called blocks
\[
\mathcal{O} = \bigoplus_{\lambda} \mathcal{O}_{\lambda}
\]
where $\lambda \in \mathfrak{h}^* / S_2$ is the image of a weight under the action of the Weyl group, (the symmetric group $S_2$), on $\mathfrak{h}^*$, and $\mathcal{O}_{\lambda}$ consists of all modules with generalized central character corresponding to $\lambda$ under the Harish-Chandra isomorphism.

For generic $\lambda$, the block $\mathcal{O}_{\lambda}$ is semisimple.  If $\lambda$ is an integral dominant weight, that is $\lambda(\Lh) \in \mathbb{Z}_{\geq 0}$,
then there are non-trivial extensions between simple modules.
For an integral dominant weight $\lambda$, the simple objects of $\mathcal{O}_{\lambda}$ are $L(\lambda)$ and $L(-\lambda-2) $.

Since $\KR_N^{\mathfrak{sl}_2}(L;R)$ is an $\mathfrak{sl}_2$-module, we have that the spaces
$\Gamma \KR_N^{\mathfrak{sl}_2}(L;R)$ and 
$\mathcal{Z} \KR_N^{\mathfrak{sl}_2}(L;R)$ are link invariants. In the examples computed in Section \ref{sec:examples} we see that for the unknot and the Hopf link that these representations are finite-dimensional.  Since the homology $\KR_N^{\mathfrak{sl}_2}(L;R)$ is in general infinitely generated, it is not clear whether or not in general applying $\Gamma$ or $\mathcal{Z}$ produces something that is finite-dimensional.  

We further note that composing Theorem \ref{thm:functorFN} with the Zuckerman or Bernstein functor yields a functor whose target category is the category of locally-finite $\mathfrak{sl}_2$-modules.

The Zuckerman functor is only left exact and is often useful in Lie theory to consider its derived functors.  In fact, the Bernstein functor is related to the Zuckerman functor by taking certain cohomology of the derived functor.  We do not pursue that direction here, but it would be interesting to investigate if we obtain any new information from the higher derived functors which potentially leads to another grading on link homology.

\subsection{The Rasmussen invariant}
\label{sec:rasmussen}

In this section we assume the link $L$ is a knot and we restrict ourselves to the case $N=2$ so we set $R=\mathbb{Q}[E_1, E_2]$.
We follow the exposition of \cite{Khfrob} in defining the Rasmussen invariant \cite{Rassslice}.

Fixing a base point $p$ on $L$, we endow the 
homology $\KR_2^{\mathfrak{sl}_2}(L;R) $ with the action of the homology of the unknot
\[
A=\mathbb{Q}[E_1, E_2][x]/(x^2-E_1x^{} + E_2) \ 
\]
by placing an unknot near the base point of $L$ and then merging the unknot to the link by the base point.  The homology of $L$ then acquires an action of $\mathbb{Q}[x]$.

For a knot $L$, denote by 
$\KR_N^{\mathfrak{sl}_2, \mathrm{tor}}(L;R) $ the torsion elements of its homology with respect to the $\mathbb{Q}[x]$-action.
Let $\KR_N^{\mathfrak{sl}_2, \mathrm{free}}(L;R) $ be the free part with respect to this action.

The next few results justify using $\mathfrak{sl}_2$ in the notation above.
\begin{prop}
The space $\KR_2^{\mathfrak{sl}_2, \mathrm{tor}}(L;R) $ is an $\mathfrak{sl}_2$-subrepresentation.
\end{prop}

\begin{proof}
Suppose $\alpha$ is a torsion element with respect to the $\mathbb{Q}[x]$-action.
Then for some $m \in \mathbb{N}$, $x^m \alpha=0$.  We need to check that $g \alpha$ is torsion where $g \in \{\Le, \Lf, \Lh  \}$.

We compute
\[
0 = \Lf \cdot (x^m \alpha) = m x^{m+1} \alpha + x^m \Lf \cdot \alpha = x^m \Lf \cdot \alpha
\]
since $x^m \alpha = 0 $.  Thus $\Lf \cdot \alpha$ is also a torsion element.

Similarly one could compute $\Lh \cdot \alpha $ is torsion.

In order to prove that $\Le \cdot \alpha $ is torsion, we note 
\[
0 = \Le \cdot (x^{m+1} \alpha) = -(m+1) x^m \alpha + x^{m+1} \Le \cdot \alpha = x^{m+1} \Le \cdot \alpha.
\]
Thus $\Le \cdot \alpha$ is torsion as well.
\end{proof}
This immediately implies the following result.
\begin{cor} \label{cor:sl2quot}
The subspace  $\KR_2^{\mathfrak{sl}_2,\mathrm{free}}(L;R)$ is an $\mathfrak{sl}_2$-quotient representation of the entire homology.
\end{cor}

The next proposition is a reinterpretation by Khovanov of a fundamental fact first established by Lee \cite{Lee1} and developed by Rasmussen \cite{Rassslice}.

\begin{prop} \cite[Proposition 8]{Khfrob} \label{prop:rassdef}
The free $\mathbb{Q}[x]$-module $\KR_2^{\mathfrak{sl}_2, \mathrm{free}}(L;R) $ is isomorphic to $\mathbb{Q}[x] \{-1-s(L)\}$ where $s(L)$ is the Rasmussen invariant of $L$.
\end{prop}

We are now able to record the relationship between the Rasmussen invariant of a link and $\mathfrak{sl}_2$-structure on its equivariant Khovanov homology.

\begin{cor}
Let $L$  be a knot and let $\mu(L)$ be the highest weight of $\KR_2^{\mathfrak{sl}_2,\mathrm{free}}(L;R)$.  
Then $s(L)=(\mu(L)-1)$.
\end{cor}

\begin{proof}
This follows from Corollary \ref{cor:sl2quot} and Proposition \ref{prop:rassdef} since the weight of an element under the action of $\Lh \in \mathfrak{sl}_2$ corresponds to negative of its internal $q$-degree.
\end{proof} 
\subsection{\texorpdfstring{$p$-DG}{$p$-DG} structure}
\label{sec:pdgstructure}

\subsubsection{General background}
Let $H=\Fp[\dif]/(\dif^p)$ where the degree of $\dif$ is two.
If $M$ is an object in $H\gmod$, then its image in the
the stable category of $H$-modules 
$ H\underline{\gmod}$ is denoted by
$ / M$.

Khovanov proved the following important result about the stable category of $H$-modules.

\begin{prop} \cite[Proposition 5]{Hopforoots}
The Grothendieck ring of the category of stable $H$-modules may be identified with the cyclotomic ring:
\[
K_0(H_p\underline{\gmod}) \cong \mathbb{Z}[q]/(q^{2p-2}+q^{2p-4} +\cdots+1) .
\]
\end{prop}

\subsubsection{Link homology coming from \texorpdfstring{$\Lf$}{$\Lf$}}

In this subsection we assume the ground ring is $R=\mathbb{F}_p[E_1,\ldots,E_N]$, and $H_{\Lf}=\mathbb{F}_p[{\Lf}]/(\Lf^p)$.
We denote the resulting link homology equipped with the $H_{\Lf}$-action by 
$\KR_N^{H_{\Lf}}(\,\cdot\,;\mathbb{F}_p[E_1,\ldots,E_N],)$.

\begin{thm} \label{thm:pdgf}
In the stable category, $/ \KR_N^{H_{\Lf}}(\,\cdot\,;\mathbb{F}_p[E_1,\ldots,E_N],)$ is a link invariant.
\end{thm}

\begin{proof}
The fact that the link invariant comes equipped with the action of the Hopf algebra $H'_{\Lf}=\mathbb{F}_p[{\Lf}]$ is a direct consequence of the fact that this Hopf algebra is a subalgebra of $\mathfrak{sl}_2$.
Since we work over a field of characteristic $p$, one could check that the ideal generated by $({\Lf}^p)$ acts trivially, so the $H'_{\Lf}$-action descends to an action of 
$H_{\Lf}=\mathbb{F}_p[{\Lf}] / ({\Lf}^p)$.
\end{proof}

\begin{rmk}
It is not clear what the image of the invariant in Theorem \ref{thm:pdgf} is in the Grothendieck group.  A computation for the unknot for $N=2$ below shows that it is not simply the Jones polynomial evaluated at a $2p$th root of unity.

While we do not understand the decategorification of the link invariant, we do expect that the invariant could be extended on the categorical level to a colored invariant.  There are a few possibilities of how to do this.  One may try to apply techniques of Khovanov's categorification of the colored Jones polynomial \cite{Khcolored}.  Alternatively, one may try to construct categorical Jones--Wenzl projectors \cite{CautisClasp, CoKr, RoJW} compatible with the differential. 

An even more ambitious problem would be to assemble these colored homological invariants into a $3$-manifold invariant. The details will appear in subsequent works by the authors.
\end{rmk}

\begin{exa}
Suppose $N=2$ and let $U$ be the unknot.  Then
\[ 
/ \KR_2^{H_{\Lf}}(U;\mathbb{F}_p[E_1, E_2])
=\mathbb{F}_p[x]/(x^2-E_1x+E_2)
\]
where the action of ${\Lf}$ is twisted by
\[
{\Lf}(1)=-\frac{1}{2}E_1 .
\]
Note that $\mathbb{F}_p[E_1,E_2]$ is a $p$-DG submodule of $\mathbb{F}_p[x]/(x^2-E_1x+E_2)$.

If $p=3$, then ${\Lf}(1)=E_1$.  By \cite[Corollary 2.10]{EQ2}, $\mathbb{F}_p[E_1,E_2]$ is acyclic.
After quotienting $\mathbb{F}_p[x]/(x^2-E_1x+E_2)$ by this acyclic submodule, one calculates that ${\Lf}(x)=2 E_1x$.  Again by \cite[Corollary 2.10]{EQ2}, this quotient module is also acyclic.
Therefore, for $p=3$, the homology of the unknot is zero.  However, the Jones polynomial, (up to a power of $q$), is $1+q^2$.  At a $6$th root of unity, this is not zero.  

Thus the link invariant does not categorify the Jones polynomial at a root of unity.
\end{exa}

\subsubsection{Link homology coming from \texorpdfstring{$\Le$}{$\Le$}}
In this section we assume $N$ is a prime $p$, and the ground ring is $R=\mathbb{F}_p$, and let $H_{\Le}=\mathbb{F}_p[{\Le}] / ({\Le}^p)$.  We recall that this is the operator considered in \cite{Shum, Wang}.
We denote the resulting link homology by
$ \KR_N^{H_{\Le}}(\,\cdot\,;\mathbb{F}_p)$. We will only mention the following simple corollary, and leave further applications to a follow-up work.  

\begin{thm} \label{thm:pdge}
In the stable category, $/ \KR_p^{H_{\Le}}(\, \cdot\,;\mathbb{F}_p)$ is a link invariant whose image in the Grothendieck group is the $\mathfrak{sl}_p$-link invariant evaluated at a $2p$th root of unity.
\end{thm}

\begin{proof}
As in the proof of Theorem \ref{thm:pdgf},
this link homology comes equipped with an action of the Hopf algebra $H_{\Le}$.

The graded Euler characteristic of 
$\KR_p^{H_{\Le}}(\,\cdot\,;\mathbb{F}_p)$ is the $\mathfrak{sl}_N$-link invariant and so taking its slash homology specializes $q$ to be a $2p$th root of unity.
\end{proof}

\begin{rmk}
In Theorem \ref{thm:pdgf}, the coefficient ring is $\mathbb{F}_p[E_1,\ldots,E_N]$ which complicates statements about the graded Euler characteristic of the link homology.  
This is in contrast to the above theorem where the coefficient ring is the field $\mathbb{F}_p$ allowing us to deduce the conclusion of the graded Euler characteristic of the link homology.
\end{rmk}

\section{Examples}
\label{sec:examples}
We restrict ourselves to $N=2$ Khovanov--Rozansky homology in this section only for presentation purposes.  The general case is similar to what is explained below.  For a good exposition on computation of homology using foams in this case, see \cite{KruWed}.

\subsection{Unknot} \label{sec:unknot}
We begin with the $\mathfrak{sl}_2$ homology of the unknot.
As a vector space, it is isomorphic to
$A=q^{-1} \Bbbk[E_1,E_2,x] / (x^2-E_1x+E_2)$.
The $\Lf$-action is twisted by $-\frac{1}{2}E_1 $.
The $\Le$-action is not twisted at all.
As noted earlier, its graded dimension is  $\frac{q^{-1}}{(1-q^{2})^2}$. 
Let $A_{\lambda}$ denote the $\lambda$-weight space.
We will show that as an $\mathfrak{sl}_2$-module, this homology is isomorphic to 
\begin{equation} \label{unknotsl2act}
P(-3) \oplus M(-1) \oplus \bigoplus_{r=3}^{\infty} M(1-2r)
\end{equation}
where we recall that $M(\lambda)$ is a Verma module with highest weight $\lambda$ and $P(\lambda)$ is an indecomposable projective cover of $L(\lambda)$.
The $q$-grading could be read off from the weight space decomposition where an element in the weight space $M_{\mu}(\lambda)$ has $q$-degree $-\mu$.

The graded dimension of the homology implies that in degree  $2n-1$, (which corresponds to the weight space $1-2n$), is $n+1$.
Since the operator $\Le$ is a map of degree $2$ between weight spaces, we know that there is an element in the kernel of $\Le$, and thus a highest weight vector in every weight space.

Consider the submodule $S_1$ generated by $v_1=1$.
It is easy to see that $S_1 \cong M(1)$.

Now consider the submodule generated by $v_2=E_1-2x$.  This submodule is isomorphic to $M(-1)$.
Since there are no extensions between $M(1)$ and $M(-1)$, the submodule $S_2$ generated by $v_1,v_2$ is isomorphic to $M(1) \oplus M(-1)$.

Next consider the element $v_3=E_2$.  Since $\Le(E_2)=-E_1$, we see that the submodule generated by $v_3$ intersects $S_1$.  After quotienting by $S_1$,
$v_3$ generates a submodule isomorphic to $M(-3)$.  Thus for the submodule $\mathcal{U}(\mathfrak{sl}_2) \langle v_1, v_3 \rangle$, there is a short exact sequence
\[
0 \rightarrow M(1) \rightarrow \mathcal{U}(\mathfrak{sl}_2) \langle v_1, v_3 \rangle \rightarrow M(-3) \rightarrow 0 \ .
\]
Thus $\mathcal{U}(\mathfrak{sl}_2) \langle v_1, v_3 \rangle \cong P(-3)$.  Since there are no extensions between $P(-3)$ and $M(-1)$, the submodule $S_3$ generated by $v_1,v_2,v_3$ is isomorphic to $P(-3) \oplus M(-1)$.

Note that $A_{-5}$ is 4-dimensional and $A_{-3}$ is 3-dimensional.  It is not difficult to show that $\Le \colon A_{-5} \rightarrow A_{-3}$ has a 1-dimensional kernel.  Let $v_4$ denote an element in this kernel.  It generates a submodule isomorphic to $M(-5)$.  Since there are no non-trivial extensions between $M(-5)$ and $P(-3) \oplus M(-1)$, we get that the submodule $S_4$ generated by $v_1,v_2,v_3,v_4$ is isomorphic to
$P(-3) \oplus M(-1) \oplus M(-5)$.

Continuing in this way we get \eqref{unknotsl2act} for the homology of the unknot.

Note that all of the Verma modules appearing in \eqref{unknotsl2act} are simple and so the dual Zuckerman functor applied to them is zero.  Applying the dual Zuckerman functor to the short exact sequence \eqref{sesproj} shows that $\mathcal{Z} P(-3) =0$.  Thus for the unknot, the dual Zuckerman functor annihilates all homology.  A similar analysis gives that the Zuckerman functor also annihilates all homology.

\subsection{Hopf link}
We now compute the homology of the Hopf link $L$ with the associated $\mathfrak{sl}_2$-action.  We take $t_1=t_2=\frac{1}{2}$.
\begin{equation}
L = 
    { \NB{\tikz[font= \tiny,
  scale=0.6]{\begin{scope}
  \draw (3, 0) .. controls +(0, 0.2) and +(0, -0.2) ..  +(-1,1);
  \fill[white] (2.5, 0.5) circle (2mm);
  \draw (2, 0) .. controls +(0, 0.2) and +(0, -0.2) ..  +(1,1);
  %%%
   \draw (3, 1) .. controls +(0, .2) and +(0, -0.2) ..  +(-1,1);
  \fill[white] (2.5, 1.5) circle (2mm);
  \draw (2, 1) .. controls +(0, .2) and +(0, -0.2) ..  +(1,1); 
   \draw (2, 0) arc (180:0:-0.5) -- (1,2) arc (180:0:0.5) coordinate[pos = 0.25] (X) coordinate[pos = 0.75] (Y);
    \draw (3, 0) arc (-180:0:0.5) -- (4,2) arc (-180:0:-0.5) coordinate[pos = 0.25] (X) coordinate[pos = 0.75] (Y);
\end{scope}}} } \ .
\end{equation}

The complex is given by
\begin{equation}
  \NB{\tikz[xscale = 3, yscale = 3]{
       \node (i0) at (0, 0) { \NB{\tikz[font= \tiny,
  scale=0.6]{\input{\imagesfolder/pdg_hopf5}}} };
    \node (i1) at (-1, .5) { \NB{\tikz[font= \tiny,
  scale=0.6]{\begin{scope}
  \coordinate (bl) at (-0.5, -1);
  \coordinate (br) at ( 0.5, -1);
  \coordinate (bm) at (  0,-0.3);
  \coordinate (tl) at (-0.5,  1);
  \coordinate (tr) at ( 0.5,  1);
  \coordinate (tm) at (  0, 0.3);
  \draw[>-]  (bl) .. controls +( 0, 0.5) and +(0,0) .. (bm)
  node[right, pos = 0] {};
  \draw[>-]  (br) .. controls +( 0, 0.5) and +(0,0) .. (bm)
  node[left, pos = 0] {};
  \draw[<-]  (tl) .. controls +( 0, -0.5) and +(0,0) .. (tm)
  node[right, pos = 0] {} coordinate[pos = 0.25] (ga) ;
    \filldraw[draw= green!50!black, fill = white] (ga) circle (1mm)
  node[left, green!50!black] {$t_1$};

  \draw[<-]  (tr) .. controls +( 0, -0.5) and +(0,0) .. (tm)
  node[left, pos = 0] {} coordinate[pos = 0.25] (gb) ;
    \filldraw[draw= green!50!black, fill = white] (gb) circle (1mm)
  node[left, green!50!black] {$t_2$};
  \draw [double] (bm) -- (tm) node[left, pos = 0.5] {};
  \draw (.5, 1) arc (180:0:0.5) -- (1.5,-1) arc (180:0:-0.5);
  \draw (-.5, 1) arc (-180:0:-0.5) -- (-1.5,-1) arc (-180:0:0.5);
\end{scope}}} };
    \node (i2) at (-1, -.5) { \NB{\tikz[font= \tiny,
  scale=0.6]{}} };
      \node (i3) at (-2, 0) { \NB{\tikz[font= \tiny,
  scale=0.6]{\begin{scope}
  \coordinate (bl) at (-0.5, -1);
  \coordinate (br) at ( 0.5, -1);
  \coordinate (bm) at (  0,-0.3);
  \coordinate (tl) at (-0.5,  1);
  \coordinate (tr) at ( 0.5,  1);
  \coordinate (tm) at (  0, 0.3);
    \coordinate (brr) at ( 1.5, -1);
        \coordinate (trr) at ( 1.5, 1);
\draw[>-]  (bl) .. controls +( 0, 0.5) and +(0,0) .. (bm)
  node[below, pos = 0] {};
  \draw[>-]  (br) .. controls +( 0, 0.5) and +(0,0) .. (bm)
  node[below, pos = 0] {};
  \draw[<-]  (tl) .. controls +( 0, -0.5) and +(0,0) .. (tm)
  node[above, pos = 0] {} coordinate[pos = 0.25] (ga) coordinate[pos = 0] (gaa) ;
    \filldraw[draw= green!50!black, fill = white] (ga) circle (1mm)
  node[left, green!50!black] {$t_1$};
    %  \filldraw[draw= black!50!black, fill = black] (gaa) circle (1mm)
  %node[left, green!50!black] {};

  \draw[<-]  (tr) .. controls +( 0, -0.5) and +(0,0) .. (tm)
  node[above, pos = 0] {} coordinate[pos = 0.25] (gb) ;
    \filldraw[draw= green!50!black, fill = white] (gb) circle (1mm)
  node[left, green!50!black] {$t_2$};
  \draw [double] (bm) -- (tm) node[left, pos = 0.5] {};
% \draw [->-] (brr) -- (trr);
 %%%%%
   \coordinate (blU) at (-0.5, 1);
  \coordinate (brU) at ( 0.5, 1);
  \coordinate (bmU) at (  0,1.7);
  \coordinate (tlU) at (-0.5,  3);
  \coordinate (trU) at ( 0.5,  3);
  \coordinate (tmU) at (  0, 2.3);
    \coordinate (brrU) at ( 1.5, 1);
        \coordinate (trrU) at ( 1.5, 3);
\draw[]  (blU) .. controls +( 0, 0.5) and +(0,0) .. (bmU)
  node[left, pos = 0] {};
  \draw[]  (brU) .. controls +( 0, 0.5) and +(0,0) .. (bmU)
  node[right, pos = 0] {};
  \draw[<-]  (tlU) .. controls +( 0, -0.5) and +(0,0) .. (tmU)
  node[above, pos = 0] {} coordinate[pos = 0.25] (gaU) ;
    \filldraw[draw= green!50!black, fill = white] (gaU) circle (1mm)
  node[left, green!50!black] {${t}_1$};

  \draw[<-]  (trU) .. controls +( 0, -0.5) and +(0,0) .. (tmU)
  node[above, pos = 0] {} coordinate[pos = 0.25] (gbU) ;
    \filldraw[draw= green!50!black, fill = white] (gbU) circle (1mm)
  node[right, green!50!black] {${t}_2$};
  \draw [double] (bmU) -- (tmU) node[left, pos = 0.5] {};
% \draw [->-] (brr) -- (trrU)  node[below, pos = 0] {$1$};
  \draw (.5, -1) arc (-180:0:0.5) -- (1.5,3) arc (-180:0:-0.5) coordinate[pos = 0.25] (X) coordinate[pos = 0.75] (Y);

    \draw (-.5, -1) arc (180:0:-0.5) -- (-1.5,3) arc (180:0:0.5) coordinate[pos = 0.25] (XX) coordinate[pos = 0.75] (YY);

\end{scope}}} };
  \draw[->] (i3) -- (i1) node[pos=0.5, above] {$\mapH$};
    \draw[->] (i3) -- (i2) node[pos=0.5, above] {$-\mapH$};
      \draw[->] (i1) -- (i0) node[pos=0.5, above] {$\mapH$};
  \draw[->] (i2) -- (i0) node[pos=0.5, above] {$\mapH$};
  }} .
\end{equation}

Let $\alpha_1,\alpha_2$ be elements in the state space in the leftmost homological degree and let $\beta_1, \beta_2$ be elements in the state space in the rightmost homological degree.

\begin{equation}
\alpha_1 = 
    { \NB{\tikz[font= \tiny,
  scale=0.6]{\begin{scope}
  \coordinate (bl) at (-0.5, -1);
  \coordinate (br) at ( 0.5, -1);
  \coordinate (bm) at (  0,-0.3);
  \coordinate (tl) at (-0.5,  1);
  \coordinate (tr) at ( 0.5,  1);
  \coordinate (tm) at (  0, 0.3);
    \coordinate (brr) at ( 1.5, -1);
        \coordinate (trr) at ( 1.5, 1);
\draw[>-]  (bl) .. controls +( 0, 0.5) and +(0,0) .. (bm)
  node[below, pos = 0] {};
  \draw[>-]  (br) .. controls +( 0, 0.5) and +(0,0) .. (bm)
  node[below, pos = 0] {};
  \draw[<-]  (tl) .. controls +( 0, -0.5) and +(0,0) .. (tm)
  node[above, pos = 0] {} coordinate[pos = 0.25] (ga) ;
 %   \filldraw[draw= green!50!black, fill = white] (ga) circle (1mm)
%  node[left, green!50!black] {$t_1$};

  \draw[<-]  (tr) .. controls +( 0, -0.5) and +(0,0) .. (tm)
  node[above, pos = 0] {} coordinate[pos = 0.25] (gb) ;
%    \filldraw[draw= green!50!black, fill = white] (gb) circle (1mm)
 % node[left, green!50!black] {$t_2$};
  \draw [double] (bm) -- (tm) node[left, pos = 0.5] {};
% \draw [->-] (brr) -- (trr);
 %%%%%
   \coordinate (blU) at (-0.5, 1);
  \coordinate (brU) at ( 0.5, 1);
  \coordinate (bmU) at (  0,1.7);
  \coordinate (tlU) at (-0.5,  3);
  \coordinate (trU) at ( 0.5,  3);
  \coordinate (tmU) at (  0, 2.3);
    \coordinate (brrU) at ( 1.5, 1);
        \coordinate (trrU) at ( 1.5, 3);
\draw[]  (blU) .. controls +( 0, 0.5) and +(0,0) .. (bmU)
  node[left, pos = 0] {};
  \draw[]  (brU) .. controls +( 0, 0.5) and +(0,0) .. (bmU)
  node[right, pos = 0] {};
  \draw[<-]  (tlU) .. controls +( 0, -0.5) and +(0,0) .. (tmU)
  node[above, pos = 0] {} coordinate[pos = 0.25] (gaU) ;
   % \filldraw[draw= green!50!black, fill = white] (gaU) circle (1mm)
  %node[left, green!50!black] {$-\bar{t}_1$};

  \draw[<-]  (trU) .. controls +( 0, -0.5) and +(0,0) .. (tmU)
  node[above, pos = 0] {} coordinate[pos = 0.25] (gbU) ;
 %   \filldraw[draw= green!50!black, fill = white] (gbU) circle (1mm)
 % node[right, green!50!black] {$-\bar{t}_2$};
  \draw [double] (bmU) -- (tmU) node[left, pos = 0.5] {};
% \draw [->-] (brr) -- (trrU)  node[below, pos = 0] {$1$};
  \draw (.5, -1) arc (-180:0:0.5) -- (1.5,3) arc (-180:0:-0.5) coordinate[pos = 0.25] (X) coordinate[pos = 0.75] (Y);

    \draw (-.5, -1) arc (180:0:-0.5) -- (-1.5,3) arc (180:0:0.5) coordinate[pos = 0.25] (XX) coordinate[pos = 0.75] (YY);
    
       \filldraw[draw= black!50!black, fill = black] (XX) circle (1mm)
  node[right, green!50!black] {};
\end{scope}}} }
  - 
      { \NB{\tikz[font= \tiny,
  scale=0.6]{\begin{scope}
  \coordinate (bl) at (-0.5, -1);
  \coordinate (br) at ( 0.5, -1);
  \coordinate (bm) at (  0,-0.3);
  \coordinate (tl) at (-0.5,  1);
  \coordinate (tr) at ( 0.5,  1);
  \coordinate (tm) at (  0, 0.3);
    \coordinate (brr) at ( 1.5, -1);
        \coordinate (trr) at ( 1.5, 1);
\draw[>-]  (bl) .. controls +( 0, 0.5) and +(0,0) .. (bm)
  node[below, pos = 0] {};
  \draw[>-]  (br) .. controls +( 0, 0.5) and +(0,0) .. (bm)
  node[below, pos = 0] {};
  \draw[<-]  (tl) .. controls +( 0, -0.5) and +(0,0) .. (tm)
  node[above, pos = 0] {} coordinate[pos = 0.25] (ga) coordinate[pos = 0] (gaa) ;
  %  \filldraw[draw= green!50!black, fill = white] (ga) circle (1mm)
 % node[left, green!50!black] {$t_1$};
      \filldraw[draw= black!50!black, fill = black] (gaa) circle (1mm)
  node[left, green!50!black] {};

  \draw[<-]  (tr) .. controls +( 0, -0.5) and +(0,0) .. (tm)
  node[above, pos = 0] {} coordinate[pos = 0.25] (gb) ;
%    \filldraw[draw= green!50!black, fill = white] (gb) circle (1mm)
 % node[left, green!50!black] {$t_2$};
  \draw [double] (bm) -- (tm) node[left, pos = 0.5] {};
% \draw [->-] (brr) -- (trr);
 %%%%%
   \coordinate (blU) at (-0.5, 1);
  \coordinate (brU) at ( 0.5, 1);
  \coordinate (bmU) at (  0,1.7);
  \coordinate (tlU) at (-0.5,  3);
  \coordinate (trU) at ( 0.5,  3);
  \coordinate (tmU) at (  0, 2.3);
    \coordinate (brrU) at ( 1.5, 1);
        \coordinate (trrU) at ( 1.5, 3);
\draw[]  (blU) .. controls +( 0, 0.5) and +(0,0) .. (bmU)
  node[left, pos = 0] {};
  \draw[]  (brU) .. controls +( 0, 0.5) and +(0,0) .. (bmU)
  node[right, pos = 0] {};
  \draw[<-]  (tlU) .. controls +( 0, -0.5) and +(0,0) .. (tmU)
  node[above, pos = 0] {} coordinate[pos = 0.25] (gaU) ;
 %   \filldraw[draw= green!50!black, fill = white] (gaU) circle (1mm)
 % node[left, green!50!black] {$-\bar{t}_1$};

  \draw[<-]  (trU) .. controls +( 0, -0.5) and +(0,0) .. (tmU)
  node[above, pos = 0] {} coordinate[pos = 0.25] (gbU) ;
 %   \filldraw[draw= green!50!black, fill = white] (gbU) circle (1mm)
 % node[right, green!50!black] {$-\bar{t}_2$};
  \draw [double] (bmU) -- (tmU) node[left, pos = 0.5] {};
% \draw [->-] (brr) -- (trrU)  node[below, pos = 0] {$1$};
  \draw (.5, -1) arc (-180:0:0.5) -- (1.5,3) arc (-180:0:-0.5) coordinate[pos = 0.25] (X) coordinate[pos = 0.75] (Y);

    \draw (-.5, -1) arc (180:0:-0.5) -- (-1.5,3) arc (180:0:0.5) coordinate[pos = 0.25] (XX) coordinate[pos = 0.75] (YY);

\end{scope}}} }
\end{equation}

\begin{equation}
\alpha_2 = 
    { \NB{\tikz[font= \tiny,
  scale=0.6]{\begin{scope}
  \coordinate (bl) at (-0.5, -1);
  \coordinate (br) at ( 0.5, -1);
  \coordinate (bm) at (  0,-0.3);
  \coordinate (tl) at (-0.5,  1);
  \coordinate (tr) at ( 0.5,  1);
  \coordinate (tm) at (  0, 0.3);
    \coordinate (brr) at ( 1.5, -1);
        \coordinate (trr) at ( 1.5, 1);
\draw[>-]  (bl) .. controls +( 0, 0.5) and +(0,0) .. (bm)
  node[below, pos = 0] {};
  \draw[>-]  (br) .. controls +( 0, 0.5) and +(0,0) .. (bm)
  node[below, pos = 0] {};
  \draw[<-]  (tl) .. controls +( 0, -0.5) and +(0,0) .. (tm)
  node[above, pos = 0] {} coordinate[pos = 0.25] (ga) coordinate[pos = 0] (gaa) ;
 %   \filldraw[draw= green!50!black, fill = white] (ga) circle (1mm)
%  node[left, green!50!black] {$t_1$};
      \filldraw[draw= black!50!black, fill = black] (gaa) circle (1mm)
  node[left, green!50!black] {};

  \draw[<-]  (tr) .. controls +( 0, -0.5) and +(0,0) .. (tm)
  node[above, pos = 0] {} coordinate[pos = 0.25] (gb) ;
 %   \filldraw[draw= green!50!black, fill = white] (gb) circle (1mm)
%  node[left, green!50!black] {$t_2$};
  \draw [double] (bm) -- (tm) node[left, pos = 0.5] {};
% \draw [->-] (brr) -- (trr);
 %%%%%
   \coordinate (blU) at (-0.5, 1);
  \coordinate (brU) at ( 0.5, 1);
  \coordinate (bmU) at (  0,1.7);
  \coordinate (tlU) at (-0.5,  3);
  \coordinate (trU) at ( 0.5,  3);
  \coordinate (tmU) at (  0, 2.3);
    \coordinate (brrU) at ( 1.5, 1);
        \coordinate (trrU) at ( 1.5, 3);
\draw[]  (blU) .. controls +( 0, 0.5) and +(0,0) .. (bmU)
  node[left, pos = 0] {};
  \draw[]  (brU) .. controls +( 0, 0.5) and +(0,0) .. (bmU)
  node[right, pos = 0] {};
  \draw[<-]  (tlU) .. controls +( 0, -0.5) and +(0,0) .. (tmU)
  node[above, pos = 0] {} coordinate[pos = 0.25] (gaU) ;
%    \filldraw[draw= green!50!black, fill = white] (gaU) circle (1mm)
 % node[left, green!50!black] {$-\bar{t}_1$};

  \draw[<-]  (trU) .. controls +( 0, -0.5) and +(0,0) .. (tmU)
  node[above, pos = 0] {} coordinate[pos = 0.25] (gbU) ;
 %   \filldraw[draw= green!50!black, fill = white] (gbU) circle (1mm)
%  node[right, green!50!black] {$-\bar{t}_2$};
  \draw [double] (bmU) -- (tmU) node[left, pos = 0.5] {};
% \draw [->-] (brr) -- (trrU)  node[below, pos = 0] {$1$};
  \draw (.5, -1) arc (-180:0:0.5) -- (1.5,3) arc (-180:0:-0.5) coordinate[pos = 0.25] (X) coordinate[pos = 0.75] (Y);

    \draw (-.5, -1) arc (180:0:-0.5) -- (-1.5,3) arc (180:0:0.5) coordinate[pos = 0.25] (XX) coordinate[pos = 0.75] (YY);
    
    \filldraw[draw= black!50!black, fill = black] (XX) circle (1mm)
  node[right, green!50!black] {};    

\end{scope}}} }
  - 
   E_1   { \NB{\tikz[font= \tiny,
  scale=0.6]{}} }
  + E_2
  { \NB{\tikz[font= \tiny,
  scale=0.6]{\begin{scope}
  \coordinate (bl) at (-0.5, -1);
  \coordinate (br) at ( 0.5, -1);
  \coordinate (bm) at (  0,-0.3);
  \coordinate (tl) at (-0.5,  1);
  \coordinate (tr) at ( 0.5,  1);
  \coordinate (tm) at (  0, 0.3);
    \coordinate (brr) at ( 1.5, -1);
        \coordinate (trr) at ( 1.5, 1);
\draw[>-]  (bl) .. controls +( 0, 0.5) and +(0,0) .. (bm)
  node[below, pos = 0] {};
  \draw[>-]  (br) .. controls +( 0, 0.5) and +(0,0) .. (bm)
  node[below, pos = 0] {};
  \draw[<-]  (tl) .. controls +( 0, -0.5) and +(0,0) .. (tm)
  node[above, pos = 0] {} coordinate[pos = 0.25] (ga) coordinate[pos = 0] (gaa) ;
  %  \filldraw[draw= green!50!black, fill = white] (ga) circle (1mm)
 % node[left, green!50!black] {$t_1$};
    %  \filldraw[draw= black!50!black, fill = black] (gaa) circle (1mm)
  %node[left, green!50!black] {};

  \draw[<-]  (tr) .. controls +( 0, -0.5) and +(0,0) .. (tm)
  node[above, pos = 0] {} coordinate[pos = 0.25] (gb) ;
 %   \filldraw[draw= green!50!black, fill = white] (gb) circle (1mm)
 % node[left, green!50!black] {$t_2$};
  \draw [double] (bm) -- (tm) node[left, pos = 0.5] {};
% \draw [->-] (brr) -- (trr);
 %%%%%
   \coordinate (blU) at (-0.5, 1);
  \coordinate (brU) at ( 0.5, 1);
  \coordinate (bmU) at (  0,1.7);
  \coordinate (tlU) at (-0.5,  3);
  \coordinate (trU) at ( 0.5,  3);
  \coordinate (tmU) at (  0, 2.3);
    \coordinate (brrU) at ( 1.5, 1);
        \coordinate (trrU) at ( 1.5, 3);
\draw[]  (blU) .. controls +( 0, 0.5) and +(0,0) .. (bmU)
  node[left, pos = 0] {};
  \draw[]  (brU) .. controls +( 0, 0.5) and +(0,0) .. (bmU)
  node[right, pos = 0] {};
  \draw[<-]  (tlU) .. controls +( 0, -0.5) and +(0,0) .. (tmU)
  node[above, pos = 0] {} coordinate[pos = 0.25] (gaU) ;
  %  \filldraw[draw= green!50!black, fill = white] (gaU) circle (1mm)
 % node[left, green!50!black] {$-\bar{t}_1$};

  \draw[<-]  (trU) .. controls +( 0, -0.5) and +(0,0) .. (tmU)
  node[above, pos = 0] {} coordinate[pos = 0.25] (gbU) ;
  %  \filldraw[draw= green!50!black, fill = white] (gbU) circle (1mm)
%  node[right, green!50!black] {$-\bar{t}_2$};
  \draw [double] (bmU) -- (tmU) node[left, pos = 0.5] {};
% \draw [->-] (brr) -- (trrU)  node[below, pos = 0] {$1$};
  \draw (.5, -1) arc (-180:0:0.5) -- (1.5,3) arc (-180:0:-0.5) coordinate[pos = 0.25] (X) coordinate[pos = 0.75] (Y);

    \draw (-.5, -1) arc (180:0:-0.5) -- (-1.5,3) arc (180:0:0.5) coordinate[pos = 0.25] (XX) coordinate[pos = 0.75] (YY);

\end{scope}}} }
\end{equation}

\begin{equation}
\beta_1 = 
    { \NB{\tikz[font= \tiny,
  scale=0.6]{\input{\imagesfolder/pdg_hopf5}}} }
\end{equation}

\begin{equation}
\beta_2 = 
    { \NB{\tikz[font= \tiny,
  scale=0.6]{\begin{scope}
  \coordinate (bl) at (-0.5, -1);
  \coordinate (br) at ( 0.5, -1);
  \coordinate (tl) at (-0.5,  1);
  \coordinate (tr) at ( 0.5,  1);
    \coordinate (ml) at (-0.5,  -.8);
        \coordinate (Ml) at (-0.5,  .8);
 \coordinate (mr) at (0.5,  -.6);
\coordinate (Mr) at (0.5,  .6);

 %\draw[>->] (bl) -- (tl) node[pos = 0, below] {$2$} node[pos = 1,
  %above] {$1$} coordinate[pos = 0.4] (ml);
   \draw[->] (bl) -- (tl) node[pos = 0, below] {} node[pos = 1,
  above] {} coordinate[pos = 0.5] (XX);
  %\draw[>->] (br) -- (tr) node[pos = 0, below] {$1$} node[pos = 1, above] {$2$} coordinate[pos = 0.6] (mr);
  
    \draw[->] (br) -- (tr) node[pos = 0, below] {} node[pos = 1, above] {} coordinate[pos = 0.5] (YY);
  
  %\draw[->-] (ml) -- (mr) node [pos= 0.5, above] {$1$};
  %  \draw[->-] (Ml) -- (Mr) node [pos= 0.5, above] {$1$};
  
    \draw (.5, -1) arc (-180:0:0.5) -- (1.5,1) arc (-180:0:-0.5) coordinate[pos = 0.25] (X) coordinate[pos = 0.75] (Y);
    
       \draw (-.5, -1) arc (180:0:-0.5) -- (-1.5,1) arc (180:0:0.5) ;
%%%%
%coordinate[pos = 0.75] (YY);
    
      \filldraw[draw= black!50!black, fill = black] (XX) circle (1mm)
  node[right, green!50!black] {};
  
%     \filldraw[draw= black!50!black, fill = black] (YY) circle (1mm)
 % node[right, green!50!black] {};

\end{scope}}} }
  +
     { \NB{\tikz[font= \tiny,
  scale=0.6]{\begin{scope}
  \coordinate (bl) at (-0.5, -1);
  \coordinate (br) at ( 0.5, -1);
  \coordinate (tl) at (-0.5,  1);
  \coordinate (tr) at ( 0.5,  1);
    \coordinate (ml) at (-0.5,  -.8);
        \coordinate (Ml) at (-0.5,  .8);
 \coordinate (mr) at (0.5,  -.6);
\coordinate (Mr) at (0.5,  .6);

 %\draw[>->] (bl) -- (tl) node[pos = 0, below] {$2$} node[pos = 1,
  %above] {$1$} coordinate[pos = 0.4] (ml);
   \draw[->] (bl) -- (tl) node[pos = 0, below] {} node[pos = 1,
  above] {} coordinate[pos = 0.5] (XX);
  %\draw[>->] (br) -- (tr) node[pos = 0, below] {$1$} node[pos = 1, above] {$2$} coordinate[pos = 0.6] (mr);
  
    \draw[->] (br) -- (tr) node[pos = 0, below] {} node[pos = 1, above] {} coordinate[pos = 0.5] (YY);
  
  %\draw[->-] (ml) -- (mr) node [pos= 0.5, above] {$1$};
  %  \draw[->-] (Ml) -- (Mr) node [pos= 0.5, above] {$1$};
  
    \draw (.5, -1) arc (-180:0:0.5) -- (1.5,1) arc (-180:0:-0.5) coordinate[pos = 0.25] (X) coordinate[pos = 0.75] (Y);
    
       \draw (-.5, -1) arc (180:0:-0.5) -- (-1.5,1) arc (180:0:0.5) ;
%%%%
%coordinate[pos = 0.75] (YY);
    
  %    \filldraw[draw= black!50!black, fill = black] (XX) circle (1mm)
%  node[right, green!50!black] {};
  
    \filldraw[draw= black!50!black, fill = black] (YY) circle (1mm)
  node[right, green!50!black] {};

\end{scope}}} }
\end{equation}

As a vector space, the homology of the Hopf link is
\[
\Bbbk[E_1,E_2] \langle \alpha_1, \alpha_2 \rangle
\oplus
\Bbbk[E_1,E_2] \langle \beta_1, \beta_2 \rangle .
\]
On generating elements $\alpha_1, \alpha_2, \beta_1, \beta_2$, the $\mathfrak{sl}_2$-action is given by 
\[
\Lf \cdot \alpha_1=0, \hspace{.1in}
\Lf \cdot \alpha_2= E_1 \alpha_2+E_2 \alpha_1, \hspace{.1in}
\Lf \cdot \beta_1=-E_1 \beta_1-\beta_2, \hspace{.1in}
\Lf \cdot \beta_2=2 E_2 \beta_1-2E_1 \beta_2,
\]
\[
\Le \cdot \alpha_1=0, \quad
\Le \cdot \alpha_2=\alpha_1, \quad
\Le \cdot \beta_1=0, \quad
\Le \cdot \beta_2=-2 \beta_1,
\]
\[
\Lh \cdot \alpha_1=0, \quad
\Lh \cdot \alpha_2=-2\alpha_2, \quad
\Lh \cdot \beta_1=4\beta_1, \quad
\Lh \cdot \beta_2=2\beta_2.
\]
For the action of $\Lf$, note that there was some twisting in addition
to the twisting coming from the complex and the natural action on
dots.  The foam representing $\alpha_1$ is
comprised of a digon cup, a zip, as well as two cups.  Thus the action of $\Lf$ gets twisted using \eqref{eq:e-act-dig-cup}, \eqref{eq:e-act-zip}, and \eqref{eq:e-act-cup}.

Similar to the earlier analysis for the homology of the unknot, in order to determine the $\mathfrak{sl}_2$-structure for the homology of the Hopf link, one looks for highest weight vectors and also uses sequence \eqref{sesdualverma}. Details will appear in \cite{Roz}. 
As a graded 
$\Bbbk[E_1,E_2]\# \mathfrak{sl}_2$-module, 
 the homology of the Hopf link is isomorphic to 
\begin{equation} \label{hopsl2act}
q^{-4}t^2\left(\dfrac{\Bbbk[E_1,E_2][x]}{(x^2-E_1x+E_2)}\right)^{-E_1-x}\oplus  \left(\dfrac{\Bbbk[E_1,E_2][x]}{(x^2-E_1x+E_2)}\right)^{}
\end{equation}
When forgetting about the $q$-degrees, the $\mathfrak{sl}_2$-module one obtains is
\begin{equation} \label{hopsl2act2}
 t^0 (M^*(0) \oplus \bigoplus_{r=1}^{\infty} M(-2r)) \oplus
t^2(M^*(4) \oplus M^*(2) \oplus P(-2) \oplus  \bigoplus_{r=0}^{\infty} M(-8-2r)) .
\end{equation}

Note that $\oplus_{r=1}^{\infty} M(-2r)$ and $\oplus_{r=0}^{\infty} M(-8-2r)$ are direct sum of irreducible Verma modules and the Zuckerman and dual Zuckerman functors annihilate these objects.  The dual Zuckerman functor also sends $M^*(0)$, $M^*(2), M^*(4)$ to zero.

However, the Zuckerman functor $\Gamma$ sends $M^*(0)$, $M^*(2)$, and $M^*(4)$ to $L(0)$, $L(2)$, and $L(4)$ respectively.  Thus the Zuckerman functor applied to the homology is isomorphic to $ t^0 L(0) \oplus t^2 L(4) \oplus  t^2 L(2)$.

\bibliographystyle{alphaurl}
\bibliography{biblio}

\end{document}